\documentclass{amsart}

\usepackage{amsfonts,mathrsfs}
\usepackage{amsmath,amssymb,amsthm, amscd}
\usepackage{tikz}
\usetikzlibrary{arrows,calc,matrix,topaths,positioning,scopes,shapes,decorations,decorations.markings} 
\usepackage{dsfont}
\usepackage{epsfig}
\usepackage{tikz-cd}
\usepackage{hyperref}
\usepackage{fullpage}
\usepackage{adjustbox}
\usepackage[foot]{amsaddr}

\makeatletter
\renewcommand{\email}[2][]{%
  \ifx\emails\@empty\relax\else{\g@addto@macro\emails{,\space}}\fi%
  \@ifnotempty{#1}{\g@addto@macro\emails{\textrm{(#1)}\space}}%
  \g@addto@macro\emails{#2}%
}
\makeatother

\def\restriction#1#2{\mathchoice
              {\setbox1\hbox{${\displaystyle #1}_{\scriptstyle #2}$}
              \restrictionaux{#1}{#2}}
              {\setbox1\hbox{${\textstyle #1}_{\scriptstyle #2}$}
              \restrictionaux{#1}{#2}}
              {\setbox1\hbox{${\scriptstyle #1}_{\scriptscriptstyle #2}$}
              \restrictionaux{#1}{#2}}
              {\setbox1\hbox{${\scriptscriptstyle #1}_{\scriptscriptstyle #2}$}
              \restrictionaux{#1}{#2}}}
\def\restrictionaux#1#2{{#1\,\smash{\vrule height .8\ht1 depth .85\dp1}}_{\,#2}}

\newcommand{\quotient}[2]{{\raisebox{.2em}{$#1$}\left/\raisebox{-.2em}{$#2$}\right.}}

\newcommand{\sslash}{\mathbin{/\mkern-6mu/}}
\newcommand{\Hom}{\operatorname{Hom}}
\newcommand{\tr}{\operatorname{tr}}
\newcommand{\Tr}{\operatorname{Tr}}
\newcommand{\qtr}{\operatorname{qtr}}
\newcommand{\ptr}{\operatorname{ptr}}
\newcommand{\SL}{\operatorname{SL}}

\newcommand{\id}{id}
\newcommand{\Span}{\operatorname{Span}}
\newcommand{\End}{\operatorname{End}}

\newcommand{\PGL}{\operatorname{PGL}}

\newcommand{\Specm}{\operatorname{MaxSpec}}

\newcommand{\Mod}{\operatorname{Mod}}

\newcommand{\Mat}{\operatorname{Mat}}

\newcommand{\Aut}{\operatorname{Aut}}

\newcommand{\Arc}{\operatorname{Arc}}

\newcommand{\Uq}{D_qB}
\newcommand{\Ad}{Ad}

\newcommand{\Arcs}{\operatorname{Arcs}}
\newcommand{\Alg}{\operatorname{Alg}}
\newcommand{\Proj}{\operatorname{Proj}}
\newcommand{\Map}{\operatorname{Map}}
\newcommand{\PVect}{\operatorname{PVect}}
\newcommand{\Irrep}{\operatorname{Irrep}}

\newcommand{\heightexch}[3]{
	\begin{tikzpicture}[baseline=-0.4ex,scale=0.5, >=stealth]
	\draw [fill=gray!60,gray!45] (-.7,-.75)  rectangle (.4,.75)   ;
	\draw[#1] (0.4,-0.75) to (.4,.75);
	\draw[line width=1.2] (0.4,-0.3) to (-.7,-.3);
	\draw[line width=1.2] (0.4,0.3) to (-.7,.3);
	\draw (0.65,0.3) node {\scriptsize{$#2$}}; 
	\draw (0.65,-0.3) node {\scriptsize{$#3$}}; 
	\end{tikzpicture}
}
\newcommand{\heightcurve}{
\begin{tikzpicture}[baseline=-0.4ex,scale=0.5]
\draw [fill=gray!20,gray!45] (-.7,-.75)  rectangle (.4,.75)   ;
\draw[-] (0.4,-0.75) to (.4,.75);
\draw[line width=1.2] (-.7,-0.3) to (-.4,-.3);
\draw[line width=1.2] (-.7,0.3) to (-.4,.3);
\draw[line width=1.15] (-.4,0) ++(-90:.3) arc (-90:90:.3);
\end{tikzpicture}
}

\author{Julien Korinman}
\address{ Institut Montpelli\'erain Alexander Grothendieck - UMR 5149 Universit\'e de Montpellier. Place Eug\'ene Bataillon, 34090 Montpellier France}
\email{julien.korinman@gmail.com}

\subjclass{$57$R$56$, $57$N$10$, $57$M$25$.}


\begin{document}

\theoremstyle{plain}
\newtheorem{theorem}{Theorem}[section]
\newtheorem{main_theorem}[theorem]{Main Theorem}
\newtheorem{proposition}[theorem]{Proposition}
\newtheorem{corollary}[theorem]{Corollary}
\newtheorem{corollaire}[theorem]{Corollaire}
\newtheorem{lemma}[theorem]{Lemma}
\theoremstyle{definition}
\newtheorem{notations}[theorem]{Notations}
\newtheorem*{notations*}{Notations}
\newtheorem{convention}[theorem]{Convention}
\newtheorem{definition}[theorem]{Definition}
\newtheorem{problem}[theorem]{Problem}
\newtheorem{question}[theorem]{Question}
\newtheorem{Theorem-Definition}[theorem]{Theorem-Definition}
\theoremstyle{remark}
\newtheorem{remark}[theorem]{Remark}
\newtheorem{conjecture}[theorem]{Conjecture}
\newtheorem{example}[theorem]{Example}


\title[Q.  holonomic link invariants from S.Skein]{Quantum holonomic link invariants  derived  from stated skein algebras}
%
%
%

\date{}
\maketitle


\begin{abstract} 
We define invariants for a framed link equipped with a $\SL_2$ local system in its complement and additional combinatorial data based on the theory of representations of stated skein algebras at roots of unity of punctured bigons and the geometric interpretation of their centers. The gauge invariance of the link invariant is derived from De Concini-Kac quantum coadjoint action lifted at the level of stated skein algebras.  A key feature is the fact that the Drinfeld double of the quantum Borel algebra admits a natural interpretation as the reduced stated skein algebra of a once-punctured bigon from which we deduce a relation between our link invariants and quantum group constructions of  Blanchet-Geer-Patureau Mirand-Reshetikhin. Using Bonahon-Wong's quantum trace, we also relate our construction to quantum hyperbolic geometry, hence to Kashaev and  Baseilhac-Benedetti constructions. We deduce from this relation explicit formulas for the R-matrices,   which permit to compute the link invariants explicitly.  In particular, we derive an alternative conceptual proof of the Murakami-Murakami relation between the Kashaev invariant and the colored Jones polynomials.  
\end{abstract}

\tableofcontents


\section{Introduction}

\paragraph{\textbf{Background}}

Jones discovery of a new polynomial link invariant \cite{Jones85} and its interpretation by Witten in the context of Topological Quantum Field Theory \cite{Wi2}, initiated fruitful interactions between Von Neumann algebras, quantum groups and low dimensional topology. Though initially stated in the context of sub-factors, the construction of Jones polynomials was soon reformulated in the context of small representations of the quantum group $U_q\mathfrak{sl}_2$ \cite{Turaev_YB}. Drinfel'd $R$-matrix \cite{DrinfeldqRMatrix} permits to associate to any small representation a solution of the Yang-Baxter equation which, in turn, induces representations of the braid groups. Using an appropriate trace, these representations induce invariants of framed links (see \cite{OhtsukiBook, ChariPressley, Kassel} for surveys). In \cite{GeerPatureauTuraev_Trace, GeerPatureau_LinksInv, CGPInvariants}, Costantino,  Geer,  Patureau-Mirand and Turaev generalized this construction by defining framed link invariants from some non-small, projective indecomposable highest and lowest modules of  $U_q\mathfrak{sl}_2$ by modifying the trace and introducing a refined version of $U_q\mathfrak{sl}_2$, namely the unrolled quantum group. These invariants contain the Akutsu-Deguchi-Ohtsuki invariants of \cite{ADO_inv}. Concerning the family of $U_q\mathfrak{sl}_2$- indecomposable modules which are not highest and lowest weight, namely the cyclic and semi-cyclic modules, an approach to define link invariants was started by Kashaev and Reshetikhin in \cite{KashaevReshetikhin04,KashaevReshetikhin05} and continued in \cite{GeerPatureau_GLinks} using semi-cyclic modules and further developed in  \cite{BGPR_Biquandle} to include cyclic modules as well.

\vspace{2mm}
\par In parallel to these quantum groups constructions, Kashaev initiated in \cite{Kashaev_6jSymbol} an original way to produce solutions of the Yang-Baxter equation from the cyclic quantum dilogarithms and the Pentagon identity they satisfy \cite{BazhanovBaxter_StarTriangle, FadeevKashaevQDilog, BazhanovReshetikhin_QDilog, KashaevVolkov}. This leads to a knot invariant \cite{KashaevLinkInv} which was conjectured to be related to the hyperbolic volume of the exterior of the knot \cite{Kashaev97}, when the latter is hyperbolic (see also \cite{Zagier_QModForms} for a generalization). Murakami and Murakami proved in \cite{MM01} that Kashaev's $R$-matrix is gauge equivalent to the Drinfel'd $R$-matrix giving rise to the $N$-th Jones polynomial at $N$-th root of unity, hence that these two framed link invariants are related. Kashaev's techniques were refined by Baseilhac and Benedetti to produce new families of link and $3$-manifold invariants \cite{BaseilhacBenedettiInvariant, BaseilhacBenedetti05, BaseilhacBenedettiLinkInv, BaseilhacBenedetti15, BaseilhacBenedetti_NonAmbiguousStructures}. 

\vspace{2mm}
\par Kashaev \cite{Kashaev98} and Chekhov and Fock \cite{ChekhovFock} independently defined a family of algebras, the quantum Teichm\"uller spaces, associated to a triangulated punctured surface. The quantum dilogarithms and the Pentagon identity naturally appear in quantum Teichm\"uller theory by means of change of coordinates. Bonahon and Wong defined in \cite{BonahonWongqTrace} a refinement of the quantum Teichm\"uller spaces, namely the balanced Chekhov-Fock algebras, and defined an algebra morphism, the quantum trace, between the Kauffman-bracket skein algebra of a closed punctured surface and the balanced Chekhov-Fock algebra. A marked surface $\mathbf{\Sigma}=(\Sigma, \mathcal{A})$ is a compact oriented surface $\Sigma$ with a finite collection $\mathcal{A}$ of arcs in its boundary. 
Generalizations $\mathcal{S}_{A}(\mathbf{\Sigma})$ and $\overline{\mathcal{S}}_A(\mathbf{\Sigma})$ of the Kauffman-bracket skein algebras (and of the quantum trace) for marked surfaces $\mathbf{\Sigma}$  were defined by L\^e in \cite{LeStatedSkein}, following \cite{BonahonWongqTrace}, under the names \textbf{stated skein algebras} and \textbf{reduced stated skein algebras} respectively. The representation theory of these algebras, when the deforming parameter $A$ is a root of unity of odd order, were studied in \cite{KojuQuesneyClassicalShadows, KojuAzumayaSkein, KojuKaruo_RepRSSkein}.

\vspace{2mm}
\paragraph{\textbf{Main results}}
The purpose of this paper is to construct link invariants derived from the representation theory of reduced stated skein algebras and to relate them to both quantum group constructions \cite{GeerPatureau_LinksInv, BGPR_Biquandle} and quantum hyperbolic geometric constructions \cite{KashaevLinkInv, BaseilhacBenedetti05}. Let $\zeta$ be root of unity of odd order $N\geq 3$. The $n$-th punctured bigon $\mathbb{D}_n= \adjustbox{valign=c}{\includegraphics[width=0.7cm]{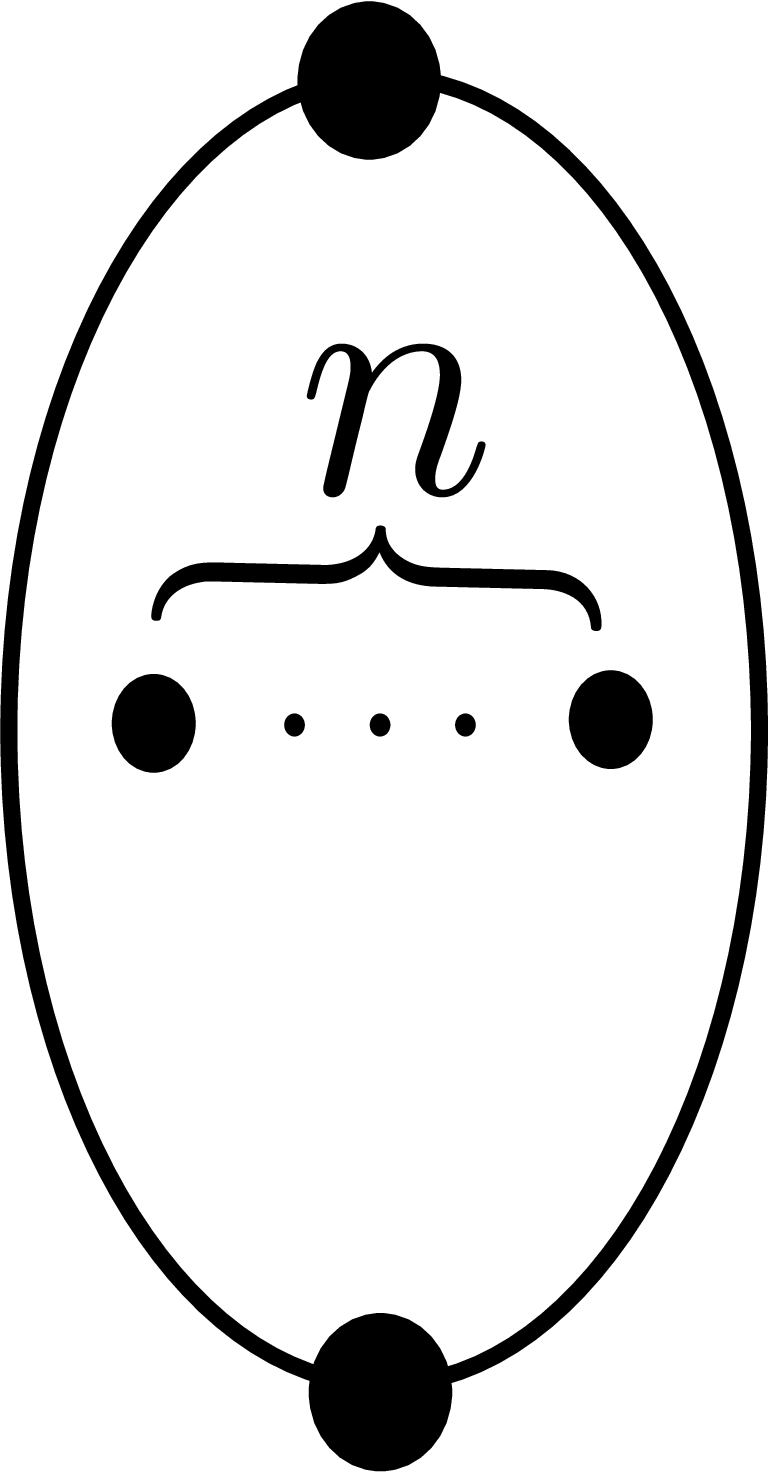}} $ is a disc with two boundary edges and $n$ open subdiscs removed from its interior. Our construction is based on the representation theory of the reduced stated skein algebra $\overline{\mathcal{S}}_{\zeta}(\mathbb{D}_n)$  described in \cite{KojuKaruo_RepRSSkein}; we sketch it briefly here and refer to Section \ref{sec_RepSSkein} for details. Let $Z_{\mathbb{D}_n}$ denote the center of $\overline{\mathcal{S}}_{\zeta}(\mathbb{D}_n)$ and write $\widehat{X}(\mathbb{D}_n):=\Specm(Z_{\mathbb{D}_n})$. By \cite{KojuQuesneyClassicalShadows, KojuAzumayaSkein}, as a set, $\widehat{X}(\mathbb{D}_n)$ is in natural bijection with the set of tuples $\widehat{x}= (\rho, h_{p_1}, \ldots, h_{p_n}, h_{\partial})$ where $\rho: \pi_1(\mathbb{D}_n, \mathbb{V})\to \SL_2$ is a representation from the fundamental groupoid of $\mathbb{D}_n$ with two base points to $\SL_2$ (satisfying some conditions), $h_{p_i}\in \mathbb{C}$ is such that if $\delta^{(i)}$ is a closed loop encircling the $i$-th puncture, then $T_N(h_{p_i})=-\tr(\rho(\delta^{(i)}))$, where $T_N(X)$ is the Chebyshev polynomial of the first kind, and $h_{\partial}\in \mathbb{C}^*$ is such that $h_{\partial}^N$ is determined by $\rho$. Consider the character map 
$$ \chi: \Irrep\left(\overline{\mathcal{S}}_{\zeta}(\mathbb{D}_n)\right) \to \widehat{X}(\mathbb{D}_n)$$
sending an isomorphism class of irreducible representation $r: \overline{\mathcal{S}}_{\zeta}(\mathbb{D}_n) \to \End(V)$ to the kernel of the induced central character. It is proved in \cite{KojuAzumayaSkein} that $\chi$ is "almost" a bijection in the sense that there exists an open dense subset $\mathcal{AL}(\mathbb{D}_n) \subset \widehat{X}(\mathbb{D}_n)$, the \textit{Azumaya locus}, such that the restriction $\chi: \chi^{-1}(\mathcal{AL}(\mathbb{D}_n)) \to \mathcal{AL}(\mathbb{D}_n)$ is a bijection. Moreover every representation in $\chi^{-1}(\mathcal{AL}(\mathbb{D}_n)) $ has dimension $N^n$. The set $\mathcal{AL}(\mathbb{D}_n)$ was computed in \cite{KojuKaruo_RepRSSkein}: it is the subset of these $\widehat{x}=(\rho, h_{p_1}, \ldots, h_{p_n}, h_{\partial})$ such that for each $1\leq i \leq n$, then either $\rho(\delta^{(i)})\neq \pm \mathds{1}_2$ or $h_{p_i}=\pm 2$. The braid group $B_n$ naturally (right) acts on $\overline{\mathcal{S}}_{\zeta}(\mathbb{D}_n)$ by preserving the center so induces a left action of $\widehat{X}(\mathbb{D}_n)$. Let $\beta \in B_n$ be a braid, $\widehat{x} \in \mathcal{AL}(\mathbb{D}_n)$ be such that $\beta\cdot \widehat{x}=\widehat{x}$ and let $r_{\widehat{x}}: \overline{\mathcal{S}}_{\zeta}(\mathbb{D}_n) \to \End(V_{\widehat{x}})$ be an irreducible representation with character $\widehat{x}$ (unique up to isomorphism). By unicity, the representation $\beta\cdot r_{\widehat{x}}: X \mapsto r_{\widehat{x}}(\beta^*X)$ is isomorphic to $r_{\widehat{x}}$ so there exists an intertwiner $L_{V_{\widehat{x}}}(\beta) \in \PGL(V_{\widehat{x}})$, unique up to multiplication by a  scalar, such that 
$$ r_{\widehat{x}}(\beta^* X) = L_{V_{\widehat{x}}}(\beta) r_{\widehat{x}}(X) L_{V_{\widehat{x}}}(\beta)^{-1}, \quad \mbox{ for all }X\in \overline{\mathcal{S}}_{\zeta}(\mathbb{D}_n).$$
By adding twist operators, we will also obtain intertwiners $L_{V_{\widehat{x}}}(\beta) \in \End(V_{\widehat{x}})$ for framed braids $\beta\in fB_n$ and using a suitable normalization, we will be able to reduce the projective ambiguity in the definition of $L_{V_{\widehat{x}}}(\beta) $ to an ambiguity up to multiplication by a $N^2$-th root of unity only. When $\rho$ is such that there exists $i$ such that either $\tr(\rho(\delta^{(i)}))\neq \pm 2$ or $h_{p_i}\neq \pm 2$,  using Geer-Patureau Mirand \textit{renormalized trace} $\tau$ from \cite{GeerPatureau_TraceQG}, we can consider the trace $\tau(L_{V_{\widehat{x}}}(\beta))\in \mathbb{C}$ well defined up to multiplication by a $N^2$-th root of unity. Let $L\subset S^3$ be the framed link which is the Markov closure of the framed braid $\beta$ and let $M_L:= S^3\setminus \mathring{N}(L)$ denote the link exterior. The representation $\rho$ induces a representation $\underline{\rho}: \pi_1(M_L)\to \SL_2$; we denote by $[\rho]$ its conjugacy class. Each component $L_i$ of $L=L_1\sqcup \ldots \sqcup L_k$ corresponds to a puncture $p_i$ of $\mathbb{D}_n$. We set $h_{L_i}:=h_{p_i}$ so $T_N(h_{L_i})=-\tr(\underline{\rho}(\mu_i))$ where $\mu_i$ is a meridian of $L_i$. Write $\mathbf{h}=(h_{L_1}, \ldots, h_{L_k})$ and consider the triple $\mathbb{L}=(L, [\underline{\rho}], \mathbf{h})$. Note that this triple satisfies: 
\begin{enumerate}
\item[(i)] for all $1\leq i \leq k$, either $\tr(\underline{\rho}(\mu_i))\neq \pm \mathds{1}_2$ or $h_{L_i}=\pm 2$ (since $\widehat{x}\in \mathcal{AL}(\mathbb{D}_n)$) and 
\item[(ii)] there exists $1\leq i \leq k$ such that either $\tr(\rho(\mu_i)) \neq \pm 2$ or $h_{L_i}=\pm 2$.
\end{enumerate}
A triple  $\mathbb{L}=(L, [\underline{\rho}], \mathbf{h})$ satisfying $(i)$ and $(ii)$ will be called an \textbf{admissible decorated link} and their set will be denoted by $\mathcal{DL}$. 

\begin{theorem}\label{theorem_intro_LinkInv}(Quantum Holonomic Invariants) The class $\tau(L_{V_{\widehat{x}}}(\beta))\in \quotient{\mathbb{C}}{\mu_{N^2}}$ only depends on $\mathbb{L}$ so we obtain a link invariant
$$ \left< \cdot \right> : \mathcal{DL} \to \quotient{\mathbb{C}}{\mu_{N^2}}, \quad \left< \mathbb{L}\right>:= \tau(L_{V_{\widehat{x}}}(\beta)).$$
\end{theorem}

Note that, unlike quantum groups constructions based on the R-matrix (so on the canonical element of Drinfeld doubles) or quantum hyperbolic constructions based on cyclic quantum dilogarithms, our definition does not take any abstract algebraic formulas as input but replace them by intertwiners associated to braids and  defined in a natural (topological) way. The drawback of this approach is that, at this stage, we have no explicit formulas for the intertwiner $L_{V_{\widehat{x}}}(\beta)$ so it is still not clear how to compute the link invariants explicitly.  
\vspace{2mm}\par 
The key to relate this link invariant to quantum groups constructions is the following. There is a natural embedding of marked surfaces $\iota: \mathbb{D}_1 \hookrightarrow \mathbb{D}_n$ which induces an algebra morphism $\iota_*: \overline{\mathcal{S}}_{A}(\mathbb{D}_1) \to \overline{\mathcal{S}}_{A}(\mathbb{D}_n)$; in particular every modules $V_{\widehat{x}}$ have a structure of $\overline{\mathcal{S}}_{A}(\mathbb{D}_1) $ as well. Also by gluing $n$ copies of $\mathbb{D}_1$ together we obtain $\mathbb{D}_n$ so we have a \textit{splitting morphism} $\theta: \overline{\mathcal{S}}_{A}(\mathbb{D}_n) \hookrightarrow (\overline{\mathcal{S}}_{A}(\mathbb{D}_1))^{\otimes n}$. The composition $\Delta:= \theta \circ \iota_*: \overline{\mathcal{S}}_{A}(\mathbb{D}_1) \to \overline{\mathcal{S}}_{A}(\mathbb{D}_1)^{\otimes 2}$ defines a coproduct which endows $\overline{\mathcal{S}}_{A}(\mathbb{D}_1)$ with a structure of Hopf algebra. Let $\Uq$ denote the Drinfeld double of the quantum Borel algebra $B_q$ where $q=A^2$. It has generators $K^{\pm 1/2}, L^{\pm 1/2}, E, F$ which satisfies $EF-FE=\frac{K -L}{q-q^{-1}}$. In particular the element $K^{1/2}L^{1/2}$ is central and group-like and the quotient $\quotient{\Uq}{(K^{1/2}L^{1/2}-1)}$ is isomorphic to the simply connected $\widetilde{U}_q\mathfrak{sl}_2$. 
 The following theorem is the heart of this paper and is  inspired from Bigelow's work in  \cite{BigelowQGroups}.

\begin{theorem}\label{theorem_intro_QG}(Skein construction of Quantum Groups) 
\\ There is an isomorphism of Hopf algebras $\Psi: \Uq \xrightarrow{\cong} \overline{\mathcal{S}}_A(\mathbb{D}_1)$ defined by:
\begin{align*}
{}&\Psi(K^{1/2}) = \adjustbox{valign=c}{\includegraphics[width=1.5cm]{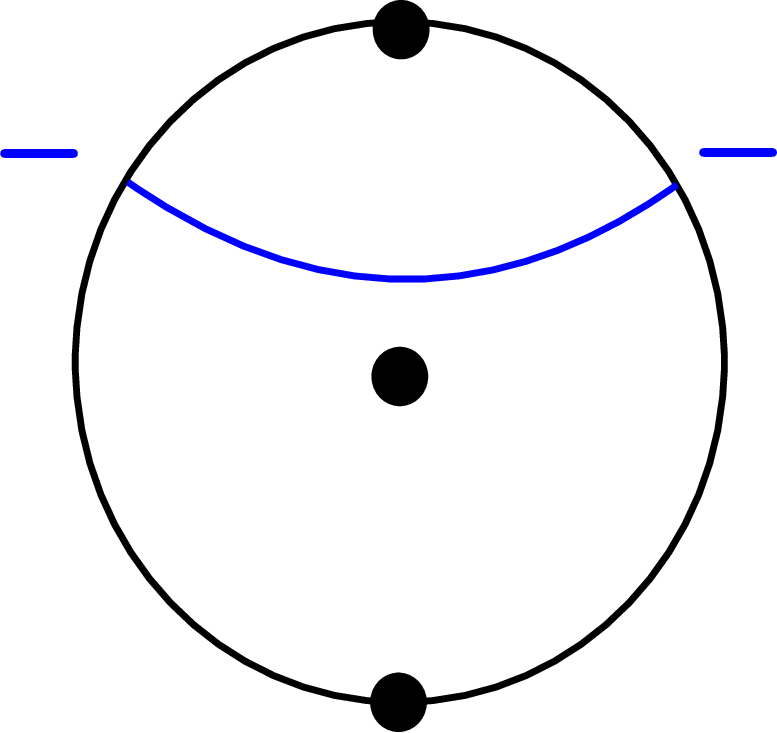}}, \quad \Psi(L^{1/2})=  \adjustbox{valign=c}{\includegraphics[width=1.5cm]{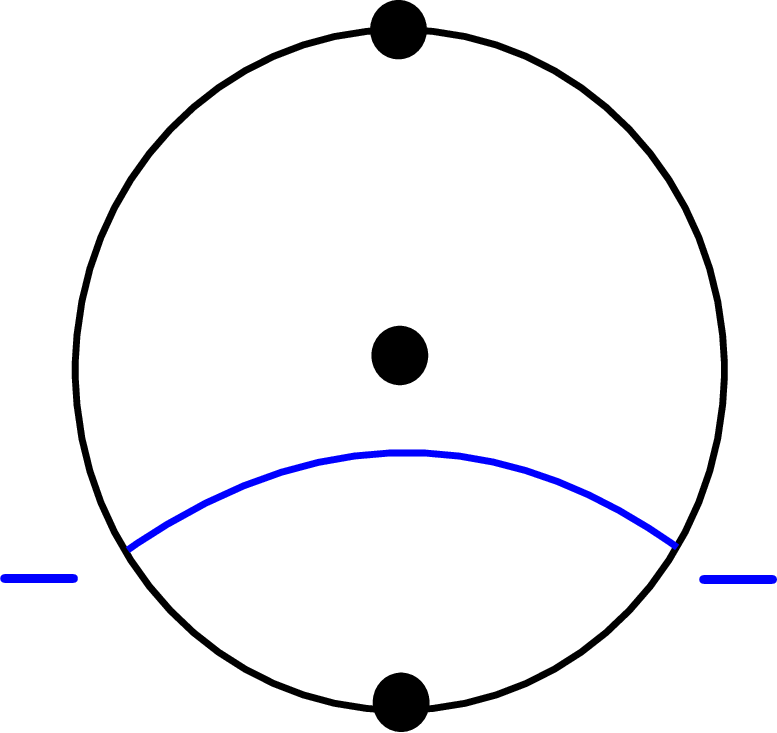}}, \quad \Psi(E)=  -\frac{A}{q-q^{-1}} \adjustbox{valign=c}{\includegraphics[width=1.5cm]{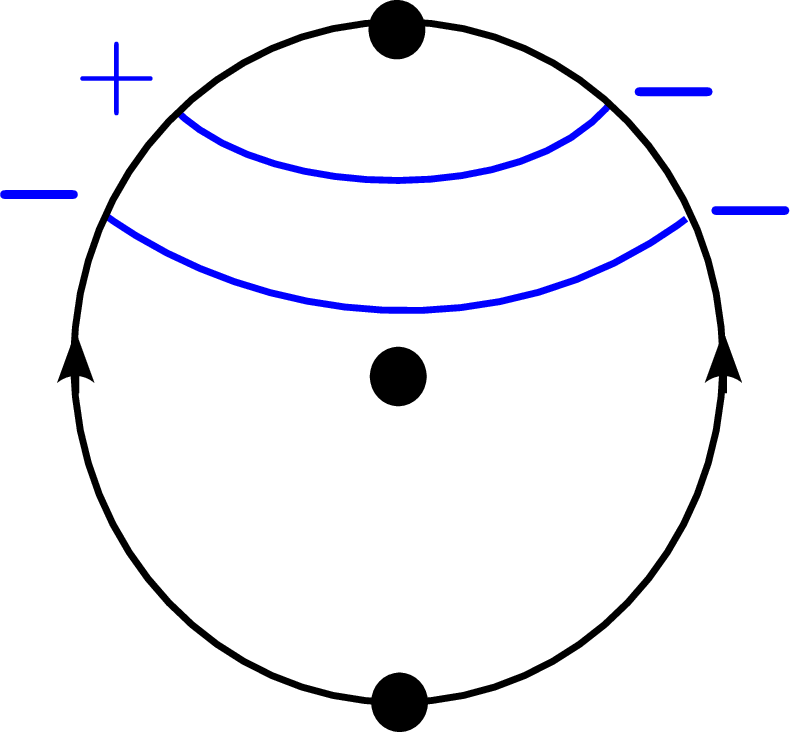}}, \quad \Psi(F) =  \frac{A^{-1}}{q-q^{-1}}\adjustbox{valign=c}{\includegraphics[width=1.5cm]{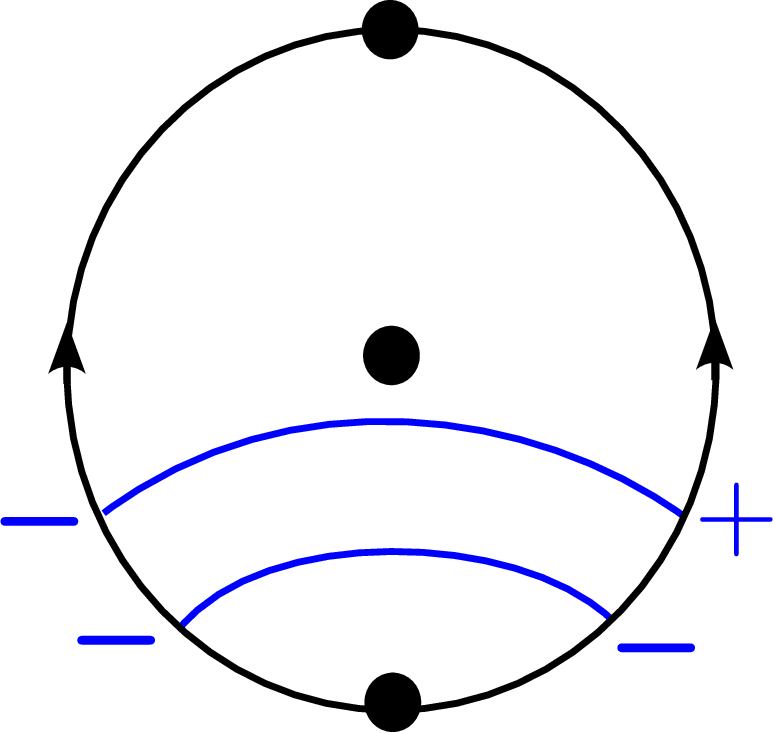}}. 
\end{align*}
\end{theorem}

Now if $A=\zeta$ is a root of unity, the category $\mathcal{C}$ of finite dimensional semi-weight $\Uq$ modules (i.e. modules on which the central elements $E^N, F^N, K^{N/2}, L^{N/2}$ acts semi-simply) is pivotal and its full subcategory $\Proj(\mathcal{C})$ generated by projective objects admits a trace $\tau: \mathrm{HH}_0(\Proj(\mathcal{C})) \to \mathbb{C}$ considered in \cite{GeerPatureau_TraceQG}  and characterized by its compatibility with the pivotal structure (see Section \ref{sec_trace}). An Azumaya simple $\mathcal{S}_{\zeta}(\mathbb{D}_n)$ module  $V_{\widehat{x}}$ has a structure of weight $\Uq$ module using $\iota_*$ and $\Psi$ and $V_{\widehat{x}}\in \Proj(\mathcal{C})$ if and only if condition $(ii)$ above is satisfied. Since the braid group acts trivially on the image of $\iota_*$, the braid operator $L_{V_{\widehat{x}}}(\beta))$ is a morphism in $\End_{\Proj(\mathcal{C})}(V_{\widehat{x}})$ so the trace $\tau(L_{V_{\widehat{x}}}(\beta))$ makes sense. 
\par The proof of Theorem \ref{theorem_intro_LinkInv} relies on the facts that $(1)$ the value of the trace $\tau(L_{V_{\widehat{x}}}(\beta))$ is invariant under framed Markov moves on $\beta$: this fact is an immediate consequence of the compatibility of $\tau$ with the pivotal structure (i.e. that $\tau$ is a "right trace" in the terminology of \cite{GeerPatureau_TraceQG}) and $(2)$ the value $\tau(L_{V_{\widehat{x}}}(\beta))$ only depends on the conjugacy class $[\rho]$  of the representation $\rho$ and the puncture invariant $h_{p_i}$: this will result from the fact that De Concini-Kac quantum coadjoint action acts transitively on the modules $V_{(\rho, h_{p_1}, \ldots, h_{p_n}, h_{\partial})}$ where $\rho$ belongs to a given conjugacy class. This latter fact will be proved by computing explicitly the Hamiltonian flows of the generators $E,F,K^{\pm 1/2}, L^{\pm 1/2}$ on the Poisson variety $X(\mathbb{D}_n)$ equipped with its Fock-Rosly Poisson structure. 
\par The relation with quantum hyperbolic geometry will be established using Bonahon-Wong's quantum trace $\Tr_{\zeta}^{\Delta} : \overline{\mathcal{S}}_{\zeta}(\mathbb{D}_n) \hookrightarrow \mathcal{Z}_{\zeta}(\mathbb{D}_n, \Delta_n)$ from \cite{BonahonWongqTrace},  where $\mathcal{Z}_{\zeta}(\mathbb{D}_n, \Delta_n)$ is the balanced Chekhov-Fock algebra of $\mathbb{D}_n$ with its canonical triangulation $\Delta_n$. When $n=1$, the embedding 
$$\Uq \cong \overline{\mathcal{S}}_{\zeta}(\mathbb{D}_1) \hookrightarrow \mathcal{Z}_{\zeta}(\mathbb{D}_1, \Delta_1)$$
embeds the quantum group $\Uq$ into a quantum torus. This is the Kashaev-Volkov embedding \cite{KashaevVolkov}, further studied in \cite{FadeevModularDouble} and generalized to the $\SL_n$ case in \cite{SchraderShapiro}. Our framework permits to see this embedding as a particular case of the quantum trace. 
\par Quantum Teichm\"uller theory permits to compute the intertwiners $L_{V_{\widehat{x}}}(\beta)$ explicitly; 
we briefly sketch the idea here and refer to Section \ref{sec_QT} for details.  A braid $\beta\in B_n$ sends the triangulation $\Delta_n$ to another one $\beta(\Delta_n)$ and induces an isomorphism $\beta^* : \mathcal{Z}_{\zeta}(\mathbb{D}_n, \beta(\Delta_n)) \to \mathcal{Z}_{\zeta}(\mathbb{D}_n, \Delta_n)$. Consider the \textbf{change of coordinates} isomorphism $\Phi_{\zeta}^{\beta(\Delta_n), \Delta_n}: \widehat{\mathcal{Z}}_{\zeta}(\mathbb{D}_n, \Delta_n) \to \widehat{\mathcal{Z}}_{\zeta}(\mathbb{D}_n, \beta(\Delta_n))$ from quantum Teichm\"uller theory. Suppose that $V_{\widehat{x}}$ is  a simple $\overline{\mathcal{S}}_{\zeta}(\mathbb{D}_n)$ Azumaya module which extends to an irreducible representation  $r : \mathcal{Z}_{\zeta}(\mathbb{D}_n, \Delta_n)\to \End(V_{\widehat{x}})$ module through $\Tr_{\zeta}^{\Delta_n}$ (we will prove that this is always possible up to gauge equivalence). Then $L_{V_{\widehat{x}}}(\beta)$ can be alternatively characterized as being the  intertwiner, unique up to multiplication by a scalar, such that
$$ r( \beta^* \circ \Phi_{\zeta}^{\beta(\Delta_n), \Delta_n} (Y)) = L_{V_{\widehat{x}}}(\beta) r(Y) L_{V_{\widehat{x}}}(\beta)^{-1}, \quad \mbox{ for all } Y \in \mathcal{Z}_{\zeta}(\mathbb{D}_n, \Delta_n).$$
Such an intertwiner for $\beta^* \circ \Phi_{\zeta}^{\beta(\Delta_n), \Delta_n} $ can be explicitly computed: if $\Delta_n$ and $\beta(\Delta_n)$ are related by a sequence of flips then the intertwiner $L_{V_{\widehat{x}}}(\beta)$ will be a composition of cyclic quantum dilogarithms (one for each flip) and a quadratic term. This strategy will provide us explicit formulas for the link invariant $ \left< \mathbb{L}\right>$ closely related to Baseilhac Benedetti QHI. 
\vspace{2mm}
\par The intertwiner $L_{V_{\widehat{x}}}(\beta)$ can be decomposed into smaller pieces named \textbf{skein braiding operators} associated to generators of $B_n$ which exchange two adjacent punctures in the clockwise direction. More precisely, consider $V_1, V_2, V_3, V_4 \in \mathcal{C}$ some simple $\Uq$ modules. Using the splitting morphism $\theta: \overline{\mathcal{S}}_{\zeta}(\mathbb{D}_2) \hookrightarrow (\Uq)^{\otimes 2}$, both $V_1\otimes V_2$ and $V_3\otimes V_4$ acquire some structure of simple $\overline{\mathcal{S}}_{\zeta}(\mathbb{D}_2)$ modules; denote by $\widehat{x}, \widehat{x}' \in \widehat{X}(\mathbb{D}_2)$ their induced central characters and by $r, r'$ their associated $\overline{\mathcal{S}}_{\zeta}(\mathbb{D}_2) $ representations. The two pairs $(V_1, V_2)$ and $(V_3,V_4)$ as said \textbf{compatible} if $\tau \cdot \widehat{x} = \widehat{x}'$, where $\tau \in B_2$ is the standard generator. If $\widehat{x}' \in \mathcal{AL}(\mathbb{D}_2)$ there exists an isomorphism $c= c_{V_1, V_2}^{V_3,V_4}: V_1\otimes V_2 \to V_3 \otimes V_4$, unique up to multiplication by a scalar, such that 
$$ r'(\beta^*X) = c r(X) c^{-1}, \quad \mbox{ for all }X\in \overline{\mathcal{S}}_{\zeta}(\mathbb{D}_2).$$
Such a skein braiding operator $c= c_{V_1, V_2}^{V_3,V_4}$ will be called \textbf{normalized} if $\det(c)=1$. Normalized braiding operators are uniquely defined up to multiplication by a root of unity of order $N^2$ and $L_{V_{\widehat{x}}}(\beta)$ is obtained by composing, tensoring and taking partial traces of such operators; this is how we reduce the projective ambiguity to a $N^2$ root of unity only. The relations with other braiding operators appearing in literature is summarized in the following.

\begin{theorem}[Braiding operators]\label{theorem_intro_braidings}
\begin{enumerate}
\item Let $V_1, V_2$ be two simple $\Uq$-modules  such that $r_1(E^N)=0$ and $r_2(F^N)=0$ and consider
$$R^D : V_1\otimes V_2 \rightarrow V_2\otimes V_1, \quad R^D= \sigma \circ (r_1\otimes r_2 )\left( q^{-\frac{H\otimes G}{2}} \exp_q^{<N}\left( (q-q^{-1})E\otimes F \right) \right)$$ the   Drinfeld braiding from quantum group theory (obtained from the canonical element of the Drinfeld double).  Then $R^D$ is a skein braiding.
\item The Kashaev-Reshetikhin braidings defined in \cite{KashaevReshetikhin04,KashaevReshetikhin05} are skein braidings.
\item  An explicit formula for general skein braiding operators $c= c_{V_1, V_2}^{V_3,V_4}$ is given by 
$$ c=\frac{1}{d} \sigma \circ Q \circ  \Phi_{w'_4}( -w_4[D^{-2}A] \otimes [A^{-1}D^2]) \Phi_{w''_3}(-w_3'[D^{-2}A] \otimes A^{-1})\Phi_{w''_2}(-w_2'A \otimes [A^{-1}D^2]) \Phi_{w'_1}( -w_1A \otimes A^{-1}) $$ 
where $\frac{1}{d}$ is a scalar which ensures that $\det(R)=1$ (see Section \ref{sec_normalization} for an explicit expression), $\sigma(x\otimes y)=y\otimes x$ is the flip, $Q\in \End(\mathbb{C}^N \otimes \mathbb{C}^N)$ and  $A, D\in \End(\mathbb{C}^N)$ are linear maps defined in canonical bases $\{v_i\}_{i}$ by 
  $$ Q\cdot v_i\otimes v_j = q^{2ij} v_i \otimes v_j, \quad D\cdot v_i = q^i v_i, \quad A\cdot v_i=v_{i-1},$$
  and $\Phi_w(X)$ is a polynomial related to the cyclic dilogarithms. 

\end{enumerate}
\end{theorem}

The last formula for skein braiding operators is a reformulation of Kashaev's R matrices in  \cite{Kashaev_6jSymbol, KashaevLinkInv}. In particular, our study reproves the fact, proved by Murakami-Murakami in \cite{MM01}, that the R-matrix of Kashaev invariant is gauge equivalent to the quantum group R-matrix of the $N$-th colored Jones polynomial at $N$-th root of unity. Moreover, it provides an explanation for this relation: the quantum group R matrix is defined by a term $\exp_q^{<N}\left( (q-q^{-1})E\otimes F \right) $ whereas Kashaev's R-matrices are contraction of $4$ cyclic quantum dilogarithms associated to quantum shape parameters $\mathbf{w}_i$. 
The idea is that for a degenerate parameter $w=1$ then $\phi_{w=1}(X)= \exp_q^{<N}(\frac{-X}{q-q^{-1}})$ so the q-exponential is a particular case of cyclic dilogarithm corresponding to a degenerate tetrahedron. Moreover, under some good conditions on $x,y$ (when $xy=q^2yx$ and $x^ny^{N-n}=0$ for all $0\leq n \leq N$), we have the equality $\exp_q^{<N}(x+y)=\exp_q^{<N}(x) \exp_q^{<N} (y)$. Using this formula, we will prove that, under the hypotheses of the first item in Theorem \ref{theorem_intro_braidings}, the product of the four quantum dilogarithms in the third item whence considered in the degenerate case, is a product of four q-exponentials which is the evaluation of a q-exponential of a sum; this sum is the image by the quantum trace of $E\otimes F$. 

\par By comparing skein braiding operators with other braiding operators, we relate our link invariant to preceding constructions.
\begin{theorem}[Comparison with previous constructions]\label{theorem_intro_relations}
\begin{enumerate}
\item If $\rho_0: \pi_1(M_L)\to \SL_2$ is the trivial representation sending every loop to the identity matrix $\mathds{1}_2\in \SL_2$, $h=(h_{L_1}, \ldots, h_{L_k})$ is such that $h_{L_i}=-2$ for all $1\leq i \leq k$ and $\mathbb{L}=(L, [\rho_0], h)$ then $\left<\mathbb{L}\right>$ is the (class in $\quotient{\mathbb{C}}{\mu_{N^2}}$ of the) Kashaev invariant (\cite{KashaevLinkInv}) of $L$.
\item If $\mathbb{L}=(L, [\rho], h) \in \mathcal{DL}$ is such that $\tr(\rho(\mu_i))\neq \pm 2$ for all $1\leq i \leq k$, then $\left<\mathbb{L}\right>$ is equal to the link invariant defined in \cite{BGPR_Biquandle}. 
\end{enumerate}
\end{theorem}

The skein construction of link invariants developed in this paper provides thus a bridge between quantum groups and quantum hyperbolic geometric constructions as illustrated in Figure \ref{fig_QInvariants}. 
 
 \begin{figure}[!h] 
\centerline{\includegraphics[width=9cm]{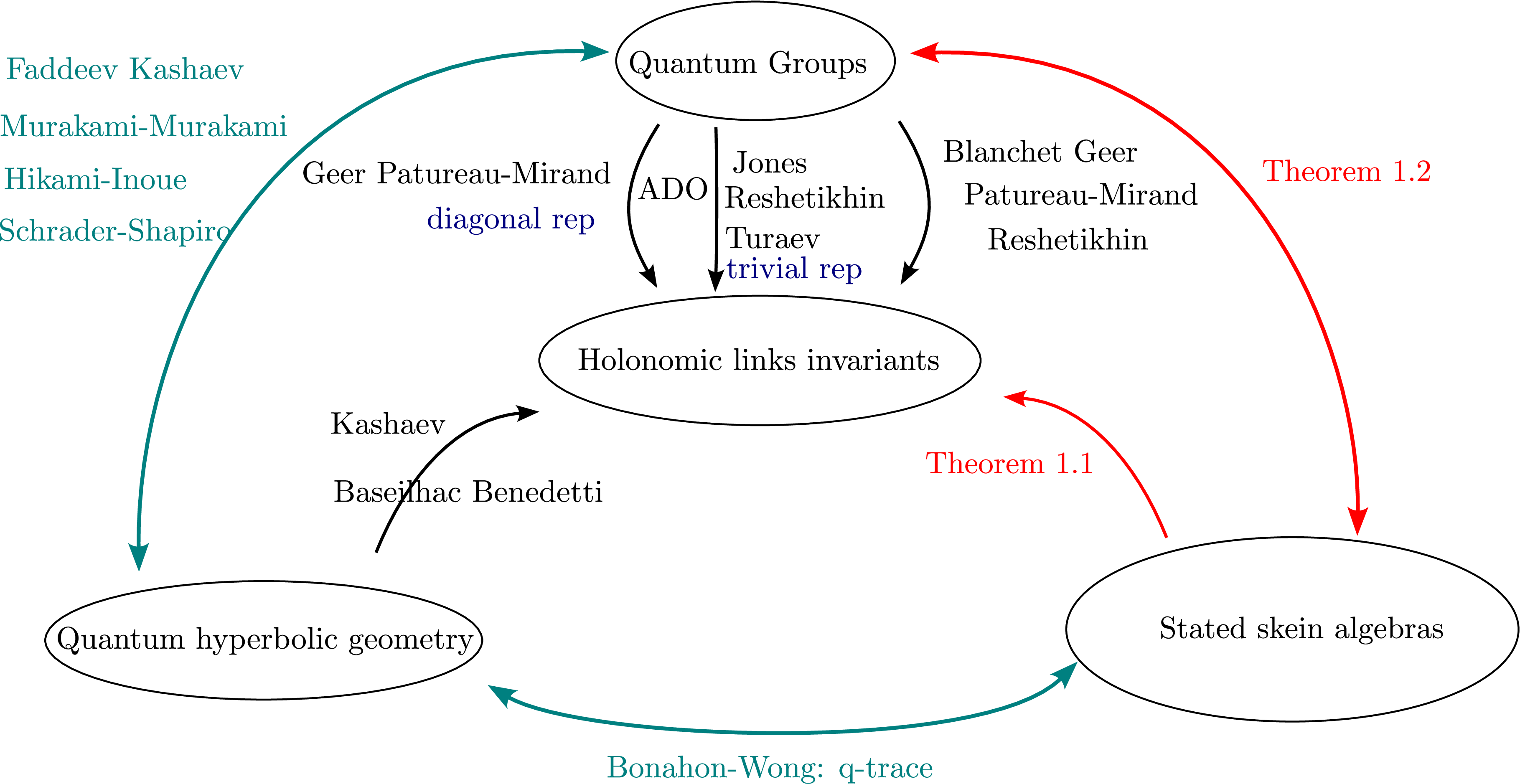} }
\caption{Three different approaches to define link invariants. The quantum trace relates skein theory to quantum hyperbolic geometry whereas Theorem \ref{theorem_intro_QG} relates quantum groups to skein theory. } 
\label{fig_QInvariants} 
\end{figure}

\vspace{2mm}
\paragraph{\textbf{Plan of the paper.}} 

After recalling the definition and main properties of reduced stated skein algebras in Section \ref{sec_sskein}, we prove Theorem \ref{theorem_intro_QG} in Section \ref{sec_QG_skein}. We next study the Poisson geometry of the moduli spaces $X(\mathbb{D}_n)$ in Section \ref{sec_moduli_space} and prove the important fact that two representations in $X(\mathbb{D}_n)$ are conjugate if and only if they belong to the same symplectic leaf. In Section \ref{sec_RepSSkein} we review the representation theory of $\overline{\mathcal{S}}_{\zeta}(\mathbb{D}_n)$ following \cite{KojuKaruo_RepRSSkein} and the renormalized trace from \cite{GeerPatureau_TraceQG}. Section \ref{sec_link_inv} is devoted to the construction of link invariants and the proof of Theorem \ref{theorem_intro_LinkInv} and of the first two items of Theorems \ref{theorem_intro_braidings} and \ref{theorem_intro_relations}. In Section \ref{sec_QT} we review the definitions and main properties of the balanced Chekhov-Fock algebra and the quantum trace and use them to obtain the explicit formula for the braiding operator appearing in the third item of Theorem \ref{theorem_intro_braidings} and explain its relation with quantum hyperbolic geometry.

\vspace{2mm}
\paragraph{\textbf{Acknowledgments.}} The author is thankful to S.Baseilhac, C.Blanchet, F.Bonahon, F.Costantino, D.Calaque, L.Funar,  A.Quesney, P.Roche,  J.Toulisse and V.Verchinine for useful discussions and to the University of South California, the Federal University of S\~ao Carlos, the University of S\~ao Paulo, the University of Montpellier and the Waseda University for their kind hospitality during the development of this work.  He acknowledges  support from the Coordena\c{c}\~ao de Aperfei\c{c}oamento de Pessoal de n\'ivel Superior (CAPES), from the GEometric structures And Representation varieties (GEAR) network, from  the Japanese Society for Promotion of Science (JSPS), from the Centre National de la Recherche Scientifique (CNRS) and from the European Research Council (ERC DerSympApp) under the European Union’s Horizon 2020 research and innovation program (Grant Agreement No. 768679).

\section{Stated skein algebras and their reduced versions}\label{sec_sskein}

\subsection{Definition and main properties of stated skein algebras}

\subsubsection{Definitions}

 \begin{definition}[Marked surfaces]\label{def_marked_surfaces}
 \begin{enumerate}
 \item A \textbf{marked surface} is a pair $\mathbf{\Sigma}=(\Sigma, \mathcal{A})$ where $\Sigma$ is a compact oriented surface and $\mathcal{A}$ is a finite set of pairwise disjoint closed intervals embedded in $\partial \Sigma$ named \textbf{boundary edges}. 
 The orientation of $\Sigma$ induces an orientation of $\partial \Sigma$ and thus an orientation of each boundary edge.
 We impose that each boundary edge $a\subset \partial \Sigma$ is equipped with a parametrization $\varphi_a: [0,1] \hookrightarrow \partial \Sigma$ such that $\varphi_a([0,1])=a$ and such that $\varphi_a$ is an oriented embedding (where $[0,1]$ is oriented from $0$ to $1$). A marked surface  $\mathbf{\Sigma}=(\Sigma, \mathcal{A})$ is called \textbf{unmarked} if $\mathcal{A}=\emptyset$ and is said \textbf{essential} if each connected component of $\Sigma$ contains at least one boundary edge.
  \item A connected component of $\partial \Sigma \setminus \mathcal{A}$ is called a \textbf{puncture} and, pictorially, we will represent them as  punctures $\bullet$. There are two kinds of punctures: the ones homeomorphic to a circle, which correspond to unmarked boundary components are called \textbf{inner punctures} and their set is denoted by $\mathring{\mathcal{P}}$; the ones homeomorphic to an open interval, which lies between two boundary edges of the same boundary component, are called \textbf{boundary punctures} and their set is denoted by $\mathcal{P}^{\partial}$. 
  \item An \textbf{embedding} of marked surfaces $f: \mathbf{\Sigma}_1 \to \mathbf{\Sigma}_2$ is an oriented embedding $f:\Sigma_1\hookrightarrow \Sigma_2$ of the underlying surfaces which induces an embedding $f: \mathcal{A}_1\hookrightarrow \mathcal{A}_2$. An embedding $f$ is said \textbf{strong} if  for every boundary edge $b\in \mathcal{A}_2$ there is at most one boundary edge $a\in \mathcal{A}_1$ such that $f(a)\subset b$. Marked surfaces with embeddings form a category $\mathcal{MS}$ with symmetric monoidal structure given by the disjoint union $\sqcup$ and marked surfaces with strong embeddings form a subcategory $\mathcal{MS}^{str}$ (so the inclusion $\mathcal{MS}^{str}\subset \mathcal{MS}$ is faithful but not full).
 \item If $\mathbf{\Sigma}_1=(\Sigma_1, \mathcal{A}_1)$ and $\mathbf{\Sigma}_2=(\Sigma_2, \mathcal{A}_2)$ are two marked surfaces and $a_1, a_2$ are boundary edges of $\mathbf{\Sigma}_1$ and $\mathbf{\Sigma}_2$ respectively, we denote by $\mathbf{\Sigma}_1\cup_{a_1\# a_2} \mathbf{\Sigma}_2:= (\Sigma_1\cup_{a_1\# a_2}\Sigma_2, (\mathcal{A}_1 \cup \mathcal{A}_2)\setminus \{a_1, a_2\})  $ the marked surface obtained by gluing $a_1$ to $a_2$ where 
 $$ \Sigma_1\cup_{a_1\# a_2}\Sigma_2:= \quotient{ \Sigma_1 \bigsqcup \Sigma_2}{ (\varphi_{a_1}(1-t)=\varphi_{a_2}(t), t\in [0,1])}.$$
 \item The \textbf{bigon} $\mathbb{B}=(D^2, \{a_L, a_R\})=\adjustbox{valign=c}{\includegraphics[width=0.5cm]{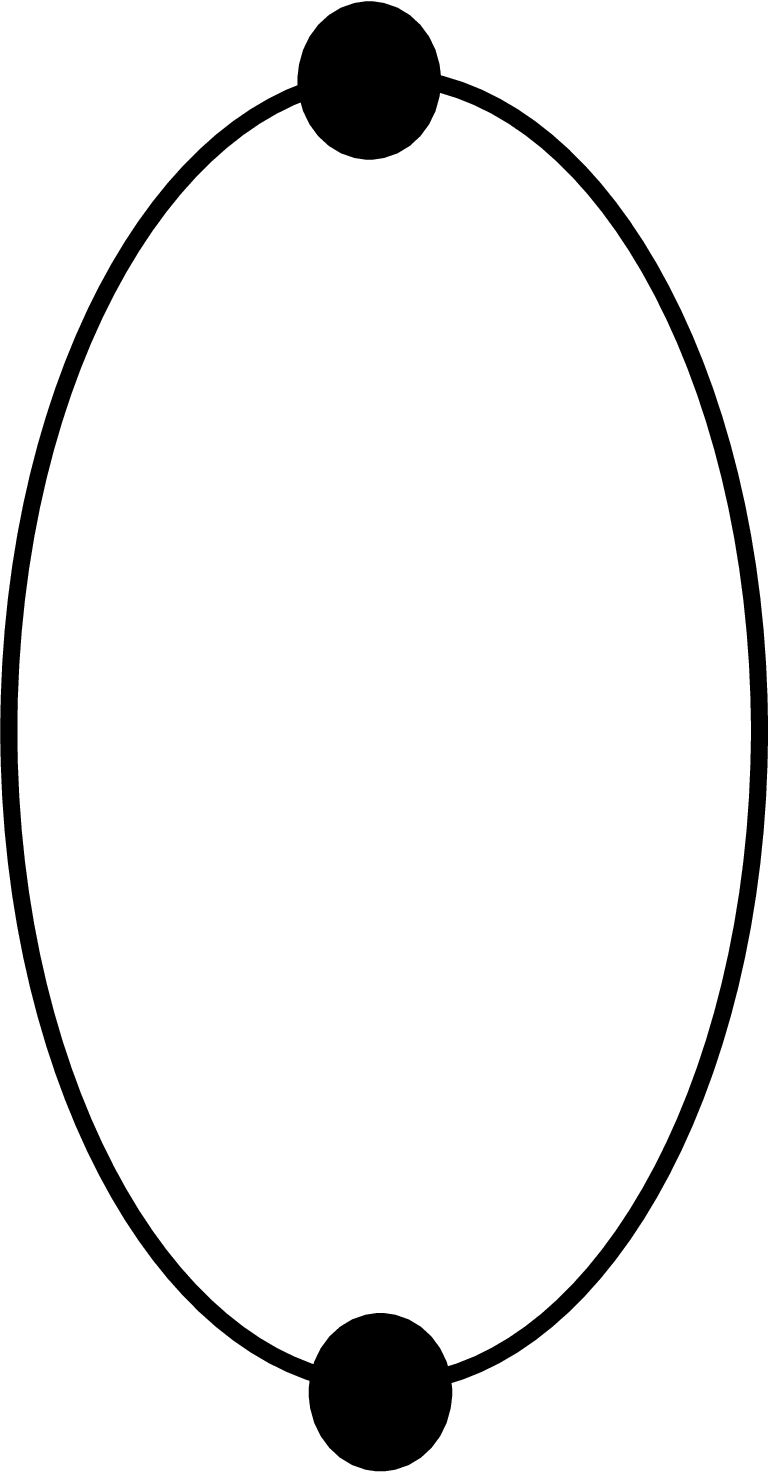}} $ is a disc with two boundary edges (thus two boundary punctures). The \textbf{triangle} $\mathbb{T}=(D^2, \{a,b,c\})=\adjustbox{valign=c}{\includegraphics[width=0.9cm]{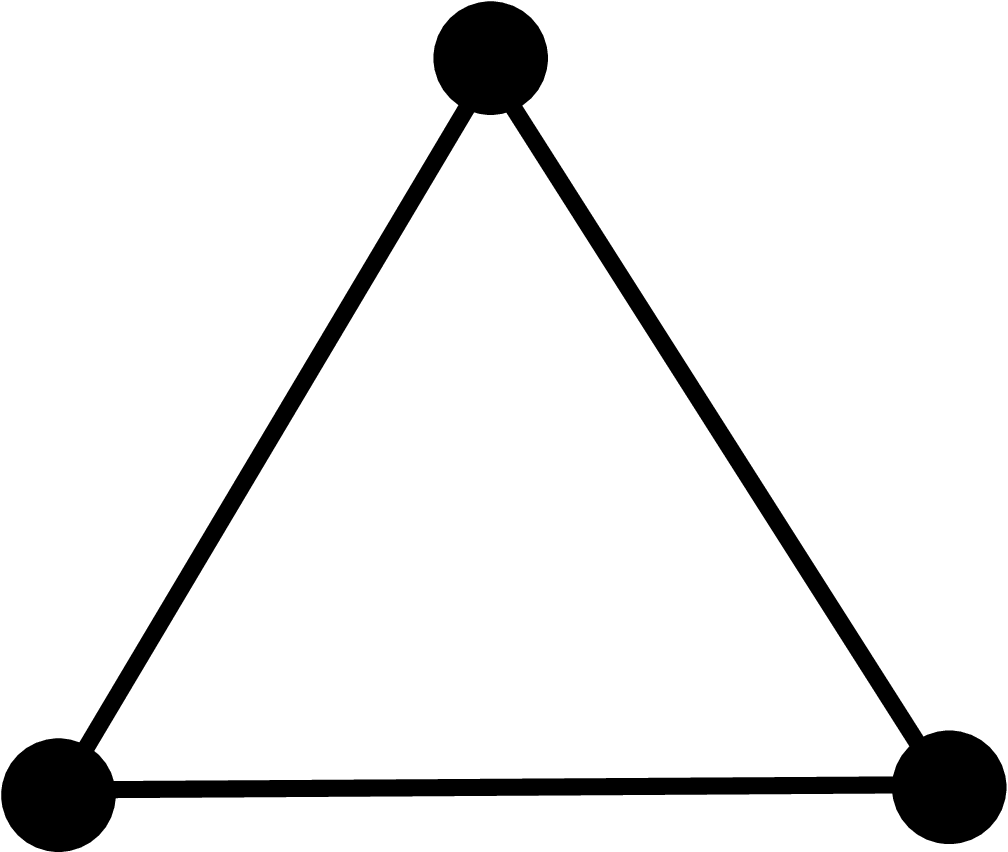}} $ is a disc with three boundary edges (thus three boundary punctures). For $n\geq 0$, the $n$-th punctured bigon $\mathbb{D}_n=(D_n, \{b_L, b_R\})=\adjustbox{valign=c}{\includegraphics[width=0.7cm]{Dn.eps}} $ is a disc with two boundary edges and $n$ open subdiscs removed from its interior (so it has two boundary punctures and $n$ inner punctures).
\end{enumerate}
 \end{definition}

 \begin{definition}[Tangles and diagrams]\label{def_tangles} Let $\mathbf{\Sigma}=(\Sigma, \mathcal{A})$ be a marked surface.
 \begin{enumerate}
 \item  A \textbf{tangle} in $ \mathbf{\Sigma} \times (0,1)$   is a  compact framed, properly embedded $1$-dimensional manifold $T\subset \Sigma\times (0,1)$ such that  $\partial T \subset \mathcal{A}\times (0,1)$ and 
 for every point of $\partial T$ the framing is parallel to the $(0,1)$ factor and points to the direction of $1$.  Here, by framing, we refer to a section of the unitary normal bundle of $T$. The \textbf{height} of $(v,h)\in \Sigma\times (0,1)$ is $h$.  If $a\in \mathcal{A}$ is a boundary edge, we impose that no two points in $\partial_aT:= \partial T \cap a\times(0,1)$  have the same heights, hence the set $\partial_aT$ is totally ordered by the heights. Two tangles are isotopic if they are isotopic through the class of tangles that preserve the boundary height orders. By convention, the empty set is a tangle only isotopic to itself.
 \item Let $\pi : \Sigma\times (0,1)\to \Sigma$ be the projection with $\pi(v,h)=v$. A tangle $T$ is in a \textbf{standard position} if for each of its points, the framing is parallel to the $(0,1)$ factor and points in the direction of $1$ and is such that $\pi_{\big| T} : T\rightarrow \Sigma$ is an immersion with at most transversal double points in the interior of $\Sigma$. Every tangle is isotopic to a tangle in a standard position. We call \textbf{diagram}  the image $D=\pi(T)$ of a tangle in a standard position, together with the over/undercrossing information at each double point. An isotopy class of diagram $D$ together with a total order of $\partial_a D:=\partial D\cap a$ for each boundary arc $a$, define uniquely an isotopy class of tangle. When choosing an orientation $\mathfrak{o}(a)$ of a boundary edge $a$ and a diagram $D$, the set $\partial_aD$ receives a natural order by setting that the points are increasing when going in the direction of $\mathfrak{o}(a)$. Pictorially, we will depict tangles by drawing a diagram and an orientation (an arrow) for each boundary edge, as in Figure \ref{fig_statedtangle}. When a boundary edge $a$ is oriented we assume that $\partial_a D$ is ordered according to the orientation.
 \item A connected open diagram without double point is called an \textbf{arc}. A closed connected diagram without double point is called a \textbf{loop}.
 \item  A \textbf{state} of a tangle is a map $s:\partial T \rightarrow \{-, +\}$. A pair $(T,s)$ is called a \textbf{stated tangle}. We define a \textbf{stated diagram} $(D,s)$ in a similar manner.
 \end{enumerate}
 \end{definition}
 
 \begin{figure}[!h] 
\centerline{\includegraphics[width=6cm]{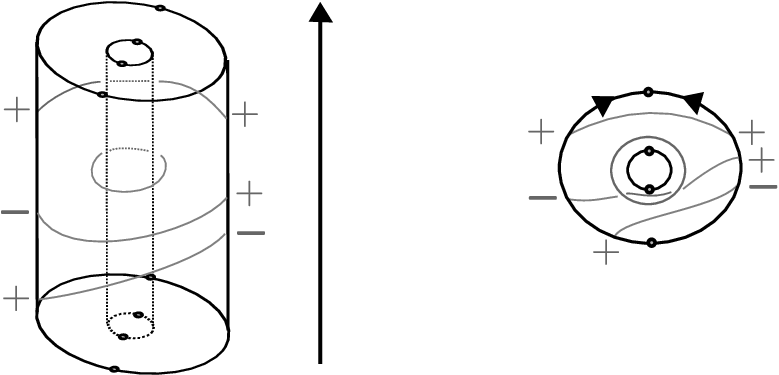} }
\caption{On the left: a stated tangle. On the right: its associated diagram. The arrows represent the height orders. } 
\label{fig_statedtangle} 
\end{figure}

 \begin{definition}[Stated skein algebras](\cite{BonahonWongqTrace, LeStatedSkein})\label{def_skein}
 Let $k$ be a commutative unital ring and $A^{1/2}\in k^{\times}$ and invertible element.
 The \textbf{stated skein algebra}  $\mathcal{S}_A(\mathbf{\Sigma})$ is the  free $k$-module generated by isotopy classes of stated tangles in $\mathbf{\Sigma}\times (0, 1)$ modulo the following relations \eqref{eq: skein 1} and \eqref{eq: skein 2}, 
  	\begin{equation}\label{eq: skein 1} 
\begin{tikzpicture}[baseline=-0.4ex,scale=0.5,>=stealth]	
\draw [fill=gray!45,gray!45] (-.6,-.6)  rectangle (.6,.6)   ;
\draw[line width=1.2,-] (-0.4,-0.52) -- (.4,.53);
\draw[line width=1.2,-] (0.4,-0.52) -- (0.1,-0.12);
\draw[line width=1.2,-] (-0.1,0.12) -- (-.4,.53);
\end{tikzpicture}
=A
\begin{tikzpicture}[baseline=-0.4ex,scale=0.5,>=stealth] 
\draw [fill=gray!45,gray!45] (-.6,-.6)  rectangle (.6,.6)   ;
\draw[line width=1.2] (-0.4,-0.52) ..controls +(.3,.5).. (-.4,.53);
\draw[line width=1.2] (0.4,-0.52) ..controls +(-.3,.5).. (.4,.53);
\end{tikzpicture}
+A^{-1}
\begin{tikzpicture}[baseline=-0.4ex,scale=0.5,rotate=90]	
\draw [fill=gray!45,gray!45] (-.6,-.6)  rectangle (.6,.6)   ;
\draw[line width=1.2] (-0.4,-0.52) ..controls +(.3,.5).. (-.4,.53);
\draw[line width=1.2] (0.4,-0.52) ..controls +(-.3,.5).. (.4,.53);
\end{tikzpicture}
\hspace{.5cm}
\text{ and }\hspace{.5cm}
\begin{tikzpicture}[baseline=-0.4ex,scale=0.5,rotate=90] 
\draw [fill=gray!45,gray!45] (-.6,-.6)  rectangle (.6,.6)   ;
\draw[line width=1.2,black] (0,0)  circle (.4)   ;
\end{tikzpicture}
= -(A^2+A^{-2}) 
\begin{tikzpicture}[baseline=-0.4ex,scale=0.5,rotate=90] 
\draw [fill=gray!45,gray!45] (-.6,-.6)  rectangle (.6,.6)   ;
\end{tikzpicture}
;
\end{equation}

\begin{equation}\label{eq: skein 2} 
\begin{tikzpicture}[baseline=-0.4ex,scale=0.5,>=stealth]
\draw [fill=gray!45,gray!45] (-.7,-.75)  rectangle (.4,.75)   ;
\draw[->] (0.4,-0.75) to (.4,.75);
\draw[line width=1.2] (0.4,-0.3) to (0,-.3);
\draw[line width=1.2] (0.4,0.3) to (0,.3);
\draw[line width=1.1] (0,0) ++(90:.3) arc (90:270:.3);
\draw (0.65,0.3) node {\scriptsize{$+$}}; 
\draw (0.65,-0.3) node {\scriptsize{$+$}}; 
\end{tikzpicture}
=
\begin{tikzpicture}[baseline=-0.4ex,scale=0.5,>=stealth]
\draw [fill=gray!45,gray!45] (-.7,-.75)  rectangle (.4,.75)   ;
\draw[->] (0.4,-0.75) to (.4,.75);
\draw[line width=1.2] (0.4,-0.3) to (0,-.3);
\draw[line width=1.2] (0.4,0.3) to (0,.3);
\draw[line width=1.1] (0,0) ++(90:.3) arc (90:270:.3);
\draw (0.65,0.3) node {\scriptsize{$-$}}; 
\draw (0.65,-0.3) node {\scriptsize{$-$}}; 
\end{tikzpicture}
=0,
\hspace{.2cm}
\begin{tikzpicture}[baseline=-0.4ex,scale=0.5,>=stealth]
\draw [fill=gray!45,gray!45] (-.7,-.75)  rectangle (.4,.75)   ;
\draw[->] (0.4,-0.75) to (.4,.75);
\draw[line width=1.2] (0.4,-0.3) to (0,-.3);
\draw[line width=1.2] (0.4,0.3) to (0,.3);
\draw[line width=1.1] (0,0) ++(90:.3) arc (90:270:.3);
\draw (0.65,0.3) node {\scriptsize{$+$}}; 
\draw (0.65,-0.3) node {\scriptsize{$-$}}; 
\end{tikzpicture}
=A^{-1/2}
\begin{tikzpicture}[baseline=-0.4ex,scale=0.5,>=stealth]
\draw [fill=gray!45,gray!45] (-.7,-.75)  rectangle (.4,.75)   ;
\draw[-] (0.4,-0.75) to (.4,.75);
\end{tikzpicture}
\hspace{.1cm} \text{ and }
\hspace{.1cm}
A^{1/2}
\heightexch{->}{-}{+}
- A^{5/2}
\heightexch{->}{+}{-}
=
\heightcurve.
\end{equation}
The product of two classes of stated tangles $[T_1,s_1]$ and $[T_2,s_2]$ is defined by  isotoping $T_1$ and $T_2$  in $\Sigma \times (1/2, 1) $ and $\Sigma \times (0, 1/2)$ respectively and then setting $[T_1,s_1]\cdot [T_2,s_2]=[T_1\cup T_2, s_1\cup s_2]$. Figure \ref{fig_product} illustrates this product.
\par An embedding $f: \mathbf{\Sigma_1} \to \mathbf{\Sigma}_2$ induces a linear map $f_*: \mathcal{S}_A(\mathbf{\Sigma}_1) \to \mathcal{S}_A(\mathbf{\Sigma}_2)$ defined by $f_*([T,s]):= [\overline{f}(T), s \circ \overline{f}^{-1}]$ where $\overline{f}: \Sigma_1\times [0,1] \to \Sigma_2\times [0,1]$ is defined by $\overline{f}(x,t):=(f(x),t)$. Moreover, if $f$ is strong, then $f_*$ is a morphism of algebras.
Stated skein algebras define  symmetric monoidal functors $\mathcal{S}_A : \mathcal{MS}^{str} \to \Alg_k$ and $\mathcal{S}_A: \mathcal{MS}\to \Mod_k$.
\end{definition}
\par For an unmarked surface, $\mathcal{S}_A(\mathbf{\Sigma})$ coincides with the classical  Kauffman-bracket skein algebra.

\begin{figure}[!h] 
\centerline{\includegraphics[width=8cm]{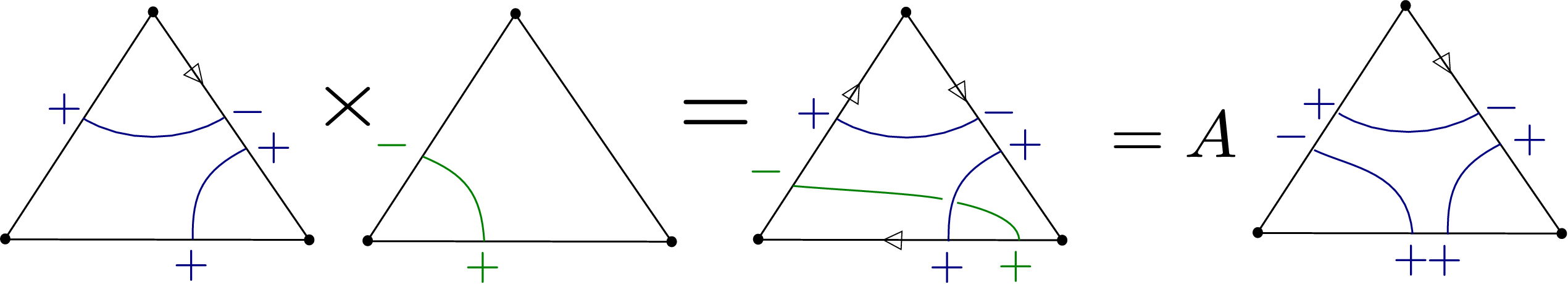} }
\caption{An illustration of the product in stated skein algebras.} 
\label{fig_product} 
\end{figure}

 \begin{definition}[Reduced stated skein algebras]
 \begin{enumerate}
 \item Let $\mathbf{\Sigma}$ a marked surface and $p\in \mathcal{P}^{\partial}$ a boundary puncture between two consecutive boundary edges $a$ and $b$ on the same boundary component $\partial\in \Gamma^{\partial}$.
The orientation of $\Sigma$ induces an orientation of $\partial$ so a cyclic ordering of the elements of $\partial \cap \mathcal{A}$ we suppose that $b$ is followed by $a$ in this ordering.  We denote by $\alpha(p)$ the arc with one endpoint $v_a\in a$ and one endpoint $v_b \in b$ such that $\alpha(p)$ can be isotoped inside $\partial$. $\alpha(p)$ is called a \textbf{corner arc}.
Let $\alpha(p)_{ij}\in \mathcal{S}_A(\mathbf{\Sigma})$ be the class of the stated arc $(\alpha(p), s)$ where $s(v_a)=i$ and $s(v_b)=j$.
 \item We call \textbf{bad arc} associated to $p$ the element $\alpha(p)_{-+}\in \mathcal{S}_A(\mathbf{\Sigma})$ (see Figure \ref{fig_bad_arc}). The \textbf{reduced stated skein algebra} $\overline{\mathcal{S}}_A(\mathbf{\Sigma})$ is the quotient of $\mathcal{S}_A(\mathbf{\Sigma})$ by the ideal generated by all bad arcs. We denote by  $\overline{\mathcal{S}}_A : \mathcal{MS}^{str} \to \Alg_k$ and $\overline{\mathcal{S}}_A: \mathcal{MS}\to \Mod_k$ the reduced stated skein functors.
 \end{enumerate}
 \end{definition}

 \begin{figure}[!h] 
\centerline{\includegraphics[width=4cm]{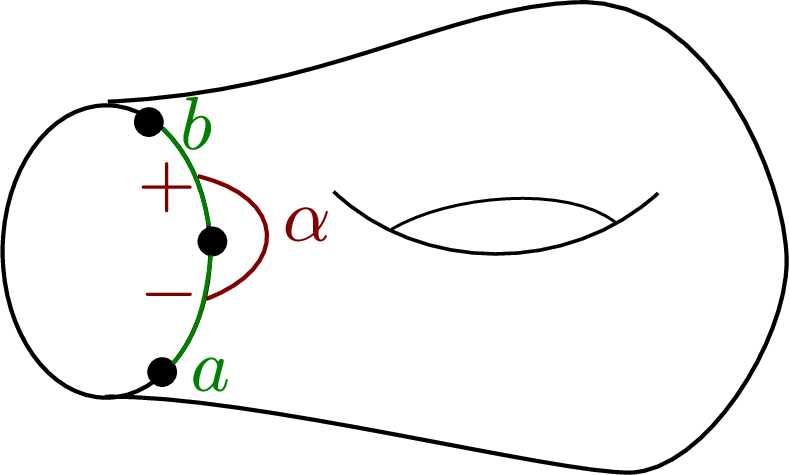} }
\caption{A bad arc.} 
\label{fig_bad_arc} 
\end{figure} 
 
 \subsubsection{Splitting morphisms}
 
 The main interest in extending skein algebras to marked surfaces is the existence of splitting morphisms which we now describe.
  Let $a$, $b$ be two distinct boundary edges of $\mathbf{\Sigma}$, denote by $\pi : \Sigma\rightarrow \Sigma_{a\#b}$ the projection and $c:=\pi(a)=\pi(b)$. Let $(T_0, s_0)$ be a stated framed tangle of $\Sigma_{a\#b}\times (0,1)$ transverse to $c\times (0,1)$ and such that the heights of the points of $T_0 \cap c\times (0,1)$ are pairwise distinct and the framing of the points of $T_0 \cap c\times (0,1)$ is vertical towards $1$. Let $T\subset \Sigma \times (0,1)$ be the framed tangle obtained by cutting $T_0$ along $c$. 
Any two states $s_a : \partial_a T \rightarrow \{-,+\}$ and $s_b : \partial_b T \rightarrow \{-,+\}$ give rise to a state $(s_a, s, s_b)$ on $T$. 
Both the sets $\partial_a T$ and $\partial_b T$ are in canonical bijection with the set $T_0\cap c$ by the map $\pi$. Hence the two sets of states $s_a$ and $s_b$ are both in canonical bijection with the set $\mathrm{St}(c):=\{ s: c\cap T_0 \rightarrow \{-,+\} \}$. 

\begin{definition}[Splitting morphism]\label{def_gluing_map}
The \textbf{splitting morphism} $\theta_{a\#b}: \mathcal{S}_{A}(\mathbf{\Sigma}_{a\#b}) \rightarrow \mathcal{S}_{A}(\mathbf{\Sigma})$ is the linear map given, for any $(T_0, s_0)$ as above, by: 
$$ \theta_{a\#b} \left( [T_0,s_0] \right) := \sum_{s \in \mathrm{St}(c)} [T, (s, s_0 , s) ].$$
\end{definition}

\begin{figure}[!h] 
\centerline{\includegraphics[width=8cm]{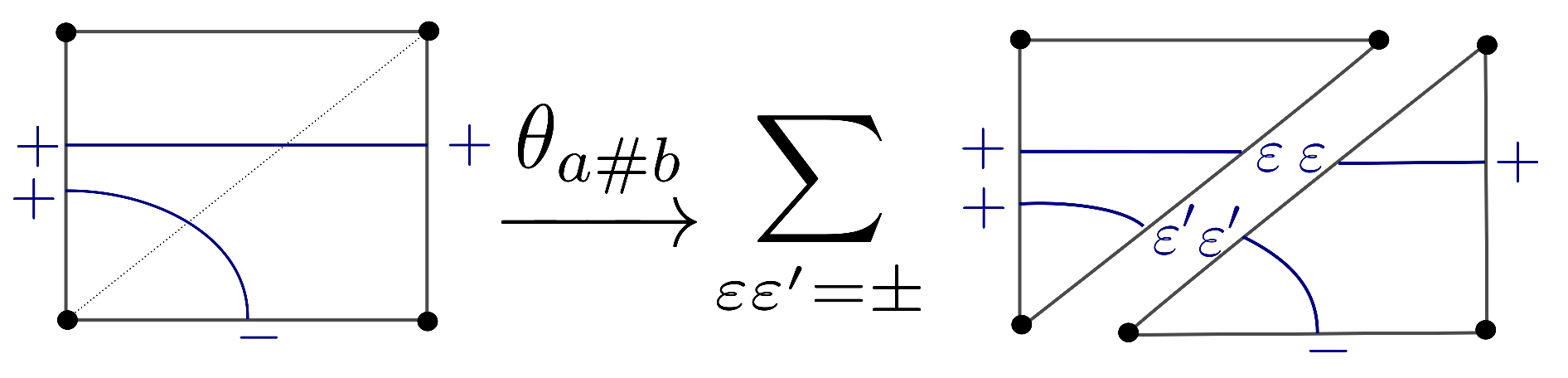} }
\caption{An illustration of the splitting morphism $\theta_{a\#b}$.} 
\label{fig_gluingmap} 
\end{figure} 

\begin{theorem}\label{theorem_gluing}(\cite[Theorem $3.1$]{LeStatedSkein}, \cite[Theorem $7.6$]{CostantinoLe19})
 The linear map $\theta_{a\#b}: \mathcal{S}_A(\mathbf{\Sigma}_{a\#b}) \hookrightarrow \mathcal{S}_A(\mathbf{\Sigma})$ is an injective morphism of algebras. It passes to the quotient to define an injective morphism (still denoted by the same letter) 
 $\theta_{a\#b}: \overline{\mathcal{S}}_A(\mathbf{\Sigma}_{a\#b}) \hookrightarrow \overline{\mathcal{S}}_A(\mathbf{\Sigma})$. 
   Moreover the gluing operation is coassociative in the sense that if $a,b,c,d$ are four distinct boundary edges, then we have $\theta_{a\#b} \circ \theta_{c\#d} = \theta_{c\#d} \circ \theta_{a\#b}$.
\end{theorem}

\subsubsection{Comodule structures}\label{sec_comodule}

Recall that the bigon $\mathbb{B}$ is a disc with two boundary edges  $a_L$ and  $a_R$. While gluing two bigons $\mathbb{B}$ and $\mathbb{B}'$ together along $a_R \# a_L'$, we get another bigon. Thus we have a  coproduct
$$ \Delta:= \theta_{a_R\# a_L'} : \mathcal{S}_A(\mathbb{B}) \to \mathcal{S}_A(\mathbb{B})^{\otimes 2}, $$
which endows $\mathcal{S}_A(\mathbb{B})$ with a structure of Hopf algebra. 

Let $\alpha_{ij}:= \adjustbox{valign=c}{\includegraphics[width=0.8cm]{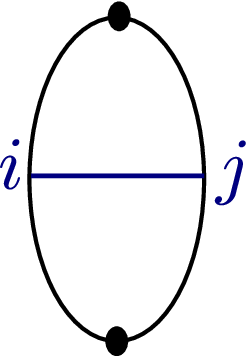}} \in \mathcal{S}_A(\mathbb{B})$ for $i,j = \pm$. Write $q:=A^2$. 

\begin{theorem}\label{theorem_bigon}(\cite{KojuQuesneyClassicalShadows}, \cite{CostantinoLe19}) The  Hopf algebra $\mathcal{S}_A(\mathbb{B})$ is isomorphic to $\mathcal{O}_q[\SL_2]$ through an isomorphism sending $\alpha_{++}$, $\alpha_{+-}$, $\alpha_{-+}$, $\alpha_{--}$ to $a,b,c,d$ respectively.
\end{theorem}

For a general marked surface $\mathbf{\Sigma}$ with a boundary edge $c$ , since $\left( \mathbb{B} \cup \mathbf{\Sigma}\right)_{a_R\# c} =\mathbf{\Sigma}$ and $\left(\mathbf{\Sigma}\cup \mathbb{B}\right)_{c\# a_L} = \mathbf{\Sigma}$, we have some left and right comodule maps 
$$\Delta_c^L := \theta_{a_R\#c} : \mathcal{S}_A(\mathbf{\Sigma}) \to \mathcal{O}_q[\SL_2] \otimes \mathcal{S}_A(\mathbf{\Sigma})$$ and 
$$\Delta_c^R:= \theta_{c\#a_L}: \mathcal{S}_A(\mathbf{\Sigma}) \to \mathcal{S}_A(\mathbf{\Sigma})\otimes \mathcal{O}_q[\SL_2] , $$
as illustrated in Figure \ref{fig_comodule_maps}. 

\begin{figure}[!h] 
\centerline{\includegraphics[width=11cm]{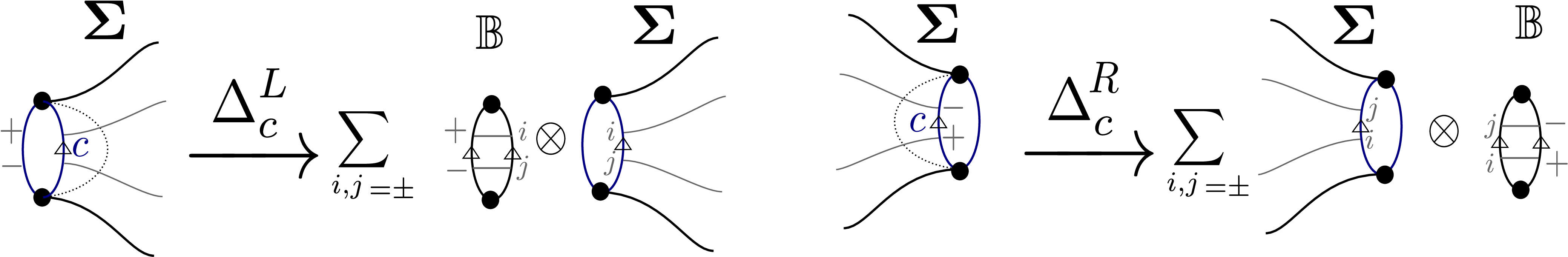} } 
\caption{Comodule maps.} 
\label{fig_comodule_maps} 
\end{figure} 

The following theorem characterizes the image of the gluing map and  was proved independently in \cite{CostantinoLe19} and \cite{KojuQuesneyClassicalShadows}.

\begin{theorem}\label{theorem_exactsequence}(\cite[Theorem $4.7$]{CostantinoLe19}, \cite[Theorem $1.1$]{KojuQuesneyClassicalShadows})
Let $\mathbf{\Sigma}$ be a punctured surface and $a,b$ two boundary edges. The following sequence is exact: 
$$ 0 \to \mathcal{S}_A(\mathbf{\Sigma}_{a\#b}) \xrightarrow{\theta_{a\#b}} \mathcal{S}_A(\mathbf{\Sigma}) \xrightarrow{\Delta^L_a - \sigma \circ \Delta^R_b}\mathcal{O}_q[\SL_2] \otimes \mathcal{S}_A(\mathbf{\Sigma}), $$
where $\sigma(x\otimes y) := y\otimes x$. 
\end{theorem}

Note that by composing the $\Delta_a^L$ for $a\in \mathcal{A}$, we get a comodule map $\Delta^L : \mathcal{S}_A(\mathbf{\Sigma}) \to (\mathcal{O}_q[\SL_2])^{\otimes \mathcal{A}} \otimes \mathcal{S}_A(\mathbf{\Sigma})$. 
By abuse of notations, we will identify $\mathcal{S}_A(\mathbb{B})$ with $\mathcal{O}_q\SL_2$. Note that both $\alpha_{-+}=c$ and $\alpha_{+-}=b$ are bad arcs and the reduced stated skein algebra $\overline{\mathcal{S}}_A(\mathbb{B})$ is isomorphic to the Hopf algebra $\mathbb{C}[X]$ with coproduct $\Delta(X)=X\otimes X$ through the isomorphism identifying $\alpha_{++}$ and $\alpha_{--}$ with $X$ and $X^{-1}$ respectively. Therefore, the splitting morphism defines a comodule map 
$$ \Delta^L: \overline{\mathcal{S}}_A(\mathbf{\Sigma}) \to (\mathbb{C}[X^{\pm 1}])^{\otimes \mathcal{A}}\otimes \overline{\mathcal{S}}_A(\mathbf{\Sigma})$$
which we call the \textbf{toric coaction}.

\subsubsection{Bases and gradings}

 We now define bases. Let us first make a preliminary remark: if $T$ is a tangle in standard position and $D$ its planar diagram projection, we cannot recover the isotopy class $T$ from the isotopy class of  $D$ since we lost the information about the height orders of the sets $\partial_aT$, $a\in \mathcal{A}$ in the projection.
 
 \begin{convention}\label{convention_diagram} Let $(D,s)$ be a stated diagram in $\mathbf{\Sigma}$. Recall that each boundary edge $a\in \mathcal{A}$ receives an orientation from the orientation of $\Sigma$ and that for $v,w \in a$ we write $v<_{\mathfrak{o}^+} w$ if $a$ is oriented from $v$ to $w$. Let $(T,s)$ be a stated tangle such that $(1)$ $(T,s)$ is in standard position (in the sense of Definition \ref{def_tangles}) and its planar projection is $(D,s)$ and $(2)$ for every $a\in \mathcal{A}$, if $v, w \in \partial_a D$ are such that $v<_{\mathfrak{o}^+}w$ then $h(v)<h(w)$. Conditions $(1)$ and $(2)$ completely determine the isotopy class of $(T,s)$ and we will write $[D,s]:= [T,s]\in {\mathcal{S}}_A(\mathbf{\Sigma})$. 
 \end{convention}
 
 \begin{definition}[Bases]\label{def_basis}
 \begin{enumerate}
 \item A closed component of a diagram $D$ is trivial if it bounds an embedded disc in $\Sigma$. An open component of $D$ is trivial if it can be isotoped, relatively to its boundary, inside some boundary edge. A diagram is \textbf{simple} if it has neither double point nor trivial component. By convention, the empty set is a simple diagram. Let $\mathfrak{o}^+$ denote the orientation of the boundary edges of $\mathbf{\Sigma}$ induced by the orientation of $\Sigma$ as in Definition \ref{def_marked_surfaces}. For each boundary edge $a$ we write $<_{\mathfrak{o}^+}$ the induced total order on $\partial_a D$. A state $s: \partial D \rightarrow \{ - , + \}$ is $\mathfrak{o}^+-$\textbf{increasing} if for any boundary edge $a$ and any two points $x,y \in \partial_a D$, then $x<_{\mathfrak{o}^+} y$ implies $s(x)\leq s(y)$, with the convention $- < +$. 
 \item We denote by $\mathcal{B}\subset \overline{\mathcal{S}}_{A}(\mathbf{\Sigma})$  the set of classes $[D,s]$ of stated diagrams such that $(i)$ $D$ is simple and $(ii)$  $s$ is $\mathfrak{o}^+$-increasing. We also denote by $\overline{\mathcal{B}}\subset \overline{\mathcal{S}}_{A}(\mathbf{\Sigma})$ the set of classes $[D,s]$ satisfying $(i)$ and $(ii)$ and such that $(iii)$ $(D,s)$ does not contain bad arc component.
 \end{enumerate}
 \end{definition}
 
 \begin{theorem}(\cite[Theorem $2.11$ ]{LeStatedSkein}\cite[Theorem $7.1$]{CostantinoLe19}) $\mathcal{B}$ and $\overline{\mathcal{B}}$  are  bases of ${\mathcal{S}}_{A}(\mathbf{\Sigma})$ and $\overline{\mathcal{S}}_{A}(\mathbf{\Sigma})$ respectively.\end{theorem}

\begin{definition}[Homological grading] For $D$ a diagram, we denote by $[D]\in \mathrm{H}_1(\Sigma, \mathcal{A}; \mathbb{Z}/2\mathbb{Z})$ its homology class in $\Sigma$ relative to $\mathcal{A}$ with coefficients in $ \mathbb{Z}/2\mathbb{Z}$. For $c \in \mathrm{H}_1(\Sigma, \mathcal{A}; \mathbb{Z}/2\mathbb{Z})$, we write 
$$ \mathcal{S}_A(\mathbf{\Sigma}, c) := \Span \left( [D,s]\in \mathcal{B} : [D]=c\right).$$
The submodule $ \overline{\mathcal{S}}_A(\mathbf{\Sigma}, c) \subset \overline{\mathcal{S}}_A(\mathbf{\Sigma})$ is defined similarly.
\end{definition}

Note that the skein relations in Definition \ref{def_skein} preserve the homology class in $\mathrm{H}_1(\Sigma, \mathcal{A}; \mathbb{Z}/2\mathbb{Z})$ (i.e. if $\sum_i c_i [D_i,s_i]=0$ is a skein relation then all $D_i$ have the same class $[D_i]$) so for any stated diagram $(D,s)$ (not necessarily simple) the class $[D,s]$ belongs to the module $ \mathcal{S}_A(\mathbf{\Sigma}, c)$ with $c=[D]$. We deduce that $\mathcal{S}_A(\mathbf{\Sigma})= \oplus_{c \in \mathrm{H}_1(\Sigma, \mathcal{A}; \mathbb{Z}/2\mathbb{Z})} \mathcal{S}_A(\mathbf{\Sigma}, c)$ is an algebra graduation in the sense that for $x_1 \in  \mathcal{S}_A(\mathbf{\Sigma}, c_1) $ and  $x_2 \in  \mathcal{S}_A(\mathbf{\Sigma}, c_2) $ then $x_1x_2 \in   \mathcal{S}_A(\mathbf{\Sigma}, c_1+c_2) $. The algebra $\overline{\mathcal{S}}_A(\mathbf{\Sigma})$ is $\mathrm{H}_1(\Sigma, \mathcal{A}; \mathbb{Z}/2\mathbb{Z})$ graded in the same manner.

\subsubsection{Finite presentations of stated skein algebras}

\begin{definition}[Fundamental groupoids and presenting graphs]
 Let $\mathbf{\Sigma}=(\Sigma, \mathcal{A})$ be an essential  marked surface.
 \begin{enumerate}
 \item The \textbf{fundamental groupoid} $\Pi_1(\Sigma)$ is the groupoid whose objects are points in $\Sigma$ and whose morphisms $\beta: v_1 \to v_2$ are homotopy classes of path $c_{\beta}: [0,1] \to \Sigma$ such that $c_{\beta}(0)=v_1$ and $c_{\beta}(1)=v_2$. We write $v_1=s(\beta)$ (the source) and $v_2= t(\beta)$ (the target). The composition is the concatenation of paths and the unit at $v\in \Sigma$ is the class $1_v$ of the constant path. For $\beta: v_1 \to v_2$, we denote by $\beta^{-1}: v_2 \to v_1$ the class of the path $c_{\beta^{-1}}(t) = c_{\beta}(1-t)$ (so $\beta \beta^{-1}=1_{t(\beta)}$). For $a\in \mathcal{A}$ we denote by $v_a \in a$ the middle point $v_a:= \varphi_a(1/2)$ and denote by $\mathbb{V}=\{ v_a, a\in \mathcal{A}\}$ the set of such points. $\Pi_1(\Sigma, \mathbb{V})$ is the full subcategory of $\Pi_1(\Sigma)$ generated by $\mathbb{V}$. By abuse of notation, we also denote by $\Pi_1(\Sigma, \mathbb{V})$ the set of morphisms of the underlying category.
\item A \textbf{presenting graph} $\Gamma$ for $\mathbf{\Sigma}$ is an embedded oriented graph $\Gamma \subset \Sigma$ whose set of vertices is $\mathbb{V}$ and such that $\Sigma$ retracts on $\Gamma$. We denote by $\mathcal{E}(\Gamma)$ the set of its oriented edges whose elements are seen as paths in $\Pi_1(\Sigma, \mathbb{V})$. 
 \item An oriented arc $\alpha$  in $\mathbf{\Sigma}$ naturally defines a path in $\Pi_1(\Sigma, \mathbb{V})$ which we abusively also denote by $\alpha$. For $i,j \in \{-, +\}$, we denote by $\alpha_{ij}\in \mathcal{S}_A(\mathbf{\Sigma})$ the class of the arc $\alpha$ with state $i$ at its source point $s(\alpha)$ at state $j$ at its target point $t(\alpha)$. 
 \end{enumerate}
 \end{definition}
 
 For instance, consider the arcs in $\mathbb{D}_n$ drawn in Figure \ref{fig_arcs_Dn}. Then $\mathbb{D}_1$ admits a presenting graph $\Gamma$ whose edges are the two arcs $\mathcal{E}(\Gamma)=\{\alpha, \beta\}$. Alternatively, another presenting graph is given by $\mathcal{E}(\Gamma')=\{\alpha, \alpha_p\}$. Also $\mathbb{D}_n$ admits the presenting graph $\Gamma$ defined by $\mathcal{E}(\Gamma)=\{ \alpha^{(n)}, \alpha_{p_1}, \ldots, \alpha_{p_n}\}$. 
 
 \begin{figure}[!h] 
\centerline{\includegraphics[width=7cm]{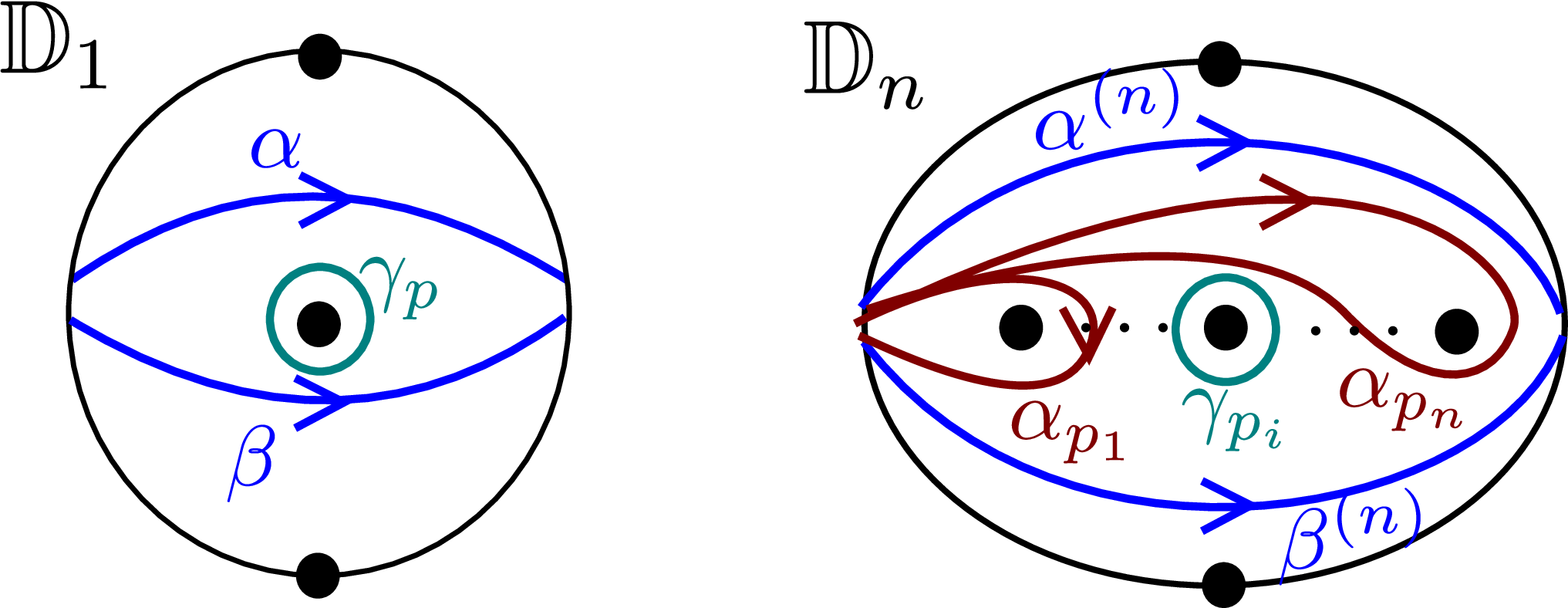} }
\caption{Some arcs and curves in $\mathbb{D}_n$.} 
\label{fig_arcs_Dn} 
\end{figure}  

\begin{theorem}(\cite[Proposition $3.4$]{KojuAzumayaSkein}, \cite[Theorem $1.1$]{KojuPresentationSSkein})\label{theorem_presentation_sskein} Let $\mathbf{\Sigma}$ be an essential marked surface with presenting graph $\Gamma$.  Then $\mathcal{S}_A(\mathbf{\Sigma})$ admits a finite presentation whose set of generators is $\mathbb{G}=\{ \alpha_{ij}: \alpha \in \mathcal{E}(\Gamma), i,j=\pm\}$ with relations described in \cite{KojuPresentationSSkein}.
\end{theorem}

We will mainly be interested in the presentation of $\mathcal{S}_A(\mathbb{D}_1)$ derived from the presenting graph $\Gamma$ with $\mathcal{E}(\Gamma)=\{\alpha, \beta\}$ so let us describe this case explicitly. Set 
$$ M(\alpha)= \begin{pmatrix} \alpha_{++} & \alpha_{+-} \\ \alpha_{-+} & \alpha_{--} \end{pmatrix}, \quad M(\beta)=\begin{pmatrix} \beta_{++} & \beta_{+-} \\ \beta_{-+} & \beta_{--} \end{pmatrix}, \quad \mathscr{R} =  
 \begin{pmatrix} A & 0 & 0 & 0 \\ 0 & 0 &A^{-1} & 0 \\ 0 & A^{-1} & A-A^{-3} & 0 \\ 0 & 0 & 0 & A \end{pmatrix}.$$
 Define the \textit{quantum determinant} $\det_q$ and the \textit{Kronecker product} $\odot$ by the formulas
 $$ {\det}_q \begin{pmatrix} a & b \\ c & d \end{pmatrix} := ad-q^{-1}bc, \quad M(\alpha) \odot M(\beta) = 
\begin{pmatrix}
 \alpha_{++} \beta_{++} & \alpha_{++} \beta_{+-} & \alpha_{+-} \beta_{++} & \alpha_{+-} \beta_{+-} \\
 \alpha_{++} \beta_{-+} &\alpha_{++} \beta_{--} &\alpha_{+-} \beta_{-+} &\alpha_{+-} \beta_{--}  \\
 \alpha_{-+} \beta_{++} &\alpha_{-+} \beta_{+-} &\alpha_{--} \beta_{++} &\alpha_{--} \beta_{+-} \\
 \alpha_{-+} \beta_{-+} &\alpha_{-+} \beta_{--} &\alpha_{--} \beta_{-+} &\alpha_{--} \beta_{--} 
 \end{pmatrix}.$$
 
 Specializing to $\mathbb{D}_1$ with presenting graph $\Gamma$, Theorem \ref{theorem_presentation_sskein} states that $\mathcal{S}_A(\mathbb{D}_1)$ is generated by the elements $\alpha_{ij}, \beta_{ij}$ for $i,j\in \{-,+\}$ with relations 
 \begin{equation}\label{eq_pres_D1}
 \mathscr{R} \left(M(y)\odot M(x)\right) = \left(M(x)\odot M(y)\right)\mathscr{R}\mbox{, for }(x,y) \in \{ (\alpha, \alpha), (\beta, \beta), (\alpha, \beta)\}  \mbox{ and } {\det}_q(M(\alpha))={\det}_q(M(\beta))=1.
 \end{equation}
 
 \begin{corollary}\label{coro_pres_D1}
$\mathcal{S}_A(\mathbb{D}_1)$ is generated by the elements $\alpha_{ij}, \beta_{ij}$ for $i,j\in \{-,+\}$ with relations 

\begin{align*}
\left.
\begin{array}{llll}
x_{++}x_{+-} &= q^{-1}x_{+-}x_{++} & x_{++}x_{-+}&=q^{-1}x_{-+}x_{++} 
\\ x_{--}x_{+-} &= q x_{+-}x_{--} & x_{--}x_{-+}&=q x_{-+}x_{--} 
\\ x_{++}x_{--}&=1+q^{-1}x_{+-}x_{-+} &  x_{--}x_{++}&=1 + q x_{+-}x_{-+} 
\\ x_{-+}x_{+-}&=x_{+-}x_{-+} & &
\end{array}
\right\} \mbox{, for }x \in \{\alpha, \beta\};  \\
\alpha_{\varepsilon \varepsilon'} \beta_{\varepsilon \varepsilon'}=\beta_{\varepsilon \varepsilon'} \alpha_{\varepsilon \varepsilon'}; \quad 
\beta_{+-}\alpha_{-+}=\alpha_{-+}\beta_{+-};  \\
\beta_{\varepsilon -}\alpha_{\varepsilon +}=q\alpha_{\varepsilon +}\beta_{\varepsilon -}; \quad
\alpha_{- \varepsilon}\beta_{+\varepsilon}= q \beta_{+\varepsilon} \alpha_{- \varepsilon};  \\
q^{-1}\beta_{\varepsilon +}\alpha_{\varepsilon -} - \alpha_{\varepsilon -}\beta_{\varepsilon +}= - (q-q^{-1})\alpha_{\varepsilon +}\beta_{\varepsilon -}; \\
q^{-1}\alpha_{+ \varepsilon}\beta_{-\varepsilon} - \beta_{-\varepsilon}\alpha_{+ \varepsilon}= -(q-q^{-1})\beta_{+ \varepsilon}\alpha_{-\varepsilon}; \\
\beta_{--}\alpha_{++}-\alpha_{++}\beta_{--} = (q-q^{-1})\beta_{+-} \alpha_{-+}; \\
\alpha_{--}\beta_{++} - \beta_{++}\alpha_{--} = (q-q^{-1})\alpha_{-+}\beta_{+-}; \\
\beta_{-+}\alpha_{+-}-\alpha_{+-}\beta_{-+}= -(q-q^{-1})(\alpha_{++}\beta_{--} - \beta_{++}\alpha_{--});
\end{align*}
for $\varepsilon, \varepsilon' \in \{-, +\}$. 
\end{corollary}

\begin{proof} The relations of Corollary \ref{coro_pres_D1} are obtained by developing the relations of Equation \eqref{eq_pres_D1}.

\end{proof}

\begin{remark} 
The Hopf algebra $\mathcal{S}_A(\mathbb{D}_1)$ is very similar to the Hopf algebra defined by Bigelow in \cite{BigelowQGroups} where our parameter $q$ is replaced by $-q$ in \cite{BigelowQGroups} and our generators $\alpha_{++}$, $\alpha_{+-}$, $\alpha_{-+}$, $\alpha_{--}$, $\beta_{++}$, $\beta_{+-}$, $\beta_{-+}$, $\beta_{--}$ are replaced by $k', e, e_0, k, l, f_0, f, l'$ respectively in \cite{BigelowQGroups}. Using this correspondence, the reader will find in \cite{BigelowQGroups} a nice topological interpretation of both the antipode and  the equality $\mu \circ (S\otimes \id)\circ \Delta = \eta\circ \epsilon$.
\end{remark}

\subsection{Skein interpretation of the  Drinfeld double of the quantum Borel algebra}\label{sec_QG_skein}

 Inspired by Bigelow's construction in \cite{BigelowQGroups}, we now endow the algebra $\mathcal{S}_A(\mathbb{D}_1)$ with a Hopf algebra structure.
Recall that  the $n$-th punctured bigon $\mathbb{D}_n=(D_n, \{b_L, b_R\})=\adjustbox{valign=c}{\includegraphics[width=0.7cm]{Dn.eps}} $ is a disc with two boundary edges and $n$ open subdiscs removed from its interior. So $\mathbb{D}_2$ is obtained from $\mathbb{D}_1$ by removing an inner puncture. Let $i: \mathbb{D}_1 \hookrightarrow \mathbb{D}_2$ be inclusion and $i_*: \mathcal{S}_A(\mathbb{D}_1) \hookrightarrow \mathcal{S}_A(\mathbb{D}_2)$ the induced morphism (so $i_*([D,s]):=[\iota(D),s]$). Alternatively, we can think of $i_*$ as the morphism which "\textit{doubles the inner puncture}". Like with the bigon, by gluing two punctured bigons  $\mathbb{D}_1$ and $\mathbb{D}_1'$ while identifying $b_R$ with $b'_L$ we obtain $\mathbb{D}_2$, so we have a splitting morphism $\theta_{b_R \#b'_L}: \mathcal{S}_A(\mathbb{D}_2) \to \mathcal{S}_A(\mathbb{D}_1)^{\otimes 2}$. Let us define a coproduct on $\mathcal{S}_A(\mathbb{D}_1)$ as the composition
$$ \Delta: \mathcal{S}_A(\mathbb{D}_1) \xrightarrow{ \iota_*} \mathcal{S}_A(\mathbb{D}_1) \xrightarrow{ \theta_{b_R \# b'_L}} \mathcal{S}_A(\mathbb{D}_1)^{\otimes 2}.$$

 \begin{figure}[!h] 
\centerline{\includegraphics[width=9cm]{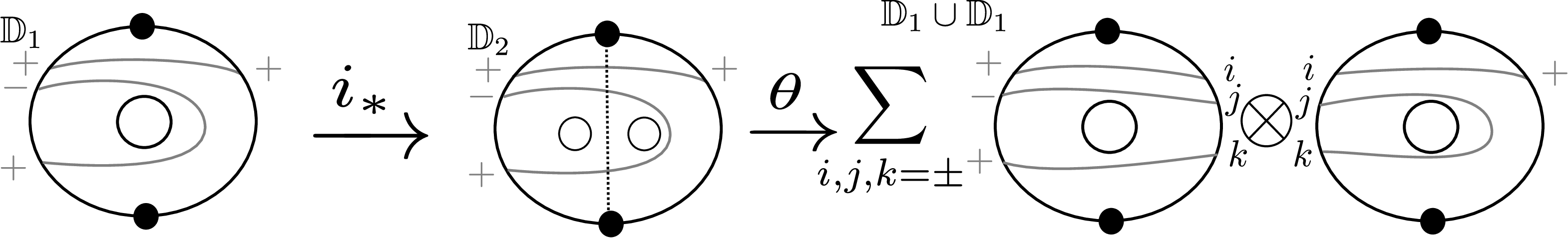} }
\caption{An illustration of the coproduct in $\mathcal{S}_A(\mathbb{D}_1)$.} 
\label{fig_skein_D1} 
\end{figure} 

It follows from the coassociativity of the gluing maps (Theorem \ref{theorem_gluing}) that $\Delta$ is coassociative. The algebra $\mathcal{S}_A(\mathbb{D}_1)$ is generated by the stated arcs $\alpha_{\varepsilon \varepsilon'}$ and $\beta_{\varepsilon \varepsilon'}$ drawn in Figure \ref{fig_skein_D1} and, for $\delta\in \{\alpha, \beta\}$, one has:

$$
 \begin{pmatrix} \Delta (\delta_{++}) & \Delta (\delta_{+-}) \\ \Delta(\delta_{-+}) & \Delta(\delta_{--}) \end{pmatrix} 
 = 
 \begin{pmatrix} \delta_{++} & \delta_{+-} \\ \delta_{-+} & \delta_{--} \end{pmatrix} 
 \otimes 
 \begin{pmatrix} \delta_{++} & \delta_{+-} \\ \delta_{-+} & \delta_{--} \end{pmatrix}. 
$$
 Further define the morphisms $\epsilon : \mathcal{S}_A(\mathbb{D}_1) \rightarrow \mathbb{C}$ and $S: \mathcal{S}_A(\mathbb{D}_1) \rightarrow \mathcal{S}_A(\mathbb{D}_1)$ by the formulas for $\delta\in \{\alpha, \beta\}$:
 
$$
 \begin{pmatrix} \epsilon(\delta_{++}) & \epsilon(\delta_{+-}) \\ \epsilon(\delta_{-+}) & \epsilon(\delta_{--}) \end{pmatrix} =
\begin{pmatrix} 1 &0 \\ 0& 1 \end{pmatrix}  
\text{ and }
\begin{pmatrix} S(\delta_{++}) & S(\delta_{+-}) \\ 	S(\delta_{-+}) & S(\delta_{--}) \end{pmatrix} 
	= 
	\begin{pmatrix} \delta_{--} & -q\delta_{+-} \\ -q^{-1}\delta_{-+} & \delta_{++} \end{pmatrix} .
$$

\begin{lemma} The coproduct $\Delta$, the counit $\epsilon$ and the antipode $S$ endow $\mathcal{S}_A(\mathbb{D}_1)$ with an Hopf algebra structure. 
\end{lemma}

\begin{proof} This follows from straightforward computations. \end{proof}

\begin{definition}[Drinfeld double quantum Borel algebra] 
The \textbf{Drinfeld double quantum Borel algebra} $\Uq$ is the Hopf algebra generated by elements $\overline{E},\overline{F}, K^{\pm 1/2}, L^{\pm 1/2}$ and relations 
\begin{align*}
&\overline{E}K^{1/2}= q^{-1}K^{1/2}\overline{E}; \quad \overline{E}L^{1/2}= qL^{1/2} \overline{E}; \quad  \overline{F}K^{1/2}=qK^{1/2}F; \quad \overline{F}L^{1/2}= q^{-1}L^{1/2}\overline{F}; \\
&xy=yx \mbox{, for all }x,y\in \{K^{\pm 1/2}, L^{\pm 1/2} \}; \quad K^{1/2}K^{-1/2}=L^{1/2}L^{-1/2}= 1; \\
&\overline{E}\overline{F}-\overline{F}\overline{E}=(q-q^{-1}) (K-L). 
\end{align*}
Here $K:= (K^{1/2})^2$ and $L:=(L^{1/2})^2$. 
The Hopf algebra structure is described by the following formulas:

\begin{align*}
& \Delta(\overline{E})= 1\otimes \overline{E} + \overline{E}\otimes K; \quad \Delta(\overline{F})= \overline{F}\otimes 1 +L\otimes \overline{F}; \\
& \Delta(x)=x\otimes x \mbox{, for all } x\in \{K^{\pm 1/2}, L^{\pm 1/2} \}; \\
& \epsilon(\overline{E})=\epsilon(F)=0; \quad \epsilon(K^{\pm 1/2})= \epsilon(L^{\pm 1/2})= 1; \\
& S(\overline{E})=-\overline{E}K^{-1}; \quad S(\overline{F})= -L^{-1}\overline{F}; \quad S(x)=x^{-1}\mbox{, for all } x\in \{K^{\pm 1/2}, L^{\pm 1/2} \}.
\end{align*}
\end{definition}

\begin{notations}\label{notations_QG}
\begin{enumerate}
\item 
If $q-q^{-1}$ is invertible in the ground ring $k$ on which we are working, we will write $E:= \frac{\overline{E}}{q-q^{-1}}$ and $F:= \frac{\overline{F}}{q-q^{-1}}$ and rather work with the generators $E,F, K^{\pm 1}, L^{\pm 1/2}$. For instance, we have the relation $EF-FE=\frac{K-L}{q-q^{-1}}$.
\item The element $H_{\partial}:= K^{-1/2}L^{-1/2}$ is clearly an invertible central element. We easily see that another central element is given (here we suppose that $q-q^{-1}$ is invertible) by the \textbf{Casimir element} defined by 
$$ C:= EF+\frac{qL +q^{-1}K}{(q-q^{-1})^2} = FE + \frac{qK+q^{-1}L}{(q-q^{-1})^2}.$$
\end{enumerate}
\end{notations}

 \begin{theorem}\label{theorem_skein_QG}
 There is an isomorphism of Hopf algebras $\Psi:  \Uq \xrightarrow{\cong}  \overline{\mathcal{S}}_{A}(\mathbb{D}_1)$ defined by 
 \begin{align*}
 {} & \Psi(K^{1/2})= \alpha_{--}, \quad \Psi(K^{-1/2})= \alpha_{++}, \quad \Psi(L^{1/2})= \beta_{--}, \quad \Psi(L^{-1/2})=\beta_{++} \\
 {} & \Psi(\overline{E})=-A \alpha_{+-}\alpha_{--}, \quad \Psi(\overline{F})= A^{-1}\beta_{--} \beta_{-+}
 \end{align*}
 In particular, supposing that $q-q^{-1}$ is invertible, $\Psi$ relates central elements $H_{\partial}^{\pm 1}$ and $C$ of $\Uq$ to the central elements $\alpha_{++}\beta_{++}$ and $\gamma_p$ of $\overline{\mathcal{S}}_A(\mathbb{D}_1)$ via
 $$ \Psi(H_{\partial})= \alpha_{++}\beta_{++}, \quad \Psi(C)= -\frac{\alpha_{--}\beta_{--}}{(q-q^{-1})^2}\gamma_p.$$
 \end{theorem}

\begin{proof}
The proof of the fact that $\Psi$ is an isomorphism of Hopf algebras follows from a straightforward computation using the fact that $\overline{\mathcal{S}}_A(\mathbb{D}_1)$ is presented by the generators $\alpha_{ij}, \beta_{ij}$ and relations obtained by adding to the relations of Corollary \ref{coro_pres_D1} the two additional relations $\alpha_{-+}=\beta_{+-}=0$. The equality relating $\Psi(C)$ to $\gamma_p$ follows from the following equality in $\overline{\mathcal{S}}_A(\mathbb{D}_1)$ obtained using a simple skein relation: 
$$ \gamma_p = -q \alpha_{++}\beta_{--}-q^{-1}\alpha_{--}\beta_{++}+\alpha_{+-}\beta_{-+}+\alpha_{-+}\beta_{+-}=  -q \alpha_{++}\beta_{--}-q^{-1}\alpha_{--}\beta_{++}+\alpha_{+-}\beta_{-+}.$$
\end{proof}

\begin{convention} For now on,  by abuse of notations, we identify the two Hopf algebras $\Uq$ and $\overline{\mathcal{S}}_A(\mathbb{D}_1)$ using $\Psi$. For instance, we write $K^{1/2}=\alpha_{--}$ instead of $\Psi(K^{1/2})=\alpha_{--}$.
\end{convention}

\begin{definition}[Cartan involution]\label{def_Cartan_involution}
Let $\theta_1 \in \mathrm{Aut}(\mathcal{S}_{\omega}(\mathbb{D}_1))$ be the involutive algebra automorphism sending a stated diagram to its image through the central symmetry of the disc (rotation of angle $\pi$). One has $\theta_1(\alpha_{\varepsilon \varepsilon'})=\beta_{\varepsilon' \varepsilon}$ and $\theta_1(\beta_{\varepsilon \varepsilon'})=\alpha_{\varepsilon' \varepsilon}$. So $\theta_1$ exchanges the two bad arcs ($\theta_1(\alpha_{-+})=\beta_{+-}$), hence it induces an involutive algebra automorphism $\Theta_1 \in \mathrm{Aut}(\Uq)$  which we call the \textbf{Cartan involution}.
 \end{definition}

Note that $\Theta_1(K^{\pm 1/2})= L^{\pm 1/2}$ and that $\Theta_1(E)=F$.

\section{Poisson geometry of representation varieties}\label{sec_moduli_space}

\subsection{Representation varieties}
The algebras $\mathcal{S}_{+1}(\mathbf{\Sigma})$ and $\overline{\mathcal{S}}_{+1}(\mathbf{\Sigma})$, where $k=\mathbb{C}$ and $A^{1/2}=+1$,  are commutative (\cite{LeStatedSkein}) and their maximal spectra admit some geometric interpretations as moduli spaces which we now recall. 

\begin{definition}[Representation spaces]
Let $\mathbf{\Sigma}$ be an essential marked surface.
\begin{enumerate}
 \item The \textbf{relative representation variety} $\mathcal{R}_{\SL_2}(\mathbf{\Sigma})$ is the set of functors $\rho: \Pi_1(\Sigma, \mathbb{V}) \to \SL_2$, where $\SL_2$ is seen as a category with only one object $*$ whose set of  endomorphisms is $\SL_2(\mathbb{C})$. It admits a structure of a complex affine variety whose algebra of regular functions is 
 $$ \mathcal{O}[\mathcal{R}_{\SL_2}(\mathbf{\Sigma})]:= \quotient{ \mathbb{C}[X_{ij}^{\beta}, i,j\in \{-,+\}, \beta \in \Pi_1(\Sigma, \mathbb{V})]}{\left( M_{\beta_1}M_{\beta_2}=M_{\beta_1 \beta_2}, \det(M_{\beta})=1 \right)}.$$
 Here, for $\beta \in \Pi_1(\Sigma, \mathbb{V})$, $M_{\beta}$ represents the $2\times 2$ matrix with coefficients in the polynomial ring $ \mathbb{C}[X_{ij}^{\beta}, i,j\in \{-,+\}, \beta \in \Pi_1(\Sigma, \mathbb{V})]$ defined by 
 $M_{\beta}=\begin{pmatrix} X_{++}^{\beta} & X_{+-}^{\beta} \\ X_{-+}^{\beta} & X_{--}^{\beta} \end{pmatrix}$ and we quotient by the relations $\det(M_{\beta}):=X_{++}^{\beta}X_{--}^{\beta}- X_{+-}^{\beta}X_{-+}^{\beta} =1$ for all $\beta \in  \Pi_1(\Sigma, \mathbb{V})$ and by the four matrix coefficients of $ M_{\beta_1}M_{\beta_2}-M_{\beta_1 \beta_2}$ for every pair of composable paths (i.e. such that $t(\beta_2)=s(\beta_1)$).
 Clearly the set of closed points $\mathcal{R}_{\SL_2}(\mathbf{\Sigma}):= \Specm(\mathcal{O}[\mathcal{R}_{\SL_2}(\mathbf{\Sigma})])$ is in canonical bijection with the set of functors  $\rho: \Pi_1(\Sigma, \mathbb{V}) \to \SL_2$.
 \item The \textbf{small Bruhat cell} is the subset of $\SL_2(\mathbb{C})$ defined by
 $$ \SL_2^{1}:= \{ M = \begin{pmatrix} a & b \\ c & d \end{pmatrix} \in \SL_2(\mathbb{C}) \mbox{ such that }a=0 \}.$$
 The \textbf{reduced relative representation variety} is the subvariety: 
 $$ \overline{\mathcal{R}}_{\SL_2}(\mathbf{\Sigma}) = \{ \rho : \Pi_1(\Sigma, \mathbb{V}) \to \SL_2 \mbox{ such that } \rho(\alpha(p))\in \SL_2^{1} \mbox{ for all }p\in {\mathcal{P}}^{\partial} \}.$$
 Said differently, its algebra of regular functions is 
 $$ \mathcal{O}[\overline{\mathcal{R}}_{\SL_2}(\mathbf{\Sigma})] := \quotient{ \mathcal{O}[\mathcal{R}_{\SL_2}(\mathbf{\Sigma})]}{ ( X_{++}^{\alpha(p)}, p\in {\mathcal{P}}^{\partial})}.$$
 \end{enumerate}
\end{definition}

Let $\Gamma$ be a presenting graph for $\mathbf{\Sigma}$. Since $\Sigma$ retracts to $\Gamma$, we obtain an isomorphism
 $$ \varphi_{\Gamma}: \mathcal{R}_{\SL_2}(\mathbf{\Sigma}) \xrightarrow{\cong} (\SL_2(\mathbb{C}))^{\mathcal{E}(\Gamma)}, \quad \varphi_{\Gamma}: \rho \mapsto (\rho(\beta))_{\beta \in \mathcal{E}(\Gamma)}.$$
 For instance, a functor $\rho: \mathcal{R}_{\SL_2}(\mathbb{D}_n)\to \SL_2$ is completely determined by the data of $\rho(\alpha^{(n)})\in \SL_2$ and of its restriction $\rho: \pi_1(\mathbb{D}_n, \{v_a\})\to \SL_2$, i.e. $\mathcal{R}_{\SL_2}(\mathbb{D}_n)\cong \SL_2 \times \Hom(\pi_1(\mathbb{D}_n, \{v_a\}), \SL_2)$. 
 \par 
 Let $\Arc(\mathbf{\Sigma})$ be the set of oriented arcs in $\mathbf{\Sigma}$.
 
 \begin{theorem}\label{theorem_classical_limit}(\cite[Theorem $3.18$]{KojuQuesneyClassicalShadows}, \cite[Corollary $3.3$]{KojuKaruo_RepRSSkein}, see also \cite[Theorem $4.7$]{KojuPresentationSSkein}) Let $\mathbf{\Sigma}$ be an essential marked surface. There exists a map $w: \Arc(\mathbf{\Sigma})\to \mathbb{Z}/2\mathbb{Z}$ and an isomorphism $\Psi_w: \mathcal{S}_{+1}(\mathbf{\Sigma}) \xrightarrow{\cong} \mathcal{O}[\mathcal{R}_{\SL_2}(\mathbf{\Sigma})]$ characterized by the formula:
$$ \Psi_w \begin{pmatrix} \alpha_{++} & \alpha_{+-} \\ \alpha_{-+} & \alpha_{--} \end{pmatrix} = (-1)^{w(\alpha)} \begin{pmatrix} 0 & -1 \\ 1 & 0\end{pmatrix} \begin{pmatrix} X_{++}^{\alpha} & X_{+-}^{\alpha} \\ X_{-+}^{\alpha} & X_{--}^{\alpha} \end{pmatrix} = (-1)^{w(\alpha)} \begin{pmatrix} -X_{-+}^{\alpha} & -X_{--}^{\alpha} \\ X_{++}^{\alpha} & X_{+-}^{\alpha} \end{pmatrix}, \quad \mbox{ for all }\alpha \in \Arc(\mathbf{\Sigma}).$$
$\Psi_w$ induces an  isomorphism $\overline{\Psi}_w: \overline{\mathcal{S}}_{+1}(\mathbf{\Sigma}) \xrightarrow{\cong} \mathcal{O}[\overline{\mathcal{R}}_{\SL_2}(\mathbf{\Sigma})]$ by passing to the quotient.
 \end{theorem}

The function $w: \Arc(\mathbf{\Sigma})\to \mathbb{Z}/2\mathbb{Z}$ is called a \textbf{spin function} and is described as follows. Let $\Gamma$ be a presenting graph of $\mathbf{\Sigma}$ and set the values $w(\beta)\in \mathbb{Z}/2\mathbb{Z}$ for $\beta\in \mathcal{E}(\Gamma)$ arbitrarily. Then, for $\beta\in \mathcal{E}(\Gamma)$, define $w(\beta^{-1})$ by imposing $w(\beta)+w(\beta^{-1})\equiv 1 \pmod{2}$. Every arc $\alpha \in \Arc(\mathbf{\Sigma})$ can be decomposed (inside $\Pi_1(\mathbf{\Sigma}, \mathbb{V})$ ) uniquely as $\alpha= \beta_{i_1}^{\varepsilon_1} \ldots \beta_{i_k}^{\varepsilon_k}$ with $\beta_{i_j}\in \mathcal{E}(\Gamma)$ and $\varepsilon_k =\pm 1$. Then we set $w(\alpha)\equiv  \sum_{k=1}^k w(\beta_{i_j}^{\varepsilon_j})$. This way, $w$ is completely determined by its values on $ \mathcal{E}(\Gamma)$ so the set of spin functions is in bijection with $(\mathbb{Z}/2\mathbb{Z})^{\mathcal{E}(\Gamma)}$. 

\begin{convention}\label{convention_w} For $\mathbb{D}_n$ with presenting graph $\Gamma$ where $\mathcal{E}(\Gamma)=\{\alpha^{(n)}, \alpha_{p_1}, \ldots, \alpha_{p_n}\}$, we consider the spin function $w: \Arcs(\mathbb{D}_n) \to \mathbb{Z}/2\mathbb{Z}$ defined by $w(\alpha^{(n)})\equiv 0 \pmod{2}$ and $w(\alpha_{p_i})\equiv 1\pmod{2}$ for $i=1, \ldots, n$. We will often identify $\overline{\mathcal{S}}_{+1}(\mathbb{D}_n)$ with $\mathcal{O}[\mathcal{R}_{\SL_2}(\mathbf{\Sigma})]$ using $\Psi_w$.
\end{convention}

Note that since $\beta^{(n)}=\alpha^{(n)}(\alpha_{p_n})^{-1} \ldots (\alpha_{p_1})^{-1}$, we have $w(\beta^{(n)})\equiv  w(\alpha^{(n)}) + \sum_{i=1}^n (1+w(\alpha_{p_i})) \equiv 0 \pmod{2}$. 

\subsection{Poisson structure on moduli spaces}

The commutative algebra $\mathcal{S}_{+1}(\mathbf{\Sigma})$ has a natural Poisson structure defined as follows. Let $\mathcal{S}_{\hbar}(\mathbf{\Sigma})$ be the skein algebra associated to the ring $k=\mathbb{C}[[\hbar]]$ of formal power series in $\hbar$ with parameter $A_{\hbar}^{1/2}:= \exp(\hbar/2)$. Since the set $\mathcal{B}$ of Definition \ref{def_basis} is a basis for both the algebras $\mathcal{S}_{\hbar}(\Sigma)$ and $\mathcal{S}_{+1}(\Sigma)$, we can define a $\mathbb{C}[[\hbar]]$-linear isomorphism $ \mathcal{S}_{+1}(\mathbf{\Sigma})\otimes_{\mathbb{C}}\mathbb{C}[[\hbar]] \xrightarrow{\cong} \mathcal{S}_{\hbar}(\mathbf{\Sigma})$ characterized by the fact that it sends a basis elements $b\in \mathcal{B}$ to itself. Let $\star$ be the pull-back in $\mathcal{S}_{+1}(\mathbf{\Sigma})\otimes_{\mathbb{C}}\mathbb{C}[[\hbar]]$ of the product in $\mathcal{S}_{\hbar}(\mathbf{\Sigma})$. 

\begin{definition}[Poisson structure on moduli spaces]\label{def_DQSkein} The Poisson bracket $\{\cdot, \cdot\}$ on $\mathcal{S}_{+1}(\mathbf{\Sigma})$ is defined by the formula
$$ x \star y - y \star x = \hbar \{x, y\} \pmod{\hbar^2}.$$
\end{definition}

In particular, the associativity of $\star$ implies the Jacobi identity for $\{\cdot, \cdot\}$. 

\begin{remark}
When $\mathbf{\Sigma}$ is essential, by fixing a spin function $w$ and using the isomorphism of Theorem \ref{theorem_classical_limit}, we obtain a Poisson structure on the representation variety $\mathcal{R}_{\SL_2}(\mathbf{\Sigma})$. By \cite[Theorem $1.3$]{KojuQuesneyClassicalShadows} and \cite[Appendix B]{KojuTriangularCharVar}, the Poisson variety $\mathcal{R}_{\SL_2}(\mathbf{\Sigma})$ is isomorphic to the moduli space defined by Fock and Rosly in \cite{FockRosly} associated to a presenting graph. When $\mathbf{\Sigma}$ is a connected marked surface with a single boundary edge, by \cite[Theorem $1.3$]{KojuQuesneyClassicalShadows} and \cite[Appendix C]{KojuTriangularCharVar},  the Poisson structure on  $\mathcal{R}_{\SL_2}(\mathbf{\Sigma})=\Hom(\pi_1(\Sigma), \SL_2)$ coincides with the Poisson structure defined by Alekseev-Kosmann-Meinrenken in  \cite{AlekseevMalkin_PoissonCharVar, AlekseevMalkin_PoissonLie, AlekseevKosmannMeinrenken}. 
\end{remark}

\begin{lemma} The ideal $\mathcal{I}^{bad} \subset \mathcal{S}_A(\mathbf{\Sigma})$ generated by bad arcs is a Poisson ideal, i.e. $\{\mathcal{I}^{bad}, \mathcal{S}_A(\mathbf{\Sigma})\}\subset \mathcal{I}^{bad}$.
\end{lemma}

\begin{proof} By \cite[Lemma $4.4(a)$]{LeYu_SSkeinQTraces}, for every bad arc $\alpha^{bad}$ and for every basis element $[D,s]\in \mathcal{B}$, there exists $n\in \mathbb{Z}$ such that $[D,s]\alpha^{bad}= A^n \alpha^{bad}[D,s]$ in $\mathcal{S}_A(\mathbf{\Sigma})$. Therefore $\{[D,s], \alpha^{bad}\}=\frac{n}{2} \alpha^{bad}[D,s]\in \mathcal{I}^{bad}$ in $\mathcal{S}_{+1}(\mathbf{\Sigma})$.
\end{proof}

So $X(\mathbf{\Sigma})=\Specm(\overline{\mathcal{S}}_{+1}(\mathbf{\Sigma}))$ is a Poisson subvariety of $\underline{X}(\mathbf{\Sigma})=\Specm(\mathcal{S}_{+1}(\mathbf{\Sigma}))$. 

\begin{lemma}\label{lemma_poisson_morphisms}
\begin{enumerate}
\item If $f: \mathbf{\Sigma}_1 \to \mathbf{\Sigma}_2$ is a strong embedding, then $f_*: \mathcal{S}_{+1}(\mathbf{\Sigma}_1) \to \mathcal{S}_{+1}(\mathbf{\Sigma}_2)$ is Poisson.
\item If $a,b$ are two boundary edges of $\mathbf{\Sigma}$, then $\theta_{a\#b}:  \mathcal{S}_{+1}(\mathbf{\Sigma}_{a\#b })\to \mathcal{S}_{+1}(\mathbf{\Sigma})$ is Poisson.
\end{enumerate}
The same assertions hold with the stated skein algebras replaced by their reduced versions.
\end{lemma}

\begin{proof} This follows from the facts that both $f_*: \mathcal{S}_{\hbar}(\mathbf{\Sigma}_1) \to \mathcal{S}_{\hbar}(\mathbf{\Sigma}_2)$ and $\theta_{a\#b}: \mathcal{S}_{\hbar}(\mathbf{\Sigma}_{a\#b })\to \mathcal{S}_{\hbar}(\mathbf{\Sigma})$ are morphisms of algebras.
\end{proof}

\subsection{Orbits for the gauge groups action}\label{sec_gauge_orbit}

Let us now define two group actions on $X(\mathbb{D}_n)\cong \overline{\mathcal{R}}_{\SL_2}(\mathbb{D}_n)$. Recall that $\mathbb{D}_n$ has two boundary edges $b_L, b_R$ with based points $v_L\in b_L$ and $v_R\in b_R$ and we have set $\mathbb{V}=\{v_L, v_R\}$. The \textbf{outer gauge group} is the group $\mathcal{G}^{out}= \Map(\mathbb{V}, \SL_2)$ of maps $g: \mathbb{V} \to \SL_2$. It acts on the representation variety $\mathcal{R}_{\SL_2}(\mathbb{D}_n)$ by the formula
$$ (g\cdot \rho) (\alpha):= g(s(\alpha)) \rho(\alpha) g(t(\alpha))^{-1}, \quad \mbox{ for }\rho \in \mathcal{R}_{\SL_2}(\mathbb{D}_n), g \in \mathcal{G}^{out}, \alpha \in \Pi_1(\mathbb{D}_n, \mathbb{V}).$$
Note that this group action is dual to the comodule map $\Delta^L : \mathcal{S}_{+1}(\mathbb{D}_n) \to (\mathcal{O}[\SL_2])^{\otimes \mathbb{V}} \otimes \mathcal{S}_{+1}(\mathbb{D}_n)$ of Section \ref{sec_comodule} (see \cite{KojuPresentationSSkein} for details). Recall from Section \ref{sec_moduli_space} the isomorphism 

$$\varphi_{\Gamma}: \mathcal{R}_{\SL_2}(\mathbb{D}_n) \xrightarrow{\cong} \SL_2 \times \Hom(\pi_1(\mathbb{D}_n, \{b_L\}), \SL_2)$$ 

which  sends $\rho$ to the pair $(\rho(\alpha^{(n)}), \restriction{\rho}{\pi_1(\mathbb{D}_n)})$ where $\pi_1(\mathbb{D}_n)$ is the fundamental group with base point $v_{L}$. 
\begin{lemma}\label{lemma_orbit1} We have a bijection 
$$ \quotient{\mathcal{R}_{\SL_2}(\mathbb{D}_n)}{\mathcal{G}^{out}} \cong \quotient{\Hom(\pi_1(\mathbb{D}_n), \SL_2)}{\SL_2}, \quad [\rho] \mapsto [\restriction{\rho}{\pi_1(\mathbb{D}_n)}].$$
Moreover, each class in the quotient $ \quotient{\mathcal{R}_{\SL_2}(\mathbb{D}_n)}{\mathcal{G}^{out}}$ is the class of an element in $\overline{\mathcal{R}}_{\SL_2}(\mathbb{D}_n)$.
\end{lemma}

\begin{proof}
The first assertion follows from the fact that for $g\in \mathcal{G}^{out}$ we have $$\varphi_{\Gamma}( g\cdot \rho) = (g(v_L) \rho(\alpha^{(n)}) g(v_R)^{-1}, g(v_L)\restriction{\rho}{\pi_1(\mathbb{D}_n)}g(v_L)^{-1}).$$ The second assertion is immediate.
\end{proof}

By \cite[Theorem $3.5$]{KojuKaruo_RepRSSkein}, $\overline{\mathcal{R}}_{\SL_2}(\mathbb{D}_n)$ is a smooth affine variety, so we can consider it as an analytic variety as well. For $f \in \mathcal{O}[\overline{\mathcal{R}}_{\SL_2}(\mathbb{D}_n)]$ a regular function and $t\in \mathbb{C}$, we can thus consider the Hamiltonian flow $\phi^t_f$, seen as an analytic automorphism of $\overline{\mathcal{R}}_{\SL_2}(\mathbb{D}_n)$. 
The \textbf{inner gauge group} is the group of analytic automorphisms $\mathcal{G}_{\mathbb{D}_1}$  generated by the automorphisms $\phi^t_{\alpha^{(n)}_{ij}}$ and $\phi^t_{\beta^{(n)}_{ij}}$ for $t\in \mathbb{C}$ and $i,j = \pm$ (recall that $\alpha^{(n)}_{ij}, \beta^{(n)}_{ij}$ are the stated arcs of Figure \ref{fig_arcs_Dn} which we consider as regular functions using the isomorphism of Theorem \ref{theorem_classical_limit}). 
Let $\alpha_{\partial} := (\alpha^{(n)})_{++} (\beta^{(n)})_{++} \in \overline{\mathcal{S}}_{+1}(\mathbb{D}_n) \cong \mathcal{O}[\overline{\mathcal{R}}_{\SL_2}(\mathbb{D}_n)]$. As we shall see in the next section, $\alpha_{\partial}$ is central in $\overline{\mathcal{S}}_{\hbar}(\mathbb{D}_n)$ so it defines a Casimir element in $\overline{\mathcal{S}}_{+1}(\mathbb{D}_n)$ and the level sets $\alpha_{\partial}^{-1}\{ h \}$ are thus preserved by any Hamiltonian flow.
For $\rho \in \overline{\mathcal{R}}_{\SL_2}(\mathbb{D}_n)$, we write
$$ \mathcal{G}^{out}_{\rho} := \{ g \in \mathcal{G}^{out} : g\cdot \rho \in \overline{\mathcal{R}}_{\SL_2}(\mathbb{D}_n) \mbox{ and } \alpha_{\partial}(g\cdot \rho) = \rho \}.$$

\begin{proposition}\label{prop_orbit} Two representations $\rho_1, \rho_2$  of $\overline{\mathcal{R}}_{\SL_2}(\mathbb{D}_n)$ are in the same $\mathcal{G}_{\mathbb{D}_1}$ orbit if and only if there exists $g\in \mathcal{G}^{out}_{\rho_1}$ such that $g\cdot \rho_1 = \rho_2$. Therefore we have a bijection 
$$\quotient{\overline{\mathcal{R}}_{\SL_2}(\mathbb{D}_n)}{\mathcal{G}_{\mathbb{D}_1}} \xrightarrow{\cong} \left(\quotient{\Hom(\pi_1(\mathbb{D}_n), \SL_2)}{\SL_2}\right)\times \mathbb{C}^*, \quad  [\rho] \mapsto ([\restriction{\rho}{\pi_1(\mathbb{D}_n)}], \alpha_{\partial}(\rho)).$$
\end{proposition}

\begin{lemma}\label{lemma_orbit2}
For $1\leq i \leq n$, the Poisson bracket of $\overline{\mathcal{S}}_{+1}(\mathbb{D}_n)$ satisfies

\begin{align*}
{}& \{\alpha_{+-}^{(n)}, \delta_{++}^{(i)}\} = - \alpha_{+-}^{(n)} \delta_{++}^{(i)} & \{\alpha_{+-}^{(n)}, \alpha^{(n)}_{++}\} &= \alpha^{(n)}_{+-}\alpha^{(n)}_{++} & \{\alpha^{(n)}_{++}, \delta^{(i)}_{++}\} &= - \alpha^{(n)}_{++}\delta^{(i)}_{++} \\
{}& \{\alpha_{+-}^{(n)}, \delta_{--}^{(i)}\} = \alpha^{(n)}_{+-}\delta^{(i)}_{--}-\alpha_{--}^{(n)}\delta_{+-}^{(i)}-3 \alpha^{(n)}_{--} \delta_{-+}^{(i)} & \{\alpha_{+-}^{(n)}, \alpha^{(n)}_{--}\} &= -\alpha^{(n)}_{+-}\alpha^{(n)}_{--} & \{\alpha^{(n)}_{++}, \delta^{(i)}_{--}\} &=  \alpha^{(n)}_{++}\delta^{(i)}_{--} \\
{}& \{\alpha_{+-}^{(n)}, \delta_{+-}^{(i)}\} = -2 \alpha_{--}^{(n)} \delta_{++}^{(i)} & \{\alpha_{++}^{(n)}, \alpha^{(n)}_{--}\} &= 0 & \{\alpha^{(n)}_{++}, \delta^{(i)}_{+-}\} &= 0 \\  
{}& \{\alpha_{+-}^{(n)}, \delta_{-+}^{(i)}\} = -2 \alpha_{+-}^{(n)} \delta_{-+}^{(i)} & {} &{}& \{\alpha^{(n)}_{++}, \delta^{(i)}_{-+}\} &= 0 
\end{align*}

\end{lemma}

\begin{proof}
The formulas for the Poisson bracket in $\overline{\mathcal{S}}_{+1}(\mathbb{D}_n)$ follow by developing at the order $1$ in $\hbar$  the following formulas in $\overline{\mathcal{S}}_{A}(\mathbb{D}_n)$ where we set $A^{1/2}=\exp(\hbar/4)$ so $q=\exp(\hbar)$: 

\begin{align*}
{}& \alpha_{+-}^{(n)} \delta_{++}^{(i)} = q^{-1} \delta_{++}^{(i)} \alpha_{+-}^{(n)}  & \alpha_{+-}^{(n)}\alpha^{(n)}_{++}\} &= q \alpha^{(n)}_{++}\alpha^{(n)}_{+-} & \alpha^{(n)}_{++} \delta^{(i)}_{++} &= q^{-1}\delta^{(i)}_{++} \alpha^{(n)}_{++} \\
{}& \alpha_{+-}^{(n)} \delta_{-+}^{(i)} = q^{-2} \delta_{-+}^{(i)} \alpha_{+-}^{(n)}  & \alpha_{+-}^{(n)} \alpha^{(n)}_{--}& = q^{-1}\alpha^{(n)}_{--}\alpha^{(n)}_{+-} & \alpha^{(n)}_{++}\delta^{(i)}_{--} &= q\delta^{(i)}_{--} \alpha^{(n)}_{++} \\
{}& [\alpha_{+-}^{(n)},  \delta_{+-}^{(i)}] = (1-q^2)\delta_{++}^{(i)}\alpha_{--}^{(n)}  & [\alpha_{++}^{(n)}, \alpha^{(n)}_{--}] &= 0 & [\alpha^{(n)}_{++}, \delta^{(i)}_{+-}] &=  [\alpha^{(n)}_{++}, \delta^{(i)}_{-+}]=0  
 \end{align*}
and the equation 
\begin{equation}\label{eq_exemple}
 \alpha_{+-}^{(n)} \delta_{--}^{(i)} = q\delta^{(i)}_{--}\alpha^{(n)}_{+-}+ (q-q^2)\delta_{+-}^{(i)}\alpha_{--}^{(n)}+(q^{-1}-q^{2})\delta_{-+}^{(i)} \alpha^{(n)}_{--}.
\end{equation}
For instance using the developments $q=1+\hbar +o(\hbar^2)$, $q-q^2= - \hbar + o(\hbar^2)$ and $q^{-1}-q^{2}=-3\hbar + o(\hbar^2)$, Equation \eqref{eq_exemple} implies
$$ [\alpha_{+-}^{(n)},  \delta_{--}^{(i)} ]= (q-1)\delta^{(i)}_{--}\alpha^{(n)}_{+-}+ (q-q^2)\delta_{+-}^{(i)}\alpha_{--}^{(n)}+(q^{-1}-q^{2})\delta_{-+}^{(i)} \alpha^{(n)}_{--} \equiv \hbar( \alpha^{(n)}_{+-}\delta^{(i)}_{--}-\alpha_{--}^{(n)}\delta_{+-}^{(i)}-3 \alpha^{(n)}_{--} \delta_{-+}^{(i)}) \pmod{\hbar^2}$$ 
from which we derive the equality $\{\alpha_{+-}^{(n)}, \delta_{--}^{(i)}\} = \alpha^{(n)}_{+-}\delta^{(i)}_{--}-\alpha_{--}^{(n)}\delta_{+-}^{(i)}-3 \alpha^{(n)}_{--} \delta_{-+}^{(i)} $. Now Equation \eqref{eq_exemple} is obtained using skein manipulations as follows: 
\begin{align*}
\alpha_{+-}^{(n)}\delta^{(i)}_{--}&=   \adjustbox{valign=c}{\includegraphics[width=1.5cm]{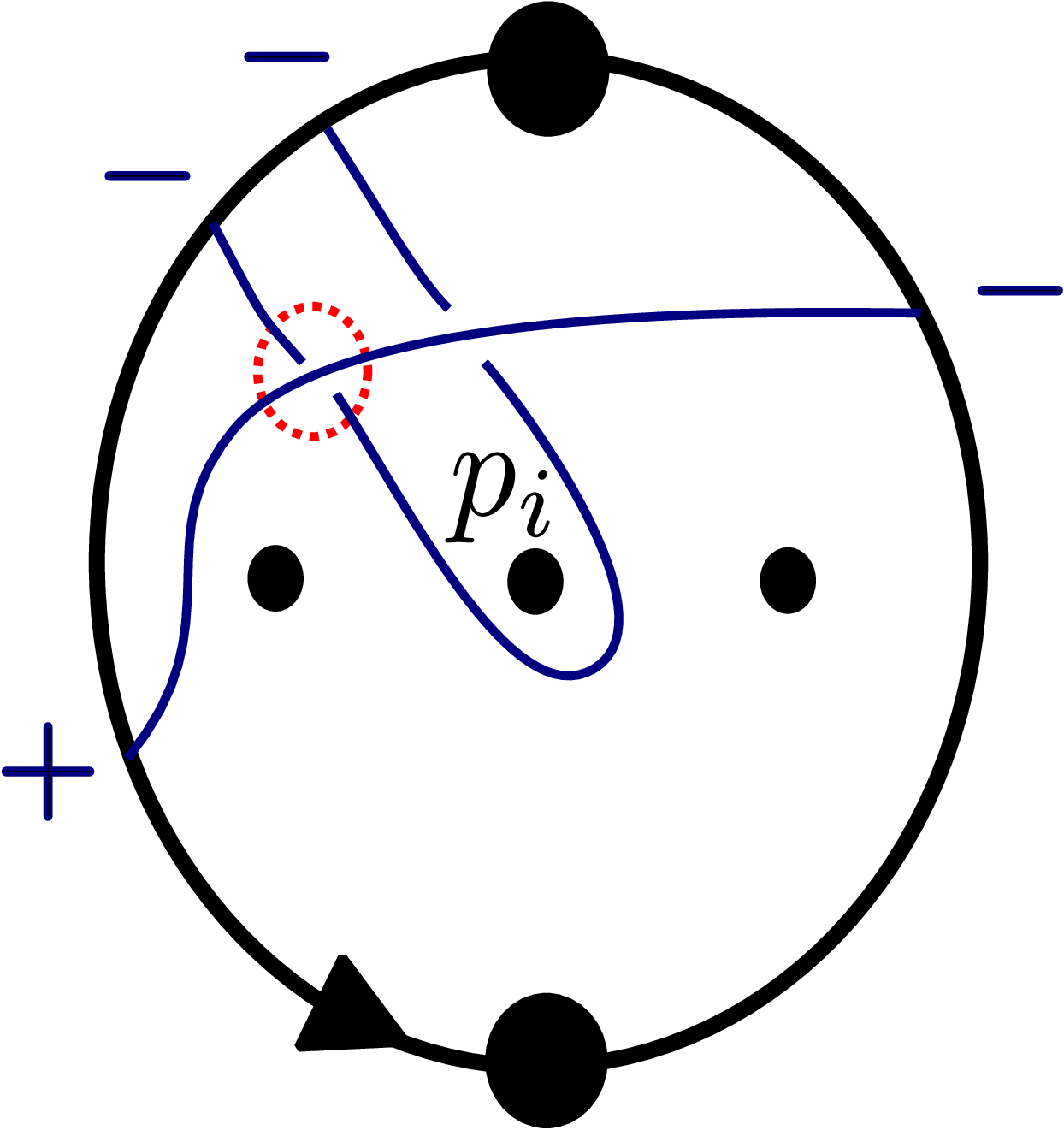}} =A \adjustbox{valign=c}{\includegraphics[width=1.5cm]{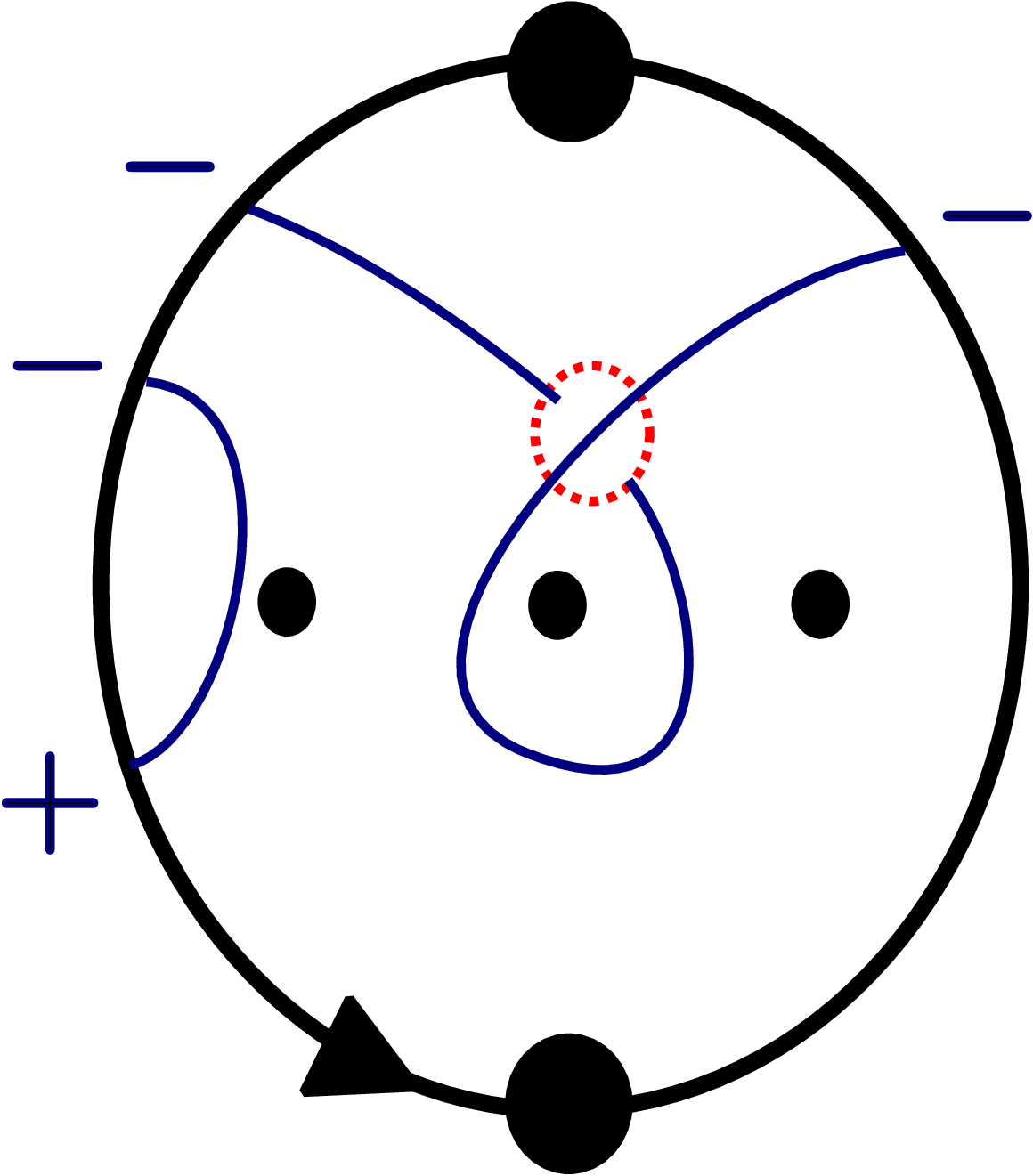}}+A^{-1} \adjustbox{valign=c}{\includegraphics[width=1.5cm]{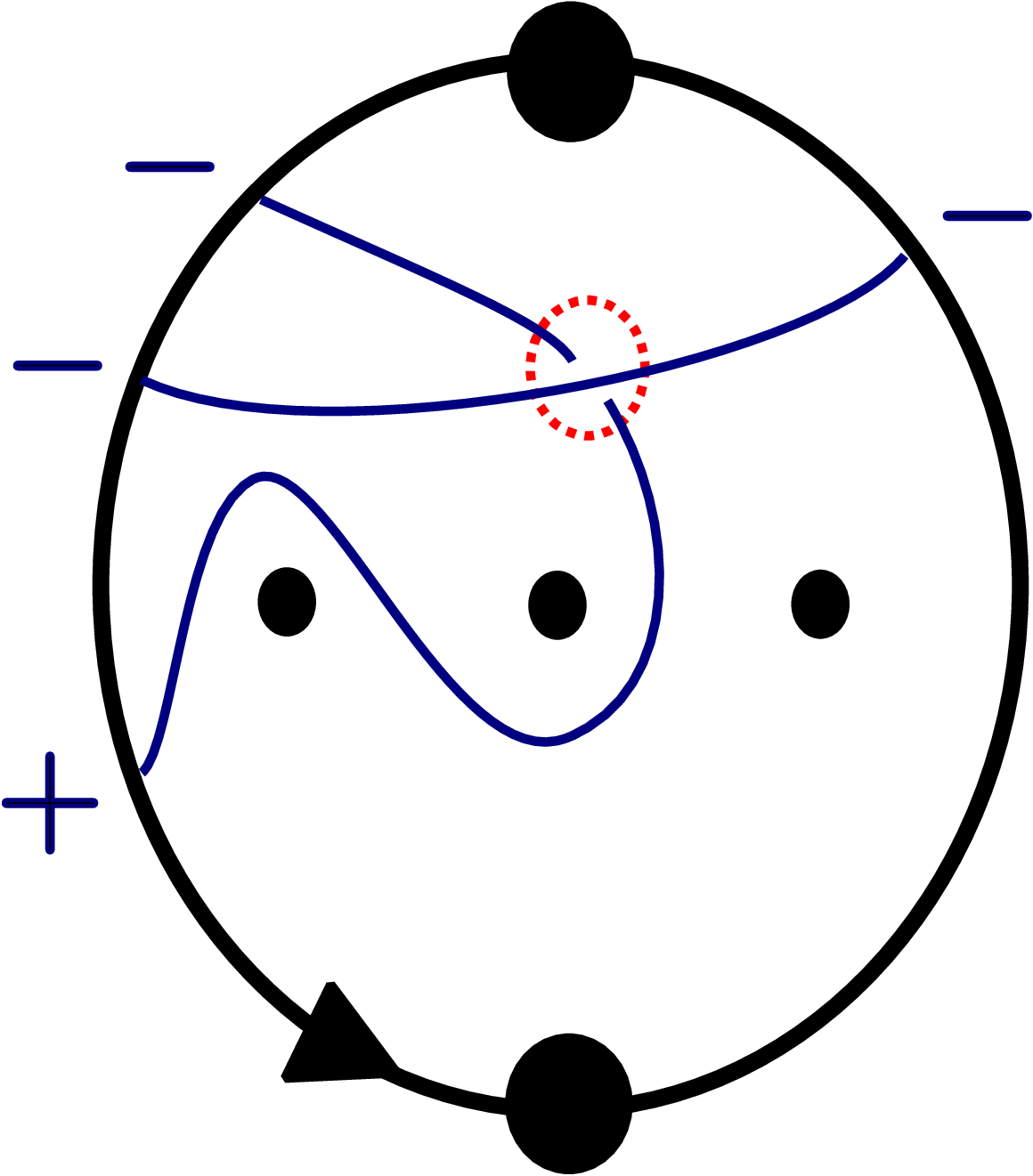}} 
{} & =A^{3/2}  \adjustbox{valign=c}{\includegraphics[width=1.5cm]{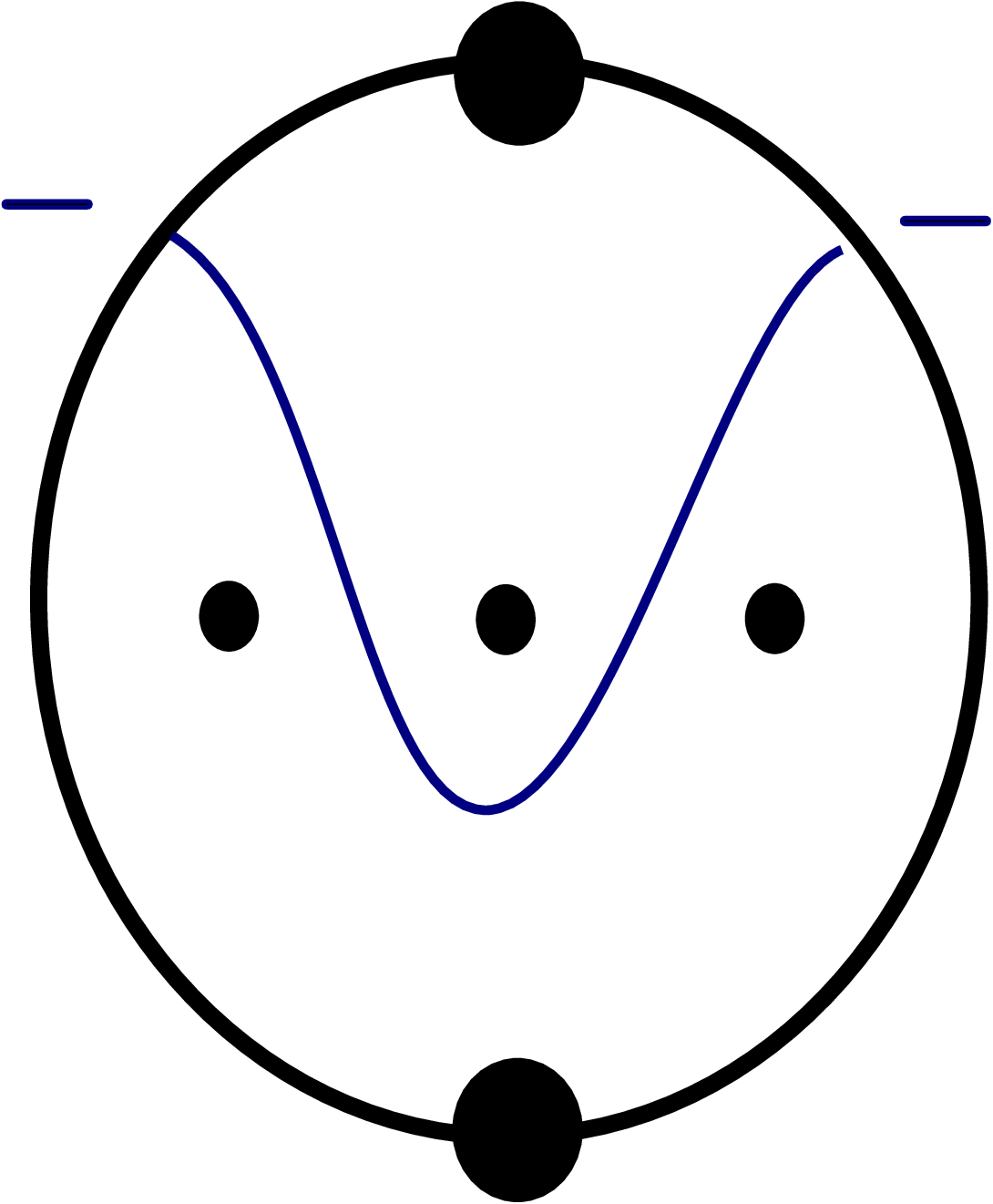}} + A^{-1/2} \adjustbox{valign=c}{\includegraphics[width=1.5cm]{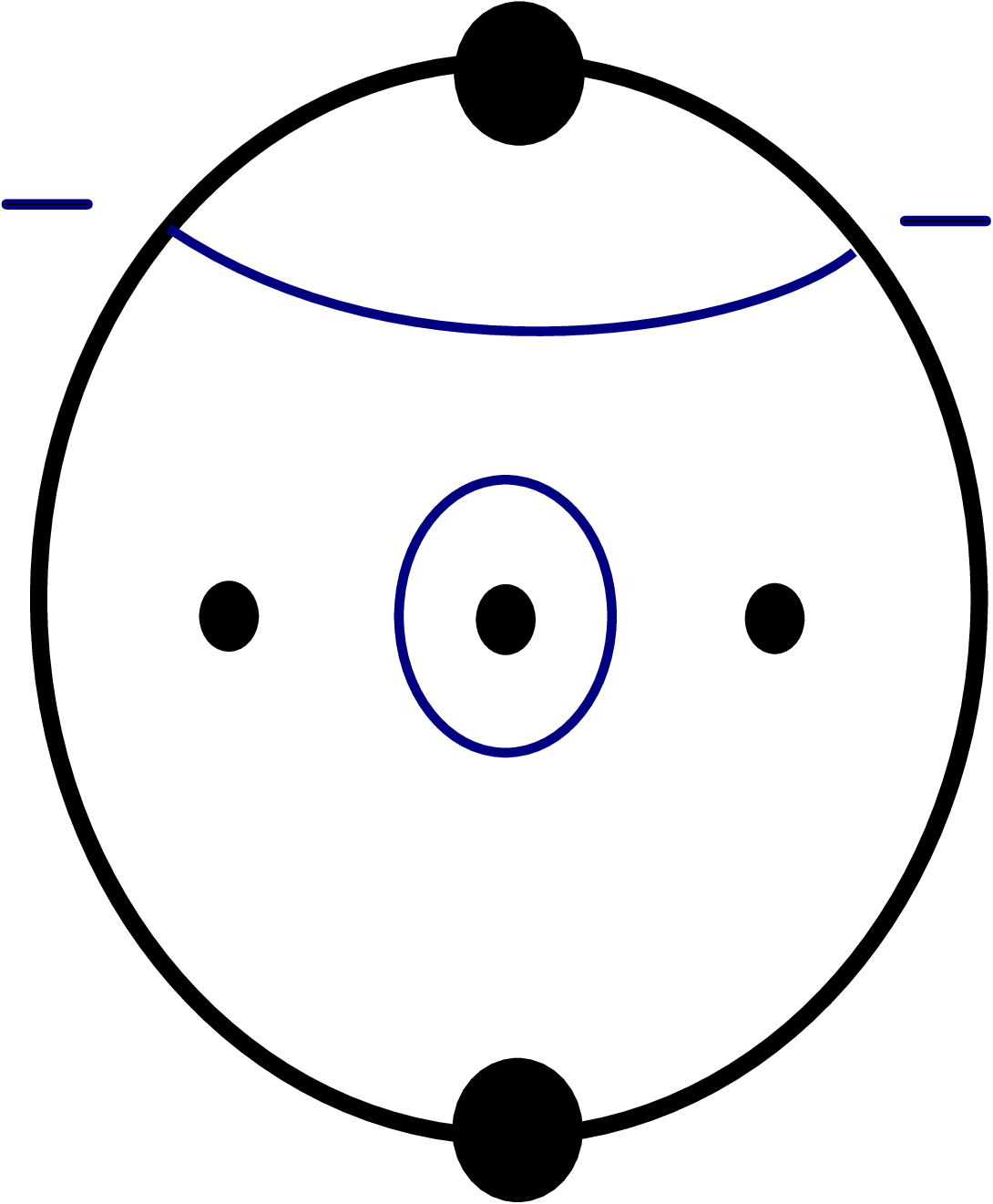}}+q^{-1} \adjustbox{valign=c}{\includegraphics[width=1.5cm]{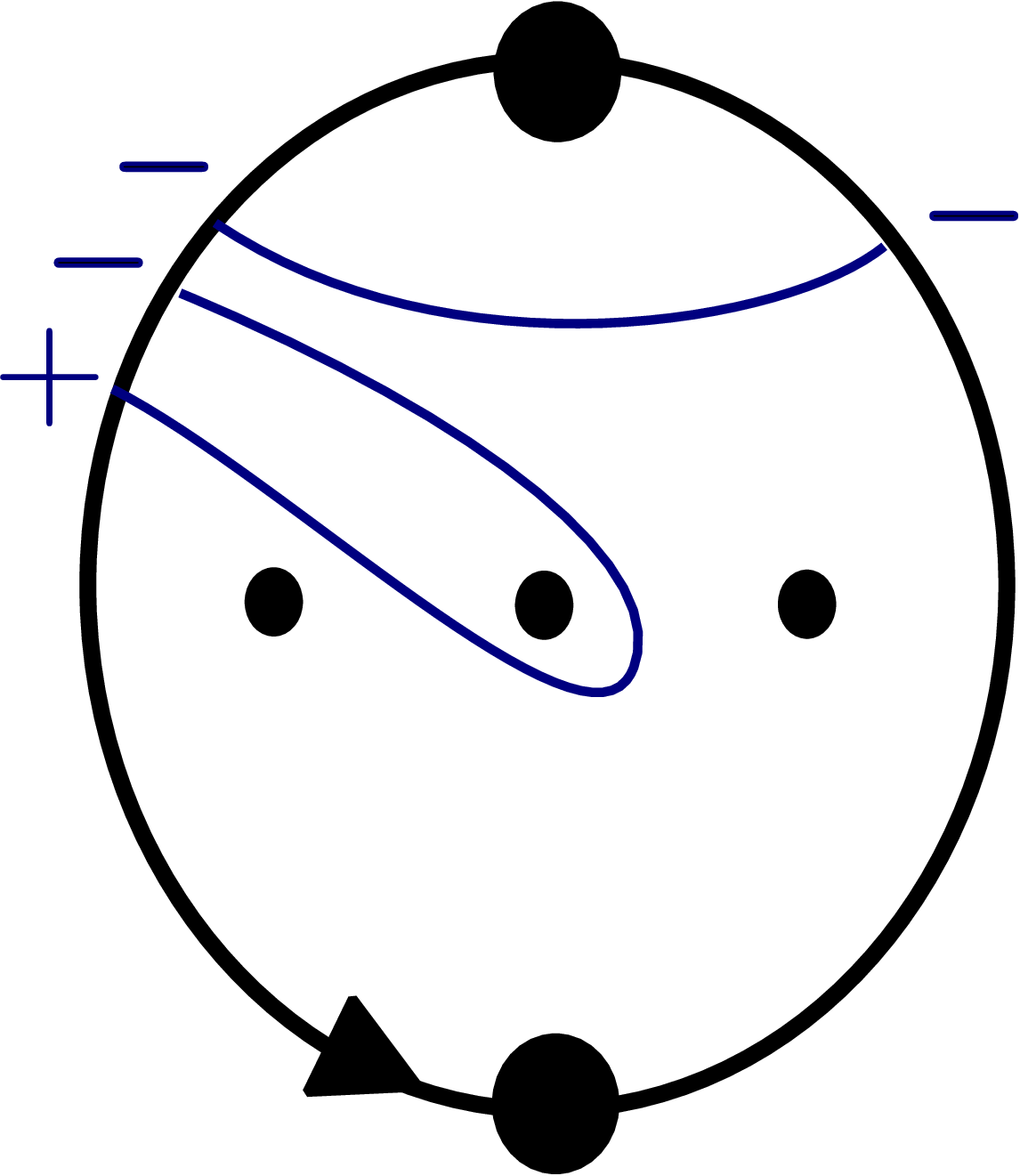}}.
\end{align*}

\begin{align*}
\delta^{(i)}_{--}\alpha_{+-}^{(n)}&=   \adjustbox{valign=c}{\includegraphics[width=1.5cm]{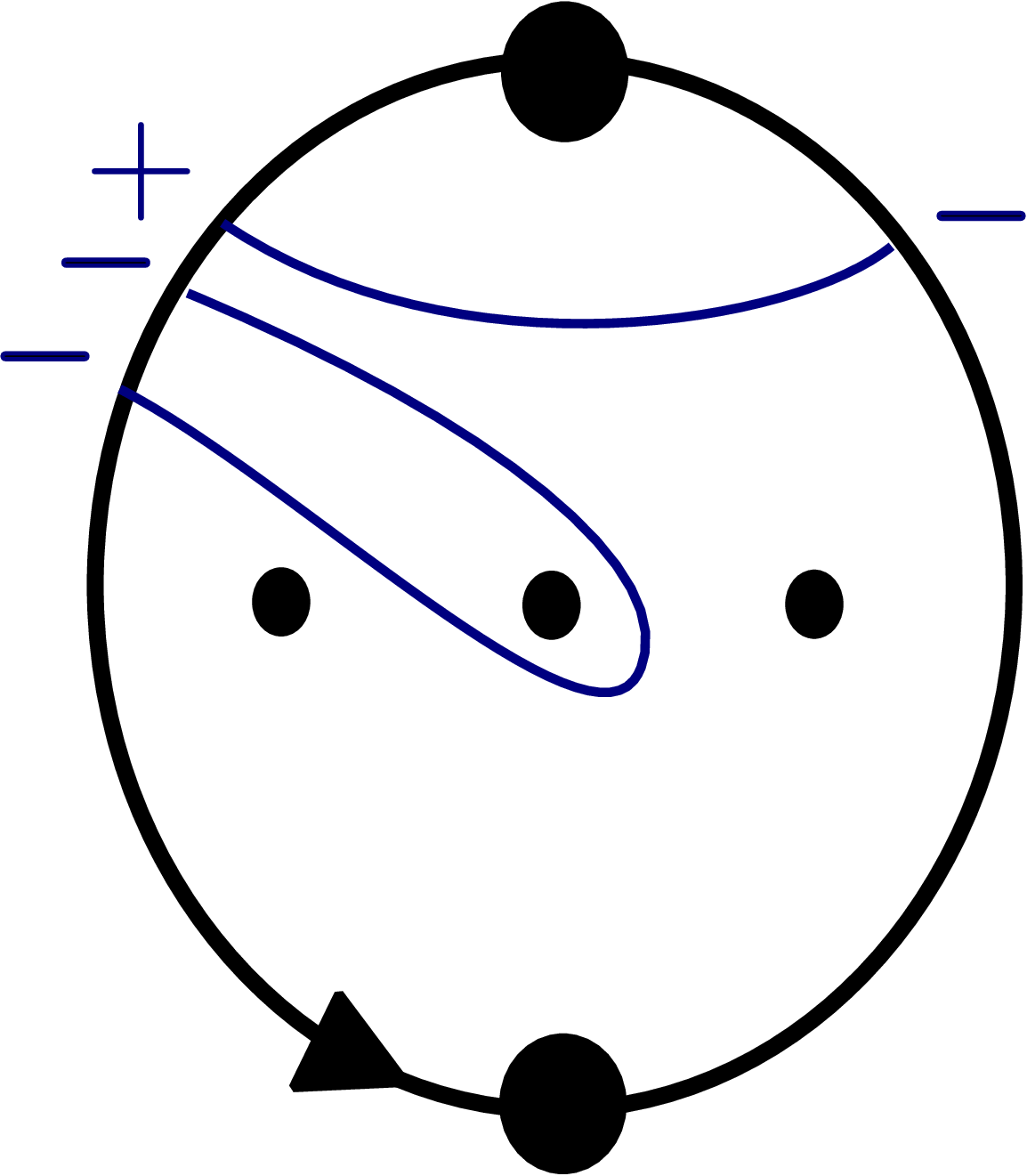}}= q   \adjustbox{valign=c}{\includegraphics[width=1.5cm]{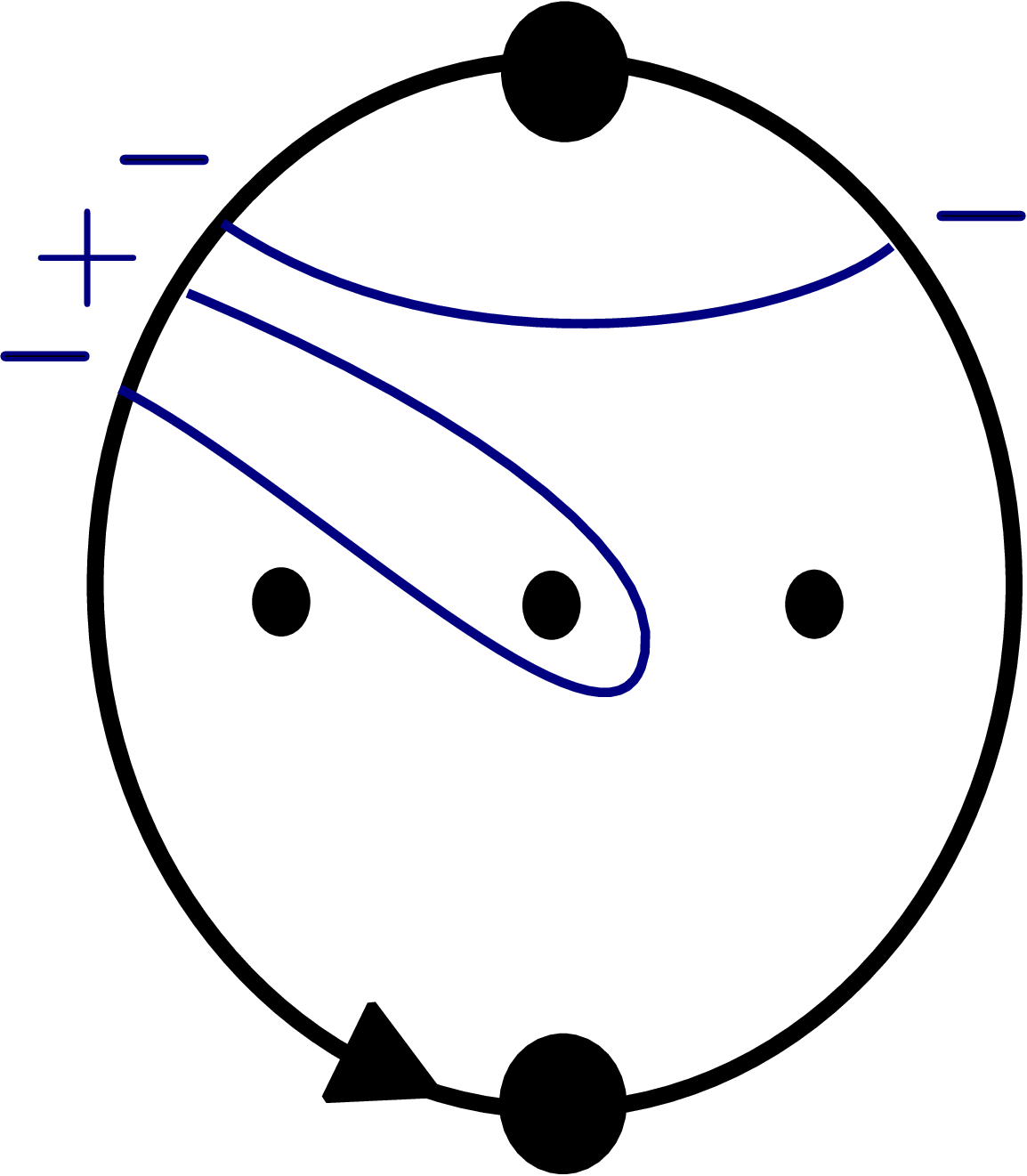}}+A^{1/2}  \adjustbox{valign=c}{\includegraphics[width=1.5cm]{Poisson4.eps}} 
{} & =q^2   \adjustbox{valign=c}{\includegraphics[width=1.5cm]{Poisson6.eps}} +A^{-1/2}  \adjustbox{valign=c}{\includegraphics[width=1.5cm]{Poisson5.eps}} + A^{-1/2}  \adjustbox{valign=c}{\includegraphics[width=1.5cm]{Poisson4.eps}}.
\end{align*}
From which we deduce:
\begin{align*}
[ \alpha_{+-}^{(n)}, \delta^{(i)}_{--}]&= (A^{3/2}-A^{-1/2}) \adjustbox{valign=c}{\includegraphics[width=1.5cm]{Poisson4.eps}} +(q^{-1}-q^2) \adjustbox{valign=c}{\includegraphics[width=1.5cm]{Poisson6.eps}}\\
 {} & =(A^{3/2}-A^{-1/2})\left( A^{1/2} \adjustbox{valign=c}{\includegraphics[width=1.5cm]{Poisson7.eps}}-A^{5/2} \adjustbox{valign=c}{\includegraphics[width=1.5cm]{Poisson8.eps}}\right) +(q^{-1}-q^2) \adjustbox{valign=c}{\includegraphics[width=1.5cm]{Poisson6.eps}} \\
 {} & =  (q-1) \delta^{(i)}_{--}\alpha^{(n)}_{+-}+ (q-q^2)\delta_{+-}^{(i)}\alpha_{--}^{(n)}+(q^{-1}-q^{2})\delta_{-+}^{(i)} \alpha^{(n)}_{--}.
\end{align*}

The $10$ other equalities are obtained using similar (and easier) skein computations which are left to the reader.
\end{proof}

Let us identify an element $g: \mathbb{V}\to \SL_2$ in $\mathcal{G}^{out}$ with the pair $(g(v_L), g(v_R))$. 

\begin{lemma}\label{lemma_orbit3}
Let $\rho \in \overline{\mathcal{R}}_{\SL_2}(\mathbb{D}_n)$ and write 
$$\rho(\alpha^{(n)})= \begin{pmatrix} 0 & -1 \\ 1 & 0 \end{pmatrix} \begin{pmatrix} \lambda & x \\ 0 & \lambda^{-1}\end{pmatrix}= \begin{pmatrix} 0 & -\lambda^{-1} \\ \lambda & x \end{pmatrix} \mbox{ and }\rho(\beta^{(n)})= \begin{pmatrix} 0 & -1 \\ 1 & 0 \end{pmatrix} \begin{pmatrix} \mu & 0 \\ y & \mu^{-1}\end{pmatrix}= \begin{pmatrix} -y & -\mu^{-1} \\ \mu & 0 \end{pmatrix}.$$
Consider the matrices
$$d_t:=  \begin{pmatrix} e^{\frac{\lambda t}{2}} &0 \\ 0 & e^{-\frac{\lambda t}{2}}\end{pmatrix}, \quad
g_t:= \begin{pmatrix} e^{\frac{xt}{2}} & \frac{1}{x\lambda}(e^{\frac{xt}{2}}-e^{-\frac{3xt}{2}}) \\ 0 & e^{-\frac{xt}{2}}\end{pmatrix}, \quad h_t:= \begin{pmatrix} e^{\frac{xt}{2}} & 0 \\ \frac{\lambda}{x} (e^{-\frac{xt}{2}}-e^{\frac{3xt}{2}}) & e^{-\frac{xt}{2}}\end{pmatrix}.
$$
When $x=0$, we set instead $g_t=h_t=\mathds{1}_2$.
Then 
$$ \phi^t_{\alpha^{(n)}_{++}} \rho= \phi^{-t}_{\alpha^{(n)}_{--}}\rho = (d_t, d_{-t}) \cdot \rho, \quad \mbox{ and }\phi^t_{\alpha^{(n)}_{+-}} \rho = (g_t,h_t) \cdot \rho. $$
Similarly, consider the matrices: 
$$\widetilde{d_t}:=  \begin{pmatrix} e^{-\frac{\mu t}{2}} &0 \\ 0 & e^{\frac{\mu t}{2}}\end{pmatrix}, \quad
\widetilde{g_t}:= \begin{pmatrix} e^{\frac{yt}{2}} & 0 \\ \frac{\mu}{y} (e^{-\frac{yt}{2}}-e^{\frac{3yt}{2}}) & e^{-\frac{yt}{2}}\end{pmatrix},
 \quad \widetilde{h_t}:= \begin{pmatrix} e^{\frac{yt}{2}} & \frac{1}{y\mu}(e^{\frac{yt}{2}}-e^{-\frac{3yt}{2}}) \\ 0 & e^{-\frac{yt}{2}}\end{pmatrix}.
$$
When $y=0$, we  set instead $\widetilde{g_t}=\widetilde{h_t}=\mathds{1}_2$.
Then 
$$ \phi^t_{\beta^{(n)}_{++}} \rho= \phi^{-t}_{\beta^{(n)}_{--}}\rho = (\widetilde{d_t}, \widetilde{d_{-t}}) \cdot \rho, \quad \mbox{ and }\phi^t_{\beta^{(n)}_{-+}} \rho =  (\widetilde{g_t},\widetilde{h_t}) \cdot \rho. $$

\end{lemma}

\begin{proof}
Let us first compute the image of $\rho$ by the Hamiltonian flow $\phi^t_{\alpha_{+-}^{(n)}}$. For $f \in \mathcal{O}[\overline{\mathcal{R}}_{\SL_2}(\mathbb{D}_n)]$ a regular function, we write $f^t:=f(\phi^t_{\alpha_{+-}^{(n)}} \rho)$ so $\frac{d}{dt} f^t = \{\alpha_{+-}^{(n)}, f \}^t$ and $f^0=f(\rho)$. 
 Using the formulas of Lemma \ref{lemma_orbit2}, we are reduced to solve the system of differential equations 
 
 \begin{align*}
  \frac{d}{dt} (\delta^{(i)}_{++})^t& = - (\delta^{(i)}_{++})^t (\alpha^{(n)}_{+-})^t   & \frac{d}{dt} (\delta^{(i)}_{+-})^t  &= -2 (\delta^{(i)}_{++})^t (\alpha^{(n)}_{--})^t \\
   \frac{d}{dt} (\delta^{(i)}_{-+})^t &=-2 (\delta^{(i)}_{-+})^t (\alpha^{(n)}_{+-})^t &
   \frac{d}{dt} (\delta^{(i)}_{--})^t & =( (\alpha^{(n)}_{+-})^t (\delta^{(i)}_{--})^t-(\alpha_{--}^{(n)})^t(\delta_{+-}^{(i)})^t-3 (\alpha^{(n)}_{--})^t (\delta_{-+}^{(i)})^t) \\
    \frac{d}{dt} (\alpha_{++}^{(n)})^t &= (\alpha_{++}^{(n)})^t (\alpha_{+-}^{(n)})^t &  \frac{d}{dt} (\alpha_{+-}^{(n)})^t &=0.
\end{align*}
Now $(\alpha_{+-}^{(n)})^t$ is the constant function equal to $x$. In particular, if $x=0$, each of the above derivative vanishes so $\rho$ is fixed by $\phi^t_{\alpha_{+-}^{(n)}}$. If $x\neq 0$, write $\rho(\delta^{(i)})= \begin{pmatrix} a & b \\ c & d\end{pmatrix}$. 
Then the system 
$$ \left\{ \begin{array}{ll}
\frac{d}{dt} (\delta^{(i)}_{++})^t& = - x(\delta^{(i)}_{++})^t \\
\frac{d}{dt} (\delta^{(i)}_{-+})^t& = - 2x(\delta^{(i)}_{-+})^t \\
\frac{d}{dt} (\alpha^{(n)}_{++})^t& =  x(\alpha^{(n)}_{++})^t 
\end{array} \right. 
\mbox{ has solution }
\left\{ \begin{array}{ll} 
(\delta_{++}^{(i)})^t & =c e^{-xt} \\
(\delta_{-+}^{(i)})^t & =-a e^{-2xt} \\
(\alpha^{(n)}_{++})^t &= \left((\alpha^{(n)}_{--})^t\right)^{-1} = \lambda e^{xt}
\end{array} \right.$$
The differential equation 
$$  \frac{d}{dt} (\delta^{(i)}_{+-})^t  = -2 (\delta^{(i)}_{++})^t (\alpha^{(n)}_{--})^t =-2 \frac{c}{\lambda} e^{-2xt} \mbox{ integrates to } (\delta^{(i)}_{+-})^t =  \frac{c}{x\lambda} e^{-2xt} +(d-\frac{c}{x\lambda}).$$
Recall that the peripheral curve $\gamma_{p_i}$ is central in $\overline{\mathcal{S}}_{A_{\hbar}}(\mathbb{D}_n)$ so it defines a Casimir element in $\overline{\mathcal{S}}_{+1}(\mathbb{D}_n)$ and $\gamma_{p_i}^t$ is contant. Thus
$$ 0= \frac{d}{dt} (\gamma_{p_i})^t = \frac{d}{dt} \left( (\delta_{+-}^{(i)})^t - (\delta_{-+}^{(i)})^t \right) \mbox{ so } -a = \frac{c}{x\lambda}.$$
This implies in particular that $c\neq 0$ and that we have $-\frac{a}{c}= \frac{1}{x\lambda}$ so $(\delta^{(i)}_{+-})^t = -a e^{-2xt} +a +d$. Eventually
$$  \frac{d}{dt} (\delta^{(i)}_{--})^t  =( (\alpha^{(n)}_{+-})^t (\delta^{(i)}_{--})^t-(\alpha_{--}^{(n)})^t(\delta_{+-}^{(i)})^t-3 (\alpha^{(n)}_{--})^t (\delta_{-+}^{(i)})^t) = x  (\delta^{(i)}_{--})^t + \frac{4a}{\lambda}e^{-3xt} - \frac{a+d}{\lambda}e^{-xt}$$
integrates to 
$$ (\delta^{(i)}_{--})^t  = \frac{1}{c} e^{xt} - \frac{a}{c}(a+d) e^{-xt} + \frac{a^2}{c} e^{-3xt}.$$

Therefore 
\begin{multline*}
 (\phi^t_{\alpha_{+-}^{(n)}} \rho) (\delta^{(i)}) = \begin{pmatrix} a e^{-2xt} & -\frac{1}{c} e^{xt} + \frac{a}{c}(a+d) e^{-xt} - \frac{a^2}{c} e^{-3xt}  \\ ce^{-xt} & -ae^{-2xt} +a+d \end{pmatrix} \\
 =  \begin{pmatrix} e^{\frac{xt}{2}} &- \frac{a}{c}(e^{\frac{xt}{2}}-e^{-\frac{3xt}{2}}) \\ 0 & e^{-\frac{xt}{2}}\end{pmatrix}
  \begin{pmatrix} a & b \\ c & d\end{pmatrix}
   \begin{pmatrix} e^{-\frac{xt}{2}} & \frac{a}{c}(e^{\frac{xt}{2}}-e^{-\frac{3xt}{2}}) \\ 0 & e^{\frac{xt}{2}}\end{pmatrix}
  = g_t \rho(\delta^{(i)}) g_t^{-1}
  \end{multline*} 
 and 
 $$  (\phi^t_{\alpha_{+-}^{(n)}} \rho) (\alpha^{(n)}) = \begin{pmatrix} 0 & - \lambda^{-1}e^{-xt} \\ \lambda e^{xt} & x \end{pmatrix} = g_t \rho(\alpha^{(n)}_{++})  h_t^{-1}.$$

This implies that $\phi^t_{\alpha_{+-}^{(n)}} \rho= (g_t, h_t)\cdot \rho$ as required.
\par Let us now compute the image of $\rho$ by the Hamiltonian flow $\phi_{\alpha_{++}^{(n)}}^t$. Writing $f^t:= f ( \phi_{\alpha_{++}^{(n)}}^t \rho)$ and using the formulas in Lemma \ref{lemma_orbit2},  we have the system of differential equations
$$ \left\{ \begin{array}{l}
\frac{d}{dt} (\alpha^{(n)}_{++})^t =\frac{d}{dt} (\alpha^{(n)}_{--})^t=\frac{d}{dt} (\delta^{(i)}_{+-})^t=\frac{d}{dt} (\delta^{(i)}_{-+})^t=0 \\
\frac{d}{dt} (\delta^{(i)}_{++})^t=- \lambda (\delta^{(i)}_{++})^t, \frac{d}{dt} (\delta^{(i)}_{--})^t= \lambda (\delta^{(i)}_{--})^t, \frac{d}{dt} (\alpha^{(n)}_{+-})^t =-\lambda (\alpha^{(n)}_{+-})^t
\end{array}
\right. $$
 which integrates to 
 $$
 \left\{ \begin{array}{l}
(\alpha^{(n)}_{++})^t = \lambda, (\alpha^{(n)}_{--})^t =\lambda^{-1},  (\delta^{(i)}_{+-})^t= d,  (\delta^{(i)}_{-+})^t=-a \\
 (\delta^{(i)}_{++})^t=c e^{-\lambda t},  (\delta^{(i)}_{--})^t=-be^{-\lambda t}, (\alpha^{(n)}_{+-})^t = x e^{-\lambda t}
 \end{array} \right.
 $$
Therefore
$$ (\phi^t_{\alpha_{++}^{(n)}} \rho) (\delta^{(i)}) = \begin{pmatrix} a & b e^{\lambda t} \\ c e^{-\lambda t} & d \end{pmatrix} = d_t \rho(\delta^{(i)}) d_{-t}, \quad (\phi^t_{\alpha_{++}^{(n)}} \rho) (\alpha^{(n)}) = 
\begin{pmatrix} 0 & - \lambda^{-1} \\ \lambda & x e^{-\lambda t} \end{pmatrix}= d_t \rho(\alpha^{(n)}) d_t.$$
So $\phi^t_{\alpha_{++}^{(n)}} \rho= (d_t, d_{-t})\cdot \rho$ as required. The formulas for the Hamiltonian flows of the function $\beta_{ij}^{(n)}$ are obtained from the ones for $\alpha^{(n)}_{ji}$ using the following symmetry. Consider $\mathbb{D}_n$ as a disc $\mathbb{D}$ (whose boundary contains the boundary edges) with $n$ subdiscs $d_1, \ldots, d_n$ removed and consider the involutive homeomorphism $\theta_n$ of $\mathbb{D}_n$ which is the central symmetry (rotation of angle $\pi$) of $\mathbb{D}$ which sends $d_i$ to $d_{n-i}$ and exchanges $b_L$ and $b_R$. Let $\Theta_n \in \Aut( \overline{\mathcal{S}}_A(\mathbb{D}_n))$ be the involutive algebra isomorphism induced by $\theta_n$. So when $n=1$, $\Theta_1$ is the Cartan involution of Definition \ref{def_Cartan_involution}. Note that $\Theta_n$  exchanges $\alpha^{(n)}_{ij}$ and $\beta^{(n)}_{ji}$. 
$\Theta_n$ defines an involutive Poisson isomorphism on $\overline{\mathcal{R}}_{\SL_2}(\mathbb{D}_n)$ which sends a representation $\rho$ to, say, $\rho^{\theta}$. Note that if  
$$\rho(\alpha^{(n)})= \begin{pmatrix} 0 & -\lambda^{-1} \\ \lambda & x \end{pmatrix}, \rho(\beta^{(n)})=\begin{pmatrix} -y & -\mu^{-1} \\ \mu & 0 \end{pmatrix} \mbox{ then }
 \rho^{\theta}(\alpha^{(n)})= \begin{pmatrix} 0 & -\mu^{-1} \\ \mu & y \end{pmatrix}, \rho(\beta^{(n)})=\begin{pmatrix} -x & -\lambda^{-1} \\ \lambda & 0 \end{pmatrix}.$$
Therefore by the preceding analysis
$$ \phi_{\alpha_{++}^{(n)}}^t \rho^{\theta} = ( \widetilde{d_{-t}}, \widetilde{d_t})\cdot \rho^{\theta} \mbox{ and }  \phi_{\alpha_{+-}^{(n)}}^t \rho^{\theta} = ( \widetilde{h_t}, \widetilde{g_t})\cdot \rho^{\theta}.$$
Since $\theta_n$ exchanges $b_L$ and $b_R$, we have that $\left( (g,h)\cdot \rho \right)^{\theta}= (h,g)\cdot \rho^{\theta}$. Since $\Theta_n$ is Poisson, then 
$$ \phi_{\beta_{++}^{(n)}}^t \rho = \left( \phi_{\alpha_{++}^{(n)}}^t \rho^{\theta}\right)^{\theta} = \left( ( \widetilde{d_{-t}}, \widetilde{d_t})\cdot \rho^{\theta}\right)^{\theta}=(\widetilde{d_t}, \widetilde{d_{-t}}) \cdot \rho \mbox{ and }
 \phi_{\beta_{-+}^{(n)}}^t \rho = \left( \phi_{\alpha_{+-}^{(n)}}^t \rho^{\theta}\right)^{\theta} = \left( ( \widetilde{h_t}, \widetilde{g_t})\cdot \rho^{\theta}\right)^{\theta}=(\widetilde{g_t}, \widetilde{h_t}) \cdot \rho.$$
 This concludes the proof.

\end{proof}

The Lie group action $\mathcal{G}^{out}  \curvearrowright \mathcal{R}_{\SL_2}(\mathbb{D}_n)$ induces, at each point $\rho$,  an infinitesimal action 
$$ a_{\rho} : Lie(\mathcal{G}^{out}) \to T_{\rho} \mathcal{R}_{\SL_2}(\mathbb{D}_n)$$
which is the derivative of the map $\mathcal{G}^{out} \to  \mathcal{R}_{\SL_2}(\mathbb{D}_n), g \mapsto g\cdot \rho$.
The infinitesimal action of the elements in $\mathcal{G}^{out}_{\rho} \subset \mathcal{G}^{out}$ is given by the subspace: 
$$ L_{\rho}:= \{ X \in Lie(\mathcal{G}^{out}): a_{ \rho} (X) \in T_{\rho} \overline{\mathcal{R}}_{\SL_2}(\mathbb{D}_n) \subset T_{\rho} \mathcal{R}_{\SL_2}(\mathbb{D}_n) \mbox{ and }D_{\rho}\alpha_{\partial} X =0 \} .$$
The identification $\mathcal{G}^{out}\cong \SL_2 \times \SL_2, g \mapsto (g(v_L), g(v_R))$ induces an identification $Lie(\mathcal{G}^{out}) \cong \mathfrak{sl}_2 \oplus \mathfrak{sl}_2$. 

\begin{lemma}\label{lemma_orbit4} 
Let $\rho \in \overline{\mathcal{R}}_{\SL_2}(\mathbb{D}_n)$ and write $\rho(\alpha^{(n)})= \begin{pmatrix} 0 & -\lambda^{-1}\\ \lambda & x \end{pmatrix}$ and $\rho(\beta^{(n)})= \begin{pmatrix} -y & -\mu^{-1} \\ \mu & 0 \end{pmatrix}$ as before. Consider the three elements 
$$X_1:= \left( \begin{pmatrix} 1 & 0 \\ 0 & -1 \end{pmatrix}, \begin{pmatrix} -1 & 0 \\ 0 & 1 \end{pmatrix} \right), 
 X_2:= \left( \begin{pmatrix} 0 & 0 \\ 1 & 0 \end{pmatrix}, \begin{pmatrix} -\frac{1}{2} & -1 \\ 0 & \frac{1}{2} \end{pmatrix} \right), 
 X_3 := \left( \begin{pmatrix} 0 & 1 \\ 0 & 0 \end{pmatrix} , \begin{pmatrix} \frac{x \lambda}{2} & 0 \\ - \lambda^2 & - \frac{x\lambda}{2} \end{pmatrix} \right).$$
 Then $(X_1, X_2, X_3)$ is a basis of $L_{\rho}$.
 \end{lemma}
 
 \begin{proof}
 Let $X= \left( \begin{pmatrix} a & c \\ b & -a \end{pmatrix}, \begin{pmatrix} a' & c' \\ b' & -a' \end{pmatrix} \right) \in Lie(\mathcal{G}^{out}) \cong \mathfrak{sl}_2 \oplus \mathfrak{sl}_2$. Then $X\in L_{\rho}$ if and only if $D_{\rho} H_{\partial}\cdot  X=  D_{\rho} \alpha^{(n)}_{-+}\cdot  X =D_{\rho} \beta^{(n)}_{+-}\cdot  X=0$. A straightforward computation shows the equivalences: 
 $$ \left\{ \begin{array}{l} 
 D_{\rho} \alpha^{(n)}_{-+}\cdot  X = 0 \\
 D_{\rho} \beta^{(n)}_{+-}\cdot  X=0 \\
 D_{\rho} \alpha_{\partial}\cdot  X= 0 
 \end{array} \right. 
 \Leftrightarrow 
 \left\{ \begin{array}{l} 
 c\lambda + \lambda^{-1} b'=0 \\
 -\mu^{-1}b - \mu c'=0 \\
 2(a+a') \lambda \mu + \mu x b' +\lambda y b = 0
 \end{array}\right.
  \Leftrightarrow 
 \left\{ \begin{array}{l} 
 b'= -c\lambda^2 \\
 c'= -\mu^{-2}b \\
 a'=-a + \frac{xc\lambda}{2} - \frac{b}{2}
 \end{array}
 \right.
 \Leftrightarrow
 X= a X_1+bX_2+cX_3.
 $$
 This concludes the proof.

 \end{proof}

\begin{proof}[Proof of Proposition \ref{prop_orbit}]
By Lemma \ref{lemma_orbit3}, for each function $f\in \{ \alpha^{(n)}_{ij}, \beta^{(n)}_{ij}, i,j =\pm \}$ and each $t\in \mathbb{C}$,  there exists $g_f^t \in \mathcal{G}^{out}_{\rho}$ such that $\phi^t_f \rho = g_f^t \cdot \rho$. So the $\mathcal{G}_{\mathbb{D}_1}$ orbit of $\rho$ is contained in the set $\{ g \cdot \rho, g\in \mathcal{G}^{out}_{\rho}\}$. Conversely, in order to prove that the whole $\{ g \cdot \rho, g\in \mathcal{G}^{out}_{\rho}\}$ is included in the $\mathcal{G}_{\mathbb{D}_1}$ orbit, we prove that the subspace of $Lie(\mathcal{G}^{out})$ spanned by the vectors $X_f$ such that $g_f^t = (\mathds{1}_2, \mathds{1}_2) + t X_f  + o(t^2)$ is equal to $L_{\rho}$. First 
\begin{align*}
{}&  (g_t, h_t)=  (\mathds{1}_2, \mathds{1}_2) + t (\frac{x}{2} X_1 + \frac{2}{\lambda} X_3)  + o(t^2), \quad \mbox{ so }\quad X_{\alpha_{+-}}=\frac{x}{2} X_1 + \frac{2}{\lambda} X_3; \\
{}& (\widetilde{g_t} , \widetilde{h_t})= (\mathds{1}_2, \mathds{1}_2) + t (\frac{y}{2} X_1 -2 \mu X_2) +o(t^2), \quad \mbox{ so }\quad X_{\beta_{-+}}=\frac{x}{2} X_1 + \frac{2}{\lambda} X_3; \\
 {}& (d_t, d_{-t}) = (\mathds{1}_2, \mathds{1}_2) + t ( \frac{x}{2} X_1 ) +o(t^2) , \quad \mbox{ so }\quad X_{\alpha_{++}}=-X_{\alpha_{--}}= \frac{x}{2}X_1.
 \end{align*}
Similarly, $X_{\beta_{++}}=-X_{\beta_{--}}= \frac{y}{2}X_1$. Using Lemma \ref{lemma_orbit4}, we see that $L_{\rho}=\Span( X_f, f \in \{ \alpha^{(n)}_{ij}, \beta^{(n)}_{ij}, i,j =\pm \})$. This concludes the proof.

\end{proof}

\subsection{Poisson Lie groups}\label{sec_poisson_lie_group}

The coproduct $\Delta$ on $\overline{\mathcal{S}}_{+1}(\mathbb{D}_1)$ is Poisson by Lemma \ref{lemma_poisson_morphisms} so it induces a structure of Poisson Lie group on $X(\mathbb{D}_1)$. As a group, $X(\mathbb{D}_1)$ has a simple description as follows. 
Let $B^+\subset \SL_2$ be the subgroup of upper triangular matrices and $B^- \subset \SL_2$ the subgroup of lower triangular matrices and define $G_0:= B^+ \times B^{-}$ with composition law defined for $g=(g_+,g_-)$ and $h=(h_+,h_-)$ by $gh:=(g_+h_+, g_-h_-)$. Define  an isomorphism of algebraic groups $f: X(\mathbb{D}_1) \xrightarrow{\cong} G_0$ by sending a point $x\in X(\mathbb{D}_1)$, described by a character $\chi_x: \overline{\mathcal{S}}_{+1}(\mathbb{D}_1)\to \mathbb{C}$ to
$$ f(x)= (g_+, g_-) := \left( \begin{pmatrix} \chi_x(\alpha_{++}) & \chi_x(\alpha_{+-}) \\ \chi_x(\alpha_{-+}) & \chi_x(\alpha_{--})\end{pmatrix},  \begin{pmatrix} \chi_x(\beta_{++}) & \chi_x(\beta_{+-}) \\ \chi_x(\beta_{-+}) & \chi_x(\beta_{--})\end{pmatrix}\right).$$
 Note that the isomorphism $\overline{\mathcal{R}}_{\SL_2}(\mathbb{D}_1)\cong G_0$ obtained by composing the isomorphism of Convention \ref{convention_w} with $f$, sends a reduced representation $\rho: \Pi_1(\mathbb{D}_1) \to \SL_2$ to the elements $g=(g_+,g_-)$ where $g_+= C \rho(\alpha)$, $g_-=C\rho(\beta)$ and $C=\begin{pmatrix} 0 & -1 \\ 1 & 0\end{pmatrix}$. 

\vspace{2mm}
\par  Consider the strong embedding $\mu: \mathfrak{m}_1 \hookrightarrow \mathbb{D}_1$ illustrated in Figure \ref{fig_mu}. 
By functoriality, it induces an algebra morphism $\mu_*: B_q\SL_2 \to \Uq$ and a Poisson morphism $\mu^* : X(\mathbb{D}_1)\to \SL_2^{STS}$, where the superscript $STS$ stands to recall that $\SL_2$ is equipped with the Poisson structure derived from the identification $X(\mathfrak{m}_1)\cong \SL_2$ (the so-called Semenov-Tian-Shansky Poisson structure). The composition $\varphi: G^0 \xrightarrow{f^{-1}} X(\mathbb{D}_1) \xrightarrow{\mu^*} \SL_2^{STS}$ is then described by $\varphi(g_+,g_-)= -\tr(g_+ g_-^{-1})$. Said differently, if $g=(g_+,g_-)$ corresponds to a representation $\rho: \Pi_1(\mathbb{D}_1)\to \SL_2$, then $\varphi(g)=-\tr(\rho(\gamma_p))$. 
As a particular case of Proposition \ref{prop_orbit} we obtain the

 \begin{figure}[!h] 
\centerline{\includegraphics[width=5cm]{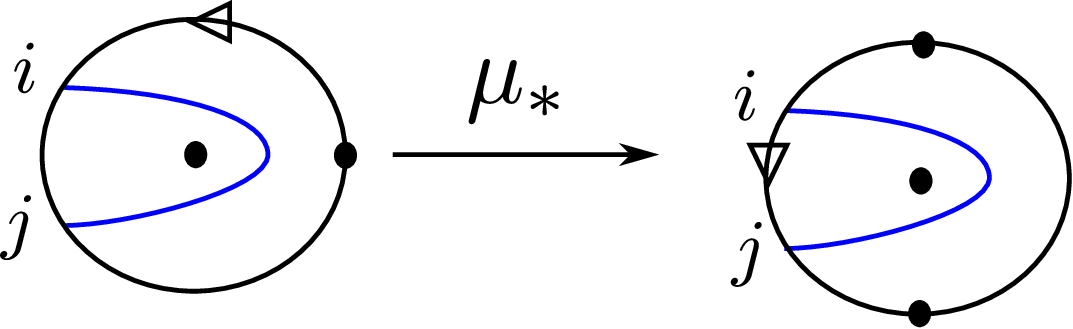} }
\caption{An illustration of the morphism $\mu_*: \mathcal{S}_A(\mathfrak{m}_1)\hookrightarrow \overline{\mathcal{S}}_A(\mathbb{D}_1)$.
} 
\label{fig_mu} 
\end{figure} 

\begin{corollary} The symplectic leaves of $G^0$ are the intersections $\varphi^{-1}(C)\cap (\alpha_{\partial})^{-1}(\{h_{\partial}\})$ where $C\subset \SL_2$ is a conjugacy class and $h_{\partial}\in \mathbb{C}^*$.
\end{corollary}

\subsection{Gluing operation for representation varieties}

Let $\mathbf{\Sigma}_1=(\Sigma_1,\mathcal{A}_1)$ and $\mathbf{\Sigma}_2=(\Sigma_2, \mathcal{A}_2)$ be two marked surfaces, $a_1\in \mathcal{A}_1$ and $a_2\in \mathcal{A}_2$ some boundary edges and let $\mathbf{\Sigma}= (\mathbf{\Sigma}_1 \sqcup \mathbf{\Sigma}_2)_{a_2\# a_1}$ be the marked surface obtained by gluing $\mathbf{\Sigma}_1$ with $\mathbf{\Sigma}_2$ while identifying $a_1$ and $a_2$. Writing $\underline{X}(\mathbf{\Sigma}):= \Specm(\mathcal{S}_{+1}(\mathbf{\Sigma}))$, the coactions $\Delta^L_{a_2}: \mathcal{S}_{+1}(\mathbf{\Sigma}_2) \to \mathcal{O}[\SL_2] \otimes \mathcal{S}_{+1}(\mathbf{\Sigma}_2)$ 
and $\Delta^R_{a_1}: \mathcal{S}_{+1}(\mathbf{\Sigma}_1) \to  \mathcal{S}_{+1}(\mathbf{\Sigma}_1)\otimes \mathcal{O}[\SL_2] $, induce some group actions $\rhd_{a_2}: \SL_2\times \underline{X}(\mathbf{\Sigma}_2) \to \underline{X}(\mathbf{\Sigma}_1)$ and $\lhd_{a_1}: \underline{X}(\mathbf{\Sigma}_1) \times \SL_2 \to \underline{X}(\mathbf{\Sigma}_1)$. Define a (left) group action $\rhd : \SL_2\times \underline{X}(\mathbf{\Sigma}_1)\times \underline{X}(\mathbf{\Sigma}_2) \to \underline{X}(\mathbf{\Sigma}_1)\times \underline{X}(\mathbf{\Sigma}_2)$ by 
$$g \rhd (x_1,x_2) := (x_1 \lhd_{a_1}g^{-1}, g\rhd_{a_2} x_2) \quad \mbox{, for } g\in \SL_2, x_1\in \underline{X}(\mathbf{\Sigma}_1), x_2 \in \underline{X}(\mathbf{\Sigma}_2).$$
We denote by $\Delta^L_{a_1,a_2}$ the associated comodule map.
Denote also by $\underline{\pi}: \underline{X}(\mathbf{\Sigma}_1)\times \underline{X}(\mathbf{\Sigma}_2) \to \underline{X}(\mathbf{\Sigma})$ the morphism defined by $\theta_{a_1\# a_2}: \mathcal{S}_{+1}(\mathbf{\Sigma}) \to \mathcal{S}_{+1}(\mathbf{\Sigma}_1)\otimes  \mathcal{S}_{+1}(\mathbf{\Sigma}_2)$. 

\begin{proposition} Suppose that the connected component of $\mathbf{\Sigma}_1$ which contains $a_1$ contains at least another boundary edge. Then $\underline{\pi}: \underline{X}(\mathbf{\Sigma}_1)\times \underline{X}(\mathbf{\Sigma}_2) \to \underline{X}(\mathbf{\Sigma})$ is a principal $\SL_2$ bundle with $\SL_2$ action defined by $\rhd$.
\end{proposition}

\begin{proof}

By Theorem \ref{theorem_exactsequence} we have a left exact sequence
$$ 0 \to \mathcal{S}_{+1}(\mathbf{\Sigma})\xrightarrow{\theta_{a_2\# a_1}} \mathcal{S}_{+1}(\mathbf{\Sigma}_1) \otimes \mathcal{S}_{+1}(\mathbf{\Sigma}_2) \xrightarrow{\Delta^R_{a_1}\otimes \id - \id \otimes  \Delta^L_{a_2}} \mathcal{S}_{+1}(\mathbf{\Sigma}_1) \otimes \mathcal{O}[\SL_2]\otimes  \mathcal{S}_{+1}(\mathbf{\Sigma}_2).$$
Consider the  isomorphism 
$$\psi:= (\sigma \otimes \id)\left(\id \otimes  \left( (\mu \otimes \id)(S\otimes \Delta_{a_2}^L)\right)\right): \mathcal{S}_{+1}(\mathbf{\Sigma}_1) \otimes \mathcal{O}[\SL_2]\otimes  \mathcal{S}_{+1}(\mathbf{\Sigma}_2)\to  \mathcal{O}[\SL_2]\otimes\mathcal{S}_{+1}(\mathbf{\Sigma}_1) \otimes   \mathcal{S}_{+1}(\mathbf{\Sigma}_2)$$
with inverse $\psi^{-1}:= \left(\id \otimes  \left( (\mu \otimes \id)(S\otimes \Delta_{a_2}^L)\right)\right)(\sigma \otimes \id)$. 
A straightforward computation shows that: 
$$ \psi \circ \left(\Delta^R_{a_1}\otimes \id \right)  = \Delta^L_{a_1, a_2} \quad \mbox{ and }\psi \circ \left( \id \otimes  \Delta^L_{a_2}\right) = \eta \otimes \id\otimes \id.$$
Therefore the following sequence is exact: 
$$ 0 \to \mathcal{S}_{+1}(\mathbf{\Sigma})\xrightarrow{\theta_{a_2\# a_1}} \mathcal{S}_{+1}(\mathbf{\Sigma}_1) \otimes \mathcal{S}_{+1}(\mathbf{\Sigma}_2) \xrightarrow{ \Delta^L_{a_1, a_2} -\eta \otimes \id \otimes \id} \mathcal{O}[\SL_2]\otimes \mathcal{S}_{+1}(\mathbf{\Sigma}_1) \otimes   \mathcal{S}_{+1}(\mathbf{\Sigma}_2).$$
So $\mathcal{S}_{+1}(\mathbf{\Sigma})$ is identified with the subset of $\mathcal{S}_{+1}(\mathbf{\Sigma}_1) \otimes \mathcal{S}_{+1}(\mathbf{\Sigma}_2)$ of coinvariant vectors for the $\mathcal{O}[\SL_2]$ coaction. Dualizing this discussion, we obtain that $\underline{X}(\mathbf{\Sigma})$ gets identified with the GIT quotient $(\underline{X}(\mathbf{\Sigma}_1)\times \underline{X}(\mathbf{\Sigma}_2))\sslash \SL_2$ with quotient map $\underline{\pi}$. In particular $\underline{\pi}$ is surjective and $\SL_2$ acts transitively on the fibers. It remains to prove that the $\SL_2$ action on $\underline{X}(\mathbf{\Sigma}_1)\times \underline{X}(\mathbf{\Sigma}_2)$ is free. Let $g\in \SL_2$ and $(x_1,x_2) \in \underline{X}(\mathbf{\Sigma}_1)\times \underline{X}(\mathbf{\Sigma}_2)$ such that $g\rhd(x_1,x_2) = (x_1,x_2)$ and let us prove that $g= \mathds{1}_2$. Fix an arbitrary spin function $w$ on $\mathbf{\Sigma}_1\sqcup \mathbf{\Sigma}_2$ and let $(\rho_1, \rho_2)$ be the image of $(x_1,x_2)$ under the identification $\underline{X}(\mathbf{\Sigma}_1)\times \underline{X}(\mathbf{\Sigma}_2)\cong \mathcal{R}_{\SL_2}(\mathbf{\Sigma}_1) \times \mathcal{R}_{\SL_2}(\mathbf{\Sigma}_2)$ of Theorem \ref{theorem_classical_limit}. 
By hypothesis, $\mathbf{\Sigma}_1$ contains another boundary edge $b_1 \neq a_1$ in the same connected component than $a_1$. Consider an arbitrary path $\alpha= v_{a_1}\to v_{b_1}$ in $\Pi_1(\Sigma_1, \mathbb{V}_1)$. 
Then $\left(g\rhd (\rho_1,\rho_2)\right) (\alpha)=g  \rho_1(\alpha)$. So if $g\rhd(x_1,x_2) = (x_1,x_2)$ then $g \rho_1(\alpha)= \rho_1(\alpha)$. This implies that $g=\mathds{1}_2$ so the $\SL_2$ action is free and the $\SL_2$ action is principal.
\end{proof}

We would like an analogous result for $X(\mathbf{\Sigma}):= \Specm(\overline{\mathcal{S}}_{+1}(\mathbf{\Sigma}))$. Dualizing the following diagram
$$ \begin{tikzcd}
0 \ar[r] & 
\mathcal{S}_{+1}(\mathbf{\Sigma}) 
\ar[r, "\theta_{a_2\#a_1}"] \ar[d] &
\mathcal{S}_{+1}(\mathbf{\Sigma}_1)\otimes \mathcal{S}_{+1}(\mathbf{\Sigma}_2) 
\ar[r, shift left,  "\Delta_{a_1,a_2}^L"] \ar[r, shift right,  "\eta\otimes \id"'] \ar[d] & 
\mathcal{O}[\SL_2] \otimes \mathcal{S}_{+1}(\mathbf{\Sigma}_1)\otimes \mathcal{S}_{+1}(\mathbf{\Sigma}_2) 
\ar[d] \\
0 \ar[r] & 
\overline{\mathcal{S}}_{+1}(\mathbf{\Sigma}) 
\ar[r, "\theta_{a_2\#a_1}"]  &
\overline{\mathcal{S}}_{+1}(\mathbf{\Sigma}_1)\otimes \overline{\mathcal{S}}_{+1}(\mathbf{\Sigma}_2) 
\ar[r, shift left,  "\Delta_{a_1,a_2}^L"] 
\ar[r, shift right,  "\eta\otimes \id"']  & 
\mathbb{C}[X^{\pm1}]  \otimes\overline{\mathcal{S}}_{+1}(\mathbf{\Sigma}_1)\otimes \overline{\mathcal{S}}_{+1}(\mathbf{\Sigma}_2) 
\end{tikzcd}$$
where the vertical arrows are the quotient maps and the first line is a coequalizer (but not the second), we obtain the following diagram
$$ \begin{tikzcd}
\mathbb{C}^* \times X(\mathbf{\Sigma}_1) \times X(\mathbf{\Sigma}_2) 
\ar[r, shift left, "\rhd"] \ar[r, shift right, "inv"']  \ar[d, hook] & 
 X(\mathbf{\Sigma}_1) \times X(\mathbf{\Sigma}_2) 
 \ar[r, "\pi"] \ar[d, hook] & 
 X(\mathbf{\Sigma})
 \ar[d,hook] \\
\SL_2 \times \underline{X}(\mathbf{\Sigma}_1) \times \underline{X}(\mathbf{\Sigma}_2) 
\ar[r, shift left, "\rhd"] \ar[r, shift right, "inv"'] & 
 \underline{X}(\mathbf{\Sigma}_1) \times \underline{X}(\mathbf{\Sigma}_2) 
 \ar[r, "\underline{\pi}"]  & 
 \underline{X}(\mathbf{\Sigma})
 \end{tikzcd}
 $$ 
where the vertical are the inclusion maps and the embedding $\iota : \mathbb{C}^* \hookrightarrow \SL_2$, $\iota(z)=\begin{pmatrix} z & 0 \\ 0 & z\end{pmatrix}$ and $inv(g, x_1,x_2):=(x_1,x_2)$. 

\begin{proposition}\label{prop_gluing_reduced} Suppose that for $i=1,2$, the boundary component of $\mathbf{\Sigma}_i$ which contains $a_i$ contains at least another boundary edge. Then ${\pi}: {X}(\mathbf{\Sigma}_1)\times {X}(\mathbf{\Sigma}_2) \to {X}(\mathbf{\Sigma})$ is a principal $\mathbb{C}^*$ bundle with $\mathbb{C}^*$ action defined by $\rhd$.
\end{proposition}

Recall that $B^+ \subset \SL_2$ denotes the subset of upper triangular matrices.
\begin{lemma}\label{lemma_gluing_reduced} Let $g\in \SL_2$ and $M,N \in B^+$ such that $Mg \in B^+$ and $g^{-1}N \in B^+$. Then $g$ is diagonal.
\end{lemma}

\begin{proof} The proof is a straightforward verification left to the reader. \end{proof}

\begin{proof}[Proof of Proposition \ref{prop_gluing_reduced}]
We need to prove that $(1)$ $\pi$ is surjective and $(2)$ if $(x_1,x_2) \in X(\mathbf{\Sigma}_1)\times X(\mathbf{\Sigma}_2)$ and $g\in \SL_2$ are such that $g \rhd(x_1,x_2) \in  X(\mathbf{\Sigma}_1)\times X(\mathbf{\Sigma}_2)$ then $g\in \iota( \mathbb{C}^*) \subset \SL_2$. 
This will imply that for $x \in X(\mathbf{\Sigma})$, the fiber $\pi^{-1}(x) \subset \underline{\pi}^{-1}$ gets identified with the subgroup $\mathbb{C}^* \subset \SL_2$ through an identification $\underline{\pi}^{-1}(x)\cong \SL_2$. So it will prove that $\pi$ is a $\mathbb{C}^*$ principal bundle. 
\par $(1)$ The surjectivity of $\pi$ is proved in \cite[Theorem $3.10$]{KojuKaruo_RepRSSkein}. $(2)$ Let $(x_1,x_2) \in X(\mathbf{\Sigma}_1)\times X(\mathbf{\Sigma}_2)$ and $g\in \SL_2$  such that $g \rhd(x_1,x_2)=(x_1\lhd_{a_1} g^{-1}, g\rhd_{a_2} x_2)  \in  X(\mathbf{\Sigma}_1)\times X(\mathbf{\Sigma}_2)$ and let us prove that $g$ is diagonal. For $i=1,2$, write $\chi_{x_i}: \mathcal{S}_{+1}(\mathbf{\Sigma}_i)\to \mathbb{C}$ for the character defined by $x_i$ and let $p_i$ be one of the endpoint of $a_i$ such that $p_1$ gets glued to $p_2$ in $\mathbf{\Sigma}$. Recall that we denote by $\alpha(p_i) \subset \Sigma_i$ the corner arcs.
 Consider the identities 
\begin{align*} {}& \Delta_{a_1}^R \begin{pmatrix} \alpha(p_1)_{++} & \alpha(p_1)_{+-} \\ \alpha(p_1)_{-+} & \alpha(p_1)_{--} \end{pmatrix} =\begin{pmatrix} \alpha(p_1)_{++} & \alpha(p_1)_{+-} \\ \alpha(p_1)_{-+} & \alpha(p_1)_{--} \end{pmatrix}  \otimes \begin{pmatrix} a & b \\ c & d \end{pmatrix} \quad \mbox{ and } \\
 {}& \Delta_{a_2}^L \begin{pmatrix} \alpha(p_2)_{++} & \alpha(p_2)_{+-} \\ \alpha(p_2)_{-+} & \alpha(p_2)_{--} \end{pmatrix} = \begin{pmatrix} a & b \\ c & d \end{pmatrix} 
 \otimes \begin{pmatrix} \alpha(p_2)_{++} & \alpha(p_2)_{+-} \\ \alpha(p_2)_{-+} & \alpha(p_2)_{--} \end{pmatrix}  \end{align*}
and write $M:= \chi_{x_1} \begin{pmatrix} \alpha(p_1)_{++} & \alpha(p_1)_{+-} \\ \alpha(p_1)_{-+} & \alpha(p_1)_{--} \end{pmatrix} $ and $N:= \chi_{x_2}  \begin{pmatrix} \alpha(p_2)_{++} & \alpha(p_2)_{+-} \\ \alpha(p_2)_{-+} & \alpha(p_2)_{--} \end{pmatrix} $. The fact
that $(x_1,x_2) \in X(\mathbf{\Sigma}_1)\times X(\mathbf{\Sigma}_2)$ implies that $M,N \in B^+$. The fact that $(x_1\lhd_{a_1}g^{-1}, g\rhd_{a_2} x_2) \in X(\mathbf{\Sigma}_1)\times X(\mathbf{\Sigma}_2)$ implies that $Mg^{-1} \in B^+$ and $Ng \in B^+$. So Lemma \ref{lemma_gluing_reduced} implies that $g$ is diagonal. This concludes the proof.
\end{proof}

Let $n\geq 2$. Then $\mathbb{D}_n$ is obtained by gluing $n$ copies of $\mathbb{D}_1$ together, so we have a splitting morphism $\theta: \overline{\mathcal{S}}_{+1}(\mathbb{D}_n) \hookrightarrow \overline{\mathcal{S}}_{+1}(\mathbb{D}_1)^{\otimes n}$ which induces a map $\pi: (G_0)^n \cong X(\mathbb{D}_1)^n \to X(\mathbb{D}_n)$. The group  $\mathcal{G}^{inner}:= (\mathbb{C}^*)^{n-1}$ acts on $(G_0)^n$ by the formula: 
\begin{multline}\label{eq_action}(z_1, \ldots, z_{n-1}) \rhd \left( g_1, \ldots , g_n \right) := (g_1 \cdot \iota(z_1)^{-1}, \ldots, \iota(z_{i-1})\cdot g_i \cdot \iota(z_i)^{-1}, \ldots, \iota(z_{n-1})\cdot g_n), \\ \mbox{, for }(z_1, \ldots, z_{n-1})\in \mathcal{G}^{inner}, (g_1, \ldots ,g_n)\in (G_0)^n, 
\end{multline}
where for $g=(g_-, g_+)$ we set $\iota(z)\cdot g= (\iota(z)g_-, \iota(z)g_+)$ and $g\cdot \iota(z)=(g_- \iota(z), g_+ \iota(z))$. Applying Proposition \ref{prop_gluing_reduced} $n-1$ times, we obtain the following.

\begin{corollary}\label{coro_gluing}
The morphism  $\pi: (G_0)^n \to X(\mathbb{D}_n)$ is a principal $\mathcal{G}^{inner}$ bundle. 
\end{corollary}




\section{Representations of reduced stated skein algebras at roots of unity}\label{sec_RepSSkein}

\begin{notations}\label{notations_RootsUnity}  Let  $\zeta^{1/2}\in \mathbb{C}$  a root of unity of odd order $N\geq 3$ and denote by $\mathcal{S}_{\zeta}(\mathbf{\Sigma})$ the corresponding stated skein algebra at $A^{1/2}=\zeta^{1/2}$. In all the paper, when a property is stated for $\mathcal{S}_A(\mathbf{\Sigma})$, it means that it holds for any ground ring $k$ with any $A^{1/2}\in k^{\times}$ whereas when a property is stated for $\mathcal{S}_{\zeta}(\mathbf{\Sigma})$ it means it only holds when $\zeta^{1/2}$ is a root of unity of odd order.
\end{notations}

\subsection{Centers of reduced stated skein algebras and shadows}

 Consider a stated arc $\alpha_{ij}$ in some marked surface $\mathbf{\Sigma}$ and denote by $\alpha_{ij}^{(N)}$ the stated tangle obtained by taking $N$ parallel copies of $\alpha_{ij}$ pushed along the framing direction. In the case where both endpoints of $\alpha$ lye in two distinct boundary edges, one has the equality
$ \alpha_{ij}^{(N)} = (\alpha_{ij})^N$
in $\overline{\mathcal{S}}_A(\mathbf{\Sigma})$, but when both endpoints lye in the same boundary edge, they are distinct. More precisely, suppose the two endpoints, say $v$ and $w$, of $\alpha$ lie in the same boundary edge  with $h(v)>h(w)$. Then  $\alpha_{ij}^{(N)}$ is defined by a stated tangle $(\alpha^{(N)}, s)$ where $\alpha^{(N)}=\alpha_1 \cup \ldots \cup \alpha_N$ represents $N$ copies of $\alpha$ and the endpoints $v_i, w_i$ of the copy $\alpha_i$ are chosen such that $h(v_N) > \ldots > h(v_1) > h(w_N)>\ldots >h(w_1)$.

\begin{definition}[Chebyshev polynomials]  The $N$-th \textbf{Chebyshev polynomial} of the first kind is the polynomial  $T_N(X) \in \mathbb{Z}[X]$ defined by the recursive formulas $T_0(X)=2$, $T_1(X)=X$ and $T_{n+2}(X)=XT_{n+1}(X) -T_n(X)$ for $n\geq 0$.
\end{definition}

\begin{theorem}\label{theorem_Frobenius}(\cite{BonahonWongqTrace} for unmarked surfaces,  \cite{KojuQuesneyClassicalShadows} for marked surfaces;  see also \cite{BloomquistLe, LePaprocki2018})
There is an embedding 
$$ Fr_{\mathbf{\Sigma}}: {\mathcal{S}}_{+1}(\mathbf{\Sigma}) \hookrightarrow \mathcal{Z}({\mathcal{S}}_{\zeta}(\mathbf{\Sigma}))$$
sending the (commutative) algebra at $+1$ into the center of the stated skein algebra at $\zeta^{1/2}$. Moreover, $Fr_{\mathbf{\Sigma}}$ is characterized by the facts that if $\gamma$ is a loop, then $Fr_{\mathbf{\Sigma}}(\gamma) = T_N(\gamma)$ and if $\alpha_{ij}$ is a stated arc, then $Fr_{\mathbf{\Sigma}}(\alpha_{ij})= \alpha_{ij}^{(N)}$. It passes to the quotient to define an embedding (still denoted by the same letter): 
$$ Fr_{\mathbf{\Sigma}}: \overline{\mathcal{S}}_{+1}(\mathbf{\Sigma}) \hookrightarrow \mathcal{Z}(\overline{\mathcal{S}}_{\zeta}(\mathbf{\Sigma})).$$
\end{theorem}
$Fr_{\mathbf{\Sigma}}$ is called the \textbf{Frobenius morphism}.
It is proved in \cite{KojuQuesneyClassicalShadows} that  ${\mathcal{S}}_A(\mathbf{\Sigma})$ is generated by the classes of loops and stated arcs so $Fr_{\mathbf{\Sigma}}$ is indeed characterized by specifying its image on these generators.
A generalization of Theorem \ref{theorem_Frobenius} for marked $3$-manifolds was done in \cite{LeKauffmanBracket} in the unmarked case and \cite{BloomquistLe} for marked $3$-manifolds (except for marked $3$-manifolds, the Frobenius morphism is no longer injective in general, see \cite{CostantinoLe_SSkeinTQFT}). 

\begin{notations} We denote by $Z_{\mathbf{\Sigma}}$ the center of $\overline{\mathcal{S}}_{\zeta}(\mathbf{\Sigma})$ and by $Z^0_{\mathbf{\Sigma}}\subset Z_{\mathbf{\Sigma}}$ the image of the Frobenius morphism. \end{notations}

 \begin{definition}[Punctures and boundary central elements]
 \begin{enumerate}
 \item   For $p\in \mathring{\mathcal{P}}$ an inner puncture, we denote by $\gamma_p \in \overline{\mathcal{S}}_A(\mathbf{\Sigma})$ the class of a peripheral curve encircling $p$ once.
 \item  For $\partial \in \Gamma^{\partial}$ a boundary component which intersects $\mathcal{A}$ non trivially, denote by $p_1, \ldots, p_n$ the boundary punctures in $\partial$ cyclically ordered by $\mathfrak{o}^+$  and define the elements in $\overline{\mathcal{S}}_A(\mathbf{\Sigma})$:
  $$ \alpha_{\partial} := \alpha(p_1)_{++} \ldots \alpha(p_n)_{++}, \quad \mbox{ and } \quad \alpha_{\partial}^{-1}:= \alpha(p_1)_{--} \ldots \alpha(p_n)_{--}.$$
 \end{enumerate}
 \end{definition}
 
 The elements $\alpha_{\partial}$ and $\alpha_{\partial}^{-1}$ are inverse to each other, hence the notation. 

\begin{theorem}[\cite{FrohmanKaniaLe_UnicityRep} for unmarked surfaces, \cite{KojuAzumayaSkein} for marked surfaces]\label{theorem_center} 
The elements $\gamma_p$ and $\alpha_{\partial}^{\pm 1}$ are central and $Z_{\mathbf{\Sigma}}$ is generated by the image of the Frobenius morphism together with these elements. More precisely, $Z_{\mathbf{\Sigma}}$ is isomorphic to the quotient of $$ \overline{\mathcal{S}}_{+1}(\mathbf{\Sigma})[ \gamma_p, \alpha_{\partial}^{\pm 1}; p \in \mathring{P}, \partial \in \Gamma^{\partial} ]$$ by the relations $T_N(\gamma_p)=Fr_{\mathbf{\Sigma}}(\gamma_p)$ and $\alpha_{\partial}^N = Fr_{\mathbf{\Sigma}}(\alpha_{\partial})$.
 \end{theorem}

 \begin{definition}[Classical schemes]
 \begin{enumerate}
 \item We write $X(\mathbf{\Sigma}):= \Specm( \overline{\mathcal{S}}_{+1}(\mathbf{\Sigma}) )$ and $\widehat{X}(\mathbf{\Sigma}):= \Specm(Z_{\mathbf{\Sigma}})$ and denote by $\pi: \widehat{X}(\mathbf{\Sigma}) \to {X}(\mathbf{\Sigma})$ the morphism induced by $Fr_{\mathbf{\Sigma}}$. We extend the Poisson bracket $\{\cdot, \cdot \}$ of $\overline{\mathcal{S}}_{+1}(\mathbf{\Sigma})$ to $Z_{\mathbf{\Sigma}}$ by declaring that the additional generators $\gamma_p$ and $\alpha_{\partial}$ are Casimir elements, that is that for any $x\in Z_{\mathbf{\Sigma}}$ then $\{\gamma_p, x\}=\{\alpha_{\partial}, x\}=0$. So $\pi$ is a Poisson morphism. 
 \item For $\mathbf{\Sigma}_1, \mathbf{\Sigma}_2$ with boundary edges $a_1$ and $a_2$, we denote by $ p_{a_1 \# a_2} : {X}(\mathbf{\Sigma}_1)\times {X}(\mathbf{\Sigma}_2) \to {X}(\mathbf{\Sigma}_1\cup_{a_1 \# a_2} \mathbf{\Sigma}_2)$ and $\widehat{p}_{a_1 \# a_2} : \widehat{X}(\mathbf{\Sigma}_1)\times \widehat{X}(\mathbf{\Sigma}_2) \to \widehat{X}(\mathbf{\Sigma}_1\cup_{a_1 \# a_2} \mathbf{\Sigma}_2)$ the dominant maps induced by $\theta_{a_1 \# a_2}$.
 \end{enumerate}
 \end{definition}
 Note that the Frobenius morphisms are compatible with the splitting morphisms in the sense that $Fr_{\mathbf{\Sigma}} \circ \theta_{a\# b} = \theta_{a\# b}\circ Fr_{\mathbf{\Sigma}}$.
 By Theorem \ref{theorem_classical_limit}, $X(\mathbf{\Sigma})$ is isomorphic to $\overline{\mathcal{R}}_{\SL_2}(\mathbf{\Sigma})$ and Theorem \ref{theorem_center} implies that,  as a set,   $\widehat{X}(\mathbb{D}_n)$ is described as 
 \begin{multline*}  \widehat{X}(\mathbb{D}_n) = \{ \widehat{x}= (\rho, h_{p_1}, \ldots, h_{p_n}, h_{\partial}): \rho\in \overline{\mathcal{R}}_{\SL_2}(\mathbf{\Sigma}),
 h_{p_i}\in \mathbb{C} \mbox{ is such that  } T_N(h_{p_i})=- \tr(\rho(\gamma_{p_i})) \mbox{ for }i=1, \ldots, n, 
\\ h_{\partial} \in \mathbb{C}^* \mbox{ is such that  } h_{\partial}^N = (\rho(\alpha)_{-+})(\rho(\beta)_{-+}) \}.\end{multline*}

 \begin{definition}[Semi-weight representations] A representation $r: \overline{\mathcal{S}}_A(\mathbf{\Sigma}) \to \End(V)$ is called a \textbf{semi-weight representation} (resp. \textbf{weight representation}) if it is semi-simple over $Z^0_{\mathbf{\Sigma}}$ (resp. over $Z_{\mathbf{\Sigma}}$). 
 \end{definition}
 
  \begin{definition}[Shadows]
  Let $r:\overline{\mathcal{S}}_A(\mathbf{\Sigma})$ be an indecomposable semi-weight representation.
  \begin{enumerate}
  \item The \textbf{classical shadow} of $r$ is the maximal ideal $x_r:= \ker(r) \cap Z^0_{\mathbf{\Sigma}} \in X(\mathbf{\Sigma})$.
  \item The \textbf{full shadow} of $r$ is  the (unique) maximal ideal $\widehat{x}_r \in \widehat{X}(\mathbf{\Sigma})$ containing $\mathcal{I}_r$. 
  \end{enumerate}
  \end{definition}
We refer to \cite{KojuKaruo_RepRSSkein} for proofs that the full shadow is indeed  unique.

\subsection{Azumaya loci}

  \begin{definition}[Almost Azumaya algebras] A $\mathbb{C}$ associative, unital algebra $\mathscr{A}$ is said \textbf{almost Azumaya} if $(i)$ $\mathscr{A}$ is affine (finitely generated as an algebra), $(ii)$ $\mathscr{A}$ is prime and $(iii)$ $\mathscr{A}$ is finitely generated as a module over its center $Z$. 
  \end{definition}
  Let $Q(Z)$ denote the fraction field of $Z$ obtained by localizing by every non zero element. By a theorem of Posner-Formaneck (see \cite[Theorem $I.13.3$]{BrownGoodearl}), $\mathscr{A}\otimes_Z Q(Z)$ is a central simple algebra and there exists a finite extension $F$ of $Q(Z)$ such that $\mathscr{A} \otimes_Z F\cong \Mat_D(F)$ is a matrix algebra of size $D$. The integer $D$ is called the \textbf{PI-degree} of $\mathscr{A}$ and is characterized by the formula $\dim_{F}(\mathscr{A}\otimes_ZF) = D^2$. Let $X:= \Specm(Z)$ and for $x\in X$ corresponding to a maximal ideal $\mathfrak{m}_x$, consider the finite dimensional algebra $\mathscr{A}_x:= \quotient{\mathscr{A}}{\mathfrak{m}_x \mathscr{A}}$. 
  Like before, a representation $\rho: \mathscr{A}\to \End(V)$ is a \textbf{weight representation} if $V$ is semi-simple as a $Z$-module. If $\rho$ is weight indecomposable, it sends central elements to scalars so it induces a point $x\in X$ such that $\rho$  induces by quotient a representation $\rho : \mathscr{A}_x \to \End(V)$. The point $x\in X$ will be referred to as the shadow of $\rho$. 
  
  \begin{definition}[Azumaya locus] The \textbf{Azumaya locus} of $\mathscr{A}$ is the subset
  $$ \mathcal{AL}:= \{x \in X : \mathscr{A}_x \cong \Mat_D(\mathbb{C}) \}.$$
  \end{definition}
  
  Note that $\Mat_D(\mathbb{C})$ is simple and every non trivial simple module is isomorphic to the standard module $\mathbb{C}^D$.
So if $x\in \mathcal{AL}$, there exists a unique indecomposable representation of $\mathscr{A}$ with shadow $x$ (up to isomorphism); this representation is irreducible and has dimension $D$ and is called \textbf{the Azumaya representation} at $x$.
The following theorem is essentially due to De Concini-Kac in their original work on quantum enveloping algebras \cite{DeConciniKacRepQGroups} and was further generalized by several authors including \cite{DeConciniLyubashenko_OqG, Brown_AL_discriminant, BrownGoodearl, FrohmanKaniaLe_UnicityRep}.
\begin{theorem}\label{theorem_AL} Let $\mathscr{A}$ be an almost Azumaya algebra with PI-degree $D$.
\begin{enumerate}
\item (\cite[Theorem III.I.7]{BrownGoodearl}) The Azumaya locus is a Zariski open dense subset.
\item  (\cite[Theorem III.I.6]{BrownGoodearl}) If $\rho: \mathscr{A}\to \End(V)$ is an irreducible representation whose  shadow is not in the Azumaya locus, then $\dim(V)<D$. So the Azumaya locus admits the following alternative description:
$$\mathcal{AL}=\{ x \in \mathcal{X} | x \mbox{ is the shadow of an irreducible representation of maximal dimension }D\}.$$
  Therefore, if $\rho:\mathscr{A} \to \End(V)$ is a $D$-dimensional central representation inducing $x\in \mathcal{X}$, then $\rho$ is irreducible if and only if $x\in \mathcal{AL}$.
\end{enumerate}
\end{theorem}

 \begin{theorem}[\cite{FrohmanKaniaLe_UnicityRep} for unmarked surfaces, \cite{KojuAzumayaSkein} for marked surfaces]\label{theorem_almost_Azumaya} 
$\overline{\mathcal{S}}_{\zeta}(\mathbf{\Sigma})$ is almost Azumaya. If $\mathbf{\Sigma}=(\Sigma_{g,n}, \mathcal{A})$ is a genus $g$ surface with $n$ boundary components with negative Euler characteristic, then $\overline{\mathcal{S}}_{\zeta}(\mathbf{\Sigma})$ has PI-degree $D_{\mathbf{\Sigma}}= N^{3g-3+n+|\mathcal{A}|}$. 
 \end{theorem}

The Azumaya locus of $\overline{\mathcal{S}}_{\zeta}(\mathbf{\Sigma})$ was computed explicitly in \cite{KojuKaruo_RepRSSkein} when $\mathbf{\Sigma}$ is essential with at least two boundary edges per connected component (see also \cite{KojuKaruo_Azumaya, FKL_GeometricSkein} for the unmarked cases). In the particular case where $\mathbf{\Sigma}=\mathbb{D}_n$ we obtain the following theorem.

\begin{theorem}(Particular case of \cite[Theorem $1.2$]{KojuKaruo_RepRSSkein}) \label{theorem_AzumayaLocus}
The Azumaya locus of $\overline{\mathcal{S}}_{\zeta}(\mathbb{D}_n)$ is the set of elements $\widehat{x}=(\rho, h_{p_1}, \ldots, h_{p_n}, h_{\partial})\in \widehat{X}(\mathbb{D}_n)$ such that for all $1\leq i \leq n$, then either $\rho(\gamma_{p_i})\neq \pm \mathds{1}_2$ or $h_{p_i}=\pm 2$. 
\end{theorem}
In this case, the Azumaya locus coincides with the regular locus of $\widehat{X}(\mathbb{D}_n)$.

\subsection{Poisson orders and gauge equivalence}\label{sec_PO_GaugeEquivalence}

The theory of Poisson orders was created by DeConcini-Kac in their original work on quantum enveloping algebras \cite{DeConciniKacRepQGroups} and generalized by Brown-Gordon in \cite{BrownGordon_PO}.

\begin{definition}[Poisson orders]
A \textbf{Poisson order} is a $4$-tuple $(\mathscr{A}, {X}, \phi, D)$ where: 
\begin{enumerate}
\item $\mathscr{A}$ is an almost Azumaya algebra with center ${Z}$; 
\item ${X}$ is a Poisson affine $\mathbb{C}$ -variety; 
\item $\phi: \mathcal{O}[{X}] \hookrightarrow {Z}$ a finite injective morphism of algebras; 
\item $D: \mathcal{O}[{X}] \to Der (\mathscr{A}) : z\mapsto D_z$ a linear map such that for all $f,g \in  \mathcal{O}[{X}]$, we have
$$ D_f(\phi(g))= \phi(\{f,g\}).$$
\end{enumerate}
\end{definition}

Let us suppose that $X$ is smooth and denote by $Z^0\subset Z$ the image of $\phi$.
Using $\phi$, $\mathscr{A}$ has a natural structure of module over $\mathcal{O}[X]$ so it defines a coherent sheaf $\mathcal{M}$ over $X$ such that $D$ plays the role of a connection over $\mathcal{M}$ compatible with the Poisson structure and  $\mathscr{A}_x=\quotient{\mathscr{A}}{\mathfrak{m}_x \mathscr{A}}$ is the space of sections of the fiber $x$.   A brilliant idea of De Concini-Kac is that if $h\in \mathcal{O}[X]$ and $\Phi^t_h$ is the Hamiltonian flow of $h$ sending a point $x\in X$ to a point $y\in X$, then we can use the connection $D$ to "lift" the flow to the sections of the sheaf $\mathcal{M}$ to an automorphism $\Psi^t_h:= \exp(t D_h)$ inducing an isomorphism between $\mathscr{A}_{x}$ and $\mathscr{A}_y$. One issue arises from the fact that we work with the ring $\mathcal{O}[X]$ of polynomial functions on $X$ and that the Hamiltonian flow is analytic in essence, that it is if $h,f\in \mathcal{O}[X]$ then the solution $f_t=\phi_t^h (f)$ of the equation $\frac{d}{dt}f_t=\{f_t, h\}$ with $f_0=f$  might not live in $\mathcal{O}[X]$ but in the larger ring $\mathcal{O}^{an}[X]$ of complex analytic functions on $X$. Similarly,  for $a \in \mathscr{A}$ and $t\in \mathbb{C}$, the power series  $\exp(t D_ha)$ might not converge in $\mathscr{A}$. Following De Concini-Kac in \cite{DeConciniKacRepQGroups}, we introduce the completions
$$ \widehat{Z}:= Z \otimes_{Z_0} \mathcal{O}^{an}[X] \mbox{ and }\widehat{\mathscr{A}}:= \mathscr{A}\otimes_{Z_0}\mathcal{O}^{an}[X].$$
For $h\in \mathcal{O}[X]$ and $z\in \mathbb{C}$ then  $\Psi^t_h := \exp(tD_h)$ is a well-defined automorphism of $\widehat{\mathscr{A}}$ and if $\phi_h^t x =y$ then $\Psi^t_h$ induces an isomorphism $\mathscr{A}_x\cong \mathscr{A}_y$. The \textbf{gauge group} is the subgroup $\mathcal{G}\subset \Aut(\widehat{\mathscr{A}})$ generated by the automorphisms $\Psi^t_h$. The main theorem of the theory is the following. 

\begin{theorem}(\cite[Theorem $4.2$]{BrownGordon_PO}) If $x,y\in X$ belong to the same symplectic leaf, there exists an element of $\mathcal{G}$ which induces an isomorphism $\mathscr{A}_x \cong \mathscr{A}_y$.
\end{theorem}

This theorem can be improved as follows.

\begin{definition}[Equivariant Poisson orders] Let $G$ be an affine Lie group. A Poisson order $(\mathscr{A}, {X}, \phi, D)$ is said $G$-\textbf{equivariant} if $G$ acts  on $\mathscr{A}$ by automorphism such that its action preserves $\phi(\mathcal{O}[{X}])\subset \mathscr{A}$ and such that it is $D$ equivariant in the sense that for every $g\in G$, $z\in \mathcal{O}[{X}]$ and $a\in \mathscr{A}$, one has
$$ D_{g\cdot z}(a) = g D_z (g^{-1}a).$$
\end{definition}

In this case the action of $g\in G$ induces an isomorphism $\mathscr{A}_x\cong \mathscr{A}_{gx}$. Let us call \textit{equivariant  symplectic leaves} the $G$-orbits of the symplectic  symplectic leaves. The \textbf{extended gauge group} is the subgroup $\widehat{\mathcal{G}}\subset \Aut(\widehat{\mathscr{A}})$ generated by $\mathcal{G}$ and $G$. 

\begin{theorem}\label{theorem_PO_equivariant}(\cite[Proposition $4.3$]{BrownGordon_PO}) If $x,y\in {X}$ belong to the same equivariant symplectic  leaf then there exists an element of $\widehat{\mathcal{G}}$ which induces an isomorphism $\mathscr{A}_x\cong \mathscr{A}_y$.
\end{theorem}

\par  Here is our main source of examples of Poisson orders; we closely follow \cite{BrownGordon_PO} to which we refer for further details. Let $\mathcal{A}_v$ a free affine $\mathbb{C}[v^{\pm 1}]$-algebra which is a domain. Let   $N\geq 1$ and, writing $t:= v^N-1$, consider  the $\quotient{\mathbb{C}[v^{\pm 1}]}{(v^N-1)}$, the algebra $\mathcal{A}_N:= \quotient{\mathcal{A}_v}{t}$ and denote by $\pi: \mathcal{A}_v \to \mathcal{A}_N$ the quotient map. By fixing a basis $\mathcal{B}$ of $\mathcal{A}_v$ we can define a linear
 embedding $\hat{\cdot}: \mathcal{A}_{N} \to \mathcal{A}_v$ sending a basis element $b\in \mathcal{B}$ seen as element in $\mathcal{A}_{N}$ to the same element $\hat{b}$ seen as an element in $\mathcal{A}_v$. Note that $\hat{\cdot}$ is a left inverse for $\pi$. Suppose that the  algebra $\mathcal{A}_{+1}=\mathcal{A}_q\otimes_{q=1}\mathbb{C}$ is commutative and suppose there exists a central embedding  $\phi_N: \mathcal{A}_{+1} \hookrightarrow \mathcal{Z}(\mathcal{A}_{N})$ into the center of $\mathcal{A}_{N}$ such that $\mathcal{A}_N$ is finitely generated as an $\mathcal{A}_{+1}$ module. The pair $(\mathcal{A}_v, \phi_N)$ will be referred to a \textbf{quantum algebra}.
 Write $\mathcal{X}:=\mathrm{Specm}(\mathcal{A}_{+1})$ and define $D: \mathcal{A}_{+1} \to Der(\mathcal{A}_N)$ by the formula
 $$ D_xy:= \pi \left( \frac{ [\widehat{\phi_N(x)}, \hat{y}]}{N(q^N-1)} \right) .$$
Clearly $D_x$ is a derivation, is independent on the choice of the basis $\mathcal{B}$ and preserves $\phi_N(\mathcal{A}_{+1})$, so it defines a Poisson bracket $\{\cdot, \cdot \}_N$ on $\mathcal{A}_{+1}$ by 
$$
 D_x \phi_N(y)= \phi_N (\{x,y\}_N).
 $$
So  $\mathcal{X}=\mathrm{Specm}(\mathcal{A}_{+1})$ is a Poisson variety. Let $\zeta$ be a root of unity of order $N$, write $\mathcal{A}_{\zeta}=\mathcal{A}_v \otimes_{v=\zeta} \mathbb{C}$ and consider $D: \mathcal{A}_{+1} \to Der(\mathcal{A}_{\zeta})$ and $\phi_{\zeta}$ obtained by tensoring by $\mathbb{C}$. 
Then $(\mathcal{A}_{\zeta}, \mathcal{X}, \phi_{\zeta}, D)$ is a Poisson order. Now let us apply this construction to the algebra $\mathcal{A}_v:= \overline{\mathcal{S}}_v(\mathbf{\Sigma})$, where $A^{1/2}$ is replaced by the formal parameter $v$, with linear embedding $Fr_N: \overline{\mathcal{S}}_{+1}(\mathbf{\Sigma})\to \overline{\mathcal{S}}_v(\mathbf{\Sigma})$ defined by the same formula than $Fr_{\mathbf{\Sigma}}$. We thus get a linear map $D: \overline{\mathcal{S}}_{+1}(\mathbf{\Sigma})\to Der(\overline{\mathcal{S}}_{\zeta}(\mathbf{\Sigma}))$ such that $(\overline{\mathcal{S}}_{\zeta}(\mathbf{\Sigma}), X(\mathbf{\Sigma}), Fr_{\mathbf{\Sigma}}, D)$ is a Poisson order.
\par Recall from Section \ref{sec_comodule} that the comodule map $\Delta^L: \mathcal{S}_{\zeta}(\mathbf{\Sigma})\to (\mathcal{O}_q[\SL_2])^{\mathcal{A}}\otimes \mathcal{S}_{\zeta}(\mathbf{\Sigma})$ induces the toric coaction $\Delta^L: \overline{\mathcal{S}}_{\zeta}(\mathbf{\Sigma})\to (\mathbb{C}[X^{\pm 1}])^{\mathcal{A}}\otimes \overline{\mathcal{S}}_{\zeta}(\mathbf{\Sigma})$ which thus defines an algebraic action of the affine Lie group $(\mathbb{C}^*)^{\mathcal{A}}=\Specm((\mathbb{C}[X^{\pm 1}])^{\mathcal{A}})$ on 
$\mathcal{S}_{\zeta}(\mathbf{\Sigma})$. This action preserves $Z_{\mathbf{\Sigma}}$ and $Z^0_{\mathbf{\Sigma}}$ and $D$ is equivariant for this action so $(\overline{\mathcal{S}}_{\zeta}(\mathbf{\Sigma}), \widehat{X}(\mathbf{\Sigma}), Fr_{\mathbf{\Sigma}}, D)$ is a $(\mathbb{C}^*)^{\mathcal{A}}$ Poisson order.

\begin{notations} We denote by $\mathcal{G}_{\mathbf{\Sigma}}$ and $\widehat{\mathcal{G}}_{\mathbf{\Sigma}}$ the gauge group and extended gauge group of $\overline{\mathcal{S}}_{\zeta}(\mathbf{\Sigma})$.
\end{notations}

\begin{definition}[Gauge equivalence of representations] 
\begin{enumerate}
\item Let $r: \Uq \to \End(V)$ be a semi-weight representation which sends the elements of  $Z^0_{\mathbb{D}_1}$ to scalar operators, so $r$ admits a well-defined classical shadow $g\in G_0$ and passes to the quotient $r: (\Uq)_g \to \End(V)$. Let $\phi \in \widehat{\mathcal{G}}_{\mathbb{D}_1}$ and consider the vectors space $V^{\phi}:= V$ with module structure $r^{\phi}: (\Uq)_{\phi^* g} \xrightarrow{\phi} (\Uq)_g \xrightarrow{r} \End(V)$. We say that the two modules $V$ and $V^{\phi}$ are \textbf{gauge equivalent}.
\item Let $f: V \to V$ be a $\Uq$ equivariant morphism and define $f^{\phi}: V^{\phi} \to V^{\phi}$ such that $f^{\phi}=f$ as a linear map. Then $f^{\phi}$ is clearly $\Uq$ equivariant as well and $f$ and $f^{\phi}$ are said \textbf{gauge equivalent}.
\end{enumerate}
\end{definition}

Let $i : \mathbb{D}_1\hookrightarrow \mathbb{D}_n$ be the strong embedding sending $\alpha, \beta$ to $\alpha^{(n)}, \beta^{(n)}$ (when $n=2$ we already used it to define the coproduct of $\overline{\mathcal{S}}_A(\mathbb{D}_1)$). Recall that an Azumaya representation $r: \overline{\mathcal{S}}_A(\mathbf{\Sigma}) \to \End(V)$ is an irreducible representation whose full shadow $\widehat{x}= (\rho, h_{p_1}, \ldots, h_{p_n}, h_{\partial}) \in \widehat{X}(\mathbb{D}_n)$ belongs to the Azumaya locus of $\overline{\mathcal{S}}_A(\mathbf{\Sigma})$, that is is such that for $1\leq i\leq n$ then either $\rho(\alpha_{p_i})\neq \pm \mathds{1}_2$ or $h_{p_i}=\pm 2$. By definition of the Azumaya locus, such a representation $r$ is determined, up to isomorphism, by $\widehat{x}$. Define a representation of $\Uq$ by the composition: 
$$ \overline{r} : \Uq \xrightarrow[\cong]{\Psi}\overline{\mathcal{S}}_A(\mathbb{D}_1) \xrightarrow{i_*} \overline{\mathcal{S}}_A(\mathbb{D}_n) \xrightarrow{r} \End(V).$$

\begin{definition}[Standard Azumaya representations]\label{def_standard_representation}
We call $\overline{r}$ the \textbf{standard Azumaya representation of }$\Uq$ with shadow $\widehat{x}$. 
\end{definition}


Recall from Section \ref{sec_moduli_space} that we have fixed identifications 
$$X(\mathbb{D}_n) \cong \overline{\mathcal{R}}_{\SL_2}(\mathbb{D}_n) \subset \mathcal{R}_{\SL_2}(\mathbb{D}_n) \xrightarrow{\cong} \SL_2 \times \Hom(\pi_1(\mathbb{D}_n, \{b_L\}), \SL_2).$$ 
The second isomorphism  sends $\rho$ to the pair $(\rho(\alpha^{(n)}), \restriction{\rho}{\pi_1(\mathbb{D}_n)})$ where $\pi_1(\mathbb{D}_n)$ is the fundamental group with base point $v_{b_L}$. 

\begin{definition}[Gauge equivalence on shadows] Two elements $x_1$ and $x_2$ of $X(\mathbb{D}_n)$, corresponding to some representations $\rho_1$ and $\rho_2$ in $\overline{\mathcal{R}}_{\SL_2}(\mathbb{D}_n)$ respectively, are \textbf{gauge equivalent} if $\restriction{\rho_1}{\pi_1(\mathbb{D}_n)}$ and $\restriction{\rho_2}{\pi_1(\mathbb{D}_n)}$ are conjugate. 
\par Two elements $\widehat{x_1}=(x_1,h_{p_1}^{(1)}, \ldots, h_{p_n}^{(1)}, h_{\partial}^{(1)})$ and $\widehat{x_2}=(x_2,h_{p_1}^{(2)}, \ldots, h_{p_n}^{(2)}, h_{\partial}^{(2)})$ of $\widehat{X}(\mathbb{D}_n)$ are \textbf{gauge equivalent} if $(i)$ $x_1$ and $x_2$ are gauge equivalent and $(ii)$ $h_{p_i}^{(1)}=h_{p_i}^{(2)}$ for all $1\leq i \leq n$.
\end{definition}

\begin{theorem}\label{theorem_GaugeRep} Two standard Azumaya representations $r_1$ and $r_2$ with shadows $\widehat{x}_1, \widehat{x}_2 \in \widehat{X}(\mathbb{D}_n)$ are gauge equivalent if and only if their shadows are gauge equivalent.
\end{theorem}

\begin{remark}
In particular, the set of gauge equivalence classes of standard Azumaya representations which factorize through $\overline{\mathcal{S}}_A(\mathbb{D}_n)$ is in $1:1$ correspondence with the set of elements $([\rho], h_{p_1}, \ldots, h_{p_n})$ where $[\rho]\in \quotient{\Hom(\pi_1(\mathbb{D}_n), \SL_2)}{\SL_2}$,  $h_{p_i}\in \mathbb{C}$ is such that $T_N(h_{p_i})=-\tr(\rho(\alpha_{p_i}))$ and such that for all $1\leq i \leq N$ either $\rho(\alpha_{p_i})\neq \pm \mathds{1}_2$ or $h_{p_i}=\pm 2$. 
\end{remark}

\begin{proof}
We need to compute the orbit of a given point $\widehat{x}=(\rho, h_{p_1}, \ldots , h_{p_n}, h_{\partial}) \in \widehat{X}(\mathbb{D}_n)$ though the full gauge group.
First, note that since any element $f \in \{ \alpha_{\partial}, \gamma_{p_i}, 1\leq i \leq n\}$ is central in $\overline{\mathcal{S}}_A(\mathbb{D}_n)$, then $D_hf=0$ for any regular function $h$ so $\Psi^t_h f = f$. So by Proposition \ref{prop_orbit}, the orbit of $\widehat{x}$ by $\mathcal{G}_{\mathbb{D}_1}$ is the set of elements $(\rho', h_{p_1}, \ldots , h_{p_n}, h_{\partial})$ such that $\restriction{\rho}{\pi_1(\mathbb{D}_n)}$ and $\restriction{\rho'}{\pi_1(\mathbb{D}_n)}$ are conjugate. Let us compute the action of the torus $(\mathbb{C}^*)^{\mathcal{A}}$. Let $\iota: \mathbb{C}^* \hookrightarrow \SL_2$ the embedding defined by $\iota(z):= \begin{pmatrix} z & 0 \\ 0 & z^{-1} \end{pmatrix}$. Let $(z_L, z_R)\in (\mathbb{C}^*)^{\mathcal{A}}$. Unfolding the definitions, we find that the image of $\widehat{x}$ by $(z_L, z_R)$ is $\widehat{x}^{(z_L,z_R)}= \left( (\iota(z_L^{-N}), \iota(z_R^{-N}))\cdot \rho, h_{p_1}, \ldots, h_{p_n}, z_L^2 z_R^2 h_{\partial} \right)$, where $(\iota(z_L^{-N}), \iota(z_R^{-N}))\cdot \rho$ denotes the outer gauge group action of Section \ref{sec_gauge_orbit}. So the torus action preserves the conjugacy class of  $\restriction{\rho}{\pi_1(\mathbb{D}_n)}$ and the values of $h_{p_i}$  and acts transitively on the values $h_{\partial}$. This concludes the proof.

\end{proof}

\subsection{The category of representations of $\Uq$}

Let $\mathcal{C}$ be the category of semi-weight finite dimensional $\Uq$ modules where $q=\zeta^2$.
\par \underline{\textbf{Grading:}} Consider the subalgebra $Z^1_{\mathbb{D}_1}\subset \overline{\mathcal{S}}_{\zeta}(\mathbb{D}_1)$ generated by $Z^0_{\mathbb{D}_1}$ and the boundary elements $\alpha_{\partial}^{\pm 1}$; so we have the inclusions $Z^0_{\mathbb{D}_1} \subset Z^1_{\mathbb{D}_1} \subset Z_{\mathbb{D}_1}$. Since $\alpha_{\partial}^{\pm 1}$ are group-like, $Z^1_{\mathbb{D}_1}$ is a sub-Hopf algebra of $\overline{\mathcal{S}}_{\zeta}(\mathbb{D}_1)$ so $\Specm(Z^1_{\mathbb{D}_1})$ is a Poisson Lie group. Recall the identification $X(\mathbb{D}_1)\cong G^0$ from Section \ref{sec_poisson_lie_group}. Under this identification, $\Specm(Z^1_{\mathbb{D}_1})$ identifies with the group 
$$G:= \{ (g,h_{\partial}) \in G_0\times \mathbb{C}^*: h_{\partial}^N = (g_- g_+)_{++}\}, $$
with composition law $(g,h_{\partial}) \cdot (g', h_{\partial}') = (gg', h_{\partial}h_{\partial}')$. The inclusion $Z^0_{\mathbb{D}_1}\subset Z^1_{\mathbb{D}_1}$ induces the projection $G\to G^0$ sending $(g,h_{\partial})$ to $g$. 
Let $r: \Uq \to \End(V)$ be a representation in $\mathcal{C}$. Since $r$ is semi-weight, then $r(\alpha_{\partial})^N$ is diagonalizable and invertible, so $r(\alpha_{\partial})$ is diagonalizable as well and $Z^1_{\mathbb{D}_1}$ acts semi-simply on $V$. It follows that $\mathcal{C}$ admits a $G$ grading $\mathcal{C}=\sqcup_{g\in G} \mathcal{C}_g$ where $\mathcal{C}_g \subset \mathcal{C}$ is the full subcategory of $\Uq$ modules on which the elements of  $Z^1_{\mathbb{D}_1}$ acts by the scalars given character $\chi_g: Z^1_{\mathbb{D}_1} \to \mathbb{C}$ corresponding to $g$. In particular we have $V_1 \in \mathcal{C}_{g_1}$ , $V_2 \in \mathcal{C}_{g_2}$ implies $V_1\otimes V_2 \in \mathcal{C}_{g_1g_2}$ and $g_1\neq g_2$ implies $\Hom_{\mathcal{C}}(V_1,V_2)=0$. 

\vspace{2mm}
\par \underline{\textbf{Pivotal structure:}} For $V\in \mathcal{C}$, we denote by $V^*:=\Hom(V, \mathbb{C})$ its dual with $\Uq$ action such that for $x\in \Uq$ and $f\in V^*$, $v\in V$, one has $x\cdot f (v):=f(S(x)\cdot v)$. The left evaluation and coevaluation maps are defined by 
$$ \overrightarrow{ev}_V: V^*\otimes V \to \mathds{1}, \quad \overrightarrow{ev}_V(f\otimes v):= f(v), \quad \mbox{ and }\quad  \overrightarrow{coev}_V: \mathds{1}\to V \otimes V^*, \quad \overrightarrow{coev}_V(1):=\sum_i v_i \otimes v_i^*.$$
Here $(v_i)_i$ is an arbitrary basis of $V$ and $(v_i^*)_i$ is the dual basis of $V^*$. 
The right evaluation and coevaluation maps are defined by 
\begin{multline*} \overleftarrow{ev}_V: V\otimes V^* \to \mathds{1}, \quad  \overleftarrow{ev}_V(v\otimes f):= f((K^{1/2}L^{-1/2})^{1-N}\cdot v), \mbox{ and } \\
  \overleftarrow{coev}_V: \mathds{1}\to V^* \otimes V, \quad \overleftarrow{coev}_V(1):=\sum_i v_i^* \otimes (K^{1/2}L^{-1/2})^{N-1}\cdot v_i.
  \end{multline*}
Note that if $V\in \mathcal{C}_g$ then $V^* \in \mathcal{C}_{g^{-1}}$. 

\vspace{2mm}
\par \underline{\textbf{Classification of semi-weight representations of $\Uq$:}}

The classification of all indecomposable modules in $\mathcal{C}$ is fully described in \cite{KojuKaruo_RepRSSkein}; let us summarize part of this classification. 
\begin{definition} An element $\underline{g}=(g,h_{\partial}) \in G$ is said: 
\begin{enumerate}
\item of type $1$ if $\tr(\varphi(g))\neq \pm 2$; 
\item of type $2$ if $\tr(\varphi(g))= \pm 2$ and $\varphi(g)\neq \pm \mathds{1}_2$; 
\item of type $3$ if $\varphi(g)=\pm \mathds{1}_2$.
\end{enumerate}
\end{definition}

Identify $\widehat{X}(\mathbb{D}_1)$ with the set $\widehat{G}=\{ \widehat{g}=(g, h_p, h_{\partial})\in G^0\times \mathbb{C}\times \mathbb{C}^*: T_N(h_p)=-\tr(\varphi(g)), h_{\partial}^N=(g_- g_+)_{++}\}$ and denote by $\pi: \widehat{G} \to G$ the projection map $\pi((g,h_{\partial},h_p))=(g,h_{\partial})$ corresponding to the inclusion $Z^1_{\mathbb{D}_1} \subset Z_{\mathbb{D}_1}$. So $\pi$ is an $N$-fold ramified covering, ramified precisely at elements which are not of type $1$. 
By Theorem \ref{theorem_AzumayaLocus}, the Azumaya locus of $\Uq$ corresponds to the subset of elements $\widehat{g}=(g,h_{\partial},h_p)$ such that either $g$ is of type $1$ or $2$ or $h_p=\pm 2$. For each such $\widehat{g}$, we thus have a unique (up to isomorphism) irreducible representation $V_{\widehat{g}}$ with full shadow $\widehat{g}$ which has dimension $N$. 
Let us say that $P\in \mathcal{C}$ is \textbf{projective} if it is a projective object in $\mathcal{C}$. Note that this does not implies that $V$ is a projective $\Uq$ modules (i.e. $P$ might not be a projective object in the whole category $\Uq-\Mod$).

\begin{theorem}(\cite[Corollary $4.8$]{KojuKaruo_RepRSSkein})\label{theorem_RepQG}
Let $g\in G$. 
\begin{enumerate}
\item If $g$ is of type $1$, then $\mathcal{C}_g$ is semi-simple. Its simple modules are the $N$ Azumaya representations $V_{\widehat{g}}$ for $\widehat{g}\in \pi^{-1}(g)$. In particular, they are all projective. 
\item If $g$ is of type $2$, then $V_{\widehat{g}}$ with $\widehat{g}=(g,h_p) \in \pi^{-1}(g)$ is projective if and only if $h_p=\pm 2$ (so there is only one such module, say $V_{g,\pm 2}$). When $h_p\neq \pm 2$, then $V_{\widehat{g}}$ admits a $2N$-dimensional projective cover $P_{\widehat{g}}$ which fits into an exact sequence: 
$$ 0 \to V_{\widehat{g}} \to P_{\widehat{g}} \to V_{\widehat{g}} \to 0.$$
 Every simple module in $\mathcal{C}_g$ is isomorphic to one of the Azumaya module $V_{\widehat{g}}$. Every projective indecomposable module in $\mathcal{C}_g$ is isomorphic to either $V_{g, \pm 2}$ or a $P_{\widehat{g}}$. Every indecomposable module in $\mathcal{C}_g$ is isomorphic to either a $V_{\widehat{g}}$ or a $P_{\widehat{g}}$ (so there are $N$ classes of indecomposable modules). 
 \item If $g$ is of type $3$, $\mathcal{C}_g$ contains an infinite number of classes of  indecomposable modules described in \cite{KojuKaruo_RepRSSkein}.
\end{enumerate}
\end{theorem}

Since we will need to work with the Azumaya modules $V_{\widehat{g}}$, let us describe them explicitly.

 \begin{definition}[Irreducible representations of $\Uq$]\label{def_QGRep}
 We define three families of indecomposable weight $\Uq$-modules as follows. For $n\in \mathbb{Z}$, we write $[n]:=\frac{q^n-q^{-n}}{q-q^{-1}}$. Consider the vector space $V:= \mathbb{C}^N$ with its canonical basis $(v_0, \ldots, v_{N-1})$. 
\begin{enumerate}
\item For $\lambda, \mu\in \mathbb{C}^*$ and $a,b\in \mathbb{C}$, the module $V(\lambda, \mu, a,b)$ is the space $V$ with module structure given by:
\begin{align*}
& K^{1/2}v_i = \lambda q^{-i}v_i; \quad L^{1/2}v_i = \mu q^i v_i \mbox{, for }i\in \{0, \ldots, N-1\}; \\
& Fv_{N-1}=bv_0; \quad Fv_i=v_{i+1} \mbox{, for }i \in \{0, \ldots, N-2\}; \\
& Ev_0=a v_{N-1}; \quad Ev_{i}= \left( \frac{q^{-i+1}\lambda^2 - q^{i-1}\mu^2}{q-q^{-1}} [i] +ab \right) v_{i-1}\mbox{, for }i\in \{1, \ldots, N-1\}.
\end{align*}

\item For $\lambda, \mu \in \mathbb{C}^*$, $c\in \mathbb{C}$, the module $\widetilde{V}(\lambda, \mu, c)$ is the space $V$ with module structure given by: 
\begin{align*}
& K^{1/2}v_i= \lambda q^{i}v_i; \quad L^{1/2}v_i= \mu q^{-i}v_i \mbox{, for }i \in \{0, \ldots, N-1\}; \\
& Ev_{N-1}= cv_0; \quad Ev_i= v_{i+1} \mbox{, for }i\in \{0, \ldots, N-2\}; \\
& F v_0= 0; \quad Fv_{i} = \left( \frac{ \mu^2q^{-i+1} - \lambda^2 q^{i-1} }{q-q^{-1}}[i] \right) v_i \mbox{, for }i\in \{1, \ldots, N-1\}.
\end{align*}

\item For $\mu \in \mathbb{C}^*$, $\varepsilon \in \{-1,+1\}$, $0\leq n \leq N-1$, the module $S_{\mu, \varepsilon, n}$ has canonical basis $(e_0, \ldots, e_n)$ and module structure given by: 
\begin{align*}
&K^{1/2}e_i = \varepsilon \mu {\zeta}^{n-2i} e_i; \quad L^{1/2}e_i= \mu {\zeta}^{2i-n}e_i \mbox{, for }i\in \{0, \ldots, n\}; \\
&Fe_n=0; \quad Fe_i = e_{i+1} \mbox{, for }i\in \{0,\ldots, n-1\}; \\
&Ee_0=0; \quad Ee_i= \mu^2 [i][n-i+1]e_{i-1}\mbox{, for }i\in \{1,\ldots, n\}.
\end{align*}

\end{enumerate}
 \end{definition}

Any Azumaya module $V_{\widehat{g}}$ is isomorphic to a module either of the form $V(\lambda, \mu ,a b)$ or $\widetilde{V}(\lambda, \mu , c)$ so the above definition gives explicit formulas for these modules. Note that the underlying vector space of $V_{\widehat{g}}$ is $V$ so is independent on $\widehat{g}$.  Moreover \cite[Corollary $4.8$]{KojuKaruo_RepRSSkein} implies that any simple $\Uq$ module is isomorphic to a module in Definition \ref{def_QGRep}.

\par Let us introduce another module. The module $P_{\mathds{1}}$ has 
has canonical basis $(x_0, \ldots, x_{N-1}, y_0, \ldots, y_{N-1})$ and module structure given by: 
\begin{align*}
& K^{1/2} x_i =  {\zeta}^{-2-2i}x_i, K^{1/2}y_i= \varepsilon \mu {\zeta}^{-2i}y_i, \quad L^{1/2} x_i=  {\zeta}^{2i+2} x_i, L^{1/2}y_i=\mu {\zeta}^{2i}y_i \mbox{, for }i\in \{0, \ldots, N-1\}; \\
& F x_i = x_{i+1}, \quad Fy_i= y_{i+1} \quad \mbox{( where }y_N=x_N=0)\mbox{, for }i\in \{0, \ldots, N-1\}; \\
& Ex_0=0, \quad Ex_i= -\mu^2[i][i+1] x_{i-1} \mbox{, for }i\in \{1, \ldots, N-1\}; \\
& Ey_0= 0, \quad E y_i = \mu^2[i] [i+1]y_{i-1}   \mbox{, for }i\in \{1, \ldots, N-1\}. 
\end{align*}
Recall that $\mathds{1}=S_{1,1,0}$ is the trivial module. The following proves that $\mathcal{C}$ is unimodular. 

\begin{lemma}\label{lemma_unimodular} $P_{\mathds{1}}$ is the projective cover of $\mathds{1}$. Moreover $\mathds{1}$ is the unique simple submodule of $P_{\mathds{1}}$. \end{lemma}
\begin{proof}
$P_{\mathds{1}}$ is indecomposable and projective by \cite[Theorem $4.6$]{KojuKaruo_RepRSSkein}. 
By \cite[Lemmas $4.10$]{KojuKaruo_RepRSSkein} we have the following non split exact sequences: 
\begin{align*}
{}&  0 \to V(q^{-1}, q, 0, 0) \to P_{\mathds{1}} \to V(1,1,0,0) \to 0; \\
{}& 0 \to \mathds{1} \to V(q^{-1}, q, 0, 0) \to S_{1,1,N-2} \to 0; \\
{}& 0 \to S_{1,1,N-2} \to V(1,1,0,0) \to \mathds{1} \to 0.
\end{align*}
So $P_{\mathds{1}}$ decomposes along the following "diamond shape" decomposition: 
$$ P_{\mathds{1}}= 
\begin{tikzcd} 
{} & \mathds{1} \ar[ld] \ar[rd] & {} \\
S_{1,1,N-2} \ar[rd]& {} & S_{1,1,N-2}\ar[ld] \\
{} & \mathds{1} &{}
\end{tikzcd}
$$
from which we read that $\mathds{1}$ is both a submodule and a quotient of $P_{\mathds{1}}$ and that the quotient map $P_{\mathds{1}}\to \mathds{1}$ is a projective cover.
\end{proof}

\subsection{Factorizable representations of $\overline{\mathcal{S}}_{\zeta}(\mathbb{D}_n)$}

Note that $\mathbb{D}_n$ is obtained by gluing $n$ copies of $\mathbb{D}_1$ together, so one has a splitting morphism
$$\theta: \overline{\mathcal{S}}_{\zeta}(\mathbb{D}_n) \to \overline{\mathcal{S}}_{\zeta}(\mathbb{D}_1)^{\otimes n} \cong (\Uq)^{\otimes n}.$$
In particular, if $V_1, \ldots, V_n \in \mathcal{C}$, then $V_1\otimes \ldots \otimes V_n$ acquires a structure of semi-weight $\overline{\mathcal{S}}_{\zeta}(\mathbb{D}_n)$ module through $\theta$. We call \textbf{factorizable} a $\overline{\mathcal{S}}_{\zeta}(\mathbb{D}_n)$-module isomorphic to a module of the form $V_1\otimes \ldots \otimes V_n$.

\begin{theorem}(\cite{KojuKaruo_RepRSSkein})\label{theorem_RepRSSkein}
\begin{enumerate}
\item A factorizable module $V_1\otimes \ldots \otimes V_n$ is semi-weight indecomposable.
\item Any simple $\overline{\mathcal{S}}_{\zeta}(\mathbb{D}_n)$ module is factorizable.
\item The class of factorizable modules is stable by the mapping class group action of $\Mod(\mathbb{D}_n)$.
\end{enumerate}
\end{theorem}

In particular, any standard Azumaya representation $r_{\widehat{x}}: \overline{\mathcal{S}}_{\zeta}(\mathbb{D}_n) \to \End(W)$ with full shadow $\widehat{x}\in \widehat{X}(\mathbb{D}_n)$ in the sense of Definition \ref{def_standard_representation} can be obtained as $W=V_{\widehat{g}_1}\otimes \ldots \otimes V_{\widehat{g}_n}$ with $V_{\widehat{g}_i}$ given by Definition \ref{def_QGRep}. Such a decomposition is not unique since different $n$-tuples $(\widehat{g}_1, \ldots, \widehat{g}_n)$ can induce isomorphic $\overline{\mathcal{S}}_{\zeta}(\mathbb{D}_n)$ modules. Let us study this ambiguity.
Since $\theta: \overline{\mathcal{S}}_{\zeta}(\mathbb{D}_n) \to \overline{\mathcal{S}}_{\zeta}(\mathbb{D}_1)^{\otimes n} \cong (\Uq)^{\otimes n}$ satisfies $\theta(Z^0_{\mathbb{D}_n}) \subset( Z^0_{\mathbb{D}_1})^{\otimes n}$ and $\theta(Z_{\mathbb{D}_n}) \subset Z_{\mathbb{D}_1}^{\otimes n}$, it defines a projection map $\pi : (G_0)^n \to X(\mathbb{D}_n)$, described in Corollary \ref{coro_gluing}, and a lift $\widehat{\pi}: (\widehat{G})^n \to \widehat{X}(\mathbb{D}_n)$. Now $W=V_{\widehat{g}_1}\otimes \ldots \otimes V_{\widehat{g}_n}$ has classical shadow $\widehat{x}\in \widehat{X}(\mathbb{D}_n)$ if and only if $\widehat{\pi}(\widehat{g}_1, \ldots , \widehat{g}_n)=\widehat{x}$. If $\widehat{g}_i=(g_i, h_{p_i}, h_{\partial}^{(i)})$, then $\widehat{\pi}(\widehat{g}_1, \ldots, \widehat{g}_n) = (\pi(g_1, \ldots, g_n), h_{p_n}, \ldots, h_{p_n}, h_{\partial})$ where $h_{\partial}=\prod_{i=1}^n h_{\partial}^{(i)}$.

\subsection{The renormalized trace}\label{sec_trace}

For $\mathcal{D}$ a $\mathbb{C}$ linear category, consider the $0$-th Hochschild homology group 
$$\mathrm{HH}_0(\mathcal{D}):= \int^{V \in \mathcal{D}} \End_{\mathcal{D}}(V) \cong \quotient{\oplus_{V \in \mathcal{D}} \End_{\mathcal{D}}(V)}{(uv-vu, u:V\to W, v:W\to V)}.$$
 A \textbf{trace} on $\mathcal{D}$ is a linear map $t : \mathrm{HH}_0(\mathcal{D})\to \mathbb{C}$. This is equivalent to the data of linear maps $t_V: \End_{\mathcal{D}}(V)\to \mathbb{C}$ for every $V\in \mathcal{D}$ such that for any morphisms $u: V\to W$ and $v: W\to V$, one has $t_V(vu)=t_W(uv)$. 
The \textbf{dimension} of $V\in \mathcal{D}$ associated to $t$ is $d(V):= t_V(\id_V)$. 
When $\mathcal{D}$ is monoidal and has a pivotal structure, it admits two natural traces $\qtr^L$ and $\qtr^R$ defined for $f\in \End(V)$ by $\qtr^L(f):=\overleftarrow{ev}_V (\id_{V^*}\otimes f)\overrightarrow{coev}_V$ and $\qtr^R(f):= \overrightarrow{ev}_V (f \otimes \id_{V^*})\overleftarrow{coev}_V$. When both traces coincide, $\mathcal{D}$ is said \textbf{spherical}. The pivotal category $\mathcal{C}$ is spherical. However, for the q-trace, all typical modules $V_{\widehat{x}}$ have dimension $0$, therefore using the q-trace leads to little interesting link invariants outside the colored Jones polynomials. To remedy this problem, the authors of \cite{GeerPatureauTuraev_Trace, GeerPatureauVirelizier_Trace,GeerPatureau_TraceQG} introduce a new trace which we briefly review. It is characterized by the following condition. For $f \in \End_{\mathcal{D}}(V\otimes W)$ define the left and right partial traces: 
\begin{align*}
 {}& \ptr^L(f):= (\overrightarrow{ev}_V \otimes \id_W)( \id_{V^*} \otimes f)(\overleftarrow{coev}_V \otimes \id_W) \in \End_{\mathcal{D}}(W), \quad \mbox{ and }
 \\ {}&  \ptr^R(f):= (\id_V \otimes \overleftarrow{ev}_W)(f\otimes \id_{W^*})(\id_V \otimes \overrightarrow{coev}_W) \in \End_{\mathcal{D}}(V).
 \end{align*}
 A trace $t$ on a pivotal category $\mathcal{D}$ is \textbf{compatible with the pivotal structure} if for every $f\in \End_{\mathcal{D}}(V\otimes W)$ then $t_V( \ptr^R(f))= t_{V\otimes W}(f) = t_W(\ptr^L(f))$. 
 Let $\Proj\subset \mathcal{C}$ denote the full subcategory of projective objects and let $V_0 \in \Proj$ be the module $V_0:= V({\zeta}^{-1/2},{\zeta}^{1/2}, 0, 0)$ as defined in Definition \ref{def_QGRep}.
 
 \begin{theorem}(Analogue of \cite[Theorem $22$]{GeerPatureau_TraceQG})\label{theorem_trace}
 \begin{enumerate}
 \item 
 There exists a unique trace $\tau : \mathrm{HH}_0(\Proj) \to \mathbb{C}$ compatible with the pivotal structure and normalized such that $d(V_0)=1$ where $d$ is the associated dimension. 
 \item For any typical module $V_{\widehat{x}}$ with $\widehat{x}=(g,h_p, h_{\partial})$, one has $d(V_{\widehat{x}})=- \frac{N}{S_N(h_p)}$ where $S_N(X)$ is the Chebyshev polynomial of second type.
 \item Let $g\in G$ be of type $1$, $V\in \mathcal{C}_g$ and $f \in \End_{\mathcal{C}}(V)$. Then for any $h \in \widehat{\mathcal{G}}_{\mathbb{D}_1}$, we have  $\tau_V(f)=\tau_{V^h}(f^h)$, i.e. $\tau$ is gauge invariant.
 \end{enumerate}
 \end{theorem}

Since Theorem \ref{theorem_trace} is formulated for the simply connected quantum group $\widetilde{U}_q\mathfrak{sl}_2\cong \quotient{\Uq}{(H_{\partial}- 1)}$ instead of $\Uq$  and is formulated slightly differently, we briefly review the arguments of \cite{GeerPatureau_TraceQG} leading to this theorem. The only originality here is the verification that $d(V_{(g,h_p,h_{\partial})})$ only depends on $h_p$ and not on $h_{\partial}$. This will imply, thanks to the torus action, that we do not lose any information by working over $U_q\mathfrak{sl}_2$ instead of $\Uq$, that is by considering standard Azumaya representations for which $h_{\partial}=1$. 

\begin{lemma}\label{lemma_trace1} There exists a unique trace  $\tau : \mathrm{HH}_0(\Proj) \to \mathbb{C}$ such that $d(V_0)=1$ and which is "right" compatible with the pivotal structure in the sense that for every $f\in \End_{\mathcal{D}}(V\otimes W)$ then $t_V( \ptr^R(f))= t_{V\otimes W}(f) $. 
\end{lemma}

$\tau$ is called a "right trace" in \cite{GeerPatureau_TraceQG}.

\begin{proof} By \cite[Corollary $3.2.1$]{GeerKujawaPatureau_AmbidextrousObjects}, if $\mathcal{D}$ is a pivotal $\mathbb{C}$ category which is abelian, which admits enough projective, in which every indecomposable are absolutely indecomposable and which is unimodular, then $\Proj(\mathcal{D})$ admits a unique right trace up to scalar normalization. Clearly $\mathcal{C}$ satisfies all these assumptions, the only non trivial being the unimodularity which is proved in Lemma \ref{lemma_unimodular}.
\end{proof}

\begin{definition}(Drinfeld braidings)\label{def_Drinfeld_braiding}
\begin{enumerate}
\item 
 A system of \textbf{unrolled parameters} for the typical module $V=V(\lambda, \mu, a, b)$ is a pair $(h,g)$ of complex numbers such that $\lambda=q^{h/2}$ and $\mu=q^{g/2}$. We can then define the operators $H,G \in \End_{\mathcal{C}}(V)$ by $Hv_i= (h-2i)v_i$ and $Gv_i=(g+2i)v_i$ so $K^{1/2}=q^{H/2}$ and $L^{1/2}=q^{G/2}$. 
\item For $V_1, V_2$ two typical diagonal modules with unrolled parameters $(h_1,g_1)$ and $(h_2,g_2)$ and such that $E^N\otimes F^N$ acts as $0$ on $V_1\otimes V_2$, the \textbf{Drinfeld braiding} is the morphism $c_{V_1,V_2}: V_1\otimes V_2 \to V_2 \otimes V_1$ defined by 
$$ c_{V_1,V_2} = \sigma \circ q^{-H\otimes G /2} \exp_q^{<N}\left( (q-q^{-1})E\otimes F\right), $$ 
where $\sigma (x\otimes y)=y\otimes x$ ,  $q^{-H\otimes G/2}v_i\otimes v_j = q^{-(h-2i)(g+2j)/2}v_i \otimes v_j$ and $\exp_q^{<N}(X)=\sum_{n=0}^{N-1} \frac{q^{n(n-1)/2}}{[n]!}x^n$. 
\end{enumerate}
\end{definition}

\begin{lemma}\label{lemma_trace2}
If $V=V_{g, h_p, h_{\partial}}$ is a typical module, then $d(V)=- \frac{N}{S_N(h_p)}$ where $S_N(X)$ is the Chebyshev polynomial of second type and $d$ the dimension associated to the right trace $\tau$. In particular $d(V)=d(V^*)$ and $d(V^g)=d(V)$ for any $g\in \widehat{\mathcal{G}}_{\mathbb{D}_1}$.
\end{lemma}

\begin{proof} 
Let $V_0:= V({\zeta}^{-1/2},{\zeta}^{1/2},0,0)=V_{(\mathds{1}_2,\mathds{1}_2), -2, 1}$ with unrolled parameters $(h_0,g_0)= (N-1,1-N)$ and note that $V_0^*\cong V_0$ and both $E^N$ and $F^N$ act as $0$ on $V_0$.  Then an easy adaptation of the arguments in \cite[Section $4.5$]{GeerPatureau_TraceQG}, based on the fact that $V_0 \cong V_0^*$,  shows that for any endomorphism $f_V \in \End_{\mathcal{C}}(V_0\otimes V)$ then 
$$d(V)=\frac{\left<\ptr^R(f_V)\right>}{\left< \ptr^L(f_V)\right>}.$$
Write  $V=V(\lambda, \mu, a,b)$ and let $(h,g)$ be a system of unrolled parameters for $V$. Consider the morphism $f_V \in \End_{\mathcal{C}}(V_0\otimes V)$, $f_V:= c_{V, V_0}c_{V_0,V}$.
 Using the formula 
$$ (K^{1/2}L^{-1/2})^{1-N} v_i= \left\{ \begin{array}{ll} q^{(\frac{\alpha-\beta}{2}-2i)(1-N)} v_i& \mbox{, for }v_i \in V; \\ q^{-(1+2i)}v_i & \mbox{, for }v_i\in V_0 \end{array}\right.$$
and 
$$
f_V v_i \otimes v_j  = q^{-\frac{(H\otimes G + G\otimes H)}{2}} v_i\otimes v_j + w = q^{(2i+1-N)(\frac{\beta -\alpha}{2} +2j)} v_i\otimes v_j + w, $$
where $w\in \Span( v_a \otimes v_b, (a,b)\neq (i,j))$, we find that 

\begin{multline*}\ptr^R(f_V) v_0= (\id \otimes \overleftarrow{ev}_V) (f_V \otimes \id) (\id \otimes \overrightarrow{coev}_V) v_0 = 
\sum_i (\id \otimes \overrightarrow{ev}_V) (q^{\frac{(H\otimes G +G\otimes H)}{2}}\otimes \id)( v_0 \otimes (K^{1/2}L^{-1/2})^{1-N}v_i \otimes v_i^*) \\
= \left(\sum_i q^{(\frac{\alpha-\beta}{2} -2i)(1-N)} q^{(1-N)(\frac{\beta-\alpha}{2}+2i)}\right) v_0  =(\sum_i 1)v_0 =Nv_0.
\end{multline*}
So $\left<\ptr^R(f_V)\right>=N$.
Similarly,  setting $z:=\frac{-2+\beta- \alpha}{2}$ such that $V$ has puncture invariant $h_p= -\lambda\mu^{-1}q - \lambda^{-1}\mu q^{-1} = -q^z -q^{-z}$, we find that 
\begin{multline*}
 \ptr^L(f_V) v_0 =  (\overrightarrow{ev}_{V_0} \otimes \id)(\id \otimes f_V)(\overleftarrow{coev}_{V_0}\otimes \id) v_0 = \sum_i  (\overrightarrow{ev}_{V_0} \otimes \id)(\id \otimes q^{\frac{H\otimes G+H\otimes H}{2}})(v_i^* \otimes (K^{1/2}L^{-1/2})^{N-1}v_i \otimes v_0) \\ 
 = \left( \sum_i q^{-(2i+1-N)(\frac{\beta-\alpha}{2})} q^{-(1+2i)}\right)v_0
=\left( \frac{q^{Nz}-q^{-Nz}}{q^z -q^{-z}}\right)v_0= S_N(q^z+q^{-z})v_0= -S_N(h_p)v_0.
\end{multline*}
Therefore $ \left< \ptr^L(f_V)\right>= -S_N(h_p)$ and we find
$$ d(V) =\frac{\left<\ptr^R(f_V)\right>}{\left< \ptr^L(f_V)\right>} = - \frac{N}{S_N(h_p)}.$$

\end{proof}

\begin{lemma}\label{lemma_trace3} The right trace $\tau$ is a trace compatible with the pivotal structure.\end{lemma}

\begin{proof}
This follows from \cite[Theorem $14$]{GeerPatureau_TraceQG} using the fact that $d(V)=d(V^*)$ for all typical module $V$.
\end{proof}

\begin{lemma}\label{lemma_trace4} 
Let $g\in G$ be of type $1$, $V\in \mathcal{C}_g$ and $f \in \End_{\mathcal{C}}(V)$. Then for any $h \in \widehat{\mathcal{G}}_{\mathbb{D}_1}$, we have  $\tau_V(f)=\tau_{V^h}(f^h)$, i.e. $\tau$ is gauge invariant.
\end{lemma}

\begin{proof}
Since $g$ is of type $1$, then $\mathcal{C}_g$ is semi simple with simple objects the $V_{g,h_p}$ so $V\cong \oplus_{h_p/ T_N(h_p)=-\tr(\varphi(g))} n_{h_p} V_{g, h_p}$ for some integers $n_{h_p}\in \mathbb{N}$. Under this isomorphism we have $f = \oplus_{h_p} n_{h_p} c_{h_p} \id_{V_{g, h_p}}$ for some complex numbers $c_{h_p} \in \mathbb{C}$. Thus $\tau_V(f)=\sum_{h_p} n_{h_p}c_{h_p} d(V_{g, h_p})$.  For $h\in \widehat{\mathcal{G}}_{\mathbb{D}_1}$ then $g^h$ is of type $1$ as well so $V^h= \oplus_{h_p} n_p V_{g^h, h_p}$. Since the elements of the full gauge group leave $\gamma_p$ invariant, we have  $(V_{g,h_p})^h =V_{g^h, h_p}$. We deduce that  $f^h = \oplus_{h_p} n_{h_p} c_{h_p} \id_{V_{g^h, h_p}}$ so the fact that $d(V_{g,h_p})= d(V_{g^h, h_p})$ implies that 
$$ \tau_{V^h}(f^h) = \sum_{h_p}n_{h_p} c_{h_p} d(V_{g^h, h_p})= \sum_{h_p} n_{h_p}c_{h_p} d(V_{g, h_p})= \tau_V(f).$$
This concludes the proof.
\end{proof}

\begin{proof}[Proof of Theorem \ref{theorem_trace}] The theorem follows from Lemmas \ref{lemma_trace1}, \ref{lemma_trace2}, \ref{lemma_trace3} and \ref{lemma_trace4}.
\end{proof}

\section{Link invariants derived from stated skein algebras}\label{sec_link_inv}

\subsection{Braids intertwiners}

The \textbf{braid group} $B_n$ is presented by the  generators $\sigma_i^{\pm 1}$ for $1\leq i \leq n-1$ and relations $\sigma_i \sigma_j \sigma_i = \sigma_j \sigma_i \sigma_j$ if $\lvert i-j \rvert =1$ and $\sigma_i \sigma_j= \sigma_j \sigma_i$ if $\lvert i-j \rvert \geq 2$. 
There is a surjective group morphism $B_n \to \mathbb{S}_n, \beta \mapsto \pi_{\beta}$ to the permutation group of $\{1, \ldots, n\}$ sending $\sigma_i$ to the transposition exchanging $i$ and $i+1$. 
 The \textbf{framed braid group} $fB_n$ is presented by the same generators $\sigma_1, \ldots, \sigma_{n-1}$ and additional generators $t_1, \ldots, t_n$ with the same relations than $B_n$ and additional relations $\sigma_j t_i  =  t_{\pi_{\sigma_j}(i)} \sigma_j$ and $t_it_j=t_jt_i$. So the framed braid group is a semi-direct product  $fB_n = B_n \rtimes \mathbb{Z}^n$ and we consider $B_n$ as a subgroup of $fB_n$. 
Recall that $\mathbb{D}_n$ is a disc $\mathbb{D}$ with the interior of $n$ subdiscs $D_1,\ldots, D_n$ removed. The framed braid group can be embedded into the group $\Mod(\mathbb{D}_n)$ of mapping classes of $\mathbb{D}_n$ whose restriction to the boundary $\partial \mathbb{D}_n$ is the identity. The generator $\sigma_i$ is the class of the half-twist exchanging $D_i$ and $D_{i+1}$ in the clockwise order and $t_i$ is the class of the Dehn twist around a curve parallel to $\partial D_i$. Thus we have the inclusions $B_n \subset fB_n \subset \Mod(\mathbb{D}_n)$. 
\par Define a right action of $\Mod(\mathbb{D}_n)$ on $\overline{\mathcal{S}}_{\zeta}(\mathbb{D}_n)$ by 
$$ \phi^* [D,s] = [\phi^{-1}(D), s\circ \phi], \mbox{ for } [D,s] \in \overline{\mathcal{S}}_{\zeta}(\mathbb{D}_n), \phi \in \Mod(\mathbb{D}_n).$$
This action preserves both the center $Z_{\mathbb{D}_n}$ and the small center $Z^0_{\mathbb{D}_n}$ thus induces  left actions on $X(\mathbb{D}_n)$ and $\widehat{X}(\mathbb{D}_n)$. Under the identification $X(\mathbb{D}_n) \cong \overline{\mathcal{R}}_{\SL_2}(\mathbb{D}_n)$,  the left action sends a mapping class $\phi \in \Mod(\mathbb{D}_n)$ and a representation $\rho: \Pi_1(\mathbb{D}_n, \mathbb{V}) \to \SL_2$ to $\phi\cdot \rho$ defined by $\phi \cdot \rho(\alpha):= \rho(\phi^{-1}(\alpha))$.

\begin{definition}[Representations of braid groupoids]\label{def_braid_rep}
\begin{enumerate}
\item The  \textbf{braid groupoid} $B_n^{\SL_2}$ is the groupoid whose objects are the points $\widehat{x} \in \widehat{X}(\mathbb{D}_n)$ which are in the Azumaya locus of $\overline{\mathcal{S}}_A(\mathbb{D}_n)$ and whose morphisms between $\widehat{x}$ and $\widehat{x}'$ is the set of braids $\beta \in B_n$ such that $\beta\cdot \widehat{x}= \widehat{x}'$. Composition is just the braid group composition.  The \textbf{framed braid groupoid} $fB_n^{\SL_2}$ is defined similarly with $B_n$ replaced by $fB_n$. 
\item Define a functor $L: B_n^{\SL_2} \to \PVect$ as follows. For each $\widehat{x}$ in the Azumaya locus, fix a standard Azumaya representation $r_{\widehat{x}}: \overline{\mathcal{S}}_A(\mathbb{D}_n) \to \End(V_{\widehat{x}})$ with full shadow $\widehat{x}$ (it is unique up to isomorphism) and set $L(\widehat{x}):=V_{\widehat{x}}$. By Theorem \ref{theorem_RepRSSkein} we can (and do) suppose that $V_{\widehat{x}}= V_1\otimes \ldots \otimes V_n$ with $V_i$ a typical $\Uq$ module of Definition \ref{def_QGRep}. In particular, as a vector space,  
$V_{\widehat{x}}= V^{\otimes n}$ does not depend on $\widehat{x}$.  Suppose that $\beta \cdot \widehat{x}= \widehat{x}'$ for $\beta \in B_n$ and consider the representation $\beta \cdot r_{\widehat{x}} : [D,s]\mapsto r_{\widehat{x}}( [\beta^{-1}(D), s\circ \beta])$. Since $\beta \cdot r_{\widehat{x}}$ has dimension the PI-degree of $\overline{\mathcal{S}}_{\zeta}(\mathbb{D}_n)$ and full shadow $\widehat{x}'$, then  $\beta \cdot r_{\widehat{x}}$ and $r_{\widehat{x}'}$ are isomorphic. So there exists an isomorphism $L_{V_{\widehat{x}}}(\beta) : V_{\widehat{x}} \xrightarrow{\cong} V_{\widehat{x}'}$, unique up to multiplication by a non zero scalar, such that 
$$ r_{\widehat{x}'}(X) = L_{V_{\widehat{x}}}(\beta) r_{\widehat{x}}(\beta^* X) L_{V_{\widehat{x}}}(\beta)^{-1}, \mbox{ for all }X\in \overline{\mathcal{S}}_{\zeta}(\mathbb{D}_n).$$
The assignation $\beta \mapsto L_{V_{\widehat{x}}}(\beta)$ is clearly functorial. 
\item We refine the projective ambiguity as follows. Let $\tau \in B_2$ be the standard generator. The isomorphism $L_{V_{\widehat{x}}}(\tau): V_{\widehat{x}} \to V_{\sigma \cdot \widehat{x}}$ which intertwines two representations of $\overline{\mathcal{S}}_{\zeta}(\mathbb{D}_2)$ will be called a \textbf{braiding operator} or \textbf{skein braiding}. A braiding operator is \textbf{normalized} if it has determinant $1$. Since standard Azumaya representations of $\overline{\mathcal{S}}_{\zeta}(\mathbb{D}_2)$ have dimension $N^2$, normalized braiding operators are uniquely defined up to multiplication by an $N^2$-th root of unity. Thanks to Theorem \ref{theorem_RepRSSkein}, a standard Azumaya representation of $\overline{\mathcal{S}}_{\zeta}(\mathbb{D}_n)$ factorizes as $V_{\widehat{x}}=V_1\otimes \ldots \otimes V_n$ and the image of the generator $\sigma_i$ has the form $L_{V_{\widehat{x}}}(\sigma_i)= \id^{\otimes i-1} \otimes L_{V_i\otimes V_{i+1}}(\sigma) \otimes \id^{\otimes n-i-1}$ with $L_{V_i \otimes V_{i+1}}(\sigma)$ a braiding operator. $L_{V_{\widehat{x}}}(\sigma_i)$ is said normalized if $L_{V_i \otimes V_{i+1}}(\sigma)$ is normalized. 
Since the $\sigma_i^{\pm 1}$ generate $B_n^{\SL_2}$, every morphism in $B_n^{\SL_2}$ is a composition $L_{V_{\widehat{x}}}(\beta)= L_{V_1}(\sigma_{i_1})^{\pm 1} \ldots L_{V_k}(\sigma_{i_k})^{\pm 1}$. Such a morphism is said \textbf{normalized} if each of the $L_{V_j}(\sigma_j)$ is normalized. Normalized intertwiners are defined up to multiplication by a $N^2$-th root of unity. Recall that using the embedding $j: \Uq\xrightarrow{\Psi} \overline{\mathcal{S}}_{\zeta}(\mathbb{D}_1) \xrightarrow{i^*} \overline{\mathcal{S}}_{\zeta}(\mathbb{D}_n)$, each standard Azumaya module $V_{\widehat{x}}$ is a semi-weight $\Uq$ module so an element of $\mathcal{C}$. Note that for every $\beta \in B_n$ and for every $X \in D_qB$ then $\beta^* j(X)=j(X)$ (the braid group acts trivially on the stated arcs $\alpha_{ij}^{(n)}$ and $\beta_{ij}^{(n)}$) so the morphism $L_{V_{\widehat{x}}}(\beta) : V_{\widehat{x}} \to V_{\beta \cdot \widehat{x}}$ is an intertwiner of $D_qB$ modules (a morphism in $\mathcal{C}$). 
Let $\mathcal{C}_{(N)}$ be the category with the same objects as $\mathcal{C}$  but whose  morphisms are only considered up to multiplication by a $N^2$-th root of unity. Then by imposing that each intertwiner $L_{V_{\widehat{x}}}(\beta)$ is normalized, we obtain a refinement (still denoted by the same letter): 
$$ L : B_n^{\SL_2} \to \mathcal{C}_{(N)}.$$
\end{enumerate}
\end{definition}

\subsection{Braiding operators and comparison with Drinfeld and  Kashaev-Reshetikhin braidings}

Recall that $\tau \in \Mod(\mathbb{D}_2)$ is the mapping class exchanging the two inner punctures in the clockwise order and denote by $\mathscr{R}:= (\tau^{-1})_* \in \Aut(\overline{\mathcal{S}}_{\zeta}(\mathbb{D}_2))$ the automorphism induced by $\tau^{-1}$ (seen as a strong embedding). By definition, a braiding operator is an intertwiner for $\mathscr{R}$ so we devote this subsection to the study of this automorphism. 
Recall the splitting morphism $\theta: \overline{\mathcal{S}}_{\zeta}(\mathbb{D}_2) \hookrightarrow \overline{\mathcal{S}}_{\zeta}(\mathbb{D}_1)^{\otimes 2} \cong (\Uq)^{\otimes 2}$. By abuse of notations in this subsection we will consider $ \overline{\mathcal{S}}_{\zeta}(\mathbb{D}_2)$ as a subalgebra of $(\Uq)^{\otimes 2}$ using $\theta$. It is natural to try to extend $\mathscr{R}$ to an automorphism of the larger algebra $(\Uq)^{\otimes 2}$. In order to make sense to such an extension, we need to localize $(\Uq)^{\otimes 2}$ by the element $1-q(q-q^{-1})^2 FL^{-1}\otimes K^{-1}E $ (we easily verify that it satisfies the Ore condition); let $\widehat{(\Uq)^{\otimes 2}}$ denote this localization. 
The following proposition will prove that $\mathscr{R}$ extends uniquely to an automorphism of $\widehat{(\Uq)^{\otimes 2}}$. In the following proposition, $\mathscr{R}(a\otimes b)$ is seen as an element living in $\widehat{(\Uq)^{\otimes 2}}$.

\begin{proposition}\label{prop_KR}
The automorphism $\mathscr{R}$ satisfies the following equalities:
\begin{eqnarray}
\mathscr{R}(1\otimes K) &=&(K\otimes 1) \left( 1-q(q-q^{-1})^2 FL^{-1}\otimes K^{-1}E \right)^{-1}; \label{eq_KR1} \\
\mathscr{R}(1\otimes L^{-1}) &=& (L^{-1} \otimes 1)  \left( 1-q(q-q^{-1})^2 FL^{-1}\otimes K^{-1}E \right)^{-1}; \label{eq_KR2} \\
\mathscr{R}(E\otimes 1) &=& L^{-1} \otimes E; \label{eq_KR3} \\
\mathscr{R}(1 \otimes F) &=& F \otimes K^{-1}; \label{eq_KR4} \\
\mathscr{R}(\Delta(x)) &=& \Delta(x) \quad \mbox{, for all }x \in D_qB  \label{eq_KR5}
\end{eqnarray}
\end{proposition}

\begin{proof}
First, remember that the coproduct of $\overline{\mathcal{S}}_{\zeta}(\mathbb{D}_1)$ is defined by doubling the inner puncture, thus one has $\mathscr{R} \cdot \Delta(X) = \Delta(X)$ for all $X\in \Uq$ and  Equation \eqref{eq_KR5} is trivially satisfied.

Recall that $\Theta_2 \in \Aut(\overline{\mathcal{S}}_{\zeta}(\mathbb{D}_2))$ is the automorphism associated to the mapping class $\theta_2$ which is a central symmetry exchanging the two inner punctures. Since $\tau$ and $\theta_2$ commute, then 
\begin{equation}\label{eq_Cartan}
\mathscr{R} \circ \Theta_2 (X) = \Theta_2 \circ \mathscr{R} (X) \quad \mbox{, for all }X \in \overline{\mathcal{S}}_{\zeta}(\mathbb{D}_2).
\end{equation}

 Let $\delta_{\varepsilon \varepsilon'} \in \overline{\mathcal{S}}_{\zeta}(\mathbb{D}_1)$ be the class of the stated arc drawn in Figure \ref{fig_twist_skein}. Using the defining relation \eqref{eq: skein 2}, one finds that 
\begin{equation}\label{eq_Delta}
\delta_{\varepsilon \varepsilon'} = {\zeta}^{-1/2}\beta_{\varepsilon' +} \alpha_{\varepsilon - }  - {\zeta}^{-5/2}\beta_{\varepsilon' -} \alpha_{\varepsilon +}
\Rightarrow \left\{ \begin{array}{l}
\delta_{++} = -(q-q^{-1}) {\zeta}^{-1/2} E H_{\partial};  \\
\delta_{-+} = {\zeta}^{1/2} K H_{\partial}; \\
\delta_{--} = {\zeta}^{3/2}(q-q^{-1})K H_{\partial} F.
\end{array} \right.
.
\end{equation}

 \begin{figure}[!h] 
\centerline{\includegraphics[width=9cm]{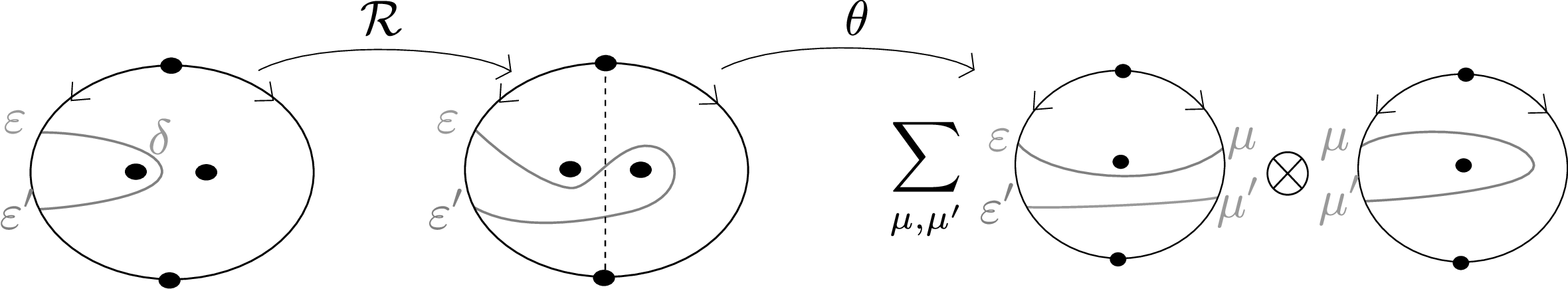} }
\caption{The figure illustrates how the Kashaev-Reshetikhin automorphism can be derived from elementary skein manipulations.} 
\label{fig_twist_skein} 
\end{figure}

By a simple skein computation drawn in Figure \ref{fig_twist_skein}, one finds:
\begin{equation}\label{eq_twist_truc}
 \mathscr{R}  (\delta_{\varepsilon \varepsilon'}\otimes 1) = \sum_{\mu, \mu' = \pm} \beta_{\varepsilon' \mu'}\beta_{\varepsilon \mu}  \otimes \delta_{\mu \mu'}.
 \end{equation}
 When  $(\varepsilon, \varepsilon')=(+,+)$  Equation \eqref{eq_twist_truc} reads
$$  \mathscr{R} (\delta_{++}\otimes 1)  = (\beta_{++})^2\otimes \delta_{++} \stackrel{\eqref{eq_Delta}}{\Rightarrow} \mathscr{R}(EH_{\partial}\otimes 1) = L^{-1}\otimes EH_{\partial}.$$
Since $\Delta(H_{\partial}^{-1})=H_{\partial}^{-1}\otimes H_{\partial}^{-1}$, Equation \eqref{eq_KR5} gives $\mathscr{R}(H_{\partial}^{-1}\otimes H_{\partial}^{-1})=H_{\partial}^{-1}\otimes H_{\partial}^{-1}$. Multiplying the previous equality with this equality gives 
 Equation \eqref{eq_KR3}. By composing Equation \eqref{eq_KR3} with the automorphism $\Theta_2$ and using Equation \eqref{eq_Cartan}, one finds Equation \eqref{eq_KR4}.
  Using Equation \eqref{eq_twist_truc} with  $(\varepsilon, \varepsilon')= (-,+)$  one gets:
\begin{equation}\label{eq_RRR}
  \mathscr{R}(\delta_{-+}\otimes 1)  = \beta_{++}\beta_{-+}\otimes \delta_{++} + \beta_{++}\beta_{--}\otimes \delta_{-+}  \stackrel{\eqref{eq_Delta}}{\Rightarrow} \mathscr{R} (KH_{\partial} \otimes 1) = 1\otimes KH_{\partial} - q^{-1}(q-q^{-1})^2 FL^{-1}\otimes E H_{\partial}.
\end{equation}

 Equation \eqref{eq_twist_truc} with $(\varepsilon, \varepsilon')= (-,-)$ and using Equation \eqref{eq_Delta},  gives:
\begin{multline*}
 \mathscr{R} (\delta_{--}\otimes 1)  ={\zeta}^{3/2}(q-q^{-1}) \left( -(q-q^{-1})^2 q^{-3} F^2L^{-1} \otimes EH_{\partial} \right.
\\ \left.+ L \otimes KFH_{\partial} -F\otimes (K-L-q^{-1}(q-q^{-1})EF)H_{\partial}\right).
\end{multline*}
so, using again Equation \eqref{eq_Delta},  we obtain:

\begin{multline}\label{eq_RR}
\mathscr{R}(KFH_{\partial} \otimes 1) =  -(q-q^{-1})^2 q^{-3} F^2L^{-1} \otimes EH_{\partial} \\ + L \otimes KFH_{\partial} -F\otimes (K-L-q^{-1}(q-q^{-1})EF)H_{\partial}.
\end{multline}
Combining Equations \eqref{eq_RRR} and \eqref{eq_RR}, one finds:
$$
\mathscr{R}(F\otimes 1) = F\otimes 1 +L\otimes F - \left( 1 -q^{-1}(q-q^{-1})^2 FL^{-1} \otimes K^{-1}E \right)^{-1} (F\otimes LK^{-1}).
$$
Using the fact that $\mathscr{R}(\Delta(F))= \Delta(F)$, we derive the equality:
$$\mathscr{R}(L\otimes F) =  \left( 1 -q^{-1}(q-q^{-1})^2 FL^{-1} \otimes K^{-1}E \right)^{-1} (F\otimes LK^{-1}).$$
Using the equality $\mathscr{R}(1\otimes F) = F \otimes K^{-1}$, one finds:
\begin{eqnarray*}
 \mathscr{R}(L\otimes 1) &=&  \left( 1 -q^{-1}(q-q^{-1})^2 FL^{-1} \otimes K^{-1}E \right)^{-1} (1\otimes L) \\
  &=& (1\otimes L)  \left( 1 -q(q-q^{-1})^2 FL^{-1} \otimes K^{-1}E \right)^{-1}.
  \end{eqnarray*}
 Composing the above equation with the automorphism $\Theta_2$ and using Equation \eqref{eq_Cartan}, one finds Equation \eqref{eq_KR1}, while combining the same equation with the equality $\mathscr{R}(L^{-1}\otimes L^{-1}) = L^{-1}\otimes L^{-1}$ one finds \eqref{eq_KR2}.

\end{proof}
Proposition \ref{prop_KR} implies that the automorphism $\mathscr{R} \in \Aut(\overline{\mathcal{S}}_{\zeta}(\mathbb{D}_2))$ extends in a unique manner to an automorphism $\widehat{\mathscr{R}}\in \Aut(\widehat{(\Uq)^{\otimes 2}})$. 
\par 
In \cite{KashaevReshetikhin04,KashaevReshetikhin05}, Kashaev and Reshetikhin defined a Hopf algebra $\mathcal{U}$ generated by some elements $\mathbf{K}^{\pm 1}, \mathbf{L}^{\pm 1}, \mathbf{E}, \mathbf{F}$. By comparing the presentations of $\mathcal{U}$ and $\Uq$, we observe the existence of an injective Hopf algebra morphism $\Psi : \mathcal{U} \hookrightarrow \Uq$ defined by $\Psi(\mathbf{K}^{\pm 1})= K^{\pm 1}, \Psi(\mathbf{L}^{\pm 1})= L^{\mp 1}, \Psi(\mathbf{E})= (q-q^{-1}) E$ and $\Psi(\mathbf{F})= (q-q^{-1})F$.
Let $\widehat{\mathcal{U}^{\otimes 2} }$ be the algebra obtained from $\mathcal{U}^{\otimes 2}$ by localizing by the element $1-q\mathbf{K}^{-1}\mathbf{E} \otimes \mathbf{F}\mathbf{L}$. 
 The authors of  \cite{KashaevReshetikhin04,KashaevReshetikhin05} defined an automorphism $\mathscr{R}^{KR} \in \mathrm{Aut}\left( \widehat{\mathcal{U}^{\otimes 2} } \right)$ by the formulas: 
\begin{eqnarray*}
\mathscr{R}^{KR} (1\otimes \mathbf{K}) &=& (1\otimes \mathbf{K})\left( 1-q\mathbf{K}^{-1}\mathbf{E} \otimes \mathbf{F}\mathbf{L} \right)^{-1}; \\
\mathscr{R}^{KR} (1\otimes \mathbf{L}) &=& (1\otimes \mathbf{L})\left( 1-q\mathbf{K}^{-1}\mathbf{E} \otimes \mathbf{F}\mathbf{L} \right)^{-1}; \\
\mathscr{R}^{KR}(\mathbf{E}\otimes 1) &=& \mathbf{E}\otimes \mathbf{L}; \\
\mathscr{R}^{KR}(1\otimes \mathbf{F}) &=& \mathbf{K}^{-1} \otimes \mathbf{F}; \\
\mathscr{R}^{KR}( \Delta(x)) &=& \sigma \circ \Delta(x) \quad \mbox{, for all }x\in \mathcal{U}; 
\end{eqnarray*}
where $\sigma(a\otimes b):=b\otimes a$. 

\begin{definition}[Kashaev-Reshetikhin braidings]\label{def_KRBraidings} 
Let $V_1,V_2, V_3, V_4$ be four simple typical $\Uq$ modules. Denote by $r_{12}: \widehat{\mathcal{U}^{\otimes 2} } \to \End(V_1\otimes V_2)$ and $r_{43}: \widehat{\mathcal{U}^{\otimes 2} } \to \End(V_4\otimes V_3)$ the associated (surjective) morphisms and suppose that the representation $r'_{12}:= r_{12} \circ \mathscr{R}^{KR}$ is isomorphic to $r_{43}$. In this case, we have an isomorphism $R : V_1\otimes V_2 \xrightarrow{\cong} V_4\otimes V_3$ (unique up to multiplication by a scalar) such that 
$$ r_{12}(\mathscr{R}^{KR}(x\otimes y)) = R r_{43}(x\otimes y) R^{-1} \quad \mbox{, for all } x, y \in \Uq.$$
Let $\sigma: V_4\otimes V_3 \to V_3 \otimes V_4$, $\sigma(x\otimes y):=y\otimes x$. A \textbf{Kashaev-Reshetikhin braiding} associated to the compatible pairs $(V_1,V_2)$ and $(V_3,V_4)$ is the operator: 
$$ c_{V_1,V_2}^{V_3,V_4} := \sigma \circ R : V_1\otimes V_2 \xrightarrow{\cong} V_3\otimes V_4.$$
It is \textbf{normalized} if it has determinant $1$.
\end{definition}

\begin{remark}
The definition of the Kashaev-Reshetikhin braidings is designed such that when $V_1,V_2$ are two simple diagonal typical modules, the Drinfeld braiding 
$$c_{V_1,V_2}:= \sigma q^{-H\otimes G/2}\exp_q^{<N}( (q-q^{-1})E\otimes F): V_1\otimes V_2 \to V_2\otimes V_1$$
of Definition \ref{def_Drinfeld_braiding} is a particular case of Kashaev-Reshetikhin braiding. Indeed, the origin of the operator $\mathscr{R}^{KR}$ can be traced back as follows. When working over the ring $\mathbb{C}[[\hbar]]$ and setting $q=\exp(\hbar)$, the element $\mathcal{R}:= q^{-H\otimes G/2}\exp_q((q-q^{-1})E\otimes F)$ is well defined and is in fact the canonical element of the Drinfeld double $D_{\hbar}B$. So working in the $\hbar$-adic completion $\widehat{D_{\hbar}B^{\otimes 2}}$, the operator $\mathscr{R}^{KR}_{\hbar}:= \Ad(\mathcal{R}), x\mapsto \mathcal{R}x\mathcal{R}^{-1}$ is a well defined automorphism of $\widehat{D_{\hbar}B^{\otimes 2}}$ and it is proved in \cite{KashaevReshetikhin04} that it satisfies the same relations than in Definition \ref{def_KRBraidings}. So $\mathscr{R}^{KR}$ extends word-by-word to an automorphism of $D_qB^{\otimes 2}$ over the ring $\mathbb{Q}[q^{\pm 1}]$ provided that we localize by  $1-q\mathbf{K}^{-1}\mathbf{E} \otimes \mathbf{F}\mathbf{L}$. Replacing $q$ by a root of unity,  the intertwiner $\mathcal{R}^{KR}$ of Definition \ref{def_KRBraidings} will coincide with the term $q^{-H\otimes G/2}\exp_q^{<N}((q-q^{-1})E\otimes F)$ for diagonal representations (for which $(E\otimes F)^N=0$) and provided a generalization of this term for more general pairs of representations including cyclic and semi-cyclic ones.
\end{remark}

Let $\Psi_2 : \widehat{\mathcal{U}^{\otimes 2} } \hookrightarrow \widehat{(D_qB)^{\otimes 2}}$ the injective morphism induced by $\Psi$. 
We deduce from Proposition \ref{prop_KR} the following:

\begin{corollary}\label{coro_KR}
The following diagram commutes:
$$
\begin{tikzcd}
\widehat{\mathcal{U}^{\otimes 2} }  \arrow[r, "\cong"', "\sigma \circ \mathscr{R}^{KR}"] \arrow[d, hook, "\Psi_2"] &
\widehat{\mathcal{U}^{\otimes 2} }   \arrow[d, hook, "\Psi_2"] \\
 \widehat{(D_qB)^{\otimes 2}} \arrow[r, "\cong"', "\widehat{\mathscr{R}}"] &
 \widehat{(D_qB)^{\otimes 2}} \\
 \overline{\mathcal{S}}_{\zeta}(\mathbb{D}_2) \ar[u, hook, "\theta"] \ar[r, "\cong"', "\mathscr{R}"]&
 \overline{\mathcal{S}}_{\zeta}(\mathbb{D}_2) \ar[u, hook, "\theta"] 
 \end{tikzcd}
 $$
 Hence the Kashaev-Reshetikhin braidings in \cite{KashaevReshetikhin04,KashaevReshetikhin05} associated to pairs of simple typical $D_qB$ modules are braiding operators in the sense of Definition \ref{def_braid_rep}. In particular, for $V_1,V_2$ two diagonal simple typical modules, the Drinfeld braidings $c_{V_1,V_2} : V_1\otimes V_2\to V_2\otimes V_1$ of Definition \ref{def_Drinfeld_braiding} are braiding operators in the sense of Definition \ref{def_braid_rep}.
 \end{corollary}

\subsection{Twist operators}

The goal of this subsection is to extend the braid groupoid representation $L: B_n^{\SL_2} \to \mathcal{C}_{(N)}$ of Definition \ref{def_braid_rep}  to a representation $L: fB_n^{\SL_2} \to \mathcal{C}_{(N)}$ of the framed braid groupoid. So we need to define the image $L_{V_{\widehat{x}}}(t_i)$ of the Dehn twists $t_i$. If $V_{\widehat{x}}=V_1\otimes \ldots \otimes V_n$, we will set $L_{V_{\widehat{x}}}(t_i) := \id_{V_1\otimes \ldots \otimes V_{i-1}} \otimes \theta_{V_i} \otimes \id_{V_{i+1}\otimes V_n}$ where $\theta_{V_i} \in \End_{\mathcal{C}_{(N)}}(V_i)$ will be called a \textbf{twist operator}. In order to obtain framed link invariants, we will need in the next subsection that for $\beta\in fB_n$ a framed link, the trace $\tau(L_{V_{\widehat{x}}}(\beta))$ to be invariant under some framed Markov moves on $\beta$. This leaves only one possible definition for $\theta_{V}$: if $\overline{V}$ is another typical $\Uq$ module such that $\tau_* (V\otimes \overline{V}) \cong (V\otimes \overline{V})$ as $\overline{\mathcal{S}}_{\zeta}(\mathbb{D}_2)$ modules (recall that $\tau \in \Mod(\mathbb{D}_2)$ is the half twist which permutes the two inner punctures in the clockwise order), and if $c_{V, \overline{V}}^{V, \overline{V}}$ is the associated braiding operator, we will set: 
$$ \theta_V := \ptr^R \left( c_{V,\overline{V}}^{V, \overline{V}}\right) =  \adjustbox{valign=c}{\includegraphics[width=4cm]{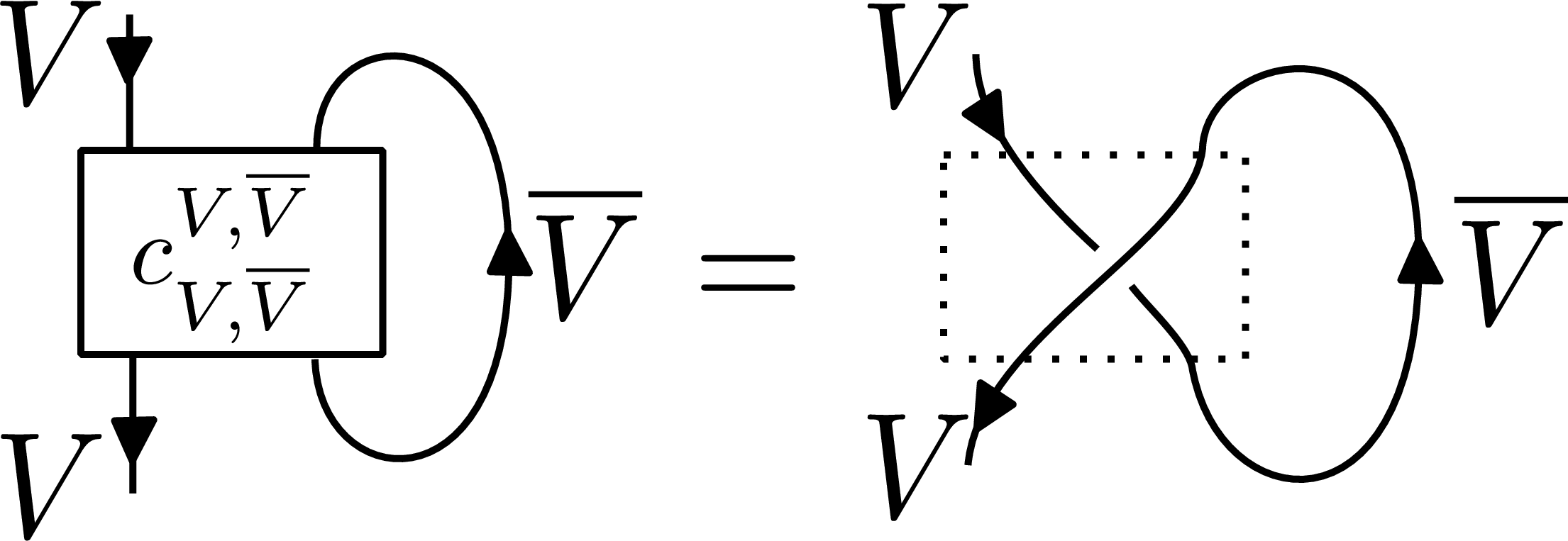}}  .$$

\begin{lemma} \label{lemma_twist1}
\begin{enumerate}
\item Let $\rho: \Pi_1(\mathbb{D}_2, \mathbb{V}) \to \SL_2$ be a representation. Then $\tau_* \rho = \rho$ if and only if we have $\rho(\delta^{(1)}) = \rho(\delta^{(2)})$. 
\item Let $V=V_{(g,h_p, h_{\partial})}$ and $\overline{V}=V_{(\overline{g}, \overline{h_p}, \overline{h_{\partial}})}$ be some typical $\Uq$ modules with $g=(g_+,g_-), \overline{g}=(\overline{g}_+, \overline{g}_-) \in G_0$ such that $V\otimes \overline{V}$ is a standard Azumaya $\overline{\mathcal{S}}_A(\mathbb{D}_2)$ module. Then $\tau_* (V\otimes \overline{V})$ is isomorphic to $V\otimes \overline{V}$ as a $\overline{\mathcal{S}}_{\zeta}(\mathbb{D}_2)$ module if and only if we have 
$$ h_p = \overline{h_p} \quad \mbox{ and } \quad \overline{g}_- g_+ = \overline{g}_+ g_-.$$
\item If $g=\left( \begin{pmatrix} \lambda & x \\ 0 & \lambda^{-1} \end{pmatrix}, \begin{pmatrix} \mu & 0 \\ y & \mu^{-1} \end{pmatrix} \right)$ and $\overline{g}=\left( \begin{pmatrix} \overline{\lambda} & \overline{x} \\ 0 & \overline{\lambda}^{-1} \end{pmatrix}, \begin{pmatrix} \overline{\mu} & 0 \\ \overline{y} & \overline{\mu}^{-1} \end{pmatrix} \right)$, then 
$$\overline{g}_- g_+ = \overline{g}_+ g_- \Leftrightarrow \left\{ \begin{array}{l} 
\overline{y}= y (\lambda \overline{\lambda})^{-1} \\ 
\overline{x} = x \mu \overline{\mu} \\
\lambda\mu^{-1} - \overline{\lambda}\overline{\mu}^{-1} = xy
\end{array} \right.
$$

\end{enumerate}
\end{lemma}

\begin{proof}
$(1)$ Let $\rho : \Pi_1(\mathbb{D}_2, \mathbb{V}) \to \SL_2$ be such that $\rho(\tau^{-1}(\alpha))= \rho(\alpha)$ for all path $\alpha$ in $\Pi_1(\mathbb{D}_2, \mathbb{V})$. Since $\tau^{-1}(\delta^{(2)})=\delta^{(1)}$, we have $\rho(\delta^{(1)})=\rho(\delta^{(2)})$. Conversely, since the paths $\delta^{(1)}, \delta^{(2)}, \alpha^{(2)}$ generate the groupoid $\Pi_1(\mathbb{D}_2, \mathbb{V}) $, to prove that a representation $\rho$ satisfies $\tau_* \rho = \rho$, it suffices to prove the three equalities $\rho(\tau^{-1}(\alpha))=\rho(\alpha)$ for $\alpha \in \{\delta^{(1)}, \delta^{(2)}, \alpha^{(2)}\}$. Since $\tau^{-1}(\alpha^{(2)})= \alpha^{(2)}$, the equality $\rho(\tau^{-1}(\alpha^{(2)}))=\rho(\alpha^{(2)})$ is trivially satisfied. Suppose that the equality $\rho(\delta^{(2)})=\rho(\delta^{(1)})$ holds. Since $\tau^{-1}(\delta^{(1)})=(\delta^{(1)})^{-1}\delta^{(2)}\delta^{(1)}$, then $\rho(\tau^{-1}(\delta^{(1)}))= \rho(\delta^{(1)})^{-1}\rho(\delta^{(2)}) \rho (\delta^{(1)})= \rho (\delta^{(1)})$ so $\tau_* \rho = \rho$ and we have proved the first item.
\par $(2)$ By assumption, $V\otimes \overline{V}$ is an Azumaya representation of $\overline{\mathcal{S}}_{\zeta}(\mathbb{D}_2)$ so $V\otimes \overline{V}$ is isomorphic to $\tau_*(V\otimes \overline{V})$ if and only if they have the same maximal shadow. Let $\widehat{x}=(\rho, h_{p}, \overline{h_{p}}, h_{\partial}\overline{h_{\partial}})$ and $\widehat{x}'= (\rho', h'_{p_1}, h'_{p_2}, h'_{\partial})$ be the maximal shadows of  $V\otimes \overline{V}$ and $\tau_*(V\otimes \overline{V})$ respectively. First, since $(\tau^{-1})_* (\alpha_{\partial})= \alpha_{\partial}$ we have $h_{\partial}'=h_{\partial}\overline{h_{\partial}}$. Next, since $(\tau^{-1})_* (\gamma_{p_1})=\gamma_{p_2}$ and $(\tau^{-1})_* (\gamma_{p_2})=\gamma_{p_1}$ we have $h_{p_1}'=\overline{h_{p}}$ and $h_{p_2}'=h_{p}$. Lastly $\rho'= \tau_* \rho$. Therefore, using assertion $(1)$, $\tau_* (V\otimes \overline{V})$ is isomorphic to $V\otimes \overline{V}$ as a $\overline{\mathcal{S}}_A(\mathbb{D}_2)$ module if and only if $\widehat{x}=\widehat{x}'$ if and only if we have $(i)$ $h_p=\overline{h_p}$ and $(ii)$ $\rho(\delta^{(2)})=\rho(\delta^{(1)})$. 
Using the definition of $\Psi_w$ in Theorem \ref{theorem_classical_limit} and Convention \ref{convention_w}  we see that 
$$ \rho(\delta^{(1)}) = (g_-)^{-1}g_+ \quad \mbox{ and }\quad \rho(\delta^{(2)})= (\overline{g}_- g_+)^{-1} \overline{g}_+ g_+.$$
Therefore
$$ \tau_* \rho = \rho \Leftrightarrow (g_-)^{-1}g_+ =  (\overline{g}_- g_+)^{-1} \overline{g}_+ g_+ \Leftrightarrow \overline{g}_- g_+= \overline{g}_+ g_-.$$
$(3)$ The last assertion is a straightforward computation left to the reader.

\end{proof}

\begin{definition}(Twist operators)
 For $V=V_{(g,h_p, h_{\partial})}$ and $\overline{V}=V_{(\overline{g}, \overline{h_p}, \overline{h_{\partial}})}$ two typical modules, we say that $\overline{V}$ is a \textbf{conjugate} of $V$ if $ h_p = \overline{h_p}$  and $ \overline{g}_- g_+ = \overline{g}_+ g_-$. In this case, if $c_{V, \overline{V}}^{V, \overline{V}}: V\otimes \overline{V} \to V\otimes \overline{V}$ is the associated normalized braiding operator, we call \textbf{twist operator} of $V$ the operator: 
 $$ \theta_V:= \ptr^R\left(c_{V, \overline{V}}^{V, \overline{V}}\right) \in \End_{\mathcal{C}_{(N)}}(V).$$
 \end{definition}
 
 \begin{remark}
 \begin{enumerate}
 \item If $\overline{V}$ is a conjugate of $V$ then $V$ is a conjugate of $\overline{V}$.
 \item If $V$ is a typical module, by Lemma \ref{lemma_twist1},  the set of isomorphism classes of its conjugates $\overline{V}$ is a $\mathbb{C}^*$ torsor: if $\overline{V}=V_{(\overline{g}, h_p, \overline{h}_{\partial})}$ is a conjugate of $V$ with $\overline{g}=\left( \begin{pmatrix} \overline{\lambda} & \overline{x} \\ 0 & \overline{\lambda}^{-1} \end{pmatrix}, \begin{pmatrix} \overline{\mu} & 0 \\ \overline{y} & \overline{\mu}^{-1} \end{pmatrix} \right)$, then the other conjugates have the form 
 $$ V_{(z\cdot \overline{g}, h_p, z^2 h_{\partial})}, \quad \mbox{ where } z\in \mathbb{C}^* \mbox{ and } z\cdot \overline{g}= \left( \begin{pmatrix} \overline{\lambda}z & \overline{x}z \\ 0 & (\overline{\lambda}z)^{-1} \end{pmatrix}, \begin{pmatrix} \overline{\mu}z & 0 \\ \overline{y}z^{-1} & (\overline{\mu}z)^{-1} \end{pmatrix} \right).$$
 \item If $V_D$ is a diagonal representation, then $V_D$ is self-conjugate and we can (and will) use the Drinfeld braiding to compute the associated twist operator.  
 \item At first sight, the definition we gave of the twist operator $\theta_V$ seems to depend on the choice of a conjugate $\overline{V}$. We will prove that this is not the case.
 \end{enumerate}
 \end{remark}

\begin{lemma}\label{lemma_twist2}
Let $V$ be a typical diagonal $\Uq$ module, so $V$ is self-conjugate. Let $(h,g)$ be the unrolled parameters of the normalized Drinfeld braiding $c_{V,V}: V\otimes V \to V\otimes V$ of Definition \ref{def_Drinfeld_braiding}. Then the associated twist operator $\theta_V= \ptr^R(c_{V, V})$ satisfies
$$ \theta_V = \left< \theta_V \right> \id_V, \quad \mbox{ where } \left< \theta_V \right> = q^{\frac{1}{2}\left( (h+(N-1))(-g+N-1) -1\right)}.$$
\end{lemma}

\begin{proof}
By definition, we have 
\begin{multline*}
\theta_V v_0= \ptr^R(c_{V,V})\cdot v_0 = \sum_i \left< \id \otimes v_i^*, c_{V,V}(v_0 \otimes (K^{1/2}L^{-1/2})^{N-1}v_i) \right> \\
 = \sum_i q^{\frac{1}{2}(h -g-2i)(N-1)}\left< \id \otimes v_i^* , \tau q^{- \frac{H\otimes G}{2}} \exp_q^{<N} ((q-q^{-1})E\otimes F) v_0 \otimes v_i \right> \\
  = \sum_i q^{\frac{1}{2}(h-g-2i)(N-1)} \sum_n \frac{q^{\frac{n(n-1)}{2}}}{[n]_q} \left< \id \otimes v_i^* , \tau q^{- \frac{H\otimes G}{2}} E^n v_0 \otimes F^n v_i \right> \\
  = q^{\frac{1}{2}(h-g)(N-1)}q^{-\frac{hg}{2}}= q^{\frac{1}{2}\left( (h+(N-1))(-g+N-1) -1\right)}.
 \end{multline*}
In the last equality we used the fact that $\left< \id \otimes v_i^* , \tau q^{- \frac{H\otimes G}{2}} E^n v_0 \otimes F^n v_i \right> =0$ if $(i,n)\neq (0,0)$.
\end{proof}

\begin{lemma}\label{lemma_twist3}
\begin{enumerate}
\item The twist operator $\theta_V= \ptr^R(c_{V, \overline{V}}^{V, \overline{V}})$ does not depend on the choice of the conjugate $\overline{V}$. 
\item If $V$ and $V'$ are gauge equivalent typical simple modules, then $\left<\theta_V\right>= \left<\theta_{V'}\right>$. 
\item If $\overline{V}$ is conjugate to $V$ then $\theta_V= \ptr^L\left(c_{\overline{V}, V}^{\overline{V}, V}\right)$. 
\item For $\beta\in B_n$ a braid with associated permutation $\pi_{\beta} \in \mathbb{S}_n$, $V_{\widehat{x}}= V_1\otimes \ldots \otimes V_n$ a standard Azumaya $\overline{\mathcal{S}}_{\zeta}(\mathbb{D}_n)$ module and $1\leq i \leq n$, setting $j:= \pi_{\beta}(i)$  we have: 
$$ L_{V_{\widehat{x}}}(\beta) (\id_{V_1\otimes \ldots \otimes V_{i-1}} \otimes \theta_{V_i} \otimes \id_{V_{i+1}\otimes \ldots \otimes V_n})=  (\id_{V_1\otimes \ldots \otimes V_{j-1}} \otimes \theta_{V_j} \otimes \id_{V_{j+1}\otimes \ldots \otimes V_n})L_{V_{\widehat{x}}}(\beta) .$$
\end{enumerate}
\end{lemma}

\begin{proof}
Let us first prove the gauge invariance of $\left<\theta_V\right>$. The other assertions will easily follow. 
Let $V=V_{(g,h_p,h_{\partial})}$ a typical module, $\overline{V}=V_{(\overline{g}, h_p, \overline{h_{\partial}})}$ a conjugate of $V$. Suppose that $V'=V_{(g',h_p,h_{\partial})}$ is gauge equivalent to $V$ in the sense that  there exists $h\in \mathcal{G}_{\mathbb{D}_1}$ such that $V'=V^h$. Let $\overline{V}'=V_{(\overline{g}', h_p, \overline{h_{\partial}})}$ be a conjugate of $V'$ with the same boundary invariant $\overline{h_{\partial}}$ than $\overline{V}$. Let $h_2 \in \mathcal{G}_{\mathbb{D}_2}$ be the image of $h$ through the embedding $\mathcal{G}_{\mathbb{D}_1} \hookrightarrow \mathcal{G}_{\mathbb{D}_2}$ defined in Section \ref{sec_PO_GaugeEquivalence}.  Let us prove the following: 
\\ \textbf{Claim}: As  $\overline{\mathcal{S}}_{\zeta}(\mathbb{D}_2)$ modules, we have $(V\otimes \overline{V})^{h_2} \cong V' \otimes \overline{V}'$.
\par 
Let $\rho, \rho' : \Pi_1(\mathbb{D}_2, \mathbb{V})\to \SL_2$ be the classical shadows of $V\otimes \overline{V}$ and $V'\otimes \overline{V}'$ respectively. Let $(h_L, h_R)\in \mathcal{G}^{out}_{\mathbb{D}_1}$

\par This proves the claim. The claim together with Lemma  \ref{lemma_trace4} (gauge invariance of the renormalized trace) imply that  we have 
$$\tau_{V\otimes \overline{V}}\left( c_{V, \overline{V}}^{V, \overline{V}} \right) =\tau_{V\otimes \overline{V}}(L(\tau))= \tau_{(V\otimes \overline{V})^{h_2}}(L(\tau)^{h_2})= \tau_{V'\otimes \overline{V'}}\left( c_{V', \overline{V'}}^{V', \overline{V}'} \right).$$
Using the compatibility of the trace with the pivotal structure we obtain:
$$ d(\overline{V}) \left< \theta_{V} \right> = \tau_{V\otimes \overline{V}}\left( c_{V, \overline{V}}^{V, \overline{V}} \right) =\tau_{V'\otimes \overline{V'}}\left( c_{V', \overline{V'}}^{V', \overline{V}'} \right)= d(\overline{V}') \left< \theta_{V} \right>.$$
By Lemma \ref{lemma_twist1}, two type 1 typical modules which are conjugate are gauge equivalent, so have the same dimension by Lemma \ref{lemma_trace2}. In particular $d(\overline{V})=d(V)=d(V')=d(\overline{V}')$. So the above equality implies $ \left< \theta_{V} \right> = \left< \theta_{V'} \right> $ and we have proved both assertion $(2)$ and $(1)$ (since for $(1)$, the above reasoning works for $V=V'$ but possibly two distinct conjugate $\overline{V}$ and $\overline{V}'$). 
\par To prove $(3)$, write $\theta_V^L:= ptr^L\left(c_{\overline{V}, V}^{\overline{V}, V}\right)$. Using the compatibility of the trace with the pivotal structure we find the equalities
$$ \tau_{V\otimes \overline{V}}\left(c_{V, \overline{V}}^{V, \overline{V}} \right)= d(V) \left< \theta_{\overline{V}} \right>= d(\overline{V}) \left<\theta_{V}^L \right>.$$
By gauge invariance, we have  $d(V)=d(\overline{V})$ and $\left< \theta_{\overline{V}} \right>= \left< \theta_V \right>$ from which we deduce that $\left<\theta_V\right>=\left<\theta_V^L \right>$ which proves $(3)$. 
\par Let us prove $(4)$. Fix $\beta\in B_n$ a braid with associated permutation $\pi_{\beta} \in \mathbb{S}_n$, $V_{\widehat{x}}= V_1\otimes \ldots \otimes V_n$ a standard Azumaya $\overline{\mathcal{S}}_{\zeta}(\mathbb{D}_n)$ module and $\beta_* V_{\widehat{x}}=V_1'\otimes \ldots \otimes V_n'$. For $1\leq i \leq n$, set $j:= \pi_{\beta}(i)$, then since $\beta_*(\gamma_{p_i})=\gamma_{p_j}$ then $V'_j$ is gauge equivalent to $V_i$ so $\left<\theta_{V_i}\right>=\left<\theta_{V'_j}\right>$. Therefore: 
\begin{multline*}
 L_{V_{\widehat{x}}}(\beta) (\id_{V_1\otimes \ldots \otimes V_{i-1}} \otimes \theta_{V_i} \otimes \id_{V_{i+1}\otimes \ldots \otimes V_n})= \\
\left<\theta_{V_i}\right> L_{V_{\widehat{x}}}(\beta)= \left<\theta_{V'_j}\right> L_{V_{\widehat{x}}}(\beta)=
 (\id_{V_1\otimes \ldots \otimes V_{j-1}} \otimes \theta_{V_j} \otimes \id_{V_{j+1}\otimes \ldots \otimes V_n})L_{V_{\widehat{x}}}(\beta) .
 \end{multline*}
This concludes the proof.
\end{proof}

\begin{remark} If $V$ is a type $1$ typical $\Uq$ module, then $V$ is gauge equivalent to a diagonal module $V_D$ so $\left<\theta_V \right>=\left<\theta_{V_D} \right>$ can be computed using Lemma \ref{lemma_twist2}. 
\end{remark}

\begin{definition}(Representation of framed braid groupoids)\label{rep_framedbraid_rep}
We extend the functor $L: B_n^{\SL_2} \to \mathcal{C}_{(N)}$ of Definition \ref{def_braid_rep} to a representation $L: fB_n^{\SL_2} \to \mathcal{C}_{(N)}$ of the framed braid groupoid $fB_n^{\SL_2}$ by sending a Dehn twist $t_i$, seen as an endomorphism of $\widehat{x}$, to the endomorphism $L_{V_{\widehat{x}}}(t_i) \in \End_{\mathcal{C}}(V_{\widehat{x}})$ where $V_{\widehat{x}}= V_1 \otimes \ldots \otimes V_n$,   defined by the braiding $L_{V_{\widehat{x}}}(t_i):= \id_{V_1\otimes \ldots \otimes V_{i-1}} \otimes \theta_{V_i} \otimes \id_{V_{i+1}\otimes \ldots \otimes V_n}$. 
\end{definition}

That the above extension is well defined follows from the last item of Lemma \ref{lemma_twist3}.

\subsection{Definition of the link invariants}

\begin{definition}[Decorated links] A \textbf{decorated link} is a triple $\mathbb{L}=(L, [\rho], h)$ where: 
\begin{enumerate}
\item $L\subset S^3$ is a framed link with connected components $L=L_1\sqcup \ldots \sqcup L_k$; 
\item $[\rho]$ is the conjugacy class of a representation $\rho: \pi_1(S^3\setminus L) \to \SL_2$; 
\item $h=(h_{L_1}, \ldots, h_{L_k})\in \mathbb{C}^k$ is a tuple such that for each $1\leq i \leq k$, if $\mu_i \in \pi_1(S^3 \setminus L)$ is the class of the meridian encircling $L_i$, then $T_N(h_{L_i})=-\tr (\rho (\mu_i))$. 
\end{enumerate}
A decorated link is \textbf{admissible} if 
\begin{enumerate}
\item[(i)] for all $1\leq i \leq k$ then either $\rho(\mu_i) \neq \pm \mathds{1}_2$ or $h_{L_i}=\pm 2$; 
\item[(ii)] there exists $1 \leq i \leq k$ such that either $\tr (\rho(\mu_i)) \neq \pm 2$ or $h_{L_i}=\pm 2$.
\end{enumerate}
We denote by $\mathcal{DL}$ the set of admissible decorated links. 
\end{definition}

Let $\mathbb{C}_{(N)}:= \quotient{\mathbb{C}}{\mu_{N^2}}$ be the set of complex numbers considered up to multiplication by a $N^2$-th root of unity (so $\mathbb{C}_{(N)}$ is isomorphic to $\mathbb{C}$ through the map $z\mapsto z^{N^2}$). The purpose of this section, and one of the main achievement of the paper, it to define a map $\left< \cdot \right>: \mathcal{DL} \to \mathbb{C}_{(N)}$ as follows. Suppose that $\beta \in fB_n$ is a framed braid whose Markov closure is $L$. Let $M_L:= S^3 \setminus \mathring{N}(L)$ be obtained from $S^3$ by removing the interior of a tubular neighborhood of $L$. Then $\beta$ defines an open book decomposition of $M_L$ in the sense that 
$$ M_L \cong \quotient{ (\mathbb{D}_n \times [0,1])}{\left( (x, 1) \sim (\beta(x), 0) \mbox{ and } (y,t)\sim (y,t') \mbox{ for }x\in \mathbb{D}_n, y\in \partial \mathbb{D}\right)}.$$
The projection $\pi: (\mathbb{D}_n\times[0,1]) \to M_L$ induces an embedding $\Hom(\pi_1(M_L), \SL_2) \to \Hom(\pi_1(\mathbb{D}_n), \SL_2)$ which identifies the set of representations $\Hom(\pi_1(M_L), \SL_2) $ with the set of these representations $\rho: \pi_1(\mathbb{D}_n) \to \SL_2$ such that $\beta^* \rho= \rho$. Let $\rho_{\mathbb{D}}\in \overline{\mathcal{R}}_{\SL_2}(\mathbb{D}_n) \cong X(\mathbb{D}_n)$ be a representation whose restriction $\restriction{\rho_{\mathbb{D}}}{\pi_1(\mathbb{D}_n)}$ has class $[\rho]$. Fix an arbitrary boundary invariant $h_{\partial}\in \mathbb{C}^*$ compatible with $\rho_{\mathbb{D}}$. Each $p_i$ corresponds to a link component $L_{j_i}$ in the sense that $\pi ( \partial D_i)= \mu_{j_i}$; we set $p_i:= p_{L_{j_i}}$.
Eventually consider the element $\widehat{x}:= (\rho_{\mathbb{D}}, h_{p_1}, \ldots, h_{p_n}, h_{\partial}) \in \widehat{X}(\mathbb{D}_n)$. By Theorem \ref{theorem_AzumayaLocus}, hypothesis $(i)$ ensures that $\widehat{x}$ is in the Azumaya locus of $\overline{\mathcal{S}}_A(\mathbb{D}_n)$. Consider  $V_{\widehat{x}}$  the associated standard Azumaya representation of $\overline{\mathcal{S}}_{A}(\mathbb{D}_n)$ and the intertwiner $L_{V_{\widehat{x}}}(\beta) \in \End_{\mathcal{C}}(V_{\widehat{x}})$. Hypothesis $(ii)$ together with Theorem \ref{theorem_RepQG}  ensure that if $V_{\widehat{x}}=V_1\otimes \ldots \otimes V_n$ then at least one of the $V_i$ is of type $1$ so is a projective $D_qB$ module.
So $V_{\widehat{x}}$ is projective as a $D_qB$ module, we can apply the renormalized trace to get $\tau(L_{V_{\widehat{x}}}(\beta)) \in \mathbb{C}_{(N)}$. 
\begin{theorem}\label{theorem_LinksInvariants} The scalar $\tau(L_{V_{\widehat{x}}}(\beta)) \in \mathbb{C}_{(N)}$ only depends on $\mathbb{L}$ and thus defines a link invariant
$$ \left< \cdot \right>: \mathcal{DL} \to \mathbb{C}_{(N)}, \quad \left< \mathbb{L}\right> = \tau(L_{V_{\widehat{x}}}(\beta)).$$
\end{theorem}

We need to prove that $\tau(L_{V_{\widehat{x}}}(\beta))$ does not depend on: 
\begin{enumerate}
\item the choice of the framed braid $\beta$ with Markov closure $L$; 
\item the choice of the representation $\rho_{\mathbb{D}}\in \overline{\mathcal{R}}_{\SL_2}(\mathbb{D}_n)$ which induces the conjugacy class $[\rho]$; 
\item the choice of the boundary invariant $h_{\partial}$.
\end{enumerate}

\begin{lemma}\label{lemma_LI1} The scalar $\tau(L_{V_{\widehat{x}}}(\beta))$ does not depend on  the choice of the representation $\rho_{\mathbb{D}}\in \overline{\mathcal{R}}_{\SL_2}(\mathbb{D}_n)$ which induces the conjugacy class $[\rho]$ nor on the choice of  $h_{\partial}$.
\end{lemma}

\begin{proof}
Let $\widehat{x}=(\rho_{\mathbb{D}}, h_{p_1}, \ldots, h_{p_n}, h_{\partial})$ and $\widehat{x}'=(\rho_{\mathbb{D}}', h_{p_1}, \ldots, h_{p_n}, h_{\partial}')$ be two elements of $\widehat{X}(\mathbb{D}_n)$ such that $\restriction{\rho_{\mathbb{D}}}{\pi_1(\mathbb{D}_n)}$ and $\restriction{\rho'_{\mathbb{D}}}{\pi_1(\mathbb{D}_n)}$ are conjugate. By Theorem \ref{theorem_GaugeRep}, $V_{\widehat{x}}$ and $V_{\widehat{x}'}$ are gauge equivalent, so there exists $h\in \widehat{\mathcal{G}}_{\mathbb{D}_1}$ such that $V_{\widehat{x}}^h =V_{\widehat{x}}$ and $L_{V_{\widehat{x}'}}(\beta)=L_{V_{\widehat{x}}}(\beta)^h$. We conclude using the fact that $\tau$ is gauge invariant (Theorem \ref{theorem_trace} $(3)$).

\end{proof}

Let $\beta \in fB_n$ and denote by $\beta\otimes 1 \in fB_{n+1}$ the framed link obtained from $\beta$ by adding a trivial strand on the right.  Consider the two framed links $\beta^+ := t_{n+1}^{-1} \sigma_n (\beta\otimes 1)$ and $\beta^- := t_{n+1} \sigma_{n+1}^{-1}(\beta \otimes 1)$ (see Figure \ref{fig_Markov}). We say that $\beta_+$ (resp. $\beta_-$) is obtained from $\beta$ by a \textbf{framed Markov move}.

 \begin{figure}[!h] 
\centerline{\includegraphics[width=6cm]{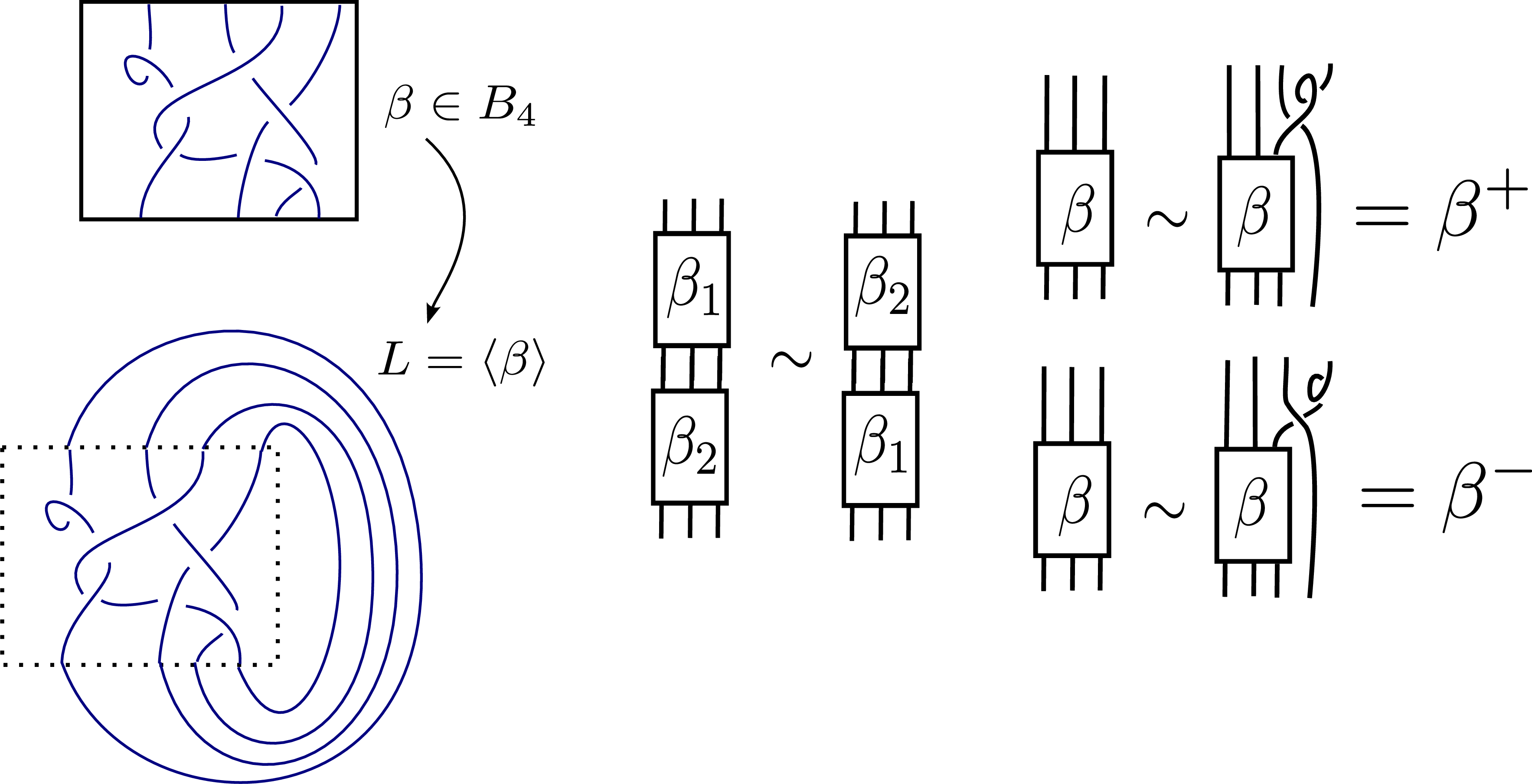} }
\caption{On the left: the Markov closure of a framed braid. On the right: two framed braids have the same Markov closure if and only if they are related by a sequence of Markov moves. } 
\label{fig_Markov} 
\end{figure}

 Clearly the Markov closures of $\beta$, $\beta^+$ and $\beta^-$ are equal to the same framed link. Conversely, the following is the framed version of a celebrated theorem of Markov:  
\begin{theorem}(Markov)\label{theorem_Markov}
Two framed links $\beta$ and $\beta'$ have the same Markov closure if and only if we can pass from $\beta$ to $\beta'$ by a finite sequence of the following elementary moves: 
\begin{enumerate}
\item replacing $\beta_1\beta_2$ by $\beta_2\beta_1$; 
\item replacing $\beta$ by $\beta^+$ or $\beta^-$ or conversely, replacing $\beta^{\pm}$ by $\beta$.
\end{enumerate}
\end{theorem}

Since $\tau$ is a trace and $L$ is functorial, we clearly have $\tau(L_{V_1}(\beta_1)L_{V_2}(\beta_2))=\tau(L_{V_2}(\beta_2)L_{V_1}(\beta_1))$. 

\begin{lemma}\label{lemma_LI2}
Let $\beta \in fB_n$ and $V_{\mathbb{D}}=V_1\otimes \ldots \otimes V_n$ a standard Azumaya $\overline{\mathcal{S}}_A(\mathbb{D}_n)$ module and set $V'_{\mathbb{D}}= V_{\mathbb{D}} \otimes \overline{V_n}$. 
Then we have 
$$ \tau(L_{V_{\mathbb{D}}}(\beta)) = \tau( L_{V'_{\mathbb{D}}}(\beta^+)) = \tau (L_{V'_{\mathbb{D}}}(\beta^-)).$$
\end{lemma}

\begin{proof}
We compute: 
\begin{align*}
\tau( L_{V'_{\mathbb{D}}}(\beta^+)) &=& \tau \left( \ptr^R(L_{V'_{\mathbb{D}}}(\beta^+)) \right)= \tau \left( (\id^{\otimes n}\otimes \theta_{\overline{V_n}}^{-1})(\id^{\otimes n-1}\otimes c_{V_n, \overline{V_n}}^{V_n, \overline{V_n}})(L_{V_{\mathbb{D}}}(\beta)\otimes \id)\right) & \mbox{, because }\tau \mbox{ is a right trace,} \\
 {} &=&  \tau \left( (\id^{\otimes n-1}\otimes c_{V_n, \overline{V_n}}^{V_n, \overline{V_n}})(\id^{\otimes n}\otimes \theta_{{V_n}}^{-1})(L_{V_{\mathbb{D}}}(\beta)\otimes \id)\right) 
 & \mbox{, by Lemma \ref{lemma_twist3}}(4), \\
 {} &=&  \tau \left(  (\id^{\otimes n-1}\otimes \theta_{V_n}) (\id^{\otimes n-1}\otimes \theta_{{V_n}}^{-1})L_{V_{\mathbb{D}}}(\beta)\right) = \tau (L_{V_{\mathbb{D}}}(\beta))& \mbox{, by definition of the twist}. 
\end{align*}

\end{proof}

\begin{proof}[Proof of Theorem \ref{theorem_LinksInvariants}] The theorem follows from Theorem \ref{theorem_Markov} and Lemmas \ref{lemma_LI1}, \ref{lemma_LI2}.
\end{proof}

\subsection{Comparison with previous constructions}

\begin{theorem}\label{theorem_comparison} Let $L=L_1\sqcup \ldots \sqcup L_k$ a framed link in $S^3$ with exterior $M_L:=S^3\setminus \mathring{N}(L)$ and meridian $\mu_i$ around $L_i$.
\begin{enumerate}
\item If $\rho_0: \pi_1(M_L)\to \SL_2$ is the trivial representation sending every loop to the identity matrix $\mathds{1}_2\in \SL_2$, $h=(h_{L_1}, \ldots, h_{L_k})$ is such that $h_{L_i}=-2$ for all $1\leq i \leq k$ and $\mathbb{L}=(L, [\rho_0], h)$ then $\left<\mathbb{L}\right>$ is the (class in $\mathbb{C}_{(N)}$ of the) Kashaev invariant (\cite{KashaevLinkInv}) of $L$.
\item If $\mathbb{L}=(L, [\rho], h) \in \mathcal{DL}$ is such that $\tr(\rho(\mu_i))\neq \pm 2$ for all $1\leq i \leq k$, then $\left<\mathbb{L}\right>$ is equal to the link invariant defined in \cite{BGPR_Biquandle}. 
\end{enumerate}
\end{theorem}

\begin{proof}
The link invariants from \cite{KashaevLinkInv, BGPR_Biquandle} are obtained by taking the renormalized trace of some operators associated to a (decorated) framed braid. This operator is obtained by composing elementary operators which are braiding operators ($R$ matrices) and twist operators ($\theta_V$) like we did; we thus need to prove that our braidings and twist operators coincide with the ones in \cite{KashaevLinkInv, BGPR_Biquandle}. The fact that the braiding operators used in \cite{BGPR_Biquandle} and ours are equal is proved in Corollary \ref{coro_KR}. For the Kashaev invariant, the fact that the Kashaev $R$-matrix coincides with the Drinfeld braiding (in the sense of Definition \ref{def_Drinfeld_braiding}) $c_{V_0, V_0}$, where $V_0= V({\zeta}^{-1/2}, {\zeta}^{1/2}, 0, 0)$ is the so-called Kashaev module, is proved in \cite{MM01}; an alternative proof will be given in Section \ref{sec_KashaevBradingsRootsUnity}; that this Drinfeld braiding coincides with our skein braiding is again a consequence of Corollary \ref{coro_KR}. In each case, the twist operators are obtained from the braidings by the formula $\theta_V= \ptr^R (c_{V, \overline{V}}^{V, \overline{V}})$ as we defined so the twist operators in \cite{KashaevLinkInv, BGPR_Biquandle} coincide with ours. 

\end{proof}

\begin{remark} The colored Jones polynomials and the ADO invariants are not part of our invariants though they are defined in a similar manner. The main difference is that they are associated to decorated links of the form $(L, [\rho_0], h)$ where $\rho_0$ is the trivial representation and $h_{L_i}\neq -2$ so which are not admissible. Hypothesis $(i)$ is used to obtain representations $V_{\widehat{x}}$ in the Azumaya locus of $\overline{\mathcal{S}}_{\zeta}(\mathbb{D}_n)$ whereas Hypothesis $(ii)$ ensures that $V_{\widehat{x}}$ is a projective $\Uq$ module so we can apply the renormalized trace. The fact that $\widehat{x}$ belongs to the Azumaya locus ensures the existence of the intertwiner $L_{V_{\widehat{x}}}(\beta)$ whereas the fact that $V_{\widehat{x}}$ is irreducible ensure the unicity (up to scalar) of the intertwiner. Still it might happen that for indecomposable $\overline{\mathcal{S}}_{\zeta}(\mathbb{D}_n)$ modules $V$ with full shadow $\widehat{x}$ outside the Azumaya locus we still have existence and unity of the intertwiner. If $V$ is not projective as a $\Uq$ module, we can take the q-trace $\qtr$ of the intertwiner instead of the renormalized one. For instance fix $1\leq i \leq N-2$ and consider the $\Uq$ module $S_i:= S_{+1, +1, i}$ of Definition \ref{def_QGRep} (it is simple not projective). Set $V:= S_i^{\otimes n}$. As a $\overline{\mathcal{S}}_A(\mathbb{D}_n)$ module $V$ is simple though its classical shadow $([\rho_0], h_{L_j}=-q^{i+1}-q^{-(i+1)}, h_{\partial}=1)$ is not in the Azumaya locus. Still, the existence of the Drinfeld R-matrix ensures the existence of the intertwiners $L_V(\beta)$ and the $i$-th colored Jones polynomial is defined as the q-trace: $J_i(L):= \qtr(L_V(\beta))$ when $L$ is the closure of $\beta$. 
\end{remark}

\section{Quantum hyperbolic geometry}\label{sec_QT}

Quantum hyperbolic geometry provides a recipe to obtain explicit formulas for the intertwiners $L_V(\beta)$ and thus for the link invariants $\left<\mathbb{L}\right>$ closely related to classical hyperbolic geometry which we now investigate.
 
 \subsection{Balanced Chekhov-Fock algebras and the quantum trace}.
 \par 
 \underline{\textbf{Triangulations:}} A marked surface $\mathbf{\Sigma}$ is \textbf{triangulable} if it can be obtained from a disjoint union of triangles $\mathbf{\Sigma}_{\Delta}=\sqcup_{\mathbb{T} \in F(\Delta)} \mathbb{T}$  by gluing some pairs of boundary edges; we then call \textbf{triangulation} $\Delta$  the combinatoric data made of $\mathbf{\Sigma}_{\Delta}$  and the pairs of glued edges. The connected components of $\mathbf{\Sigma}_{\Delta}$ are the \textbf{faces} of $\Delta$ with set $F(\Delta)$. The image of the boundary edges by the quotient map $\mathbf{\Sigma}_{\Delta}\to \mathbf{\Sigma}$ are the \textbf{edges} of $\Delta$ with set $\mathcal{E}(\Delta)$. An edge $e\in \mathcal{E}(\Delta)$ is a  \textbf{boundary edge} it it is a boundary edge of $\mathbf{\Sigma}$, else it is an \textbf{inner edge} and we consider the partition $\mathcal{E}(\Delta)=\mathcal{E}^{\partial}(\Delta)\sqcup \mathring{\mathcal{E}}(\Delta)$. By convention, we consider that each triangulation is equipped with a bijection $I: \mathcal{E}(\Delta)\xrightarrow{\cong} \{1, \ldots, |\mathcal{E}(\Delta)|\}$ named its \textbf{indexation}. 
 \par 
  The splitting morphism defines an embedding $\theta^{\Delta} : \mathcal{S}_A(\mathbf{\Sigma}) \hookrightarrow \otimes_{\mathbb{T}\in F(\Delta)} \mathcal{S}_A(\mathbb{T})$ and Theorem \ref{theorem_exactsequence} implies that $\mathcal{S}_A(\mathbf{\Sigma}) $ is completely determined by the triangle algebra $\mathcal{S}_A(\mathbb{T})$ together with the combinatoric of the triangulation.
 
 \vspace{2mm} 
 \par \underline{\textbf{The balanced Chekhov-Fock algebras:}}
 Let $(\mathbf{\Sigma}, \Delta)$ be a triangulated punctured surface. A map $\mathbf{k} : \mathcal{E}(\Delta) \rightarrow \mathbb{Z}$ is called \textbf{balanced} if for any three edges $e_1,e_2,e_3$ bounding a face of $\Delta$, the integer $\mathbf{k}(e_1)+\mathbf{k}(e_2)+\mathbf{k}(e_3)$ is even. If $\mathbf{k}_1, \mathbf{k}_2$ are balanced, their sum $\mathbf{k}_1+\mathbf{k}_2$ is balanced. We denote by $K_{\Delta}$ the abelian group of balanced maps. For  $e$ and $e'$ two edges, denote by $a_{e,e'}$ the number of faces $\mathbb{T}\in F(\Delta)$ such that $e$ and $e'$ are edges of $\mathbb{T}$ and such that we pass from $e$ to $e'$ in the counter-clockwise direction in $\mathbb{T}$. The \textbf{Weil-Petersson}  form $\left(\cdot, \cdot \right)^{WP}: K_{\Delta} \times K_{\Delta} \rightarrow \mathbb{Z}$ is the skew-symmetric form defined  by $\left( \mathbf{k}_1, \mathbf{k}_2\right)^{WP}:= \sum_{e,e'} \mathbf{k}_1(e)\mathbf{k}_2(e')(a_{e,e'}-a_{e',e})$.

\begin{definition}(Balanced Chekhov-Fock algebras)\label{def_CF}
The \textbf{balanced Chekhov-Fock algebra} $\mathcal{Z}_A(\mathbf{\Sigma}, \Delta)$ is the quotient of the associative algebra freely generated by elements $Z^k$, $k\in K_{\Delta}$ by the ideal generated by the relations $Z^{k_1+k_2}=A^{-\frac{1}{2}(k_1, k_2)^{WP}}Z^{k_1}Z^{k_2}$.
\end{definition}

\begin{notations} 
We will write $[Z_{e_1}^{k_1} \ldots Z_{e_m}^{k_m}]$ for the monomial $Z^k$ where $k(e_i)= k_i$ for $1\leq i \leq m$ and $k(e)=0$ if $e\notin\{e_1,\ldots, e_m\}$. For $e \in \mathcal{E}(\Delta)$ we denote by $X_e$ the element $Z^{k_e}$ where $k_e(e')=2 \delta_{e, e'}$. Like for skein algebras, $\mathcal{Z}_{\zeta}(\mathbf{\Sigma}, \Delta)$ will refer to the balanced Chekhov-Fock algebra where $A^{1/2}=\zeta^{1/2}$ is a root of unity of odd order $N\geq 1$. 
\end{notations}

For $a,b$ two boundary edges of $(\mathbf{\Sigma}, \Delta)$, we denote by $\Delta_{a\#b}$  the triangulation of $\mathbf{\Sigma}_{a\#b}$ obtained from $\Delta$ by gluing $a$ and $b$; in particular $\mathcal{E}(\Delta_{a\#b})=\mathcal{E}(\Delta)\setminus \{a, b\}$. 
 Balanced Chekhov-Fock algebras admit a splitting morphism $\theta_{a\# b}: \mathcal{Z}_A(\mathbf{\Sigma}_{a\# b}, \Delta_{a\#b}) \hookrightarrow \mathcal{Z}_A(\mathbf{\Sigma}, \Delta)$ defined as follows. Let $c\in \mathcal{E}(\Delta_{a\# b})$ be the common image of $a$ and $b$. We set $\theta_{a\#b}(Z^k):= Z^{k'}$ where $k'(e)=k(e)$ if $e\notin \{a,b\}$ and $k'(a)=k'(b)=k(c)$. Like for stated skein algebras, splitting morphisms define a morphism $\theta^{\Delta} : \mathcal{Z}_A(\mathbf{\Sigma}, \Delta) \hookrightarrow \otimes_{\mathbb{T} \in F(\Delta)} \mathcal{Z}_A(\mathbb{T})$. 

\vspace{2mm}
\par \underline{\textbf{Quantum trace:}} 

We now define an injective algebra morphism $\Tr_{A}^{\Delta} : \overline{\mathcal{S}}_{A}(\mathbf{\Sigma}) \hookrightarrow \mathcal{Z}_{A}(\mathbf{\Sigma}, \Delta)$. First consider the case where $\mathbf{\Sigma}=\mathbb{T}$. Let $e_1,e_2,e_3$ be the three edges of $\mathbb{T}$ cyclically ordered in the counter-clockwise order. For $i \in \{1,2,3\}$ let $\alpha_i \subset \mathbb{T}$ be an arc with endpoints in $e_{i+1}$ and $e_{i+2}$ (indexes are modulo $3$).
The algebra  $\overline{\mathcal{S}}_{A}(\mathbb{T})$ is generated by the classes of the stated arcs $(\alpha_i)_{\varepsilon \varepsilon}$, for $i=1,2,3$ and $\varepsilon \in \{-, +\}$;  moreover one has $(\alpha_i)_{--} = ((\alpha_i)_{++})^{-1}$. Define the balanced maps $\mathbf{k}_1, \mathbf{k}_2, \mathbf{k}_3 \in K_{\mathbb{T}}$ by $\mathbf{k}_i (e_i)=0$ and  $\mathbf{k}_i (e_j)=1$
for $j\neq i$. By \cite[Theorem $7.11$]{CostantinoLe19}, the linear map $\Tr_{A}^{\mathbb{T}} : \overline{\mathcal{S}}_{A}(\mathbb{T}) \rightarrow \mathcal{Z}_{A}(\mathbb{T})$, sending $(\alpha_i)_{++}$ to $Z^{\mathbf{k}_i}$ and sending $(\alpha_i)_{--}$ to $Z^{-\mathbf{k}_i}$, extends to an isomorphism of algebras.

\begin{definition}[Quantum trace \cite{BonahonWongqTrace, LeStatedSkein}]\label{def_qtrace} Let $(\mathbf{\Sigma}, \Delta)$ be a triangulated marked surface. The \textbf{quantum trace} is the unique algebra morphism  $\Tr_{A}^{\Delta} : \overline{\mathcal{S}}_{A}(\mathbf{\Sigma}) \hookrightarrow \mathcal{Z}_{A}(\mathbf{\Sigma}, \Delta)$ making the following diagram commuting: 
$$
\begin{tikzcd}
\overline{\mathcal{S}}_{A}(\mathbf{\Sigma})
\arrow[r, hook, "\theta^{\Delta}"] \arrow[d, dotted, hook, "\Tr_{A}^{\Delta}"] &
\otimes_{\mathbb{T}\in F(\Delta)} \overline{\mathcal{S}}_{A}(\mathbb{T}) 
\arrow[d, "\cong"' , "\otimes_{\mathbb{T}} \Tr_{A}^{\mathbb{T}}"] \\
\mathcal{Z}_{A}(\mathbf{\Sigma}, \Delta)
\arrow[r, hook, "\theta^{\Delta}"] &
\otimes_{\mathbb{T}\in F(\Delta)} \mathcal{Z}_{A}(\mathbb{T})
\end{tikzcd}
$$
\end{definition}

 Recall that $\overline{\mathcal{S}}_A(\mathbf{\Sigma})$ is graded by the cohomology group $H=\mathrm{H}_1(\Sigma, \mathcal{A}; \mathbb{Z}/2\mathbb{Z})$ where the class of a stated diagram $[D,s]$ is the homology class of $[D]\in H$. Note that an edge $e\in \mathcal{E}(\Delta)$ defines a Borel-Moore homology class $[e]\in \mathrm{H}_1^c(\Sigma\setminus \mathcal{P}; \mathbb{Z}/2\mathbb{Z})$. For $k\in K_{\Delta}$ let $[k]\in H$ be the class whose algebraic intersection with $[e]$ is $k(e)$ modulo $2$ for each $e\in \mathcal{E}(\Delta)$. For $\chi \in H$, define the submodule
 $$ \mathcal{Z}^{(\chi)}_A(\mathbf{\Sigma}, \Delta):= \Span( Z^k: [k]= \chi) \subset \mathcal{Z}_A(\mathbf{\Sigma}, \Delta).$$

It follows from the definition that $\mathcal{Z}_A(\mathbf{\Sigma}, \Delta)=  \oplus_{\chi \in H}  \mathcal{Z}_A^{(\chi)}(\mathbf{\Sigma}, \Delta)$ is a graded algebra. 

\begin{lemma}\label{lemma_grading}
The quantum trace map is a $\mathrm{H}_1(\Sigma, \mathcal{A}; \mathbb{Z}/2\mathbb{Z})$ graded morphism of degree $0$.
\end{lemma}

\begin{proof}
For the triangle this is follows from the fact that $[\alpha_i]=[k_i] \in H$ for $i=1,2,3$. In general, consider the commutative diagram of Definition \ref{def_qtrace}. Let $\mathbf{\Sigma}_{\Delta}:=\sqcup_{\mathbb{T} \in F(\Delta)} \mathbb{T}=: (\Sigma_{\Delta}, \mathcal{A}_{\Delta})$. The projection map $\pi: (\Sigma_{\Delta}, \mathcal{A}_{\Delta}) \to (\Sigma, \mathcal{A})$ induces a group morphism in homology $\pi_*: \oplus_{\mathbb{T}\in F(\Delta)} \mathrm{H}_1(\mathbb{T}, \{e_1,e_2,e_3\}; \mathbb{Z}/2\mathbb{Z}) \to H $ by which we  equip both $\otimes_{\mathbb{T}}\overline{\mathcal{S}}_A(\mathbb{T})$ and $\otimes_{\mathbb{T}}\mathcal{Z}_A(\mathbb{T})$ with structure of $H$ graded algebras. By definition of the splitting morphisms, the two morphisms $\theta^{\Delta}$ (both for skein and balanced CF algebras) preserve the $H$ grading therefore all three maps in the diagram of Definition \ref{def_qtrace} preserve the $H$ grading; so does $\Tr_A^{\Delta}$. 
 \end{proof}
 
 \begin{remark} \label{remark_grading}
 Note that the $0$-graded part $ \mathcal{Z}_A^{(0)}(\mathbf{\Sigma}, \Delta)$ is generated, as an algebra, by the elements $X_e$ for $e\in \mathcal{E}(\Delta)$. This $0$-graded part is the original algebra which appeared in \cite{ChekhovFock, FockGoncharovClusterVariety}.
 The following remark will be useful. 
  Suppose that for each $\chi \in \mathrm{H}_1(\Sigma, \mathcal{A}; \mathbb{Z}/2\mathbb{Z})\setminus \{0\}$ we fix $k_{\chi}\in K_{\Delta}$ such that $[k_{\chi}]=\chi$. Then the multiplication by $Z^{k_{\chi}}$ induces a linear isomorphism between $ \mathcal{Z}_A^{(0)}(\mathbf{\Sigma}, \Delta)$ and $ \mathcal{Z}_A^{(\chi)}(\mathbf{\Sigma}, \Delta)$. It follows that $\mathcal{Z}_A(\mathbf{\Sigma}, \Delta)$ is generated, as an algebra, by the elements $X_e, e\in \mathcal{E}(\Delta)$ and $Z^{k_{\chi}}, \chi \in \mathrm{H}_1(\Sigma, \mathcal{A}; \mathbb{Z}/2\mathbb{Z})\setminus \{0\}$.
 \end{remark}

 \subsection{Quantum Teichm\"uller representations}

\begin{definition}\label{def_central_elements}(Center of balanced Chekhov-Fock algebras)
\begin{itemize}
\item  Let $p\in \mathring{\mathcal{P}}$ be an inner puncture. For  $e\in \mathcal{E}(\Delta)$, denote by $\mathbf{k}_p(e)\in \{0,1,2\}$ the number of endpoints of $e$ equal to $p$. The \textbf{central inner puncture element} is  $H_p:=Z^{\mathbf{k}_p}\in \mathcal{Z}_{A}(\mathbf{\Sigma}, \Delta)$.
\item  Let  $\partial$ a connected component of $\partial \Sigma$. For each edge $e$, denote by $\mathbf{k}_{\partial}(e)\in \{0, 1, 2\}$ the number of endpoints of $e$ lying in $\partial$. The \textbf{central boundary element} is $H_{\partial}:=Z^{\mathbf{k}_{\partial}}\in \mathcal{Z}_{A}(\mathbf{\Sigma}, \Delta)$. 
\item Suppose that $\zeta^{1/2}$ is a root of unity of order $N>1$, the \textbf{Frobenius morphism} $Fr_{(\mathbf{\Sigma}, \Delta)} : \mathcal{Z}_{+1}(\mathbf{\Sigma}, \Delta) \rightarrow \mathcal{Z}_{\zeta}(\mathbf{\Sigma}, \Delta)$ sending a balanced monomial $Z^{\mathbf{k}}$ to $Z^{N\mathbf{k}}$, is an injective morphism of algebras whose image lies in the center.
\end{itemize}
\end{definition}

An important property of the quantum trace is the fact that $\Tr_{\zeta}^{\Delta}$ commutes with the Frobenius morphisms in the sense that $\Tr_{\zeta}^{\Delta} \circ Fr_{\mathbf{\Sigma}} = Fr_{\mathbf{\Sigma}, \Delta} \circ \Tr_{+1}^{\Delta}$. This is actually how the Frobenius morphism for skein algebras were first discovered in \cite{BonahonWongqTrace}. Also we have the equalities $\Tr_{\zeta}^{\Delta}(\gamma_p)=H_p +H_p^{-1}$ and $\Tr_{\zeta}^{\Delta}(\alpha_{\partial})=H_{\partial} + H_{\partial}^{-1}$ so $\Tr_{\zeta}^{\Delta}$ sends the center (resp. the small center) into the center (resp. small center). 

The center of $\mathcal{Z}_{\zeta}(\mathbf{\Sigma}, \Delta)$ will be denoted by ${Z}(\mathbf{\Sigma}, \Delta)$ whereas the small center, that is the image of the Frobenius, will be denoted by ${Z}^0(\mathbf{\Sigma}, \Delta) \subset{Z}(\mathbf{\Sigma}, \Delta)$. We write $Y(\mathbf{\Sigma}, \Delta):= \Specm({Z}^0(\mathbf{\Sigma}, \Delta))$ and $\widehat{Y}(\mathbf{\Sigma}, \Delta):= \Specm({Z}(\mathbf{\Sigma}, \Delta))$.

\begin{theorem}\label{theorem_CF_Rep}
\begin{enumerate}
\item
(\cite{BonahonWong2} when $\mathbf{\Sigma}$ is unmarked, \cite{KojuQuesneyQNonAb} in general) If $\zeta$ is a root of unity of odd order $N>1$, the center of $\mathcal{Z}_{\zeta}(\mathbf{\Sigma}, \Delta)$  is generated by the image of the Frobenius morphism together with the central inner puncture and boundary elements. Its PI-degree is equal to the PI-degree of $\overline{\mathcal{S}}_{\zeta}(\mathbf{\Sigma})$. 
\item (\cite[Proposition $7.2$]{DeConciniProcesiBook}) Like any quantum torus at roots of unity,  $\mathcal{Z}_{\zeta}(\mathbf{\Sigma}, \Delta)$ is Azumaya: it is almost Azumaya with Azumaya locus the whole variety $\widehat{Y}(\mathbf{\Sigma}, \Delta)$. 
\end{enumerate}
\end{theorem}

So isomorphism classes of irreducible representations of $\mathcal{Z}_{\zeta}(\mathbf{\Sigma}, \Delta)$ are in $1:1$ correspondence with $\widehat{Y}(\mathbf{\Sigma}, \Delta)$ through the character map and they all have the same dimension. 

\begin{definition}(Quantum Teichm\"uller representations) Let $\overline{r}: \mathcal{Z}_{\zeta}(\mathbf{\Sigma}, \Delta) \to \End(V)$ be an irreducible representation. We call \textbf{quantum Teichm\"uller representation} (or QT representation) the composition 
$$ r: \overline{\mathcal{S}}_{\zeta}(\mathbf{\Sigma}) \xrightarrow{\Tr_{\zeta}^{\Delta}} \mathcal{Z}_{\zeta}(\mathbf{\Sigma}, \Delta) \xrightarrow{\overline{r}} \End(V).$$
\end{definition}

Let $\widehat{p}: \widehat{Y}(\mathbf{\Sigma}, \Delta) \to \widehat{X}(\mathbf{\Sigma})$ and $p: Y(\mathbf{\Sigma}, \Delta) \to X(\mathbf{\Sigma})$ be the dominant maps defined by the restrictions $\Tr_{+1}^{\Delta} : Z(\mathbf{\Sigma}) \hookrightarrow Z(\mathbf{\Sigma}, \Delta)$ and  $\Tr_{+1}^{\Delta} : Z^0(\mathbf{\Sigma}) \hookrightarrow Z^0(\mathbf{\Sigma}, \Delta)$ respectively. The interested reader might consult \cite{KojuQuesneyQNonAb} for a geometric interpretation of $Y(\mathbf{\Sigma}, \Delta)$ as a moduli space and of $p$ as an algebraic non abelianization map related to Gaiotto-Moore spectral networks. Consider the following commutative diagram:
$$\begin{tikzcd}
\widehat{Y}(\mathbf{\Sigma}, \Delta) \ar[r, "\widehat{p}"] \ar[d, "\pi"] & 
\widehat{X}(\mathbf{\Sigma}) \ar[d, "\pi"] \\
Y(\mathbf{\Sigma}, \Delta) \ar[r, "p"] &
X(\mathbf{\Sigma})
\end{tikzcd}$$

Consider a QT representation $r$ corresponding to the ideal $\widehat{y}\in \widehat{Y}(\mathbf{\Sigma}, \Delta)$.
Since the dimension of  $r$ is equal to the PI-degree of $\overline{\mathcal{S}}_{\zeta}(\mathbf{\Sigma})$, $r$ is an Azumaya representation if and only if $\widehat{p}(\widehat{y})$ is in the Azumaya locus. Conversely, an Azumaya representation of $\overline{\mathcal{S}}_{\zeta}(\mathbf{\Sigma})$ with full shadow $\widehat{x}$ lifts to an irreducible representation of $\mathcal{Z}_{\zeta}(\mathbf{\Sigma}, \Delta)$ if and only if $\widehat{x}$ is in the image of $\widehat{p}$. Describing the image of $\widehat{p}$ is a difficult problem in general; we will solve it for $\mathbb{D}_n$ in the next subsection.

\subsection{Embedding quantum groups inside quantum tori}

 \begin{figure}[!h] 
\centerline{\includegraphics[width=5cm]{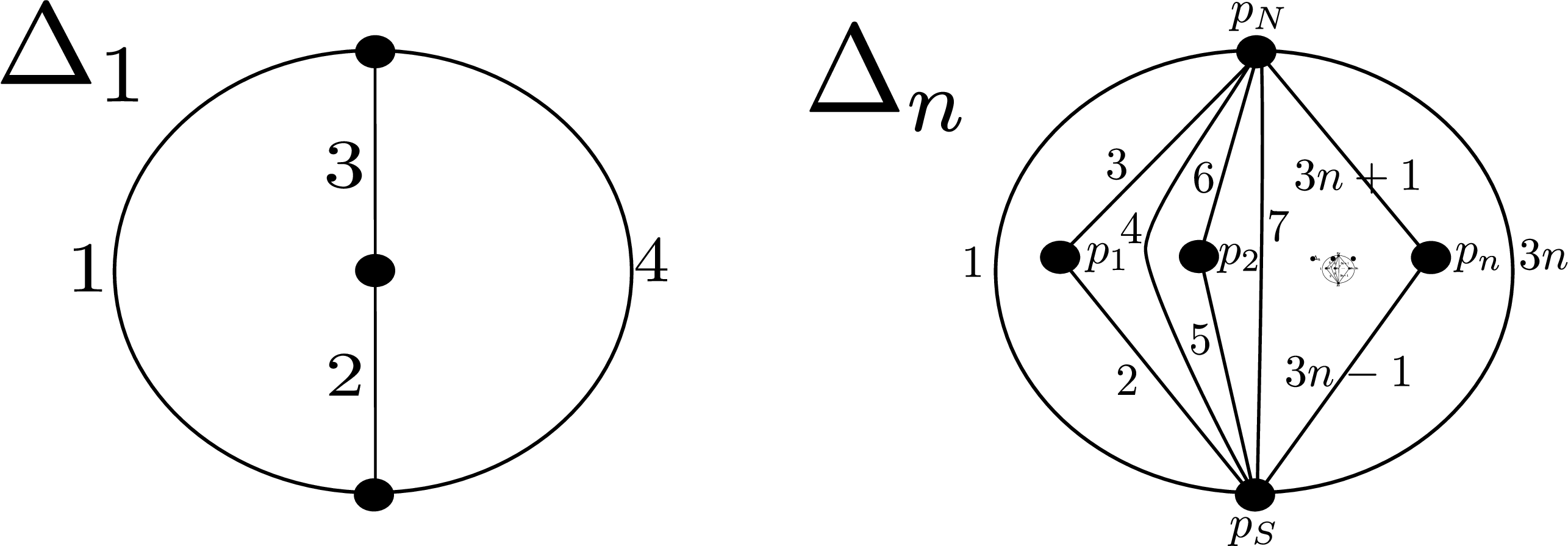} }
\caption{A triangulation of $\mathbb{D}_n$.} 
\label{fig_triangulation_Dn} 
\end{figure} 

Let $\Delta_n$ denote the triangulation of $\mathbb{D}_n$ in Figure \ref{fig_triangulation_Dn}. The algebra $\mathcal{Z}_{\zeta}(\mathbb{D}_n, \Delta_n)$ will simply be denoted by $\mathcal{Z}_{\zeta}(\mathbb{D}_n)$. Note that $\Delta_n$ is obtained by gluing $n$ copies of $\Delta_1$ together so we have a splitting morphism $\theta: \mathcal{Z}_{\zeta}(\mathbb{D}_n) \hookrightarrow \mathcal{Z}_{\zeta}(\mathbb{D}_1)^{\otimes n}$. The composition $T: \Uq \cong \overline{\mathcal{S}}_{\zeta}(\mathbb{D}_1) \xrightarrow{\Tr_{\zeta}^{\Delta_1}} \mathcal{Z}_{\zeta}(\mathbb{D}_1)$ embeds the quantum group $\Uq$ inside a quantum torus. It is characterized by the formulas: 

\begin{align*}
&\tr(K^{\pm 1/2})=  [Z_1Z_3Z_4]^{\mp 1}, \quad 
\tr(L^{\pm 1/2})= [Z_1Z_2Z_4]^{\mp 1},  \\
&\tr(E)  = \frac{-1}{q-q^{-1}} (X_4^{-1} + [X_3X_4]^{-1}), 
\quad \tr(F)= \frac{1}{q-q^{-1}} (X_1^{-1} + [X_1X_2]^{-1}).
\end{align*}

The fact that $U_q\mathfrak{sl}_2$ (and more generally $U_q\mathfrak{sl}_N$) can be embedded in some quantum tori was already found by Kashaev and Volkov in \cite{KashaevVolkov} (see also \cite{FadeevModularDouble} and \cite{SchraderShapiro}). So our skein interpretation of the quantum group permits to recast these embeddings as particular instance of the quantum trace. This embedding is the key to relate Kashaev type $R$ matrices to quantum groups $R$ matrices as we shall see. 
Using $\tr$, any simple $\mathcal{Z}_{\zeta}(\mathbb{D}_1)$ module receives a structure of $\Uq$ module of dimension $N$. We now classify them.

\begin{proposition}\label{prop_D1_lift}
Let $\mathcal{O}\subset \SL_2$ be the subset of matrices $\begin{pmatrix}a & b \\ c& d\end{pmatrix}$ such that either $(i)$ $bc\neq 0$, $(ii)$ or $b=0$, $c\neq 0$ and $a\neq \pm 1$, or $(ii)'$ $c=0$, $b\neq 0$ and $a \neq \pm 1$ or $(iii)$ $b=c=0$, $a=d=\pm 1$. 
Let $W$ be a simple $\mathcal{Z}_{\zeta}(\mathbb{D}_1)$ module. Then, as $\Uq$ module, $W$ is isomorphic to module $V(a,b,\lambda, \mu)$  of Definition \ref{def_QGRep} with classical shadow   $g\in \mathcal{O}$. Conversely any such $\Uq$ module extends to a $\mathcal{Z}_{\zeta}(\mathbb{D}_1)$ simple module.
\end{proposition}

\begin{lemma}\label{Lemma1}
The algebra $\mathcal{Z}_{\zeta}(\mathbb{D}_1)$ is presented by the generators $H_{\partial}^{\pm 1}, H_p^{\pm 1}, [Z_1Z_3Z_4]^{\pm 1}$ and  $X_1^{\pm 1}$ with relations $[Z_1Z_3Z_4]^{-1}X_1= q X_1[Z_1Z_3Z_4]^{-1}$ and $[H_{\partial}, x]=[H_p, x]=0$ for all generators $x$.
\end{lemma}

\begin{proof}
This is a particular case of \cite[Proposition $2.28$]{KojuQuesneyClassicalShadows} where it is proved the existence of an algebra isomorphism $\mathcal{Z}_{\zeta}(\mathbb{D}_1)\cong \mathcal{W}_q \otimes \mathbb{C}[H_p^{\pm 1}, H_{\partial}^{\pm 1}]$, where $\mathcal{W}_q= \mathbb{C}\left< X^{\pm 1}, Y^{\pm 1} | XY=qYX\right>$ is the Weyl algebra, which identifies $[Z_1Z_3Z_4]^{-1}, X_1, H_p, H_{\partial}$ with $X,Y,H_p, H_{\partial}$ respectively.

\end{proof}

\begin{definition}
Let $(x_1, \lambda, h_p, h_{\partial})\in (\mathbb{C}^*)^4$. The $\mathcal{Z}_{\zeta}(\mathbb{D}_1)$-module $W(x_1, \lambda, h_p, h_{\partial})$ has canonical basis $(e_0, \ldots, e_{N-1})$ and module structure given by:
\begin{align*}
& H_p e_i = h_p e_i; \quad H_{\partial} e_i = h_{\partial} e_i; \quad  [Z_1Z_3Z_4]^{-1} e_i= \lambda q^{-i} e_i  \mbox{, for i} \in \{0, \ldots, N-1\}; \\
&X_1 e_0= x_1 e_{N-1}; \quad X_1 e_i = e_{i-1} \mbox{, for }i\in \{1, \ldots N-1\}.
\end{align*}
\end{definition}

By Theorem \ref{theorem_CF_Rep}, any simple $\mathcal{Z}_{\omega}(\mathbb{D}_1)$-module is isomorphic to a module $W(x_1, \lambda, h_p, h_{\partial})$. Moreover, $W(x_1, \lambda, h_p, h_{\partial})\cong W(x'_1, \lambda', h'_p, h'_{\partial})$ if and only if $h'_p=h_p$, $h'_{\partial}=h_{\partial}$, $x'_1=x_1$ and $\lambda'=\lambda q^n$ for some $n\in \mathbb{Z}/N\mathbb{Z}$.

\begin{lemma}\label{Lemma_3}
Let $(x_1, \lambda, h_p, h_{\partial})\in (\mathbb{C}^*)^4$ and set $q:=\zeta^2$ as usual.
\begin{enumerate}
\item If $(h_{\partial}h_p\lambda^2)^N \neq -1$, set $\mu := \lambda^{-1} h_{\partial}^{-1}$ and 
\begin{align*}
& a:= -(q-q^{-1})^{N-2}x_1\frac{h_{\partial}^{-1}(h_p+h_p^{-1}) +q \lambda^2 +q^{-1}\mu^2}{1+(h_{\partial}h_p\lambda^2)^{-N}}; \\
& b:= \frac{1}{(q-q^{-1})^N} x_1^{-1}(1+ (h_{\partial}h_p\lambda^2)^{-N}).
\end{align*}
One has $W(x_1, \lambda, h_p, h_{\partial})\cong V(\lambda, \mu, a, b)$ as a  $\Uq$-module.
\item If $(h_{\partial}h_p\lambda^2)^N = -1$, set $\mu := \lambda^{-1} h_{\partial}^{-1}$ and
$$ a:= \frac{1}{N}(q-q^{-1})^{N-2}x_1h_{\partial}^{-1}(h_p^{-1}-h_p).$$
One has  $W(x_1, \lambda, h_p, h_{\partial})\cong V(\lambda, \mu, a, 0)$ as a  $\Uq$-module.
\end{enumerate}
\end{lemma}

\begin{proof}
We denote by $\chi : \mathcal{Z}\rightarrow \mathbb{C}$ the character on the center $\mathcal{Z}$ of $\Uq$ induced by $W(x_1, \lambda, h_p, h_{\partial})$ and write $\mu := \lambda^{-1} h_{\partial}^{-1}$. First one has $\chi(K^{N/2})=\lambda^N$ and $\chi(L^{N/2})= h_{\partial}^{-1}\lambda^{-N}$. Since $X_4= H_{\partial} H_p [Z_1Z_3Z_4]^{-2}$, one has $\chi(X_2^N)=(h_{\partial} h_p \lambda^2)^N$. Using the equality $F^N= \frac{1}{q-q^{-1}} X_1^{-N}(1+X_2^{-N})$, one finds $\chi(F^N)=  \frac{1}{(q-q^{-1})^N} x_1^{-1}(1+ (h_{\partial}h_p\lambda^2)^{-N})=:b$. Next, we deduce from the equality $\chi(\gamma_p)=h_p+h_p^{-1}$ that  $\chi(C)=-\frac{(h_p+h_p^{-1})h_{\partial}}{(q-q^{-1})^2}$ (recall that the Casimir $C$ is defined in Notations \ref{notations_QG} and related to $\gamma_p$ in Theorem \ref{theorem_skein_QG}).

\vspace{2mm}
\par First suppose that $(h_{\partial}h_p\lambda^2)^N \neq -1$, \textit{i.e.} that $\chi(F^N)\neq 0$. 
For $i\in \{0, \ldots, N-1\}$, define $v_i:=F^ie_0 \in W(x_1, \lambda, h_p, h_{\partial})$. It follows from the relations $FK^{1/2}=qK^{1/2}F$ and $FL^{1/2}=q^{-1}L^{1/2}F$ that one has $K^{1/2}v_i= \lambda q^{-i}v_i$ and $L^{1/2}v_i= \mu q^i v_i$ for all  $i\in \{0, \ldots, N-1\}$. Since $\chi(F^N)=b$, one has $Fv_{N-1}=bv_0$.  Consider the scalar $a$ defined in the first statement of the lemma and remark that one has the equality $\chi(C)=ab + \frac{q\lambda^2 +q^{-1}\mu^2}{(q-q^{-1})^2}$. For any $i\in \{0, \ldots, N-2\}$ one has
$$ Ev_{i+1}= EFv_i= \left(C- \frac{qL+q^{-1}K}{(q-q^{-1})^2} \right)v_i = \left(ab + \frac{q^{-i}\lambda^2 - q^i\mu^2}{q-q^{-1}}[i+1] \right)v_i.$$
Similarly, one finds
$$ Ev_0 = b^{-1}EFv_{N-1} = b^{1}\left( C- \frac{qL+q^{-1}K}{(q-q^{-1})^2} \right)v_{N-1} = a v_{N-1}.$$
We have thus proved that  $W(x_1, \lambda, h_p, h_{\partial})\cong V(\lambda, \mu, a, b)$ as a  $\Uq$-module.

\vspace{2mm}
\par Next suppose that  $(h_{\partial}h_p\lambda^2)^N = -1$, \textit{i.e.} that $\chi(F^N)= 0$. Then there exists $i_0\in \{0, \ldots, N-1\}$ such that $h_{\partial}h_p\lambda^2=-q^{-2i_0-1}$. Using that $F=\frac{1}{q-q^{-1}}X_1^{-1}\left(1+q H_{\partial}^{-1}H_p^{-1}[Z_1Z_3Z_4]^2 \right)$, one finds that 
\begin{equation}\label{eqreloute}
 F e_i = \left\{ \begin{array}{ll}
\frac{1+q^{2i+1} h_{\partial}^{-1}h_p^{-1}\lambda^{-2} }{q-q^{-1}}e_{i+1} & \mbox{, if }i\neq N-1; \\
x_1^{-1}\frac{1+q^{-1} h_{\partial}^{-1}h_p^{-1}\lambda^{-2} }{q-q^{-1}}e_{0} & \mbox{, if }i=N-1.
\end{array} \right. 
\end{equation}
In particular $Fe_{i_0}=0$. For $i\in \{0, \ldots, N_1\}$, define $v_i:= F^i e_{i_0+1}$. Set  $\lambda':= \lambda q^{-(i_0+1)}$ and $\mu' := \lambda^{-1} h_{\partial}^{-1}q^{i_0+1}$. It follows from the relations $FK^{1/2}=qK^{1/2}F$ and $FL^{1/2}=q^{-1}L^{1/2}F$ that one has $K^{1/2}v_i= \lambda' q^{-i}v_i$ and $L^{1/2}v_i= \mu' q^i v_i$ for all  $i\in \{0, \ldots, N-1\}$. Next it follows from the equality $h_{\partial}h_p\lambda^2=-q^{-2i_0-1}$ that one has $\chi(C)= \frac{q \lambda'^2 + q^{-1}\mu'^2}{(q-q^{-1})^2}$.  For any $i\in \{0, \ldots, N-2\}$ one has
$$ Ev_{i+1}= EFv_i= \left(C- \frac{qL+q^{-1}K}{(q-q^{-1})^2} \right)v_i = \left(\frac{q^{-i}\lambda'^2 - q^i\mu'^2}{q-q^{-1}}[i+1] \right)v_i.$$
It remains to compute $Ev_0$.
Since  $X_3^{-1}=H_{\partial} H_p^{-1}[Z_1Z_3Z_4]^2$ and $X_4^{-1}=H_{\partial}^{-1}H_pX_1$, one finds  $E= \frac{-1}{q-q^{-1}}H_{\partial}^{-1}H_pX_1\left( 1+qH_{\partial} H_p^{-1}[Z_1Z_3Z_4]^2 \right)$ from which we deduce that
$$ E e_i= \left\{ \begin{array}{ll}
- \frac{h_{\partial}^{-1}h_p + \lambda^2 q^{-2i+1}}{q-q^{-1}} e_{i-1} & \mbox{, if }i\neq 0; \\
-x_1 \frac{h_{\partial}^{-1}h_p + \lambda^2 q}{q-q^{-1}} e_{N-1} & \mbox{, if }i= 0.
\end{array} \right. $$
In particular, one has $Ev_0= \left\{ \begin{array}{ll} - \frac{h_{\partial}^{-1}h_p + \lambda^2 q^{-1-2i_0}}{q-q^{-1}} e_{i_0} & \mbox{, if }i_0\neq N-1; \\
-x_1 \frac{ h_{\partial}^{-1}h_p + \lambda^2 q}{q-q^{-1}} e_{N-1} & \mbox{, if }i_0=N-1. \end{array} \right.$
To express the relation between $e_{i_0}$ and $v_{N-1}$, we use Equation \eqref{eqreloute} and the equality $h_{\partial}h_p\lambda^2=-q^{-2i_0-1}$ and deduce that

\begin{eqnarray*}
v_{N-1} &=& F^{N-1}e_{i_0+1} = \left\{ \begin{array}{ll}
\frac{x_1^{-1} }{(q-q^{-1})^{N-1}} \prod_{k=i_0+1}^{N+i_0-1} \left( 1-q^{-(2i_0+1)}q^{2k+1} \right)e_{i_0} & \mbox{, if }i_0 \neq N-1; \\
\frac{1 }{(q-q^{-1})^{N-1}} \prod_{k=0}^{N-2} \left( 1-q^{2k+1} \right)e_{N-1} & \mbox{, if }i_0 = N-1. 
\end{array} \right. \\
&=&  \left\{ \begin{array}{ll}
\frac{x_1^{-1} }{(q-q^{-1})^{N-1}} P_N(1)e_{i_0} & \mbox{, if }i_0 \neq N-1; \\
\frac{1 }{(q-q^{-1})^{N-1}}P_N(1)e_{N-1} & \mbox{, if }i_0 = N-1. 
\end{array} \right.
\end{eqnarray*}

where the polynomial $P_N(X)$ is defined by  $P_N(X):= \prod_{i=1}^{N-1} (X-q^i)$. Since $P_N(X)(X-1)= X^N -1$, one has $P_N(X)= \sum_{k=0}^{N-1} X^k$, hence $P_N(1)=N$. Eventually we finds that 
$$ E v_0= \frac{1}{N} (q-q^{-1})^{N-2} x_1 h_{\partial}^{-1}(h_p^{-1}-h_p)v_{N-1}=av_{N-1}.$$
Hence $W(x_1, \lambda, h_p, h_{\partial})\cong V(\lambda', \mu', a, 0)$ as a  $\Uq$-module. Clearly  this module is isomorphic to $V(\lambda, \mu, a ,0)$.
This concludes the proof.

\end{proof}

\begin{lemma}\label{Lemma_4}
Let $\lambda, \mu \in \mathbb{C}^*$ and $a,b\in \mathbb{C}$ and consider the following system with variable $h_p$:
\begin{equation}\label{system}
\left\{ \begin{array}{lll}
ab &=& -\frac{\lambda\mu (h_p+h_p^{-1}) + q\lambda^2 + q^{-1}\mu^2}{(q-q^{-1})^2} \\
h_p^N &\neq& - (\lambda \mu^{-1})^N
\end{array}\right.
\end{equation}
The system \eqref{system} admits a solution $h_p\in \mathbb{C}^*$ if and only if
\begin{enumerate}
\item either $S:= \prod_{i=0}^{N-1} \left( ab + [i]\frac{\lambda^2 q^{1-i} -\mu^2 q^{i-1}}{q-q^{-1}} \right)$ is non null, 
\item or $S=0$ and $(\lambda \mu^{-1})^{2N} \neq 1$.
\end{enumerate}
\end{lemma}

\begin{proof}
Consider the polynomial $$P(X)=X^2 +(\lambda \mu)^{-1}\left( ab(q-q^{-1})^2 + q\lambda^2 + q^{-1}\mu^2 \right) X +1.$$ Since its constant term is $+1$, there exists $h_p \in \mathbb{C}^*$ such that $P(X)=(X-h_p)(X-h_p^{-1})$. First, if $h_p^N = - (\lambda \mu^{-1})^N$, one deduce from the equality $P(h_p)=0$ that $S=0$. Hence if $S\neq 0$, the system \eqref{system} admits a solution. Next if $S=0$, there exists $n\in \{0, \ldots, N-1\}$ such that $ab+[n]\frac{\lambda^2q^{1-n}-\mu^2 q^{n-1}}{q-q^{-1}}=0$. It follows that $P(X)=(X+\lambda^{-1}\mu q^{2n-1})(X+ \lambda \mu^{-1} q^{1-2n})$, thus either $h_p$ or $h_p^{-1}$ is equal to $-\lambda^{-1}\mu q^{2n-1}$. It follows that the system \eqref{system} admits a solution such that $S=0$ if and only if  $(\lambda \mu^{-1})^{2N} \neq 1$.

\end{proof}

\begin{proof}[Proof of Proposition \ref{prop_D1_lift}]
By Lemma \ref{Lemma_3} any simple $\mathcal{Z}_{\zeta}(\mathbb{D}_1)$-module is isomorphic, as $\Uq$-module, to a module $V(\lambda, \mu, a, b)$. Conversely, consider such a module $V=V(\lambda, \mu, a, b)$, denote by $\chi: \mathcal{Z}\rightarrow \mathbb{C}$ its central character and write  $S:=  \prod_{i=0}^{N-1} \left( ab + [i]\frac{\lambda^2 q^{1-i} -\mu^2 q^{i-1}}{q-q^{-1}} \right)$. Note that $S=\chi(E^NF^N)$. 
\par First, if $V$ is cyclic then $S\neq 0$ and by Lemma \ref{Lemma_4} one can find $h_p\in \mathbb{C}^*$ solution of the system \eqref{system}. Set $x_1:= \frac{1+ \lambda^{2N}h_p^N (\lambda \mu)^N}{(q-q^{-1})^N}$. By Lemma \ref{Lemma_3}, one has $V\cong W(\lambda, x_1, h_p, (\lambda \mu)^{-1})$, hence $V$ extends to a $\mathcal{Z}_{\zeta}(\mathbb{D}_1)$ simple module.
\par  Next suppose that $\chi(F^N)\neq 0$ and that $\chi(E^N)=0$, thus $S=0$. In this case, by Lemma \ref{Lemma_3}, $V$ extends to a $\mathcal{Z}_{\zeta}(\mathbb{D}_1)$ module if and only if the system \eqref{system} admits a solution, hence if and only if $(\lambda \mu^{-1})^{2N} \neq 1$ by Lemma \ref{Lemma_4}.
\par  Now suppose that $\chi(F^N)=0$ and that $\chi(E^N)\neq 0$. The hypothesis $\chi(E^N)\neq 0$ implies that $a\neq 0$ and that for all $i\in \{1, \ldots, N-1\}$ one has $\lambda^2\mu^{-2} \neq q^{2(i-1)}$. Thus, if $(\lambda \mu^{-1})^{2N}=1$ then $\lambda \mu^{-1}= \pm q^{-1}$. Since $\chi(\gamma_p) = -q\lambda \mu^{-1} -q^{-1}\lambda^{-1}\mu$, one has $\chi(\gamma_p)= \pm 2$. Hence, if $V$ would extend to a simple $\mathcal{Z}_{\zeta}(\mathbb{D}_1)$-module, its puncture invariant would be $h_p=\pm 1$ and Lemma \ref{Lemma_3} would implies that $a=0$ which contradicts the hypothesis $\chi(E^N)\neq 0$. Conversely, if $(\lambda \mu^{-1})^{2N}\neq 1$, set $h_p:= -q\lambda \mu^{-1}$ and  $x_1:= \frac{aN}{(q-q^{-1})^{N-2}\lambda \mu (h_p - h_p^{-1})}$.  Note that $a\neq 0$ and,  since $(\lambda \mu^{-1})^{2N}\neq 1$, then $h_p-h_p^{-1}\neq 0$, hence $x_1\in \mathbb{C}^*$ is well defined. By Lemma \ref{Lemma_3} $V$ is isomorphic to $W(\lambda, x_1, h_p, (\lambda\mu)^{-1})$. 
\par Eventually, suppose that $\chi(E^N)=\chi(F^N)=0$. Since $b=0$ and $\chi(E^N)=0$, one has either $a=0$, or there exists $i \in \{1, \ldots, N-1\}$ such that $(\lambda\mu^{-1})^2=q^{2(i-1)}$. The latter case implies $(\lambda \mu^{-1})^{2N}=1$. If $a=0$ and $V$ extends to a simple $\Uq$-module, then by Lemma \ref{Lemma_3}, one has $h_p=\pm 1$, hence $\chi(T)=\pm 2=-q\lambda \mu^{-1} -q^{-1}\lambda^{-1}\mu$ and one has again  $(\lambda \mu^{-1})^{2N}=1$.  Conversely, if $(\lambda \mu^{-1})^{2N}=1$, then one can write $\lambda^2 \mu^{-2}=q^{2(i-1)}$ for some $i \in \{0, \ldots, N-1\}$. If $i\neq 0$, then $a\neq 0$. In this case we define $h_p:= -q\lambda \mu^{-1}$ (different from $\pm 1$ by hypothesis) and $x_1:= \frac{aN}{(q-q^{-1})^{N-2}\lambda \mu (h_p - h_p^{-1})}$. By Lemma \ref{Lemma_3},  $V$ is isomorphic to $W(\lambda, x_1, h_p, (\lambda\mu)^{-1})$. Finally, if $\lambda^2\mu^{-2}=q^{-2}$, then the hypothesis $\chi(E^N)=0$ implies $a=0$. We choose $x_1\in \mathbb{C}^*$ and $h_p=\pm 1$ arbitrarily  and  Lemma \ref{Lemma_3} implies that  $V$ is isomorphic to $W(\lambda, x_1, h_p, (\lambda\mu)^{-1})$.

\end{proof}

\subsection{Azumaya standard representations are gauge equivalent to QT representations}

 The following theorem will ensure that quantum hyperbolic geometry can always be used to compute an intertwiner $L_{V_{\widehat{x}}}(\beta)$.

\begin{theorem}\label{theorem_Dn_lift}
Every standard Azumaya $\overline{\mathcal{S}}_{\zeta}(\mathbb{D}_n)$ module is gauge equivalent to a QT representation.
\end{theorem}

\begin{proof}
By Proposition \ref{prop_D1_lift} and Theorem \ref{theorem_RepRSSkein}, a standard Azumaya  $\overline{\mathcal{S}}_{\zeta}(\mathbb{D}_n)$ module with classical shadow $\rho \in \overline{\mathcal{R}}_{\SL_2}(\mathbb{D}_n)$ extends to a simple $\mathcal{Z}_{\zeta}(\mathbb{D}_n)$ module if and only if $\rho(\delta^{(i)})\in \mathcal{O}$ for all $1\leq i \leq n$. For every conjugacy class $C \subset \SL_2$ the intersection $C\cap \mathcal{O}$ is dense in $C$. It follows that every representation $\rho$ is conjugate to a representation $\rho'$ such that $\rho'(\delta^{(i)})\in \mathcal{O}$. We conclude using Proposition \ref{prop_orbit}.

\end{proof}

 \subsection{Change of coordinates and quantum dilogarithms}
 \subsubsection{Definition of the change of coordinates}

\begin{notations} Since the balanced Chekhov-Fock algebra $\mathcal{Z}_{A}(\mathbf{\Sigma}, \Delta)$  is Noetherian, its multiplicative subset of non-zero elements satisfies the Ore condition and the balanced Chekhov-Fock algebra admits a division algebra that will be denoted $\widehat{\mathcal{Z}}_{A}(\mathbf{\Sigma}, \Delta)$ (see \textit{e.g.} \cite{Kassel, Cohn_SkewFields} for details).
\end{notations}

In this subsection, for any two triangulations $\Delta$ and $\Delta'$, we will define an isomorphism $\phi_{\Delta, \Delta'}^{\mathbf{\Sigma}}: \widehat{\mathcal{Z}}_{A}(\mathbf{\Sigma}, \Delta') \xrightarrow{\cong} \widehat{\mathcal{Z}}_{A}(\mathbf{\Sigma}, \Delta)$. We first do it when $\Delta$ and $\Delta'$ differ by an elementary move which consists either in a reindexing or a flip.
 
 \begin{definition}\label{def_reindexing_map}
 \begin{enumerate}
\item  The triangulation $\Delta$ is obtained from $\Delta'$ by \textbf{reindexing its edges} if $\mathcal{E}(\Delta)= \mathcal{E}(\Delta')$ and there exists  a permutation  $\pi$  of $ \{1, 2,  \ldots, |\mathcal{E}(\Delta)| \}$ such that $I_{\Delta}=\pi \circ I_{\Delta'}$. We write $\Delta'= \pi_* \Delta$. The permutation $\pi$  induces an isomorphism $\phi_{\pi_* \Delta, \Delta}^{\mathbf{\Sigma}}: \mathcal{Z}_{A}(\mathbf{\Sigma}, \Delta) \xrightarrow{\cong} \mathcal{Z}_{A}(\mathbf{\Sigma}, \pi_*\Delta)$ which sends a balanced monomial $Z^{\mathbf{k}}$ to $Z^{\mathbf{k}\circ \pi}$.
 \item Let $e$ be an inner edge of $\Delta$ bounding two distinct faces $\mathbb{T}_1$ and $\mathbb{T}_2$.  The union of these two faces forms a square $Q$ for which $e$ is one diagonal. The topological triangulation $\Delta'$ is obtained from $\Delta$ by a \textbf{flip along} $e$, if the set $\mathcal{E}(\Delta')$ is obtained from $\mathcal{E}(\Delta)$ by removing $e$ and adding the other diagonal $e'$ of $Q$. The indexing $I_{\Delta'}$ is required to satisfy $I_{\Delta'}(a)=I_{\Delta}(a)$ for $a\in \mathcal{E}(\Delta)\setminus \{e\}$ and $I_{\Delta'}(e')=I_{\Delta}(e)$. We write $\Delta'= F_e \Delta$.
 \end{enumerate}
 \end{definition}
By a theorem of Penner \cite{Penner87}, any two triangulations are related by a sequence of flips and reindexings. Moreover, any two such sequences are related by some elementary moves; the most remarkable one is the Pentagon identity (see \cite{Penner87}). So it suffices to define $\phi_{\Delta, \Delta'}^{\mathbf{\Sigma}}$ when $\Delta, \Delta'$ differ either from a reindexing (this is done in Definition \ref{def_reindexing_map}) or a flip in such a way that these Penner moves are satisfied. 
 \par  Let $Q$ denote the square, \textit{i.e.} a disc with four boundary edges,  let $\Delta, \Delta'$ be its two triangulations drawn in Figure \ref{fig_square}, let $\beta^{(1)}, \ldots, \beta^{(4)}$ be the four arcs of $Q$ drawn in Figure \ref{fig_square}. By Remark \ref{remark_grading}, the algebra $\mathcal{Z}_{A}(Q, \Delta)$ (resp. the algebra $\mathcal{Z}_{A}(Q, \Delta')$)  is generated by the elements $\Tr_{A}^{\Delta} (\beta^{(i)}_{++})$,  (resp. by the elements $\tr_{A}^{\Delta'} (\beta^{(i)}_{++})$),  for $i\in \{1, 2, 3, 4\}$ , together with the elements $X_1, \ldots, X_5$ (resp. the elements $X_1', \ldots, X_5'$).

 \begin{figure}[!h] 
\centerline{\includegraphics[width=9cm]{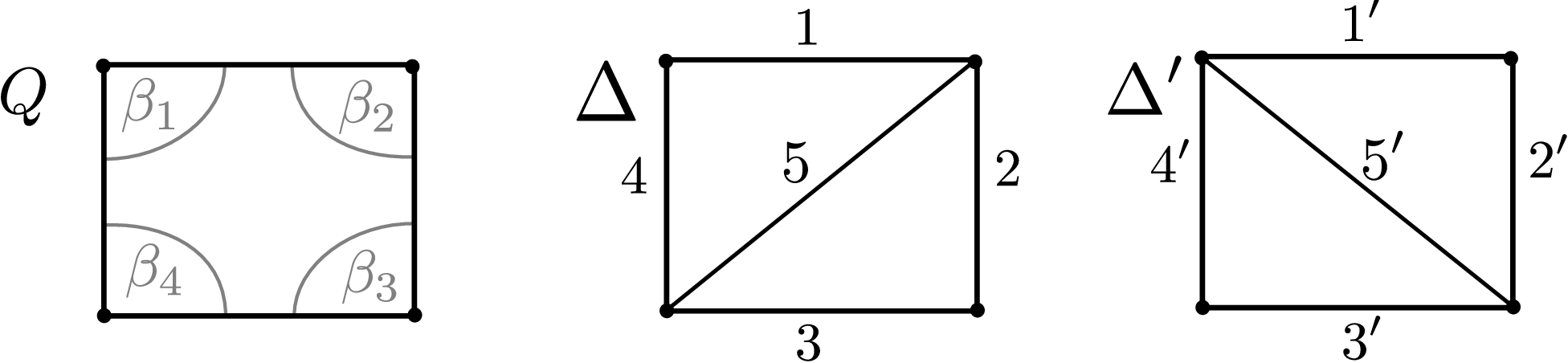} }
\caption{The square $Q$, some arcs and two  triangulations.} 
\label{fig_square} 
\end{figure}

\begin{definition}\label{def_square_changecoordinates}
The isomorphism $\Phi_{\Delta, \Delta'}^{Q} : \widehat{\mathcal{Z}}_{A}(Q, \Delta') \xrightarrow{\cong} \widehat{\mathcal{Z}}_{A}(Q, \Delta)$ is defined by $\Phi^Q_{\Delta, \Delta'}\left(\Tr_{A}^{\Delta'}(\beta^{(i)}_{++})\right)= \Tr_{A}^{\Delta}(\beta^{(i)}_{++})$, for $i\in \{1, 2, 3, 4\}$, and by:
\begin{align*}
&  \Phi^Q_{\Delta, \Delta'}(X_1')=X_1+[X_1X_5]; &\Phi^Q_{\Delta, \Delta'}(X_3')=X_3+[X_3X_5]; \\
&  \Phi^Q_{\Delta, \Delta'}(X_2')=(X_2^{-1}+[X_2X_5]^{-1})^{-1}; & \Phi^Q_{\Delta, \Delta'}(X_4')=(X_4^{-1}+[X_4X_5]^{-1})^{-1}; \\
& \Phi^Q_{\Delta, \Delta'}(X_5')=X_5^{-1}. &
\end{align*}
\end{definition}

Suppose that $\Delta'=F_e\Delta$ is obtained from $\Delta$ by flipping an inner edge $e$ inside a square $Q$ as in Definition \ref{def_reindexing_map}. The marked surface $\mathbf{\Sigma}$ 
 is obtained by gluing $Q \bigsqcup \mathbf{\Sigma}\setminus Q$ along some pair of boundary edges. The splitting morphism $\theta: \mathcal{Z}_{A}(\mathbf{\Sigma}, \Delta) \hookrightarrow \mathcal{Z}_{A}(Q, \Delta_Q) \otimes \mathcal{Z}_{A}(\mathbf{\Sigma}\setminus Q, \Delta_{\mathbf{\Sigma}\setminus Q})$  induces an injective morphism $\widehat{\theta} : \widehat{\mathcal{Z}}_{A}(\mathbf{\Sigma}, \Delta) \hookrightarrow \widehat{\mathcal{Z}}_{A}(Q, \Delta_Q) \otimes \widehat{\mathcal{Z}}_{A}(\mathbf{\Sigma}\setminus Q, \Delta_{\mathbf{\Sigma}\setminus Q})$. 
  Here $\mathbf{\Sigma}\setminus Q$ might be empty in which case we set $\mathcal{Z}_{A}(\mathbf{\Sigma}\setminus Q, \Delta_{\mathbf{\Sigma} \setminus_Q}):= k$ (think for instance of the once punctured torus). Similarly, denote by $\widehat{\theta}' : \widehat{\mathcal{Z}}_{A}(\mathbf{\Sigma}, \Delta') \hookrightarrow \widehat{\mathcal{Z}}_{A}(Q, \Delta'_Q) \otimes \widehat{\mathcal{Z}}_{A}(\mathbf{\Sigma}\setminus Q, \Delta_{\mathbf{\Sigma}\setminus Q})$ the corresponding splitting morphism.

\begin{definition}\label{def_flip_changecoordinates}
\begin{enumerate}
\item If $\Delta'=F_e \Delta$, 
the isomorphism $\Phi_{\Delta, \Delta'}^{\mathbf{\Sigma}} : \widehat{\mathcal{Z}}_{A}(\mathbf{\Sigma}, \Delta') \xrightarrow{\cong} \widehat{\mathcal{Z}}_{A}(\mathbf{\Sigma}, \Delta)$ is the unique morphism making the following diagram commuting:
$$ \begin{tikzcd}
\widehat{\mathcal{Z}}_{A}(\mathbf{\Sigma}, \Delta')  \arrow[r, hook, "\widehat{\theta}"] \arrow[d, dotted, "\Phi_{\Delta, \Delta'}^{\mathbf{\Sigma}}"] &
 \widehat{\mathcal{Z}}_{A}(Q, \Delta_Q) \otimes \widehat{\mathcal{Z}}_{A}(\mathbf{\Sigma}\setminus Q, \Delta_{\mathbf{\Sigma}\setminus Q}) \arrow[d, "\cong"', "\Phi_{\Delta, \Delta'}^{Q}\otimes \id"] \\
 \widehat{\mathcal{Z}}_{A}(\mathbf{\Sigma}, \Delta) \arrow[r, hook, "\widehat{\theta}'"] &
 \widehat{\mathcal{Z}}_{A}(Q, \Delta'_Q) \otimes \widehat{\mathcal{Z}}_{A}(\mathbf{\Sigma}\setminus Q, \Delta_{\mathbf{\Sigma}\setminus Q})
 \end{tikzcd}$$
 \item  Let $\Delta$ and $\Delta'$ be two arbitrary triangulations of $\mathbf{\Sigma}$. Let $\Delta=\Delta_0, \Delta_1, \ldots, \Delta_n=\Delta'$ be a sequence of triangulations of $\mathbf{\Sigma}$ such that $\Delta_{i+1}$ is obtained from $\Delta_i$ by either a flip or a reindexing. The \textbf{change of coordinates} isomorphism  $\Phi_{\Delta, \Delta'}^{\mathbf{\Sigma}} : \widehat{\mathcal{Z}}_{A}(\mathbf{\Sigma}, \Delta') \xrightarrow{\cong} \widehat{\mathcal{Z}}_{A}(\mathbf{\Sigma}, \Delta)$ is the composition:
 $$ \Phi_{\Delta, \Delta'}^{\mathbf{\Sigma}} := \Phi_{\Delta, \Delta_1}^{\mathbf{\Sigma}}\circ \Phi_{\Delta_1, \Delta_2}^{\mathbf{\Sigma}}\circ \ldots \circ \Phi_{\Delta_{n-1}, \Delta'}.$$
\end{enumerate}
 \end{definition}

  \begin{proposition}\label{prop_change_coordinates}\cite{LiuQTeichNC, BonahonWongqTrace, Hiatt}
 \begin{enumerate}
 \item The isomorphism  $\Phi_{\Delta, \Delta'}^{\mathbf{\Sigma}}$ is well defined in the sense that  it does not depend on the choice of sequence $\Delta=\Delta_0, \Delta_1, \ldots, \Delta_n=\Delta'$ (i.e. it preserves Penner moves).
 \item 
 One has $\Phi_{\Delta, \Delta''}^{\mathbf{\Sigma}}= \Phi_{\Delta, \Delta'}^{\mathbf{\Sigma}} \circ \Phi_{\Delta', \Delta''}^{\mathbf{\Sigma}}$ for any three  triangulations $\Delta, \Delta', \Delta''$.
 \item For $\Delta, \Delta'$ two  triangulations of $\mathbf{\Sigma}$, the following diagram commutes:
 
 $$\begin{tikzcd}
  {} &{} &  \widehat{\mathcal{Z}}_{A}(\mathbf{\Sigma}, \Delta') \arrow[dd, "\cong"', "\Phi_{\Delta, \Delta'}^{\mathbf{\Sigma}}"] \\
  \overline{\mathcal{S}}_{A}(\mathbf{\Sigma}) \arrow[urr, hook, "\Tr_{A}^{\Delta'}"] \arrow[drr, hook, "\Tr_{A}^{\Delta}"] & {} & {} \\
 {} & {} & \widehat{\mathcal{Z}}_{A}(\mathbf{\Sigma}, \Delta)
  \end{tikzcd}
 $$

\end{enumerate}
\end{proposition}

 \subsubsection{The Fock-Goncharov decomposition and the quantum exponential}
 
 Following Fock and Goncharov in \cite[Section $1.3$]{FockGoncharovClusterVariety}, we now show that the flip isomorphisms $\Phi^{\mathbf{\Sigma}}_{F_e\Delta, \Delta}$ of Definition \ref{def_flip_changecoordinates}, can be decomposed as   $\Phi_{F_e\Delta, \Delta}^{\mathbf{\Sigma}} = \nu^{\mathbf{\Sigma}}_{F_e\Delta, \Delta} \circ \mathrm{Ad}(s_q(X_e^{-1}))$, where $ \nu^{\mathbf{\Sigma}}_{F_e\Delta, \Delta}$ is an automorphism of $\mathcal{Z}_{\omega}(\Sigma,\Delta)$ (not only of its division algebra) and the automorphism $\mathrm{Ad}(s_q(X_e^{-1}))$ is inner in a suitable extension of $\widehat{\mathcal{Z}}_{\omega}(\Sigma,\Delta)$.   Recall that we defined $[n]:= \frac{q^n - q^{-n}}{q-q^{-1}}$ and write $[n]! = [n][n-1]\ldots [1]$ with the convention $[0]!=1$. We will use the notation $\mathrm{Ad}(x)y:=xyx^{-1}$.
 
 \begin{definition}\label{def_quantum_exponential}
  The  \textbf{quantum exponential} is the power series  defined as follows.
 \begin{enumerate}
 \item If $q\in k^*$ is generic (\textit{i.e.} if $q$ is not a root of unity) and $q-q^{-1}$ is invertible, we set:
 $$ \exp_q(x) := \sum_{n=0}^{\infty} \frac{q^{n(n-1)/2}}{[n]!} x^n\in k[[x]].$$
 \item If $q\in k^*$ is a root of unity of odd order $N>1$, we set:
 $$ \exp^{<N}_q(x):= \sum_{n=0}^{N-1} \frac{q^{n(n-1)/2}}{[n]!} x^n\in k[x].$$
 \end{enumerate}
 
 We also write $s_q(x):= \exp_q\left( -\frac{x}{q-q^{-1}} \right)$ and $s_q^{<N}(x):= \exp^{<N}_q\left( -\frac{x}{q-q^{-1}} \right)$.
 \end{definition}
 
 \begin{lemma}\label{lemma_expq} Consider a $k$-algebra $\mathcal{A}$ with two elements $u,v$ such that $uv=q^2vu$.
  \begin{enumerate}
 \item If $q\in k^*$ is generic and the elements $\exp_q(u)^{\pm 1}$ and $\exp_q(v)^{\pm 1}$ are well-defined, that is if $\mathcal{A}$ is a topological algebra with topological embeddings $k[[u]]\hookrightarrow \mathcal{A}$ and $k[[v]]\hookrightarrow \mathcal{A}$, 
  then the following results hold.
 \begin{enumerate}
 \item One has $s_q(u+v)=s_q(u)s_q(v)$.
 \item One has $s_q(q^2u)s_q(u)^{-1}=(1+qu)^{-1}$ and $s_q(q^{-2}u)s_q(u)^{-1}=1+q^{-1}u$. Hence $\mathrm{Ad}(s_q(u))v=v(1+qu)^{-1}$ and $\mathrm{Ad}(s_q(v))u=u(1+q^{-1}v)$. 
 \end{enumerate}
 \item  If $q\in k^*$ is a root of unity of odd order $N>1$, then the following results hold.
 \begin{enumerate}
 \item If  $u^nv^{N-n}=0$ for all $n\in \{0, \ldots, N\}$, then $s_q^{<N}(u+v)=s_q^{<N}(u)s_q^{<N}(v)$.
 \item If $u^N=0$ then  $s_q^{<N}(q^2u)s_q^{<N}(u)^{-1}=(1+qu)^{-1}$ and $s^{<N}_q(q^{-2}u)s^{<N}_q(u)^{-1}=1+q^{-1}u$. 
  \end{enumerate}
 \end{enumerate}
 \end{lemma}
 
 \begin{proof}
 Formula $(1)(a)$ is \cite[Lemma A.$3$]{OhtsukiBook}. The formulas $(1)(b)$ are immediate consequence of the following factorization, which is proved for instance in \cite[Chapter II.D, Proposition $2$]{ZagierDilogarithm}:
 $$s_q(x)=\prod_{n=0}^{\infty}(1+q^{2n+1}x).$$
 Formula $(2)(a)$ is \cite[Lemma A.$12$]{OhtsukiBook}. The formulas $(2)(b)$ follow from straightforward computations.
 \end{proof}

  \begin{remark}The q-Pochhammer symbol  $(x;q)_{\infty}:=\prod_{i=0}^{\infty}(1-q^ix)$ is related to the quantum dilogarithm $\mathrm{Li}_2(x;q):=\sum_{i=1}^{\infty}\frac{x^n}{n(1-q^n)}$ by the relation $(x;q)_{\infty}=\exp(-\mathrm{Li}_2(x;q))$ (see \cite[Proposition $2$]{ZagierDilogarithm}). Therefore,  the quantum exponential is related to the quantum dilogarithm by
   $$\exp_q\left( -\frac{x}{q-q^{-1}} \right)=s_q(x)=(-qx;q^2)_{\infty}=\exp(-\mathrm{Li}_2(-qx;q^2)).$$
    \end{remark}

 \begin{definition}\label{def_nu}
 \begin{enumerate}
 \item \textit{(The case of the square)}
 The isomorphism \\ $\nu^Q_{\Delta, \Delta'} : \mathcal{Z}_{A}(Q, \Delta') \xrightarrow{\cong} \mathcal{Z}_{A}(Q, \Delta)$ is defined by: 
 \begin{align*}
 & \nu^Q_{\Delta, \Delta'} \left(\Tr_{A}^{\Delta'} (\beta^{(i)}_{++})\right)= \Tr_{A}^{\Delta} (\beta^{(i)}_{++}) \mbox{, for }i\in \{1,2,3,4\}; \\
 & \nu^Q_{\Delta, \Delta'} (X'_1) = X_1; \quad \nu^Q_{\Delta, \Delta'}(X'_3) = X_3;  \\
 & \nu^Q_{\Delta, \Delta'} (X'_2)= [X_2X_5]; \quad \nu^Q_{\Delta, \Delta}(X'_4)= [X_4X_5]; \\
 & \nu^Q_{\Delta, \Delta'} (X'_5)= X_5^{-1}.
 \end{align*}
 \item \textit{(The general case)} Let $\Delta$ and $\Delta'=F_{e}\Delta$ be two  triangulations of $\mathbf{\Sigma}$ which differ by a flip along an inner edge $e\in\mathcal{E}(\Delta')$ and denote by $Q\subset \mathbf{\Sigma}$ the square whose $e$ is a diagonal. The isomorphism $\nu^{\mathbf{\Sigma}}_{\Delta, \Delta'} : \mathcal{Z}_{A}(\mathbf{\Sigma}, \Delta') \xrightarrow{\cong} \mathcal{Z}_{A}(\mathbf{\Sigma}, \Delta)$ is the unique isomorphism making the following diagram commuting: 
 $$ \begin{tikzcd}
\mathcal{Z}_{A}(\mathbf{\Sigma}, \Delta')  \arrow[r, hook, "\widehat{i}"] \arrow[d, dotted, "\nu_{\Delta, \Delta'}^{\mathbf{\Sigma}}"] &
 \mathcal{Z}_{A}(Q, \Delta_Q) \otimes \mathcal{Z}_{A}(\mathbf{\Sigma}\setminus Q, \Delta_{\mathbf{\Sigma}\setminus Q}) \arrow[d, "\cong"', "\nu_{\Delta, \Delta'}^{Q}\otimes \id"] \\
 \mathcal{Z}_{A}(\mathbf{\Sigma}, \Delta) \arrow[r, hook, "\widehat{i}'"] &
 \mathcal{Z}_{A}(Q, \Delta'_Q) \otimes \mathcal{Z}_{A}(\mathbf{\Sigma}\setminus Q, \Delta_{\mathbf{\Sigma}\setminus Q})
 \end{tikzcd}$$
 
 \end{enumerate}
 
 \end{definition}
To formulate the Fock-Goncharov decomposition, one needs to consider the element $s_q(X_e^{-1})$ for $X_e\in  \mathcal{Z}_{A}(\mathbf{\Sigma}, \Delta')$ and one would like to be able to consider an automorphism $\mathrm{Ad}(s_q(X_e^{-1}))$. Since the quantum exponential is a power series, let us introduce the following:

\begin{definition}\label{def_series} Let $(\mathbf{\Sigma}, \Delta)$ be a triangulated marked surface. The \textbf{completed balanced Chekhov-Fock module} is the topological $k$-module $\overline{\mathcal{Z}}_{A}(\mathbf{\Sigma}, \Delta)$ whose elements are the power series $\sum_{\mathbf{k}} c_{\mathbf{k}} Z^{\mathbf{k}}$, associated to balanced monomials $Z^{\mathbf{k}}$. It has a product only partially defined  induced by $Z^{\mathbf{k}}Z^{\mathbf{k'}} = A^{-\frac{1}{2}\left(\mathbf{k}, \mathbf{k}'\right)^{WP}}Z^{\mathbf{k} + \mathbf{k}'}$. 
\end{definition}

 The product between  two elements of $\overline{\mathcal{Z}}_{A}(\mathbf{\Sigma}, \Delta)$ is not always defined. For instance, for $x:=\sum_{\mathbf{k}}  Z^{\mathbf{k}}$, the product $x^2$ does not make sense. However, one has a natural inclusion $\mathcal{Z}_{A}(\mathbf{\Sigma}, \Delta) \subset \overline{\mathcal{Z}}_{A}(\mathbf{\Sigma}, \Delta)$ and the product of an element of $\mathcal{Z}_{A}(\mathbf{\Sigma}, \Delta)$ by an element of $\overline{\mathcal{Z}}_{A}(\mathbf{\Sigma}, \Delta)$ is well defined. For $x\in \overline{\mathcal{Z}}_{A}(\mathbf{\Sigma}, \Delta)$, we will say that the automorphism $\mathrm{Ad}(x) \in \mathrm{Aut} (\mathcal{Z}_{A}(\mathbf{\Sigma}, \Delta))$ is well defined if for any $y \in \mathcal{Z}_{A}(\mathbf{\Sigma}, \Delta)$, there exists $z\in \mathcal{Z}_{A}(\mathbf{\Sigma}, \Delta)$ such that $xy=zx$. In this case, $\mathrm{Ad}(x)\cdot y := z$ is uniquely determined and $\mathrm{Ad}(x)$ is indeed an automorphism of $\mathcal{Z}_{A}(\mathbf{\Sigma}, \Delta)$.

The following lemma is stated in a different form in \cite[Section $1.3$]{FockGoncharovClusterVariety}, where it appears as a definition. Since the authors of \cite{FockGoncharovClusterVariety} only considered the $0$-th graded part of the balanced Chekhov-Fock algebra, we prove that their decomposition generalizes to the full algebra as-well.
 
 \begin{lemma}\label{lemma_FockGoncharov}[Fock-Goncharov decomposition]
Let $\Delta'$ and $\Delta=F_{e}\Delta'$ be two  triangulations of $\mathbf{\Sigma}$ which differ from a flip along an inner edge $e\in\mathcal{E}(\Delta')$.
 If $q\in k^*$ is generic, then the automorphism $\mathrm{Ad}(s_q(X_e^{-1})) \in \mathrm{Aut}(\mathcal{Z}_{A}(\mathbf{\Sigma}, \Delta'))$  is well defined and 
 one has the equality $\Phi_{\Delta, \Delta'}^{\mathbf{\Sigma}} = \nu^{\mathbf{\Sigma}}_{\Delta, \Delta'} \circ \mathrm{Ad}(s_q(X_e^{-1}))$.
 \end{lemma}
 
 \begin{proof} In view of Definitions \ref{def_flip_changecoordinates} and \ref{def_nu}, it is sufficient to make the proof in the case where $\mathbf{\Sigma}=Q$.
 
 A straightforward computation shows that $\mathrm{Ad}(s_q(X_5'^{-1})) \cdot \Tr_{\zeta}^{\Delta'} (\beta^{(i)}_{++}) = \Tr_{\zeta}^{\Delta'} (\beta^{(i)}_{++})$ for every $i\in\{1,2,3,4\}$, hence one has $ \nu^Q_{\Delta, \Delta'} \circ \mathrm{Ad}(s_q(X_5'^{-1})) \left(   \Tr_{\zeta}^{\Delta'} (\beta^{(i)}_{++})\right) = \Phi_{\Delta, \Delta'}^Q \left( \Tr_{\zeta}^{\Delta'} (\beta^{(i)}_{++})\right) = \Tr_{\zeta}^{\Delta} (\beta^{(i)}_{++})$. The equalities $\nu_{\Delta, \Delta'} \circ \mathrm{Ad}(s_q(X_5'^{-1})) (X'_i)= \Phi_{\Delta, \Delta'}^Q(X'_i)$, for $1\leq i \leq 5$, follow from Lemma \ref{lemma_expq}.
 \end{proof}

\subsubsection{Fock-Goncharov type decomposition at roots of unity and cyclic quantum dilogarithms}

\begin{notations} For $\chi : Z(\mathbf{\Sigma}, \Delta) \to \mathbb{C}$ a character over the center of $\mathcal{Z}_{\zeta}(\mathbf{\Sigma}, \Delta)$ we denote by $\mathcal{I}_{\chi}\subset \mathcal{Z}_{\zeta}(\mathbf{\Sigma}, \Delta)$ the ideal generated by $\ker{\chi}$. In particular, by Theorem \ref{theorem_CF_Rep}, the quotient $\quotient{\mathcal{Z}_{\zeta}(\mathbf{\Sigma}, \Delta)}{\mathcal{I}_{\chi}}\cong \Mat_D(\mathbb{C})$ is a matrix algebra.Two characters $\chi$ and $\chi'$ over the centers of $\mathcal{Z}_{\zeta}(\mathbf{\Sigma}, \Delta)$ and $\mathcal{Z}_{\zeta}(\mathbf{\Sigma}, \Delta')$ are said compatible if $\chi = \Phi_{\Delta \Delta'}^{\mathbf{\Sigma}} \circ \chi'$. In this case, we get an isomorphism
$$ \Phi_{(\Delta, \chi), (\Delta', \chi')}^{\mathbf{\Sigma}} : \quotient{\mathcal{Z}_{\zeta}(\mathbf{\Sigma}, \Delta')}{\mathcal{I}_{\chi'}} \xrightarrow{\cong} \quotient{\mathcal{Z}_{\zeta}(\mathbf{\Sigma}, \Delta)}{\mathcal{I}_{\chi}}.$$

\end{notations}

 \par We now state a generalization of the Fock-Goncharov decomposition for flip operators at roots of unity. For this, we need the following generalization of the q-exponential $\exp_q^{<N}$ polynomial, which appeared in the work of Fadeev-Kashaev \cite{FadeevKashaevQDilog} and Baseilhac-Benedetti \cite{BaseilhacBenedettiLocRep}:

\begin{definition}[Cyclic quantum dilogarithm]
For $w\in \mathbb{C}^*$, we set
$$ \Phi_w(x):= \sum_{n=0}^{N-1} \frac{(-wx)^nq^{n(n-1)/2}}{(q^nw-q^{-n})\ldots (qw-q^{-1})} \in \mathbb{C}[x].$$
\end{definition}

\begin{lemma}\label{lemma_special_poly}
Consider $\mathcal{A}$ a $\mathbb{C}$-algebra and $x\in \mathcal{A}$ such that $x^N\in \mathbb{C}\setminus \{-1\}$. Choose $w\in \mathbb{C}^*$ such that $w^N(1+x^N) = 1$. Then one has:
\begin{enumerate}
\item The summand $e_n(x):=\frac{(-wx)^nq^{n(n-1)/2}}{(q^nw-q^{-n})\ldots (qw-q^{-1})}$ only depends on $n$ modulo $N$, therefore
$$ \Phi_w(x):= \sum_{n\in \mathbb{Z}/N\mathbb{Z}}e_n(x) \in \mathbb{C}[x].$$
\item One has $\Phi_w(q^{-2}x)=w(1+q^{-1}x)\Phi_w (x)$ and $\Phi_w(q^2x) = w^{-1}(1+qx)^{-1}\Phi_w(x)$.
\item If $y \in \mathcal{A}$ is such that $xy=q^2yx$, then $\mathrm{Ad}(\Phi_w(x)) y = w^{-1} y(1+qx)^{-1}$. 
\\ If $y \in \mathcal{A}$ is such that $xy=q^{-2}yx$, then $\mathrm{Ad}(\Phi_w(x)) y = w y(1+q^{-1}x)$.
\item If $w=1$, $\Phi_{w=1}(x)=\exp_q^{<N}\left(\frac{-x}{q-q^{-1}}\right)$.
\end{enumerate}
\end{lemma}

\begin{proof}
First, from the expression $e_n(x)= \frac{(-q^{-1}x)^n}{(1-w^{-1}q^{-2n}) \ldots (1-w^{-1}q^{-2})}$, we see that 
$$e_{n+N}(x)=\left( \frac{-x^N}{(1-w^{-1}q^{-2(n+N)})\ldots (1-w^{-1}q^{-2(n+1)})}\right) e_n(x) = \left( \frac{-x^N}{1-w^{-N}}\right) e_n(x) = e_n(x).$$
So the first assertion is proved. Then, using the equality $e_{n+1}(x)= \frac{-q^{-1}x}{1-w^{-1}q^{-2(n+1)}} e_n(x)$, which holds for any $n\in \mathbb{Z}/N\mathbb{Z}$, one finds that 
$$ e_{n+1}(q^{-2}x) = we_{n+1}(x) +w q^{-1}x e_n(x), $$
and the equality $\Phi_w(q^{-2}x)=w(1+q^{-1}x)\Phi_w (x)$ follows by summing over $n\in \mathbb{Z}/N\mathbb{Z}$. The equality $\Phi_w(q^2x) = w^{-1}(1+qx)^{-1}\Phi_w(x)$ is deduced by replacing $x$ by $q^2x$. The third assertion is an immediate consequence of the second one and the last assertion is immediate.

\end{proof}

\par Recall the two triangulations $\Delta$ and $\Delta'$ of the square $Q$ depicted in Figure \ref{fig_square}.
We  fix $\chi$ and $\chi'$ two compatible characters on the centers of $\mathcal{Z}_{\zeta}(Q,\Delta)$ and  $\mathcal{Z}_{\zeta}(Q,\Delta')$ and suppose that $\chi(X_5^N)\neq -1$. Write $x_i:= \chi(X_i^N), x_i':= \chi({X'}_i^N)$ and fix $w_0\in \mathbb{C}^*$ such that $w_0^N(1+x_5)=1$.

\begin{definition}\label{def_nu}
 The isomorphism 
  $$\nu^Q_{w_0} : \quotient{\mathcal{Z}_{\zeta}(Q, \Delta')}{I_{\chi'}} \xrightarrow{\cong} \quotient{\mathcal{Z}_{\zeta}(Q, \Delta)}{I_{\chi}}$$ is defined by: 
 \begin{align*}
 & \nu^Q_{w_0} \left(\Tr_{\zeta}^{\Delta'} (\beta^{(i)}_{++})\right)= \Tr_{\zeta}^{\Delta} (\beta^{(i)}_{++}) \mbox{, for }i\in \{1,2,3,4\}; \\
 & \nu^Q_{w_0} (X'_1) = w_0^{-1}X_1; \quad \nu^Q_{w_0}(X'_3) = w_0^{-1}X_3;  \\
 & \nu^Q_{w_0} (X'_2)= w_0 [X_2X_5]; \quad \nu^Q_{w_0}(X'_4)= w_0 [X_4X_5]; \\
 & \nu^Q_{w_0} (X'_5)= X_5^{-1}.
 \end{align*}
\end{definition}

\begin{lemma}\label{lemma_decomp_flip}
The isomorphism $\Phi_{(\Delta,\chi),  (\Delta', \chi')}^{Q} : \quotient{\widehat{\mathcal{Z}}_{\zeta}(Q, \Delta')}{I_{\chi'}} \xrightarrow{\cong} \quotient{\widehat{\mathcal{Z}}_{\zeta}(Q, \Delta)}{I_{\chi}}$ decomposes as:
$$ \Phi_{(\Delta,\chi),  (\Delta', \chi')}^{Q} = \nu_{w_0} \circ \mathrm{Ad}(\Phi_{w_0}({X'}_5^{-1})).$$
\end{lemma}

\begin{proof}
Like in the proof of Lemma \ref{lemma_FockGoncharov}, we show the equality $\Phi_{(\Delta,\chi),  (\Delta', \chi')}^{Q} (x)= \nu_{w_0} \circ \mathrm{Ad}(\Phi_{w_0}({X'}_5^{-1}))(x)$ for $x$ either a $X'_i, i=1, \ldots, 5$, or a $\Tr^{\Delta'}_{\zeta}(\beta^{(i)}_{++}), i=1, \ldots, 4$. When $x=\Tr^{\Delta'}_{\zeta}(\beta^{(i)}_{++})$, since $x$ commutes with $X_5'$, it is invariant under $\mathrm{Ad}(\Phi_{w_0}({X'}_5^{-1}))$ and one has $\Phi_{(\Delta,\chi),  (\Delta', \chi')}^{Q} (x)= \Tr_{\zeta}^{Q}(\beta^{(i)}_{++}) = \nu_{w_0}^Q (x)$ by definition. When $x=X_i$, the equality $\Phi_{(\Delta,\chi),  (\Delta', \chi')}^{Q} (x)= \nu_{w_0} \circ \mathrm{Ad}(\Phi_{w_0}({X'}_5^{-1}))(x)$ follows from the third assertion of Lemma \ref{lemma_special_poly}.

\end{proof}

\begin{remark}
Consider the classical shape parameter $\mathbf{z}=(z,z',z''):= (-x_5, (1+x_5)^{-1}, 1+x_5^{-1})$. The existence of a character $\chi$ compatible with $\chi'$ is equivalent to the assumption that $\mathbf{z}$ is indeed a shape parameter, \textit{i.e.} that $z\in \mathbb{C} \setminus \{0,1\}$. By Definition \ref{def_square_changecoordinates}, the character $\chi$ is then related  to $\chi'$ and $\mathbf{z}$ via the expressions: 
$$ x_5=-z, \quad x_1=z'x_1',\quad x_2=z''x_2',\quad x_3=z'x_3',\quad x_4=z''x_4'.$$
Figure \ref{fig_square_shape_parameters} illustrates how this formula can be interpreted by the operation of gluing a branched tetrahedron with shape parameters $\mathbf{z}$. The presence of a minus sign in $x_5=-z$ is purely conventional: the $z$, attached to the tetrahedron,  represents a bicross product whereas the shear bend coordinate $x_i$, parametrizing a pleated surface, is minus a bicross product.

 \begin{figure}[!h] 
\centerline{\includegraphics[width=8cm]{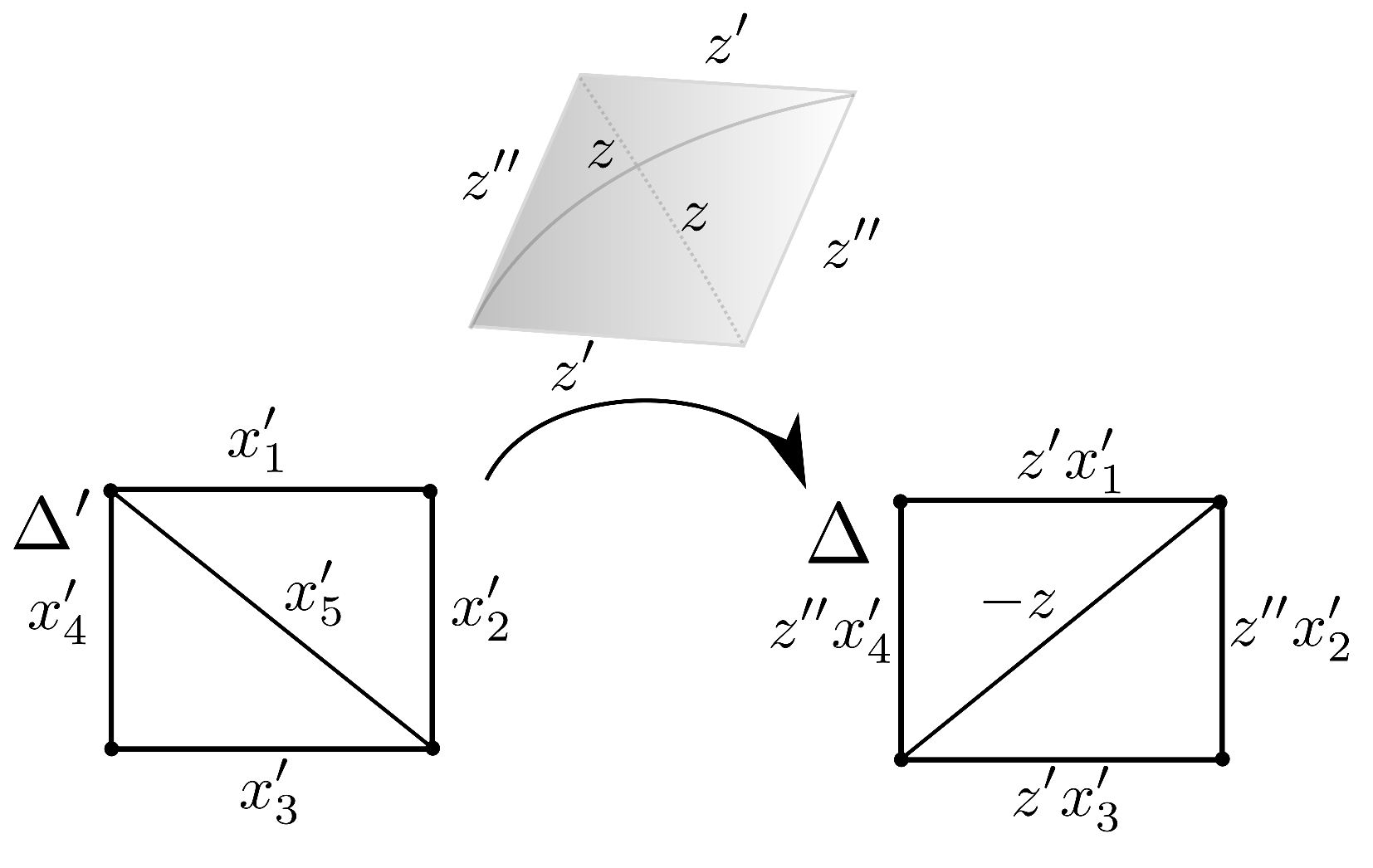} }
\caption{The figure illustrates how the flip operator $\Phi_{\Delta, \Delta'}$ at $A^{1/2}=+1$ can be interpreted as the action of gluing a tetrahedron with shape parameters.} 
\label{fig_square_shape_parameters}
\end{figure} 

\end{remark}

 \subsection{Kashaev's braidings}
 
 \subsubsection{Computing mapping class intertwiners using quantum hyperbolic geometry and Kashaev's braidings}
 
 Let $(\mathbf{\Sigma}, \Delta)$ be a triangulated marked surface, $r, r' : \overline{\mathcal{S}}_{\zeta}(\mathbf{\Sigma})\xrightarrow{\Tr_{\zeta}^{\Delta}} \mathcal{Z}_{\zeta}(\mathbf{\Sigma}, \Delta) \xrightarrow{\overline{r}, \overline{r}'} \End(V)$ two QT representations (so $\overline{r}$ and $\overline{r}'$  are irreducible) with full shadows $\widehat{x}, \widehat{x}' \in \widehat{X}(\mathbf{\Sigma})$ and $\phi \in \Mod(\Sigma)$ a mapping class such that $\phi\cdot \widehat{x} =\widehat{x}'$. If $\widehat{x}' \in \mathcal{AL}(\mathbf{\Sigma})$, there exists an intertwiner (unique up to non zero scalar) 
 $L(\phi) \in \End(V)$ such that:
 $$ r'(x) = L(\phi) r(\phi^{-1}(x)) L(\phi)^{-1}, \quad \mbox{ for all } x \in \overline{\mathcal{S}}_{\zeta}(\mathbf{\Sigma}).$$
 The study we made in the last subsection permits to compute $L(\phi)$ explicitly as follows. By Proposition \ref{prop_change_coordinates} $(3)$, it suffices to find $L(\phi) \in \End(V)$ such that: 
 $$ \overline{r}'(y) = L(\phi) \overline{r}( \phi^*\Phi^{\mathbf{\Sigma}}_{\phi(\Delta), \Delta}(y))L(\phi)^{-1}, \quad \mbox{ for all }y \in \mathcal{Z}_{\zeta}(\mathbf{\Sigma}, \Delta).$$
 So we need to find an intertwiner for the automorphism $F= \phi^* \circ \Phi^{\mathbf{\Sigma}}_{\phi(\Delta), \Delta}$.
 Let $\Delta=\Delta^{(0)}, \Delta^{(1)}, \ldots, \Delta^{(n)}=\phi(\Delta)$ be a sequence of triangulations such that $\Delta^{(i+1)}$ and $\Delta^{(i)}$ differ by a flip or a reindexing so we have a decomposition $F=\phi^*\circ \Phi^{\mathbf{\Sigma}}_{\phi(\Delta), \Delta^{(n-1)}} \circ \ldots \circ \Phi^{\mathbf{\Sigma}}_{\Delta^{(1)}, \Delta}$ by Proposition \ref{prop_change_coordinates} $(2)$. When $\Delta^{(i+1)}, \Delta^{(i)}$ differ by a flip,  Lemma \ref{lemma_decomp_flip} implies that $\Phi^{\mathbf{\Sigma}}_{\Delta^{(i+1)}, \Delta^{(i)}}$ is almost inner (up to the $\nu$ isomorphism) and it permits to compute an intertwiner for $F$.
 \par Now consider the particular case where $\mathbf{\Sigma}=\mathbb{D}_2$ and $\phi= \tau$ is the mapping class which exchanges the two inner punctures of $\mathbb{D}_2$ in the clockwise order.
 
 \begin{definition}[Kashaev's braidings]
 Let $\overline{r}, \overline{r}': \mathcal{Z}_{\zeta}(\mathbb{D}_2) \to \End(V)$ be two irreducible representations with shadows $\widehat{y}, \widehat{y}' \in \widehat{Y}(\mathbb{D}_2, \Delta_2)$ such that $\tau\cdot \widehat{y}= \widehat{y}'$. A \textbf{Kashaev braiding} is an endomorphism $R \in \End(V)$ such that 
 $$  \overline{r}'(y) = R \overline{r}( \phi^*\Phi^{\mathbf{\Sigma}}_{\tau(\Delta_2), \Delta_2}(y))R^{-1}, \quad \mbox{ for all }y \in \mathcal{Z}_{\zeta}(\mathbb{D}_2, \Delta_2).$$
 $R$ is \textbf{normalized} if $\det(R)=1$. 
 \end{definition}
 
 These braidings first appeared in \cite{KashaevLinkInv}. When the images of  $\widehat{y}, \widehat{y}'$ by $\widehat{p} : \widehat{Y}(\mathbb{D}_2,\Delta_2) \to \widehat{X}(\mathbb{D}_2)$ are in the Azumaya locus $\mathcal{AL}(\mathbb{D}_2)$, then Proposition \ref{prop_change_coordinates} $(3)$ implies that the Kashaev braiding $R$ is a skein braiding, in the sense of Definition \ref{def_braid_rep}. Conversely, a skein braiding between Azumaya representations of $\overline{\mathcal{S}}_{\zeta}(\mathbb{D}_2)$ are Kashaev braidings if and only if their full shadows lies in the image of $\widehat{p}$. So Kashaev and skein braidings are slightly different notions though they coincide generically.
 
 \subsubsection{Kashaev's braidings when $q$ is generic}
 In this subsection, we consider the case where $q=A^2\in k^*$ is not a root of unity. This study is of independent interest and serves as a warm-up for the (harder) roots of unity case in the next subsection. The study in this subsection is closely related to the work of Hikami and Inoue in \cite{Hikami_Inoue_braiding} though the formulas we will obtain are simpler and explicitly related to quantum groups $R$ matrices thanks to the quantum trace and to Theorem \ref{theorem_skein_QG}. We want to find an intertwiner for the (image through representations of the) automorphism $\mathscr{R}:= \tau^* \circ \Phi^{\mathbb{D}_2}_{\tau(\Delta), \Delta}$.
 Consider the topological triangulations $\Delta^{(2)}, \ldots, \Delta^{(5)}$ of $\mathbb{D}_2$ drawn in Figure \ref{fig_flip_triangulations} and write $\Delta^{(1)}:= \Delta_2$.

  \begin{figure}[!h] 
\centerline{\includegraphics[width=9cm]{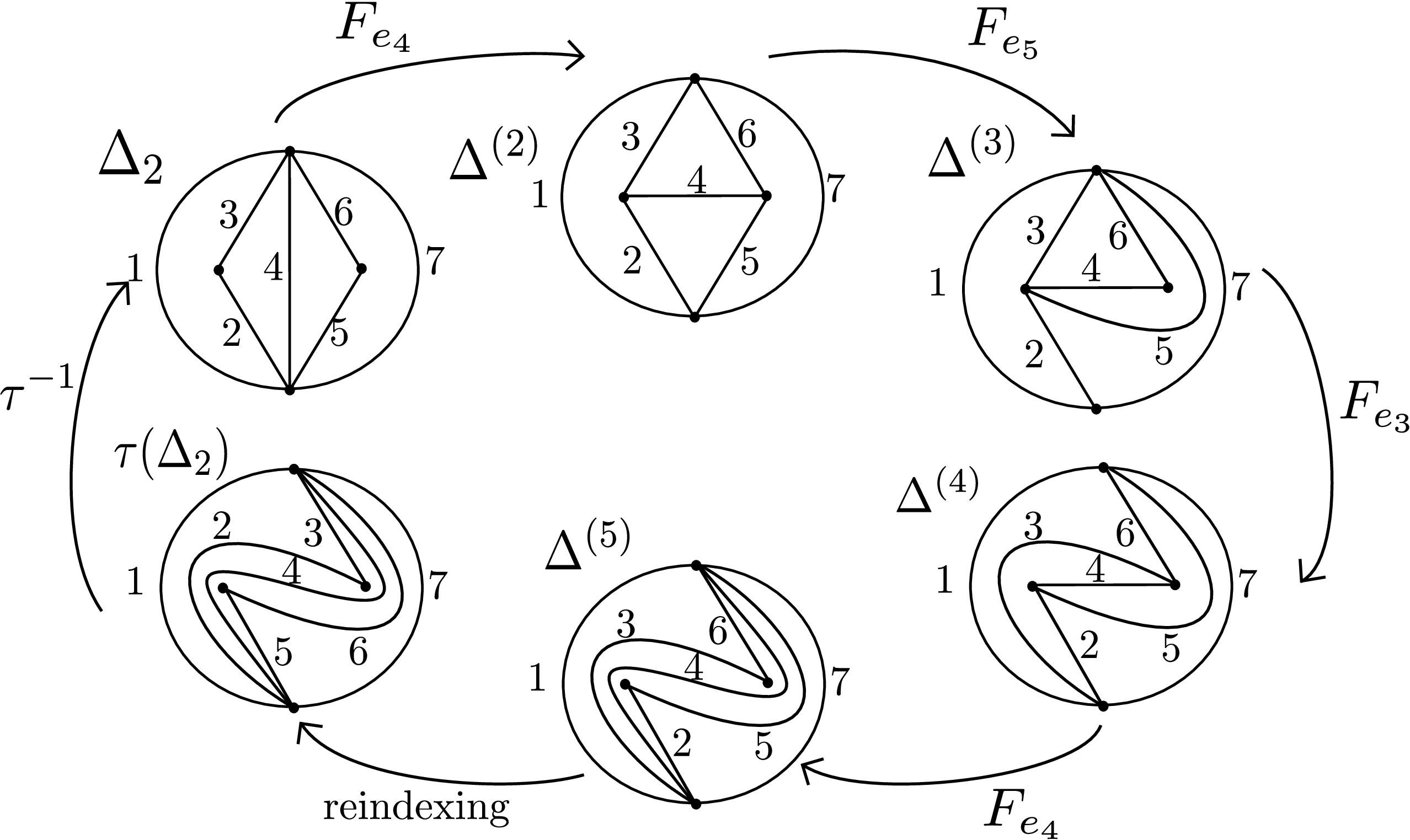} }
\caption{Six triangulations of $\mathbb{D}_2$.} 
\label{fig_flip_triangulations} 
\end{figure} 

By definition, the automorphism $\mathscr{R}$ decomposes as:
$$\mathscr{R}= \tau^* \circ \Phi^{\mathbb{D}_2}_{\tau(\Delta_2), \Delta^{(5)}} \circ  \Phi^{\mathbb{D}_2}_{\Delta^{(5)}, \Delta^{(4)}} \circ  \Phi^{\mathbb{D}_2}_{\Delta^{(4)}, \Delta^{(3)}} \circ  \Phi^{\mathbb{D}_2}_{\Delta^{(3)}, \Delta^{(2)}}  \circ  \Phi^{\mathbb{D}_2}_{\Delta^{(2)}, \Delta^{(1)}}.$$
By Lemma \ref{lemma_FockGoncharov}, for $i\in \{1, \ldots, 4\}$, the flip automorphism  $ \Phi^{\mathbb{D}_2}_{\Delta^{(i+1)}, \Delta^{(i)}}$ decomposes as  $ \Phi^{\mathbb{D}_2}_{\Delta^{(i+1)}, \Delta^{(i)}}= \nu_i \circ \mathrm{Ad}(s_q(X_{n_i}^{-1}))$, where $\nu_i:= \nu^{\mathbb{D}_2}_{\Delta^{(i+1)}, \Delta^{(i)}}$ and $(n_1,n_2,n_3, n_4)= (4, 5,3,4)$. Set
$$ \mathcal{K}:= \tau^* \circ  \Phi^{\mathbb{D}_2}_{\tau(\Delta_2), \Delta^{(5)}} \circ \nu_4 \circ \nu_3 \circ \nu_2 \circ \nu_1 \in \mathrm{Aut}(\mathcal{Z}_{A}(\mathbb{D}_2, \Delta_2)).$$
One thus has the expression:
\begin{equation*}
\mathscr{R}= \mathcal{K} \circ \mathrm{Ad}(s_q\left( (\nu_3\circ \nu_2\circ\nu_1)^{-1}(X_4^{-1})\right)) \circ  \mathrm{Ad}(s_q\left(  (\nu_2\circ\nu_1)^{-1}(X_3^{-1})\right)) 
 \circ \mathrm{Ad}(s_q\left( \nu_1^{-1}(X_5^{-1})\right)) \circ  \mathrm{Ad}(s_q\left( X_4^{-1}\right)).\end{equation*}
By a straightforward computation using Definition \ref{def_nu} (see Table $1$ below), one has:
\begin{equation}\label{eq_Rexpression1}
\mathscr{R}= \mathcal{K} \circ \mathrm{Ad}(s_q\left( [X_3X_4X_5]^{-1} \right)) \circ \mathrm{Ad}(s_q\left( [X_3X_4]^{-1} \right))\circ \mathrm{Ad}(s_q\left( [X_4X_5]^{-1} \right))
\circ \mathrm{Ad}(s_q\left( X_4^{-1} \right)).
\end{equation}

\begin{definition}
The automorphism $\mathrm{Ad}(q^{-\frac{H\otimes G}{2}}) \in \mathrm{Aut}(\mathcal{Z}_{A}(\mathbb{D}_1)^{\otimes 2})$ is defined by the formula:
$$ \mathrm{Ad}(q^{-\frac{H\otimes G}{2}}) \left( Z^{\mathbf{k}} \otimes Z^{\mathbf{k}'} \right):= (K^{1/2})^{\mathbf{k}'(e_1) - \mathbf{k}'(e_4)}Z^{\mathbf{k}} \otimes Z^{\mathbf{k}'} (L^{1/2})^{\mathbf{k}(e_4)- \mathbf{k}(e_1)}.$$
We still denote by the same symbol $\mathrm{Ad}(q^{-\frac{H\otimes G}{2}})$ its restriction to the subalgebras $\mathcal{Z}_{A}(\mathbb{D}_2)$ and $(\Uq)^{\otimes 2}$.
\end{definition}
That $\mathrm{Ad}(q^{-\frac{H\otimes G}{2}})$ is well defined and restricts to automorphisms of $\mathcal{Z}_{A}(\mathbb{D}_2)$ and $(\Uq)^{\otimes 2}$ follows from straightforward computations. 
When $q$ is generic, the simple $\Uq$ modules $S_{\mu, \varepsilon, n}$ of Definition \ref{def_QGRep} are well defined for all $n\geq 0$.
The automorphism $\mathrm{Ad}(q^{-\frac{H\otimes G}{2}})$ is not inner, so the notation is abusive. However, it is justified by the following:

\begin{lemma}\label{lemma_qq}
Let $r_1 : \Uq \rightarrow \End(V_1)$ and $r_2 : \Uq\rightarrow \End(V_2)$ be two  representations of the form $S_{\mu, \varepsilon, n}$ and consider an operator 
$(r_1\otimes r_2) ( q^{-\frac{H\otimes G}{2}})\in \End(V_1\otimes V_2)$ as in Definition \ref{def_Drinfeld_braiding}. Then the following diagram commutes:
$$ 
\begin{tikzcd}
(\Uq)^{\otimes 2} \arrow[rrr, "\cong"', "\mathrm{Ad}(q^{-\frac{H\otimes G}{2}})"] \arrow[d, "r_1 \otimes r_2"] &{}&{}&
(\Uq)^{\otimes 2}  \arrow[d, "r_1 \otimes r_2"]  \\
\End(V_1\otimes V_2) \arrow[rrr, "\cong"', "\mathrm{Ad}\left((r_1\otimes r_2) \left( q^{-\frac{H\otimes G}{2}}\right)\right)"] &{}&{}&
\End(V_1\otimes V_2)
\end{tikzcd}
$$
\end{lemma}

\begin{proof}
The restriction of $\mathrm{Ad}(q^{-\frac{H\otimes G}{2}})$ to $(\Uq)^{\otimes 2}$ is characterized by the following formulas:
\begin{eqnarray*}
\mathrm{Ad}(q^{-\frac{H\otimes G}{2}})(x\otimes y) = x\otimes y \quad \mbox{, for all }x,y\in \{K^{\pm 1/2}, L^{\pm 1/2}\}; \\
\mathrm{Ad}(q^{-\frac{H\otimes G}{2}})(E\otimes 1) =E\otimes L^{-1}; \quad \mathrm{Ad}(q^{-\frac{H\otimes G}{2}})(F \otimes 1)= F\otimes L; \\
\mathrm{Ad}(q^{-\frac{H\otimes G}{2}})(1\otimes E)= K\otimes E; \quad \mathrm{Ad}(q^{-\frac{H\otimes G}{2}})(1\otimes F) = K^{-1}\otimes F.
\end{eqnarray*}
The proof then  follows from these formulas and the fact that the representations $S_{\mu, \varepsilon, n}$ have weight bases $(e_i)_i$ such that either $e_i=F^i e_0$.
\end{proof}

\begin{lemma}\label{lemma_kappa} The following diagram commutes:
$$
\begin{tikzcd}
\mathcal{Z}_{A}(\mathbb{D}_2, \Delta_2) \arrow[rr, "\cong"', "\mathcal{K}"] \arrow[d, hook] &{}&
\mathcal{Z}_{A}(\mathbb{D}_2, \Delta_2) \arrow[d, hook] \\
\mathcal{Z}_{A}(\mathbb{D}_1)^{\otimes 2} \arrow[rr, "\cong"', "\sigma \circ \mathrm{Ad}(q^{-\frac{H\otimes G}{2}})"]   &{}&
\mathcal{Z}_{A}(\mathbb{D}_1)^{\otimes 2}   
\end{tikzcd}
$$
\end{lemma}

\begin{proof} By Remark \ref{remark_grading}, the algebra $\mathcal{Z}_{A}(\mathbb{D}_2, \Delta_2)$ is generated by the elements $\Delta(K^{1/2})$, $\Delta(L^{1/2})$,  $H_{p_1}$, $H_{p_2}$ and the $X_i$ for $i\in \{1,\ldots, 7\}$, we must thus prove that the image of these generators by $\mathcal{K}$ and $\sigma \circ \mathrm{Ad}(q^{-\frac{H\otimes G}{2}})$ are equal. For the generators $\Delta(K^{1/2})$ and $\Delta(L^{1/2})$, this fact follows from Equation \ref{eq_Rexpression1} and Equation \ref{eq_KR5} of Proposition \ref{prop_KR}. For the other generators, this follows from a straightforward computation summarized in Table $1$.
\begin{table}[h!]
  \begin{center}
    \label{table}
    \caption{Computation of $\mathcal{K}$}
\begin{tabular}{|c|c|c|c|c|c|}
\hline
$\mathcal{Z}_{A}(\mathbb{D}_2, \Delta^{(1)})$ & $\mathcal{Z}_{A}(\mathbb{D}_2, \Delta^{(2)})$ & $\mathcal{Z}_{A}(\mathbb{D}_2, \Delta^{(3)})$ & $\mathcal{Z}_{A}(\mathbb{D}_2, \Delta^{(4)})$ & $\mathcal{Z}_{A}(\mathbb{D}_2, \Delta^{(5)})$ & $\mathcal{Z}_{A}(\mathbb{D}_2, \Delta^{(1)})$ \\
\hline
$X_1$ & $X_1$ & $X_1$ & $[X_1X_3]$ & $[X_1X_3X_4]$ & $[X_1X_2X_4]$ \\
\hline
$X_2$ & $X_2$ & $X_2$ & $X_2$ & $X_2$ & $X_5$ \\
\hline
$X_3$ & $[X_3X_4]$ & $[X_3X_4X_5]$ & $[X_4X_5]$ & $X_5$ & $X_6$ \\
\hline
$X_4$ & $X_4^{-1}$ & $[X_4X_5]^{-1}$ & $[X_3X_4X_5]^{-1}$ & $[X_3X_4X_5]^{-1}$ & $[X_2X_4X_6]^{-1}$ \\
\hline
$X_5$ & $[X_4X_5]$ & $X_4$ & $[X_3X_4]$ & $X_3$ & $X_2$ \\
\hline
$X_6$ & $X_6$ &$X_6$ &$X_6$ &$X_6$ & $X_3$ \\
\hline
$X_7$ & $X_7$ & $[X_5X_7]$ & $[X_5X_7]$ & $[X_4X_5X_7]$ & $[X_4X_6X_7]$ \\
\hline
$H_{p_1}=[Z_2Z_3]$ & $[Z_2Z_3Z_4]$ & $[Z_2Z_3Z_4Z_5]$ & $[Z_2Z_4Z_5]$ & $[Z_2Z_5]$ & $H_{p_2}$ \\
\hline
$H_{p_2}=[Z_5Z_6]$ & $[Z_4Z_5Z_6]$ & $[Z_4Z_6]$ & $[Z_3Z_4Z_6]$ & $[Z_3Z_6]$ &$H_{p_1}$ \\
\hline

\end{tabular}
  \end{center}
\end{table}
In this table, for $1\leq i \leq 4$, one passes from the $i$-th column to the $i+1$-th column by composing with $\nu_i$ and one passes from the $5$-th column to the last one by composing with $ \tau^* \circ  \Phi^{\mathbb{D}_2}_{\tau(\Delta_2), \Delta^{(5)}}$. Thus we pass from the first column to the last one by composing with $\mathcal{K}$. We conclude by comparing the formulas in the last column with the image through $\sigma \circ \mathrm{Ad}(q^{-\frac{H\otimes G}{2}})$ of the corresponding generators.

\end{proof}

\begin{proposition} Suppose that  $A^{1/2}\in k^*$ is generic (not a root of unity). The automorphism \\ $\mathrm{Ad}\left(\exp_q((q-q^{-1})E\otimes F) \right) \in \mathrm{Aut}(\mathcal{Z}_{A}(\mathbb{D}_2,\Delta_2))$ is well defined. Moreover, one has:
\begin{equation*}
\mathscr{R} = \sigma \circ \mathrm{Ad}(q^{-\frac{H\otimes G}{2}}) \circ \mathrm{Ad}\left(\exp_q((q-q^{-1})E\otimes F) \right).
\end{equation*}
\end{proposition}

\begin{proof}
In $\overline{\mathcal{Z}}_{A}(\mathbb{D}_2, \Delta_2)$, one has the equality (where all products are well defined):

\begin{align*}
& \exp_q\left( (q-q^{-1})E\otimes F\right) = s_q\left( (X_4^{-1} +[X_3X_4]^{-1})\otimes (X_1^{-1}+[X_1X_2]^{-1}) \right) \\
&= s_q\left( [X_3X_4]^{-1}\otimes (X_1^{-1} + [X_1X_2]^{-1}) \right)  s_q\left( X_4^{-1}\otimes (X_1^{-1} + [X_1X_2]^{-1}) \right) \\
&= s_q \left( [X_3X_4]^{-1}\otimes [X_1X_2]^{-1} \right) s_q\left([X_3X_4]^{-1}\otimes X_1^{-1}\right)s_q\left(X_4^{-1}\otimes [X_1X_2]^{-1}\right) s_q\left(X_4^{-1}\otimes X_1^{-1}\right) \\
&= s_q\left( [X_3X_4X_5]^{-1}\right)s_q\left([X_3X_4]^{-1}\right)s_q\left([X_4X_5]^{-1}\right)s_q\left(X_4^{-1}\right)
\end{align*}
Here we used formula $(1)(a)$ of Lemma \ref{lemma_expq} three times. We conclude using Lemma \ref{lemma_kappa} and Equation \ref{eq_Rexpression1}.

\end{proof}

\begin{proposition}\label{prop_calcul_R}
Suppose that $A^{1/2}\in k^*$ is generic. The automorphism $\mathscr{R}$ is characterized by the following formulas:
\begin{align*}
&\mathscr{R}(\Delta(K^{1/2}))= \Delta(K^{1/2}); \mathscr{R}(\Delta(L^{1/2})) = \Delta(L^{1/2}); \quad\mathscr{R}(H_{p_1})=H_{p_2};  \mathscr{R}(H_{p_2})=H_{p_1}; \\
&\mathscr{R}(X_1) = \left(X_1^{-1}+  [X_1X_2]^{-1} +  [X_1X_2X_4]^{-1}  \right)^{-1}; \\
&\mathscr{R}(X_2) = X_4+X_5+[X_2X_4X_5] + [X_4X_5X_6]+[X_2X_4X_5X_6]; \\
&\mathscr{R}(X_3) = \left(X_6^{-1} +X_4+[X_2X_4] + [X_6^{-1}X_4] + [X_2X_4X_6^{-1}]\right)^{-1}; \\
&\mathscr{R}(X_5)= \left(X_2^{-1} + X_4+  [X_4X_6] + [X_2^{-1}X_4] + [X_2^{-1}X_4X_6] \right)^{-1}; \\
&\mathscr{R}(X_6) =  X_3+X_4+[X_3X_4X_6]+[X_2X_3X_4]+[X_2X_3X_4X_6]; \\
&\mathscr{R}(X_7) = \left( X_7^{-1} + [X_6X_7]^{-1}+[X_4X_6X_7]^{-1} \right)^{-1}.
\end{align*}
\end{proposition}

\begin{proof} The formulas for $\mathscr{R}(\Delta(K^{1/2}))$ and  $\mathscr{R}(\Delta(L^{1/2}))$ follow from Proposition \ref{prop_KR}. The formulas for  $\mathscr{R}(H_{p_1})$ and   $\mathscr{R}(H_{p_2})$ follow from the fact that both $H_{p_1}$ and $H_{p_2}$ commute with $E\otimes F$ (since those are central elements) together with the computation of $\mathcal{K}(H_{p_i})$ made in Table $1$. For simplicity, write $\mathcal{AD}:= \mathrm{Ad}\left(\exp_q\left( (q-q^{-1})E\otimes F\right)\right)$. Let us first compute $\mathscr{R}(X_1)$. Using the formula $(1)(a)$ in Lemma \ref{lemma_expq}, one has $\exp_q\left( (q-q^{-1}) E\otimes F\right) = s_q\left( (q-q^{-1})[X_3X_4]^{-1}\otimes F\right) s_q\left( (q-q^{-1})X_4^{-1}\otimes F\right)$. Since $X_1=X_1\otimes 1$ commutes with $X_4^{-1}\otimes F$ and $q^2$-commutes with $[X_3X_4]^{-1}\otimes F$, one has:
\begin{eqnarray*}
\mathcal{AD}(X_1) &=& X_1 s_q\left( q^2 \left( (q-q^{-1})[X_3X_4]^{-1}\otimes F \right) \right) s_q \left( (q-q^{-1})[X_3X_4]^{-1}\otimes F \right)^{-1} \\
&=& X_1\left( 1+q(q-q^{-1})[X_3X_4]^{-1}\otimes F\right)^{-1}
\end{eqnarray*}
where we used formula $(1)(b)$ of Lemma \ref{lemma_expq} to pass from the first to the second line. We obtain the expression for $\mathscr{R}(X_1)$ by composing the above formula with the automorphism $\mathcal{K}$ using Table $1$. Next, to compute $\mathscr{R}(X_2)$, note that $X_2=X_2\otimes 1$ $q^{-2}$-commutes with $E\otimes F$, thus one has:
\begin{eqnarray*}
\mathcal{AD}(X_2) &=& X_2 s_q\left( q^{-2} \left(-(q-q^{-1})^2 E\otimes F \right) \right) s_q \left(-(q-q^{-1})^2 E\otimes F \right)^{-1} \\
&=& X_2\left( 1-q^{-1}(q-q^{-1})^2 E\otimes  F\right)
\end{eqnarray*}
where we used formula $(1)(b)$ of Lemma \ref{lemma_expq}. We obtain the expression for $\mathscr{R}(X_2)$ by composing the above formula with the automorphism $\mathcal{K}$ using Table $1$. For $X_3$, note that $X_3=X_3\otimes 1$ $q^2$-commutes with $E\otimes F$, thus one has:
\begin{eqnarray*}
\mathcal{AD}(X_3) &=& X_3 s_q\left( q^{2} \left(-(q-q^{-1})^2 E\otimes F \right) \right) s_q \left(-(q-q^{-1})^2 E\otimes F \right)^{-1} \\
&=& X_3\left( 1-q(q-q^{-1})^2 E\otimes  F\right)^{-1}
\end{eqnarray*}
Again, we conclude by composing with $\mathcal{K}$ using Table $1$. The formulas for $\mathscr{R}(X_5)$, $\mathscr{R}(X_6)$, $\mathscr{R}(X_7)$ are obtained by replacing in the above computations the elements $X_1, X_2, X_3$, $X_4, X_5, X_6, X_7$ by the elements $X_7, X_6, X_5$, $X_4, X_3, X_2, X_1$ respectively.

\end{proof}

\begin{remark}
\begin{itemize}
\item In \cite{Hikami_Inoue_braiding}, the authors computed explicitly the images $\mathscr{R}^{-1}(X_i)$ directly from the formulas of Definition \ref{def_square_changecoordinates}. The formulas they obtained are much more complicated than the formulas we obtained in Proposition \ref{prop_calcul_R} using the Fock-Goncharov decomposition.
\item The reader might ask why we did not provide a formula for $\mathscr{R}(X_4)$. First because we do not need it to characterize $\mathscr{R}$: indeed it can be recovered from the above formulas using the equality $X_4= [\Delta(K^{-1})X_1^{-1}X_3^{-1}X_6^{-1}X_7^{-1}]$. The actual reason is that we will argue in the proof of Theorem \ref{theorem_Drinfeld} that the above computations for $\mathscr{R}(X_i)$ hold when $A$ is a root of unity of odd order $N>1$ with the additional assumption that $E^N\otimes1=1\otimes F^N=0$ (through the image of appropriate representations). The author did not find a straightforward computation for $\mathscr{R}(X_4)$ that generalizes to this case.
\end{itemize}
\end{remark}

 \subsubsection{Drinfeld braidings at roots of unity are Kashaev braidings}\label{sec_DrinfeldKashaevBraidings}

\begin{theorem}\label{theorem_Drinfeld}
Let $V_1, V_2$ be two standard $\Uq$-modules  such that $r_1(E^N)=0$ and $r_2(F^N)=0$ which extends to simple $\mathcal{Z}_{\zeta}(\mathbb{D}_1)$ modules. Let 
$$R^D : V_1\otimes V_2 \rightarrow V_2\otimes V_1, \quad R^D= \sigma \circ (r_1\otimes r_2 )( q^{-\frac{H\otimes G}{2}})\circ (r_1\otimes r_2) \left(  \exp_q^{<N}\left( (q-q^{-1})E\otimes F \right) \right)$$ be a  Drinfel'd braiding as in Definition \ref{def_Drinfeld_braiding}. Then $R^D$ is a Kashaev braiding.
\end{theorem}

\begin{remark} In the particular case where $V_1=V_2=V_0:=V({\zeta}^{-1/2}, \zeta^{1/2}, 0,0)$, the associated Drinfel'd braiding $R^D$ is the braiding defining the $N$-th Jones polynomial at $N$-th root of unity. Hence Theorem \ref{theorem_Drinfeld} provides an alternative proof of Murakami and Murakami result in \cite{MM01} which relates this braiding to Kashaev's $R$-matrix of \cite{KashaevLinkInv}. The proof we propose here is more enlightening: Kashaev braidings are compositions of a quadratic part and of four cyclic quantum dilogarithms. As we saw in Lemma \ref{lemma_special_poly}, $\exp_q^{<N}(\frac{x}{q-q^{-1}})= \Phi_{w=1}(x)$ corresponds to a cyclic quantum dilogarithm with quantum shape parameter $w=1$, so corresponding to a degenerate tetrahedron. Using Lemma \ref{lemma_expq}, we will show that the product of four such cyclic quantum dilogarithms is a the cyclic quantum dilogarithm of a sum: this sum is the image by the quantum trace of $E\otimes F$. Hence the term $ \exp_q^{<N}\left( (q-q^{-1})E\otimes F \right) $ is a product of four cyclic quantum dilogarithms associated to degenerate tetrahedra.
\end{remark}

  We first prove that the formulas in Proposition \ref{prop_calcul_R} also hold when $A$ is a root of unity. Consider two commutative unital rings $\mathcal{R}_1$ and $\mathcal{R}_2$ with invertible elements $A^{1/2}_1\in \mathcal{R}_1^{\times}$, $A^{1/2}_2\in \mathcal{R}_2^{\times}$, such that both $A_1^2-A_1^{-2}$ and $A_2^2-A_2^{-2}$ are invertible and suppose that we have a surjective ring morphism $\mu : \mathcal{R}_1 \rightarrow \mathcal{R}_2$ such that $\mu(A^{1/2}_1)=A^{1/2}_2$. By extension of scalars, the morphism $\mu$ induces a structure of $\mathcal{R}_2$-algebra on $\mathcal{R}_2\otimes_{\mathcal{R}_1} \widehat{\mathcal{Z}}_{A_1}(\mathbf{\Sigma}, \Delta)$ and it follows from Definition \ref{def_CF} that this algebra is isomorphic to $\widehat{\mathcal{Z}}_{A_2}(\mathbf{\Sigma}, \Delta)$. A straightforward analysis of Definitions \ref{def_reindexing_map}, \ref{def_square_changecoordinates}, \ref{def_flip_changecoordinates}, shows that the following diagram commutes:
\begin{equation}\label{eq_diaaa}
\begin{tikzcd}
\mathcal{R}_2\otimes_{\mathcal{R}_1} \widehat{\mathcal{Z}}_{A_1}(\mathbb{D}_2, \Delta_2) 
\arrow[r, "\cong"', "\mathcal{R}_2\otimes_{\mathcal{R}_1}\mathscr{R}"] \arrow[d, "\cong"] &
\mathcal{R}_2\otimes_{\mathcal{R}_1} \widehat{\mathcal{Z}}_{A_1}(\mathbb{D}_2, \Delta_2) 
\arrow[d, "\cong"] \\
\widehat{\mathcal{Z}}_{A_2}(\mathbb{D}_2, \Delta_2)
\arrow[r, "\cong"', "\mathscr{R}"] &
\widehat{\mathcal{Z}}_{A_2}(\mathbb{D}_2, \Delta_2)
\end{tikzcd}\end{equation}

\begin{lemma}\label{lemma_rootunity}
The formulas of Proposition \ref{prop_calcul_R} hold when $A^{1/2}=\zeta^{1/2} \in \mathcal{R}^{\times}$ is a root of unity with $\zeta^2-\zeta^{-2}$ invertible.
\end{lemma}

\begin{proof}
Consider the ring $\mathcal{R}_1$ obtained by localizing the ring $\mathbb{Z}[A^{\pm 1/2}]$ with respect to $A^2-A^{-2}$ and denote by $\mu : \mathcal{R}_1 \rightarrow \mathcal{R}$ the unique surjective ring morphism sending $A^{1/2}$ to the root of unity $\zeta^{1/2}\in \mathcal{R}$. Since $A\in \mathcal{R}_1$ is generic, one can apply Proposition \ref{prop_calcul_R} to $\widehat{\mathcal{Z}}_{A}(\mathbf{\Sigma}, \Delta)$ to compute $\mathscr{R}$. We conclude using the commutativity of the diagram \eqref{eq_diaaa}.
\end{proof}

\begin{proof}[Proof of Theorem \ref{theorem_Drinfeld}]
Write $\mathscr{A}:= \mathcal{K}^{-1}\circ \mathscr{R}$. By Lemma \ref{lemma_qq}, it sufficient to show that the following diagram commutes:
$$\begin{tikzcd}
\widehat{\mathcal{Z}}_{\zeta}(\mathbb{D}_2, \Delta_2) 
\arrow[rrrr, "\cong"', "\mathscr{A}"] \arrow[d, "r_1\otimes r_2"'] &{}&{}&{}&
\widehat{\mathcal{Z}}_{\zeta}(\mathbb{D}_2, \Delta_2) 
\arrow[d, "r_1\otimes r_2"] \\
\End(V_1\otimes V_2) 
\arrow[rrrr, "\cong"', "\mathrm{Ad}\left( (r_1\otimes r_2)(\exp_q^{<N}((q-q^{-1})E\otimes F))\right)"] &{}&{}&{}&
\End(V_1\otimes V_2) 
\end{tikzcd}$$
Said differently, we need to show that for each generator $X$ of $\widehat{\mathcal{Z}}_{\zeta}(\mathbb{D}_2, \Delta_2) $, one has
\begin{equation*}
(r_1\otimes r_2)\circ \mathscr{A}(X) = \mathrm{Ad}\left( (r_1\otimes r_2)(\exp_q^{<N}((q-q^{-1})E\otimes F))\right) \circ (r_1\otimes r_2)(X).
\end{equation*}
When $X= \Delta(K^{1/2})$ and $X=\Delta(L^{1/2})$ the equality  follows from Proposition \ref{prop_KR}, Lemma \ref{lemma_rootunity} and the fact that both $\Delta(K^{1/2})$ and $\Delta(L^{1/2})$ commute with $E\otimes F$. For $X=H_{p_1}$ and $X=H_{p_2}$, this follows from Table $1$, Lemma \ref{lemma_rootunity} and the fact that $H_{p_i}$ commutes with $E\otimes F$. For $X\in \{X_1, X_2, X_3, X_5, X_6, X_7\}$, the computations made in the proof of Proposition \ref{prop_calcul_R} can be translated word-by-word by replacing $\exp_q$ and $s_q$ by $(r_1\otimes r_2)\circ \exp_q^{<N}$ and $(r_1\otimes r_2)\circ s_q^{<N}$ respectively, to prove the desired equality. Indeed, the hypothesis $r_1(E^N)=r_2(F^N)=0$ implies that we can replace each computation involving formulas $(1)(a)$ and $(1)(b)$ of Lemma \ref{lemma_expq}  by equivalent computation using formulas $(2)(a)$ and $(2)(b)$ respectively. We then conclude using Lemma \ref{lemma_rootunity}.

\end{proof}

 \subsubsection{Closed formulas for Kashaev's braidings when $q$ is a root of unity}\label{sec_KashaevBradingsRootsUnity}.
\\
 \textbf{Quantum shape parameters }
\par
 
 Recall the triangulations $\Delta_2=\Delta^{(1)}, \ldots, \Delta^{(5)}$ of $\mathbb{D}_2$ depicted in Figure \ref{fig_flip_triangulations}.
\begin{definition}
\begin{enumerate}
\item A \textbf{shape parameter} is an element of $$\mathbb{C}^{**}:= \{ \mathbf{z}=(z,z',z'') | zz'z''=-1, zz'-z=-1\}.$$
Note that each projection $\mathbf{z}\mapsto z$, $\mathbf{z}\mapsto z'$ and $\mathbf{z}\mapsto z''$ defines a bijection $\mathbb{C}^{**}\cong \mathbb{C}\setminus\{0,1\}$.
\par A \textbf{q-shape parameter} is an element of 
$$\mathbb{C}_q^{**}:= \{ \mathbf{w}=(w,w',w'') | (w^N,{w'}^N, {w''}^N)\in \mathbb{C}^{**}, ww'w''=-q^{-1}\}.$$
We say that $\mathbf{w}$ is a \textbf{q-flattenting} of $\mathbf{z}=(w^N, {w'}^N, {w''}^N)$. 
\item
We consider two pairs $(V_1,V_2)$ and $(V_3,V_4)$ of $\mathcal{Z}_{\zeta}(\mathbb{D}_1)$ simple modules which are $\tau$-compatible and let $\chi_1$ denote the character over the center of $\mathcal{Z}_{\zeta}(\mathbb{D}_2, \Delta_2)$ defined by $V_1\otimes V_2$ (so the character associated to $V_3\otimes V_4$ is $\tau\cdot \chi_1$). The pairs $(V_1,V_2)$ and $(V_3,V_4)$ (or the character $\chi_1$, or a Kashaev braiding $R:V_1\otimes V_2 \rightarrow V_3\otimes V_4$) are said $4T$-\textbf{triangulable} if there exist some characters $\chi_2, \ldots, \chi_5$ over the centers of $\mathcal{Z}_{\zeta}(\mathbb{D}_2, \Delta^{(2)}), \ldots, \mathcal{Z}_{\zeta}(\mathbb{D}_2, \Delta^{(5)})$ such that $\chi_i$ and $\chi_{i+1}$ are compatible for $i=1, \ldots, 4$. 
 \item For $i=1, \ldots, 7$, write $x_i:= \chi_1(X_i^N)$.The \textbf{classical shape parameters} associated to $\chi_1$ (or to the pair $(V_1,V_2)$ or to a Kashaev braiding $R:V_1\otimes V_2 \rightarrow V_3\otimes V_4$) are the shape parameters $\mathbf{z}_1, \ldots, \mathbf{z}_4 \in \mathbb{C}^{**}$ defined by 
 \begin{align*}
 & z_1 := -\chi_2(X_4^N) = -x_4^{-1} ; \\
 & z_2':= -\chi_3(X_5^N) = -x_5^{-1}(1+x_4)^{-1}; \\
 & z_3':= -\chi_4(X_3^N) = -x_3^{-1}(1+x_4)^{-1}; \\
 & z_4:= -\chi_5(X_4^N) = - x_4(1+x_5(1+x_4))^{-1}(1+x_3(1+x_4))^{-1}.
 \end{align*}
 
 Note that $\chi_1$ is $4T$ triangulable if and only if $z_1,z_2',z_3'$ and $z_4$  are in $\mathbb{C}\setminus \{0, 1\}$, \textit{i.e.} if and only if they indeed define classical shape parameters $\mathbf{z}_i$. In particular, being $4T$ triangulable is a generic condition. Note also that they satisfy the Thurston's equation:
 $$ z_1z_2z_3z_4=1.$$ 
 Figure \ref{fig_decomposition_octahedre} illustrates how the parameters $\tau \cdot \chi_1(X_i^N)$ can be computed using the classical shape parameters and the $x_i$ (compare with the formulas obtained from Proposition \ref{prop_calcul_R}).
 \end{enumerate}
 \end{definition}
 
 \begin{figure}[!h] 
\centerline{\includegraphics[width=\textwidth]{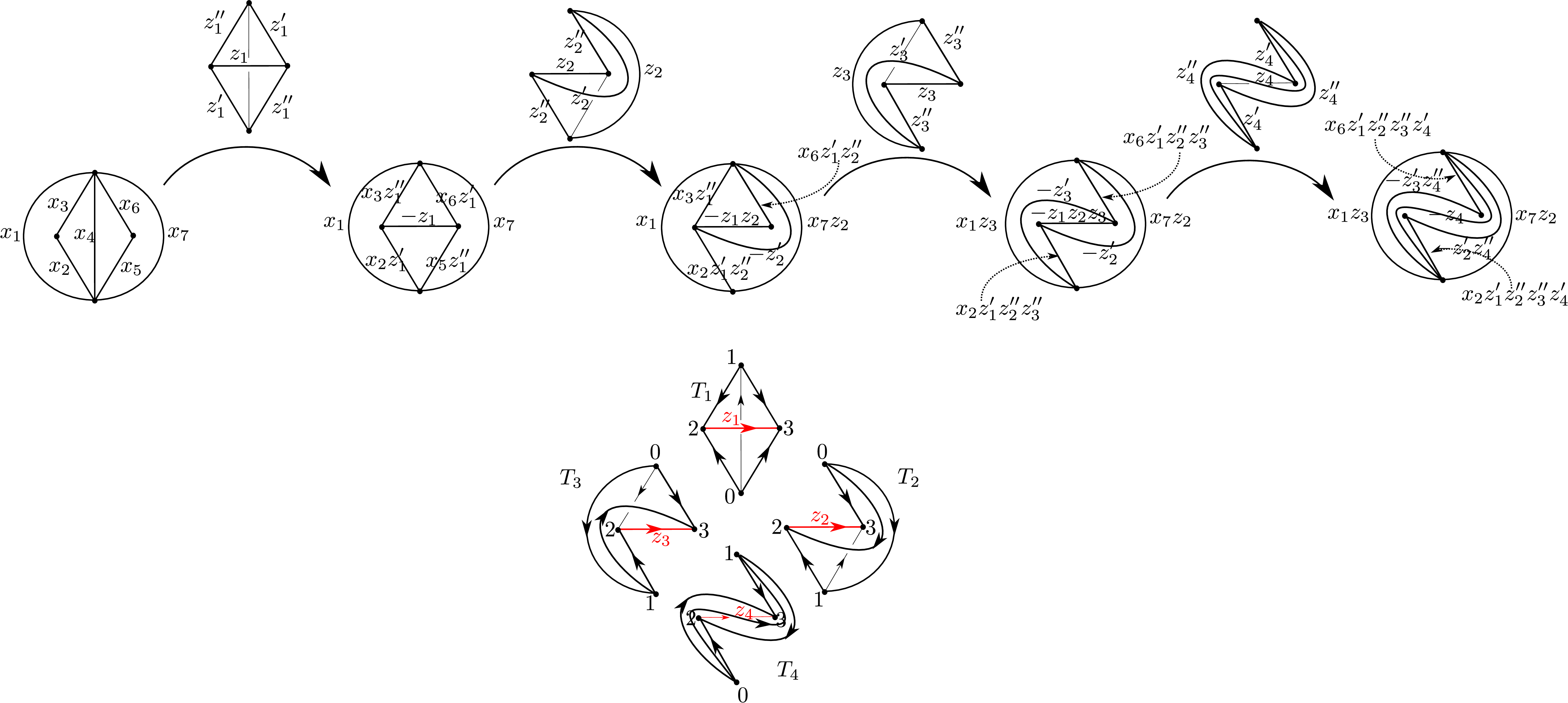} }
\caption{The top figure illustrates how the parameters $\tau \cdot \chi_1(X_i^N)$ can be computed using the classical shape parameters and the $x_i$. The bottom figure shows the four tetrahedra associated to a crossing and their branching. The edges in red are glued together inside the twisted octahedron giving the Thurston's equation $z_1z_2z_3z_4=1$.} 
\label{fig_decomposition_octahedre} 
\end{figure} 

 Fix $\chi_1$ the character over the center of $\mathcal{Z}_{\zeta}(\mathbb{D}_2)$ induced by $V_1\otimes V_2$ and consider 
 $$\mathscr{R}:= \tau^* \circ \Phi_{(\tau(\Delta_2), \chi_1\circ \tau), (\Delta_2, \chi_1)}^{\mathbb{D}_2} :  \quotient{ \widehat{\mathcal{Z}}_{\zeta}(\mathbb{D}_2)}{I_{\chi_1}} \xrightarrow{\cong}   \quotient{ \widehat{\mathcal{Z}}_{\zeta}(\mathbb{D}_2)}{I_{\tau \cdot \chi_1}}.$$
 
  By definition, a Kashaev braiding is an isomorphism $R : V_1\otimes V_2 \xrightarrow{\cong} V_3\otimes V_4$ such that $\mathscr{R}=\mathrm{Ad}(R)$.
The condition for $\chi_1$ to be $4T$ triangulable is equivalent to the possibility of decomposing the automorphism $\mathscr{R}$ as:
$$\mathscr{R}= \tau^* \circ \Phi^{\mathbb{D}_2}_{(\tau(\Delta_2), \tau\cdot \chi_1), (\Delta^{(5)}, \chi_5)} \circ  \Phi^{\mathbb{D}_2}_{(\Delta^{(5)}, \chi_5), (\Delta^{(4)}, \chi_4)} \circ  \Phi^{\mathbb{D}_2}_{(\Delta^{(4)},\chi_4), (\Delta^{(3)}, \chi_3)} \circ  \Phi^{\mathbb{D}_2}_{(\Delta^{(3)}, \chi_3),( \Delta^{(2)}, \chi_2)}  \circ  \Phi^{\mathbb{D}_2}_{(\Delta^{(2)}, \chi_2), (\Delta_2, \chi_1)}.$$

Fix $w_1', w_2'', w_3'', w_4' \in \mathbb{C}^*$ some $N$-th roots of $z_1', z_2'',z_3''$ and $z_4'$ respectively. 
By Lemma \ref{lemma_decomp_flip}, for $i\in \{1, \ldots, 4\}$, one has $ \Phi^{\mathbb{D}_2}_{(\Delta^{(i+1)}, \chi_{i+1}), (\Delta^{(i)}, \chi_i)}= \nu_{w_i^{'(')}} \circ \mathrm{Ad}(\Phi_{w_i^{'(')}}(X_{n_i}^{-1}))$.   Set
$$ \mathcal{K}:= \tau^* \circ  \Phi^{\mathbb{D}_2}_{(\tau(\Delta_2), \tau\cdot \chi_1), (\Delta^{(5)}, \chi_5)} \circ \nu_{w'_4} \circ \nu_{w''_3} \circ \nu_{w''_2} \circ \nu_{w'_1} : \quotient{ \widehat{\mathcal{Z}}_{\zeta}(\mathbb{D}_2)}{I_{\chi_1}} \xrightarrow{\cong}   \quotient{ \widehat{\mathcal{Z}}_{\zeta}(\mathbb{D}_2)}{I_{\tau \cdot \chi_1}}.$$
One has:
\begin{multline*}
\mathscr{R}= \mathcal{K} \circ \mathrm{Ad}(\Phi_{w'_4}\left( (\nu_{w''_3}\circ \nu_{w''_2}\circ\nu_{w'_1})^{-1}(X_4^{-1})\right)) \circ  \mathrm{Ad}(\Phi_{w''_3}\left(  (\nu_{w''_2}\circ\nu_{w'_1})^{-1}(X_3^{-1})\right)) \\
 \circ \mathrm{Ad}(\Phi_{w''_2}\left( \nu_{w'_1}^{-1}(X_5^{-1})\right)) \circ  \mathrm{Ad}(\Phi_{w'_1}\left( X_4^{-1}\right)).
 \end{multline*}
By a straightforward computation using Definition \ref{def_nu} (see Table \ref{table_k} below), we find:
\begin{equation}\label{eq_Rexpression1}
\mathscr{R}= \mathcal{K} \circ \mathrm{Ad}(\Phi_{w'_4}\left( {w'_1}^{2} w''_2 w''_3 [X_3X_4X_5]^{-1} \right)) \circ \mathrm{Ad}(\Phi_{w''_3}\left( w'_1 [X_3X_4]^{-1} \right))\circ \mathrm{Ad}(\Phi_{w''_2}\left( w'_1[X_4X_5]^{-1} \right))
\circ \mathrm{Ad}(\Phi_{w'_1}\left( X_4^{-1} \right)).
\end{equation}

The isomorphism $\mathcal{K}$ is described by Table \ref{table_k}:
\begin{table}[h!]
   
     \centering
  \begin{adjustbox}{max width=\textwidth}
\begin{tabular}{|c|c|c|c|c|c|}
\hline
$\mathcal{Z}_{\zeta}(\mathbb{D}_2, \Delta^{(1)})$ & $\mathcal{Z}_{\zeta}(\mathbb{D}_2, \Delta^{(2)})$ & $\mathcal{Z}_{\zeta}(\mathbb{D}_2, \Delta^{(3)})$ & $\mathcal{Z}_{\zeta}(\mathbb{D}_2, \Delta^{(4)})$ & $\mathcal{Z}_{\zeta}(\mathbb{D}_2, \Delta^{(5)})$ & $\mathcal{Z}_{\zeta}(\mathbb{D}_2, \Delta^{(1)})$ \\
\hline
$X_1$ & $X_1$ & $X_1$ & $w''_3[X_1X_3]$ & $w''_3w'_4[X_1X_3X_4]$ & $w''_3w'_4[X_1X_2X_4]$ \\
\hline
$X_2$ & $(w'_1)^{-1}X_2$ & $(w'_1w''_2)^{-1}X_2$ & $(w'_1w''_2w''_3)^{-1}X_2$ & $(w'_1w''_2w''_3w'_4)^{-1}X_2$ & $(w'_1w''_2w''_3w'_4)^{-1}X_5$ \\
\hline
$X_3$ & $w'_1[X_3X_4]$ & $(w'_1w''_2)[X_3X_4X_5]$ & $(w'_1w''_2w''_3)[X_4X_5]$ & $(w'_1w''_2w''_3w'_4)X_5$ & $(w'_1w''_2w''_3w'_4)X_6$ \\
\hline
$X_4$ & $X_4^{-1}$ & $(w''_2)^{-1}[X_4X_5]^{-1}$ & $(w''_2w''_3)^{-1}[X_3X_4X_5]^{-1}$ & $(w''_2w''_3{w'_4}^2)^{-1}[X_3X_4X_5]^{-1}$ & $(w''_2w''_3{w'_4}^2)^{-1}[X_2X_4X_6]^{-1}$ \\
\hline
$X_5$ & $w'_1[X_4X_5]$ & $(w'_1w''_2)X_4$ & $(w'_1w''_2w''_3)[X_3X_4]$ & $(w'_1w''_2w''_3w'_4)X_3$ & $(w'_1w''_2w''_3w'_4)X_2$ \\
\hline
$X_6$ & $(w'_1)^{-1}X_6$ & $(w'_1w''_2)^{-1}X_6$ &$(w'_1w''_2w''_3)^{-1}X_6$ &$(w'_1w''_2w''_3w'_4)^{-1}X_6$ & $(w'_1w''_2w''_3w'_4)^{-1}X_3$ \\
\hline
$X_7$ & $X_7$ & $w''_2[X_5X_7]$ & $w''_2[X_5X_7]$ & $(w''_2w'_4)[X_4X_5X_7]$ & $(w''_2w'_4)[X_4X_6X_7]$ \\
\hline
$H_{p_1}=[Z_2Z_3]$ & $[Z_2Z_3Z_4]$ & $[Z_2Z_3Z_4Z_5]$ & $[Z_2Z_4Z_5]$ & $[Z_2Z_5]$ & $H_{p_2}$ \\
\hline
$H_{p_2}=[Z_5Z_6]$ & $[Z_4Z_5Z_6]$ & $[Z_4Z_6]$ & $[Z_3Z_4Z_6]$ & $[Z_3Z_6]$ &$H_{p_1}$ \\
\hline

\end{tabular}
\end{adjustbox}
 \label{table_k}
 \caption{Computation of $\mathcal{K}$}
\end{table}
 In this table, for $1\leq i \leq 4$, one passes from the $i$-th column to the $i+1$-th column by composing with $\nu_{w_i^{'(')}}$ and one passes from the $5$-th column to the last one by composing with $ \tau^* \circ  \Phi^{\mathbb{D}_2}_{(\tau(\lambda_2), \tau\cdot \chi_1), (\lambda^{(5)}, \chi_5)}$. Thus one passes from the first column to the last one by composing with $\mathcal{K}$.  In order to reorder the parameters used to define Kashaev's braidings, we now introduce a class of simple $\mathcal{Z}_{\zeta}(\mathbb{D}_1)$-modules named standard.
 
 \begin{definition}
 \begin{enumerate}
 \item Let $V=\mathbb{C}^N$ and $\{ v_i, i \in \mathbb{Z}/N\mathbb{Z} \}$  its canonical basis. We define the operators $D, A \in \End(V)$ by the formulas:
 $$ D v_i= q^i v_i, \quad Av_i=v_{i-1}, \quad \forall i \in \mathbb{Z}/N\mathbb{Z}.$$
 They satisfy $AD=qDA$ and $A^N=D^N=\id$. 
 \item  Let $(\widetilde{x_1}, \lambda, h_p, h_{\partial}) \in (\mathbb{C}^*)^4$ and define a simple $\mathcal{Z}_{\zeta}(\mathbb{D}_1, \Delta_1)$-module $W(\widetilde{x_1}, \lambda, h_p, h_{\partial})$ by the formulas: 
$$ \rho(H_p)= h_p \id, \quad \rho(H_{\partial})=h_{\partial} \id, \quad \rho(X_1)= \widetilde{x_1} A, \quad \rho([Z_1Z_3Z_4])= \lambda^{-1}D.$$
We call \textbf{standard} such a simple $\mathcal{Z}_{\zeta}(\mathbb{D}_1, \Delta_1)$-module. By Theorem \ref{theorem_CF_Rep}, every simple $\mathcal{Z}_{\zeta}(\mathbb{D}_1)$-module is isomorphic to a standard one and two standards are isomorphic if and only if they have the same parameters except for $\widetilde{x_1}$ which can be multiplied by a $q^i$. 
\item Let $(V_1,V_2)$ be some $4T$ triangulable standard modules of the form $V_i=W(x_1^{(i)}, \lambda_i, h_p^{(i)}, h_{\partial}^{(i)})$. Write
$$ \widetilde{x_4}:= \frac{h_{\partial}^{(1)} \widetilde{x_1}^{(2)}}{h_p^{(1)} \widetilde{x_1}^{(1)}}, \quad \widetilde{x_3}:=\frac{h_p^{(1)}}{h_{\partial}^{(1)} \lambda_1^2}, \quad \widetilde{x_5}:= \lambda_2^2 h_p^{(2)}h_{\partial}^{(2)}.$$
Note that for $i=3,4,5$, one has $\widetilde{x_i}^N=x_i$. A \textbf{system of quantum shape parameters for} $(V_1,V_2)$ is the data of quantum shape parameters $\mathbf{w}_1, \ldots, \mathbf{w}_4$ such that $w_1',w_2'',w_3'',w_4'$ are $N$-th roots of $z_1', z_2'',z_3'',z_4'$ and such that
\begin{align}
\label{E1} & w_1=-\widetilde{x_4}^{-1}; \\
\label{E2} & w_2'= -w'_1\widetilde{x_4}^{-1} \widetilde{x_5}^{-1}; \\
\label{E3} & w_3' = -w'_1 \widetilde{x_4}^{-1} \widetilde{x_3}^{-1}; \\
\label{E4} &w_4 = -{w'_1}^2 w''_2 w''_3 \widetilde{x_3}^{-1}\widetilde{x_4}^{-1}\widetilde{x_5}^{-1}.
\end{align}
Note that $w_1,w_2',w_3',w_4$ are $N$-th roots of $z_1,z_2',z_3',z_4$ so the $\mathbf{w}_i$ define some q-flattenings of the $\mathbf{z}_i$. Note also that they satisfy the following quantum Thurston equation:
 $$ w_1w_2w_3w_4= q^{-2}.$$

\end{enumerate}
\end{definition}
 
 When $(V_1,V_2)$ are $4T$ triangulable standard representations with quantum shape parameters $\mathbf{w}_i$, Equation \eqref{eq_Rexpression1} takes the form:
 
 \begin{equation}\label{eq_Rexpression2}
\mathscr{R}= \mathcal{K} \circ \mathrm{Ad} \left[ \Phi_{w'_4}( -w_4[D^{-2}A] \otimes [A^{-1}D^2]) \Phi_{w''_3}(-w_3'[D^{-2}A] \otimes A^{-1})\Phi_{w''_2}(-w_2'A \otimes [A^{-1}D^2]) \Phi_{w'_1}( -w_1A \otimes A^{-1}) \right].
\end{equation}

\vspace{2mm}
 \textbf{Intertwiners for $\mathcal{K}$ }
\vspace{2mm}
\par Fix $(V_1,V_2)$ and $(V_3, V_4)$ some $4T$ triangulable $\tau$ compatible pairs of standard modules and denote by $\pi_i$ the character induced by $V_i=W(\widetilde{x_1}^{(i)}, \lambda_i, h_p^{(i)}, h_{\partial}^{(i)})$ on the center of $\mathcal{Z}_{\zeta}(\mathbb{D}_1)$.  In view of the decomposition \eqref{eq_Rexpression2} and in order to provide an explicit formula for a Kashaev's braiding $R:V_1\otimes V_2\to V_3 \otimes V_4$, it remains to find an intertwiner for $\mathcal{K}$.  We will make the (non restrictive) assumption that
\begin{equation}\label{eq_convention}
h_{\partial}^{(1)}=h_{\partial}^{(4)} \quad \mbox{and} \quad h_{\partial}^{(2)} = h_{\partial}^{(3)}.
\end{equation}
Note that the assumption of being $\tau$-compatible implies $h_p^{(2)}=h_p^{(3)}$ and $h_p^{(1)}=h_p^{(4)}$. We choose a system of quantum shape parameters $\mathbf{w}_i$ and fix a square root $(w'_1w''_2w''_3w'_4)^{1/2}$ of $w'_1w''_2w''_3w'_4$,  such that
 \begin{align}
 \label{E5}& w''_3w'_4=  \frac{\widetilde{x_1}^{(1)}}{\widetilde{x_1}^{(4)}} (\lambda_3 h_{\partial}^{(3)})^{-2} \\
 \label{E6}& w''_2w'_4=  \frac{\widetilde{x_1}^{(3)}}{\widetilde{x_1}^{(2)}} \lambda_4^{2}  \\
 \label{E7}& \frac{\lambda_4}{\lambda_1} = \frac{\lambda_2}{\lambda_3} = (w'_1w''_2w''_3w'_4)^{1/2}, 
 \end{align}
 The data $(\mathbf{w}_1, \mathbf{w}_2, \mathbf{w}_3, \mathbf{w}_4, (w'_1w''_2w''_3w'_4)^{1/2})$ which satisfy Equations \eqref{E1}, \eqref{E2}, \eqref{E3}, \eqref{E4}, \eqref{E5}, \eqref{E6} and \eqref{E7} will be called a \textbf{system of 4T octahedron parameters}.
 \begin{remark}
By inspection of Equations \eqref{E5}, \eqref{E6}, \eqref{E7}, we see that once a $4T$-triangulable pair $(V_1,V_2)$ of standard representations is fixed, there is a $1$-to-$1$ correspondence between the set of $\tau$-compatible standard pairs $(V_3,V_4)$ satisfying \eqref{eq_convention} and a system of 4T octahedron parameters. In particular, $(V_3, V_4)$ are obtained from the formulas:
\begin{align*}
{}&  \lambda_3=(w'_1w''_2w''_3w'_4)^{1/2}\lambda_2,\quad  \lambda_4=(w'_1w''_2w''_3w'_4)^{-1/2}\lambda_1, \quad h_{\partial}^{(3)}=h_{\partial}^{(2)}, h_{\partial}^{(4)}=h_{\partial}^{(1)} \\
{}&\widetilde{x}_1^{(4)}=(h_{\partial}^{(2)})^{-2}\widetilde{x}_1^{(1)}\lambda_2^{-2}(w'_1w''_2w''_3w'_4)^{-1/2}, \quad \widetilde{x}_1^{(3)}=\widetilde{x}_1^{(2)} \lambda_1^{-2} w_1'(w_2'')^2w_3'(w_4')^2.
\end{align*}

\end{remark}

\begin{lemma}
Let $\widehat{\mathcal{K}} : \mathcal{Z}_{\zeta}(\mathbb{D}_1)^{\otimes 2} \xrightarrow{\cong} \mathcal{Z}_{\zeta}(\mathbb{D}_1)^{\otimes 2}$ be the isomorphism defined by 
$$ \widehat{K} (Z^{\mathbf{k}} \otimes Z^{\mathbf{k}'}) := w^{(\mathbf{k}, \mathbf{k}')} \left( Z^{\mathbf{k}'} [Z_1Z_2Z_4]^{k_1-k_4} \otimes [Z_1Z_3Z_4]^{k_4' - k_1'} Z^{\mathbf{k}}\right), $$
where 
$$ w^{(\mathbf{k}, \mathbf{k}')} := (w'_1w''_2w''_3w'_4)^{(k_3-k_2+k'_2-k'_3)/2}(w''_3w'_4)^{(k_1-k_4)/2} (w''_2w'_4)^{(k_4'-k_1')/2}.$$
Note that, since $\mathbf{k}, \mathbf{k}'$ are balanced, the integers $k_1-k_4$ and $k'_4-k'_1$ are even. Then $\widehat{\mathcal{K}}$ is an algebra isomorphism which passes to the quotient to an isomorphism (still denoted by the same letter) 
$$\widehat{\mathcal{K}} : \quotient{\mathcal{Z}_{\zeta}(\mathbb{D}_1)^{\otimes 2}}{(I_{\pi_1}\otimes 1 +1\otimes I_{\pi_2})} \xrightarrow{\cong} \quotient{\mathcal{Z}_{\zeta}(\mathbb{D}_1)^{\otimes 2}}{(I_{\pi_3}\otimes 1 +1\otimes I_{\pi_4})}$$
 such that the following diagram commutes: 
$$ \begin{tikzcd}
\quotient{\mathcal{Z}_{\zeta}(\mathbb{D}_2, \Delta_2)}{I_{\chi_1}} \arrow[r, "\cong"', "\mathcal{K}"] \arrow[d, "\cong"', "\theta"] &
\quotient{\mathcal{Z}_{\zeta}(\mathbb{D}_2, \Delta_2)}{I_{\chi_2}}  \arrow[d, "\cong"', "\theta"] \\
\quotient{\mathcal{Z}_{\zeta}(\mathbb{D}_1)^{\otimes 2}}{(I_{\pi_1}\otimes 1 +1\otimes I_{\pi_2})} \arrow[r, "\cong"', "\widehat{\mathcal{K}}"] &
 \quotient{\mathcal{Z}_{\zeta}(\mathbb{D}_1)^{\otimes 2}}{(I_{\pi_3}\otimes 1 +1\otimes I_{\pi_4})}
 \end{tikzcd}
 $$
 where the isomorphisms $\theta$ are induced by the splitting morphism $\mathcal{Z}_{\zeta}(\mathbb{D}_2) \hookrightarrow \mathcal{Z}_{\zeta}(\mathbb{D}_1)^{\otimes 2}$.

\end{lemma}

\begin{proof} 
The fact that  $\widehat{\mathcal{K}} : \mathcal{Z}_{\zeta}(\mathbb{D}_1)^{\otimes 2} \xrightarrow{\cong} \mathcal{Z}_{\zeta}(\mathbb{D}_1)^{\otimes 2}$ is an algebra isomorphism is an immediate consequence of the definition and of the fact that $w^{(\mathbf{k}+\mathbf{h}, \mathbf{k}'+\mathbf{h}')}=w^{(\mathbf{k}, \mathbf{h})}w^{(\mathbf{k}',\mathbf{h}')}$. To prove the inclusion 
$$ \widehat{\mathcal{K}}( \mathcal{I}_{\pi_1}\otimes 1 + 1\otimes \mathcal{I}_{\pi_2}) \subset (\mathcal{I}_{\pi_3} \otimes 1 +1\otimes \mathcal{I}_{\pi_4}), $$
we need to show that 
$$ \pi_1(x) = (\pi_3 \otimes \pi_4)(\widehat{\mathcal{K}} (x\otimes 1)) \quad \mbox{and} \quad \pi_2(x)=(\pi_3\otimes \pi_4)(\widehat{\mathcal{K}}(1\otimes x)), $$
for each generator $x\in \{X_1^{\pm N}, K^{\pm N/2}, H_p^{\pm 1}, H_{\partial}^{\pm 1} \}$. For $x=H_p^{\pm 1}$, this follows from the equalities $h_p^{(1)}=h_p^{(4)}$ and $h_p^{(2)}=h_p^{(3)}$. For $x=H_p^{\pm 1}$, this follows from the convention  \eqref{eq_convention} that $h_{\partial}^{(1)}=h_{\partial}^{(4)}$ and $h_{\partial}^{(2)} = h_{\partial}^{(3)}$.  For $x= K^{\pm N/2}$, this follows from the equalities $\left( \frac{\lambda_4}{\lambda_1} \right)^N = \left( \frac{\lambda_2}{\lambda_3} \right)^N = (w'_1w''_2w''_3w'_4)^{N/2}$ given by Equation \eqref{E7}. For $x=X_1^{\pm N}$, this follows from Equations \eqref{E5}, \eqref{E6} by elevating both sides of the equations at the power $N$. 

\end{proof}

\begin{lemma}\label{lemma_D}
Let $\tau: V_4\otimes V_3 \to V_3\otimes V_4$ be the flip morphism $\tau(x\otimes y)=y\otimes x$ and $Q: V_1\otimes V_2 \to V_4\otimes V_3$ be the isomorphism defined by $Q (v_i \otimes v_j)=q^{2ij}v_i \otimes v_j$. Then the operator $\tau \circ Q : V_1\otimes V_2 \rightarrow V_3\otimes V_4$ is an intertwiner for $\widehat{\mathcal{K}}$, \textit{i.e.} $\mathrm{Ad}(\tau \circ Q) = \widehat{\mathcal{K}}$.
\end{lemma}

\begin{proof} Let $\sigma : \quotient{\mathcal{Z}_{\zeta}(\mathbb{D}_1)^{\otimes 2}}{(I_{\pi_3}\otimes 1 +1\otimes I_{\pi_4})} \xrightarrow{\cong} \quotient{\mathcal{Z}_{\zeta}(\mathbb{D}_1)^{\otimes 2}}{(I_{\pi_4}\otimes 1 +1\otimes I_{\pi_3})}$ be the flip $\sigma(x\otimes y):= y\otimes x$, so that $\mathrm{Ad}(\tau)=\sigma$. We need to prove the equality 
\begin{equation}\label{eq_alacon}
 \sigma \circ \widehat{\mathcal{K}} (X) = \mathrm{Ad}(Q) (X), 
 \end{equation}
for $X$ a generator of the form $a\otimes 1$ or $1\otimes a$ and $a\in \{X_1^{\pm 1}, [Z_1Z_2Z_4]^{\pm 1}, H_p^{\pm 1}, H_{\partial}^{\pm 1}\}$. When $a=H_p$, this follows from the equalities $h_p^{(1)}=h_p^{(4)}$ and $h_p^{(2)}=h_p^{(3)}$. When $a=H_{\partial}$, this follows from the convention we made in Equation \eqref{eq_convention} that $h_{\partial}^{(1)}=h_{\partial}^{(4)}$ and $h_{\partial}^{(2)} = h_{\partial}^{(3)}$. When $a=K^{1/2}$, on the one hand, we have
$$ \sigma \circ \widehat{\mathcal{K}} (K^{1/2}\otimes 1)= (w'_1w''_2w''_3w'_4)^{-1/2} (K^{1/2}\otimes 1) \quad \mbox{and} \quad \sigma \circ \widehat{\mathcal{K}}(1\otimes K^{1/2})= (w'_1w''_2w''_3w'_4)^{1/2}(1\otimes K^{1/2}).$$
On the other hand
$$ \mathrm{Ad}(Q) (K^{1/2}\otimes 1)= \frac{\lambda_1}{\lambda_4} (K^{1/2}\otimes 1) \quad \mbox{and}\quad \mathrm{Ad}(Q)(1\otimes K^{1/2}) = \frac{\lambda_2}{\lambda_3} (1\otimes K^{1/2}),$$
so Equation \eqref{eq_alacon} follows from Equation \eqref{E7} in that case. Set $d_{i,j}:= q^{2ij}$ so $Qv_i\otimes v_j= d_{i,j}v_i\otimes v_j$. 
When $a=X_1$, on the one hand, we compute
$$ \sigma \circ \widehat{\mathcal{K}} (X_1\otimes 1) \cdot (v_i\otimes v_j) = w''_3w'_4 \widetilde{x}_1^{(4)} (h_{\partial}^{(3)}\lambda_3)^{2} q^{-2j} (v_{i-1}\otimes v_j), \quad
\sigma \circ \widehat{\mathcal{K}} (1\otimes X_1) \cdot (v_i\otimes v_j) = (w''_2w'_4)^{-1} \widetilde{x}_1^{(3)} \lambda_4^{2} q^{-2i} (v_{i}\otimes v_{j-1}).$$
On the other hand
$$\mathrm{Ad}(Q)(X_1\otimes 1) \cdot(v_i \otimes v_j) = \widetilde{x}_1^{(1)} \frac{d_{i-1, j}}{d_{i,j}} (v_{i-1}\otimes v_j), \quad
 \mathrm{Ad}(Q)(1\otimes X_1) \cdot(v_i \otimes v_j) = \widetilde{x}_1^{(2)} \frac{d_{i, j-1}}{d_{i,j}} (v_{i}\otimes v_{j-1}).$$
 So Equation \eqref{eq_alacon} follows from Equations \eqref{E5} and \eqref{E6} in that case.

\end{proof}

The following is the main achievement of this section.

\begin{theorem}\label{theorem_ClosedFormula}
Let $(V_1,V_2)$ be a pair of $4T$ triangulable standard representations and $(\mathbf{w}_1, \mathbf{w}_2, \mathbf{w}_3, \mathbf{w}_4, (w'_1w''_2w''_3w'_4)^{1/2})$ a system of $4T$ octahedron parameters with associated pair $(V_3,V_4)$. Then the isomorphism $R: V_1\otimes V_2 \to V_3\otimes V_4$ defined by
$$ R= \tau \circ Q \circ  \Phi_{w'_4}( -w_4[D^{-2}A] \otimes [A^{-1}D^2]) \Phi_{w''_3}(-w_3'[D^{-2}A] \otimes A^{-1})\Phi_{w''_2}(-w_2'A \otimes [A^{-1}D^2]) \Phi_{w'_1}( -w_1A \otimes A^{-1}) $$ 
is a Kashaev braiding.
\end{theorem}

\begin{proof}
This follows from the decomposition \eqref{eq_Rexpression2} and Lemma \ref{lemma_D}.
\end{proof}

 \subsubsection{Normalization of braiding operators}\label{sec_normalization}

 It remains to normalize braiding operators so that their determinants become equal to $1$. Fix a determination of the logarithm by removing the $]-\infty, 0]$ axis, so we set $x^{\frac{1}{N}}:=\exp(\frac{1}{N}\log(x))$,  and define a function $g$ defined for $x\in \mathbb{C}\setminus \{q^n, n=1, \ldots, N-1\}$ by 
 $$ g(x):= \prod_{i=1}^{N-1}(1-xq^{2i})^{\frac{i}{N}}.$$
 
 \begin{theorem}\label{theorem_NormalizedClosedFormula} Let $R: V_1\otimes V_2 \to V_3 \otimes V_4$ be the Kashaev braiding of Theorem \ref{theorem_ClosedFormula}. Set 
 $$ d:= (-w_4'w_3''w_2''w_1')^{\frac{N-1}{2}}\frac{g((w_4')^{-1})g((w_3)^{-1})g((w_2')^{-1})g((w_1')^{-1})}{g(1)^4}.$$
 Then $dR$ is a normalized Kashaev braiding.
 \end{theorem}
 That $g(1)\neq 0$ is ensured by the equality $\lvert g(1) \rvert= N^{1/2}$ (see  \cite{BaseilhacBenedetti05}).  The following lemma was first stated in \cite[Equation $2.21$]{BazhanovBaxter_StarTriangle} where the proof is left to the reader, and subsequently used by various authors including \cite{BazhanovReshetikhin_QDilog, FadeevKashaevQDilog} which refer to \cite{BazhanovBaxter_StarTriangle} for a proof. The details of the proof can be deduced from the computations in  \cite[Appendix $8.2$]{BaseilhacBenedetti05} as we now detail for the reader convenience.
 
 \begin{lemma}\label{lemma_det_psi}(\cite{BazhanovBaxter_StarTriangle})
 Let $w,w' \in \mathbb{C}$ such that $(w')^N(1-w^N)=1$. Let $x \in \End(V)$ be an operator with characteristic polynomial $(X^N-1)^k$ for some $k\geq 1$ (i.e. the spectrum of $x$ is $\{q^n, n\in \mathbb{Z}/N\mathbb{Z}\}$ and the eigenvalues $q^n$ have all the same multiplicity). Then 
 $$ \det \left( (w')^{\frac{N-1}{2}} \frac{g((w')^{-1})}{g(1)} \Phi_{w'}(-wx) \right) = 1.$$
 \end{lemma}
 
 \begin{proof}
 Let $D:= \prod_{n\in \mathbb{Z}/N\mathbb{Z}} \Phi_{w'}(-wq^{2n})$. We need to prove the equality $D= \left( (w')^{\frac{1-N}{2}}\frac{g(1)}{g((w')^{-1})} \right)^N$.
 Set $\widetilde{g}(x):= \prod_{i=1}^{N-1}(1-xq^{-2i})^{\frac{i}{N}}$ and note that $g(x)\widetilde{g}(x)=\frac{1-x^N}{1-x}$. 
  By Lemma \ref{lemma_special_poly} $(ii)$ we have 
 $\Phi_{w'}(-wq^{2n})=\Phi_{w'}(-w)\prod_{k=1}^n \frac{(w')^{-1}}{1-wq^{2k-1}}$ from which we 
 find 
 \begin{multline*}
 D =  \prod_{n\in \mathbb{Z}/N\mathbb{Z}} \Phi_{w'}(-wq^{2n}) = \frac{ \left( \Phi_{w'}(-w) (w')^{\frac{1-N}{2}}\right)^N }{ \prod_{k=1}^{N-1}(1-wq^{2k-1})^{N-1-k}} 
 = \left( \frac{ \Phi_{w'}(-w) (w')^{\frac{1-N}{2}}}{ \widetilde{g}(wq^{-1})} \right)^N \\
 =  \left(  \Phi_{w'}(-w) (w')^{\frac{1-N}{2}} g(wq^{-1}) \frac{1-wq^{-1}}{1-w^N} \right)^N=(D^{\frac{1}{N}})^N,
 \end{multline*}
 with $D^{\frac{1}{N}}:=    \Phi_{w'}(-w) (w')^{\frac{1+N}{2}} g(wq^{-1}) (1-wq^{-1})$. 
 In \cite[Lemma $8.3(iii)$]{BaseilhacBenedetti05} the following identity is proved (with notations $x=(w')^{-1}$ and $z=wq^{-1}$):
 $$ \Phi_{w'}(-w)\equiv_N (w')^{1-N}\frac{g(1)}{g(wq^{-3})g((w')^{-1})}, $$
 where the notation $\equiv_N$ means is equal up to multiplication by a power of $q$, so 
 $$ D^{\frac{1}{N}}\equiv_N\frac{ (w')^{\frac{3-N}{2}}g(1)}{g((w')^{-1})} \frac{g(wq^{-1})}{g(wq)}(1-wq^{-1}).$$
 By \cite[Lemma $8.2$]{BaseilhacBenedetti05}, we have $\frac{g(x)}{g(q^2x)}(1-x)=(1-x^N)^{\frac{1}{N}}$ so $\frac{g(wq^{-1})}{g(wq)}(1-wq^{-1})=(w')^{-1}$ and 
 $$ D^{\frac{1}{N}}\equiv_N (w')^{\frac{1-N}{2}}\frac{g(1)}{g((w')^{-1})}, \quad \mbox{ so } D= \left( (w')^{\frac{1-N}{2}}\frac{g(1)}{g((w')^{-1})} \right)^N.$$
 This concludes the proof.
 
 \end{proof}
 
 \begin{proof}[Proof of Theorem \ref{theorem_NormalizedClosedFormula}]
 The Theorem follows from Lemma \ref{lemma_det_psi} together with the facts that $\det(Q)=1$ and $\det(\tau)=(-1)^{\frac{N-1}{2}}$. 
 \end{proof}

 

\bibliographystyle{amsalpha}
\bibliography{biblio}

\end{document}